\documentclass[12pt]{article}
\usepackage{amsmath, amsfonts, amssymb,amsthm, amscd,makeidx}

\usepackage{mathtools}
\usepackage{pictexwd,dcpic}
\usepackage{graphicx}
\usepackage{epsf}
\usepackage{epsfig}
\usepackage{color}
\usepackage{bbm}
\usepackage[colorlinks=true]{hyperref}
\usepackage{marginnote}
\usepackage{mathrsfs}
\usepackage{pdfsync}

\newtheorem{theorem}{Theorem}[section]
\newtheorem{lemma}[theorem]{Lemma}
\newtheorem{proposition}[theorem]{Proposition}
\newtheorem{corollary}[theorem]{Corollary}

\theoremstyle{definition}
\newtheorem{definition}[theorem]{Definition}
\newtheorem{example}[theorem]{Example}

\newtheorem*{OQ}{Open Question}

\theoremstyle{remark}
\newtheorem{remark}[theorem]{Remark}

\newtheoremstyle{thm}
  {12pt}
  {12pt}
  {\itshape}
  {\parindent}
  {\scshape}
  {.}
  {5pt}
  {}

\theoremstyle{thm}
\newtheorem*{T6.17}{Theorem \ref{gradientthm1}}
\newtheorem*{T5.5}{Theorem \ref{x-cc}}
\newtheorem*{T5.14}{Theorem \ref{lll}}

\newtheorem*{P3.22}{Proposition \ref{pretzel}}
\newtheorem*{L3.46}{Lemma  \ref{new_lemma_Z}}
\newtheorem*{L3.55}{Lemma  \ref{good_pos}}

\newtheoremstyle{prop}
  {12pt}
  {12pt}
  {\itshape}
  {\parindent}
  {\scshape}
  {.}
  {5pt}
  {}

\theoremstyle{prop}

\newtheoremstyle{lem}
  {12pt}
  {12pt}
  {\itshape}
  {\parindent}
  {\scshape}
  {.}
  {5pt}
  {}

\theoremstyle{lem}
\newtheorem*{L2.6}{Lemma \ref{lem2.4}}

\newtheoremstyle{defn}
  {12pt}
  {12pt}
  {\itshape}
  {\parindent}
  {\scshape}
  {.}
  {5pt}
  {}

\theoremstyle{defn}

\theoremstyle{definition}

\newtheoremstyle{examp}
  {12pt}
  {12pt}
  {}
   {\parindent}
  {\scshape}
  {.}
  {5pt}
  {}

\theoremstyle{examp}

\newtheoremstyle{cor}
  {12pt}
  {12pt}
  {\itshape}
  {\parindent}
  {\scshape}
  {.}
  {5pt}
  {}

\theoremstyle{cor}

\newtheoremstyle{recipe}
  {12pt}
  {12pt}
  {\itshape}
   {\parindent}
  {\scshape}
  {.}
  {5pt}
  {}

\theoremstyle{recipe}

\newtheoremstyle{rem}
  {12pt}
  {12pt}
  {}
   {\parindent}
  {\scshape}
  {.}
  {5pt}
  {}

\theoremstyle{rem}

\newcommand{\bp}{\begin{proof}}
\newcommand{\ep}{\end{proof}}

\providecommand{\norm}[1]{\lVert#1\rVert}
\providecommand{\abs}[1]{\lvert#1\rvert}

\newcommand{\ssc}{\text{sc}}

\renewcommand{\epsilon}{\varepsilon}

\newcommand{\what}{\widehat}
\newcommand{\wh}{\widehat}

\newcommand{\wt}{\widetilde}
\newcommand{\wtilde}{\widetilde}

\newcommand{\gr}{\text{graph}}

\newcommand{\ov}{\overline}

\newcommand{\Fred}{\operatorname{Fred}}

\newcommand{\cl}{\operatorname{cl}}

\newcommand{\codim}{\operatorname{codim}}

\newcommand{\ind}{\operatorname{ind}}

\newcommand{\real}{\operatorname{Re}}

\newcommand{\supp}{\operatorname{supp}}

\newcommand{\R}{{\mathbb R}}

\newcommand{\mo}{\mathcal O}
\providecommand{\ker}[1]{$\text{ker}\ {#1}$}
\providecommand{\im}[1]{$\text{im}\ {#1}$}
\providecommand{\coker}[1]{\text{coker}\ {#1}}

\newcommand{\N}{{\mathbb N}}

{\catcode`"=12 \gdef\hex{"}}

\mathchardef\laplace=\hex0001

\mathchardef\nabla=\hex0272

\makeatletter

\def\@@dalembert#1#2{\setbox0\hbox{$#1\mathrm I$}

  \vrule height\ht0 depth\z@ width.04\ht0

  \rlap{\vrule height\ht0 depth-.96\ht0 width.8\ht0}

  \vrule height.1\ht0 depth\z@ width.8\ht0

  \vrule height\ht0 depth\z@ width.1\ht0 }

\def\dalembert{\mathbin{\mathpalette\@@dalembert{}}\,}

\makeatother

\makeindex

\begin{document}
\title{POLYFOLD AND FREDHOLM THEORY {I}\\
Basic Theory in M-Polyfolds}
\author{Helmut Hofer, Kris Wysocki and Edi  Zehnder}
\date{\today}
\maketitle

\begin{figure}
\centering
\def\svgwidth{33ex}
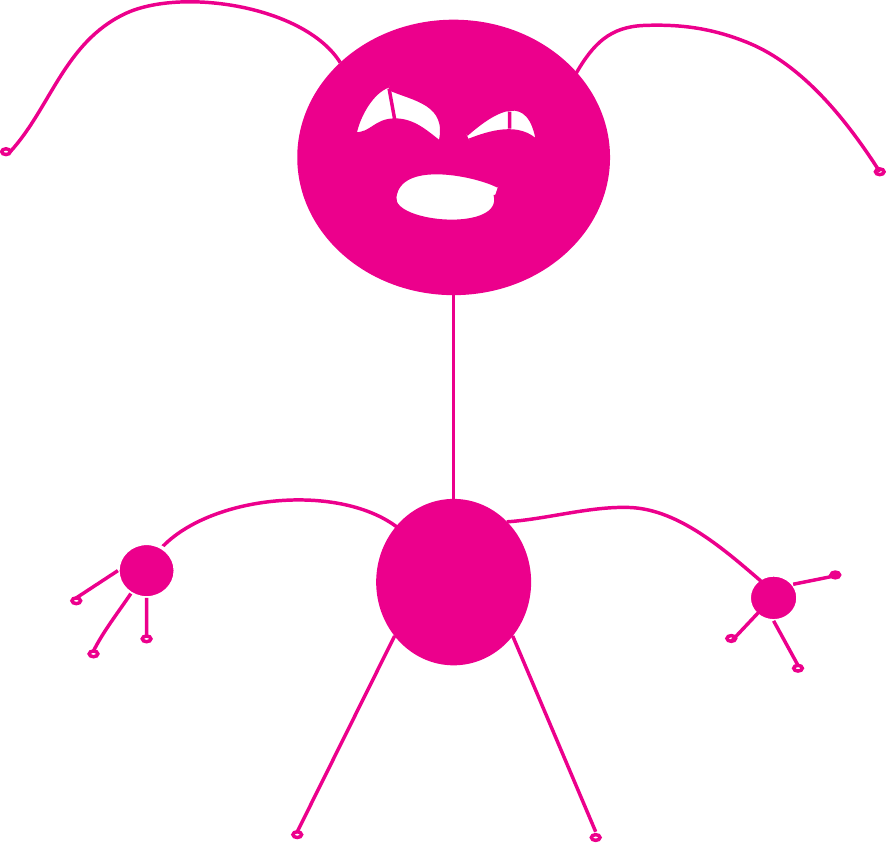
\end{figure}

\pagebreak
\tableofcontents
\pagebreak
\section*{Preface}
The present {\bf Volume I} represents a manual for parts of the polyfold theory, which has been prompted by the analytical problems of the symplectic field theory (SFT), a general theory of symplectic invariants outlined in \cite{EGH}.
The SFT constructs invariants of symplectic cobordisms by analyzing the structure of solutions of elliptic partial differential equations  of Cauchy-Riemann type, from Riemann surfaces into compact symplectic manifolds.  The partial differential equations are defined on varying domains and map into varying targets. The occurring singular limits, like bubbling off phenomena, give rise to serious compactness and transversality problems. Although such nonlinear problems could be approached by ad-hoc methods, the polyfold theory is a systematic and efficient approach providing a language and a large body of results for dealing with nonlinear elliptic equations involving compactness and transversality issues. It requires, however, a generalization of the differential geometry and of the nonlinear Fredholm theory to a class of spaces which are more general than manifolds. These spaces have (locally) varying dimensions and are described locally by retracts on decreasing sequence of Banach spaces, replacing the open sets of Banach spaces in the familiar local description of manifolds.\\

As a guide to the literature these notes provide precise definitions and formulate the important results. They also provide 
many  proofs, but refer otherwise to the papers 
\cite{BEHWZ}, \cite{Hofer}, \cite{HWZ2}-\cite{HWZ5}.
The main topic in the current {\bf Volume I} is the development of all aspects of the Fredholm theory in a class
of spaces called M-polyfolds. These spaces can be viewed as a generalization of the notion of a manifold (finite or infinite dimensional).
In  {\bf Volume {II}} the ideas will be generalized  to the more general class of spaces
called polyfolds, which is needed in the more advanced applications.  {\bf Volume {I} and {II}}
describe a wide array of nonlinear functionalanalytic tools to study perturbation and transversality questions
for a large class of so-called Fredholm sections. In a planned  {\bf Volume {III}}\footnote{Volume {III} will with all likelihood absorb our  unpublished manuscript \cite{HWZ-DM}} we intend to describe a wide array of nonlinear analysis 
tools to construct polyfolds in  applications. In particular the set-up will emphasize recyclability of the analysis
and we shall put forward a novel theory (`Black Boxes') which allows to recycle different pieces of analysis 
and guarantees that these work together, provided they have some easy to check properties.
Other approaches do not address the issue of  recyclability. The latter is a serious issue since, as is well-known,
currently the proofs in symplectic geometry are very long and usually not very transparent, so that mistakes
very often go undetected.

The series of videos \cite{H_Video}, \cite{A_Video},\cite{Wehrheim_video},  illustrate the motivation behind the polyfold theory. We also recommend \cite{FFGW} for the intuitive ideas involved in the polyfold theory.\\

Other approaches to the type of problems considered in  applications were put forward in \cite{FO}, \cite{FOOO},\cite{LiuT},\cite{LT},\cite{McW},  \cite{Yang}, and \cite{Pardon}.\\

\noindent {\bf Acknowledgement:} 
The first author was  partially supported by the NSF grants DMS-0603957 and DMS-1104470, the second author  was partially supported by  the NSF grant DMS-0906280.  The second and the third authors would like to thank the Institute for Advanced Study (IAS) in Princeton for the support and hospitality, the second author  would like to thank the Forschungs Institut f{\"u}r Mathematik (FIM) in Zurich for the support and hospitality.

The authors would like to thank Joel Fish and Katrin Wehrheim for many useful and enlightening discussions.
The first author would like to thank  the participants of the workshop on polyfolds at Pajaro Dunes in August 2012 for their valuable feedback.

\pagebreak
\section{Sc-Calculus}
The basic concept is the  sc-structure on a Banach space.

\subsection{Sc-Structures and Differentiability}
We begin with the  linear sc-theory.
\begin{definition}\label{sc-structure}
A {\bf sc-structure}  \index{D- Sc-Structure} (or scale structure) on a  Banach space $E$ consists of 
a decreasing sequence $(E_m)_{m\geq 0}$ of Banach spaces,   
$$
E=E_0\supset E_1\supset E_2\supset \ldots  ,$$
such that  the following two conditions are satisfied, 
\begin{itemize}
\item[(1)] The inclusion operators $E_{m+1}\rightarrow E_m$ are  compact.
\item[(2)] The intersection $E_\infty:=\bigcap_{i\geq 0} E_i$\index{$E_i,\ E_{\infty}$} is dense in every $E_m$.
\end{itemize}
\end{definition}

Sc-structures, where ``sc'' is short for scale, are known from  linear interpolation theory, see \cite{Tr}. However, our  interpretation is that of a smooth structure.  In the following, a sc-Banach space or a sc-smooth Banach space stands for a Banach space equipped with a sc-structure.  

A finite-dimensional Banach space $E$  has precisely one sc-structure, namely the constant structure $E_m=E$ for all $m\geq0$.  If $E$ is an infinite-dimensional Banach space, the constant structure is not a sc-structure, because it fails property (1).  We shall see that the sc-structure leads to interesting new phenomena in nonlinear analysis.

Points in $E_\infty$ are called {\bf smooth points}, points in $E_m$ are called points of regularity $m$. A subset $A$ of a sc-Banach space $E$ inherits a filtration  $(A_m)_{m\geq 0}$ defined by $A_m=A\cap E_m$. It is, of course, possible that $A_\infty=\emptyset $.   We adopt the convention that 
$A^k$ stands for the set $A_k$ equipped with the induced filtration  
$$(A^k_m)_{m\geq 0}=(A_{k+m})_{m\geq 0}.$$
The direct sum $E\oplus F$ of sc-Banach spaces is a sc-Banach space, whose sc-smooth structure is defined by $(E\oplus F)_m:=E_m\oplus F_m$ for  all $m\geq 0$.

\begin{example}\label{sc-example1}
An example of a sc-Banach space, which is relevant in our applications, is as follows.  We choose a strictly increasing sequence $(\delta_m)_{m\geq 0}$  of real numbers starting with $\delta_0=0$.  
We consider the Banach spaces  $E=L^2(\R\times S^1)$ and  $E_m=H^{m,\delta_m}(\R\times S^1)$, where  the space  $H^{m,\delta_m}(\R\times S^1)$ consists of those  elements in $E$  having
 weak  partial derivatives up to order $m$ which, if  weighted by $e^{\delta_m |s|}$, belong to $E$.  Using Sobolev's compact embedding theorem for bounded domains and the assumption that the sequence $(\delta_m)$ is strictly increasing, one sees that the sequence $(E_m)_{m\geq 0}$ defines a  sc-structure on  $E$. 
\end{example}

\begin{definition}  A linear operator $T\colon 
E\rightarrow F$ between sc-Banach spaces is called {\bf sc-operator},  \index{D- Sc-operator}  if  $T(E_m)\subset F_m$ and the induced operators $T\colon 
E_m\rightarrow F_m$  are continuous for all $m\geq 0$.  A linear {\bf sc-isomorphism}\index{S- Sc-isomorphism}  is a bijective sc-operator whose  inverse is also a sc-operator. 
\end{definition}

A special class of linear sc-operators  are $\ssc^+$-operators,  defined as follows.
\begin{definition}\label{sc_plus_operators}
 A sc-operator $S\colon 
E\to F$ between sc-Banach spaces  is called a {\bf $\ssc^{\pmb{+}}$-operator}\index{D- Sc$^+$-operator},  if 
 $S(E_m)\subset F_{m+1}$  and  $S\colon 
E\rightarrow F^1$ is a sc-operator.
\end{definition}
In view of Definition \ref{sc-structure}, the inclusion operator $F_{m+1}\to F_m$ is compact, implying that a $\ssc^+$-operator $S\colon 
E\to F$ is a sc-compact operator  in the sense that  $S\colon 
E_m\to F_m$ is compact for every $m\geq 0$. Hence, given a sc-operator $T\colon 
E\rightarrow F$ and a sc$^+$-operator $S\colon 
E\rightarrow F$, the 
 operator $T+S$ can be viewed, on every level, as a  perturbation of $T$ by the compact operator $S$.
 
\begin{definition}\label{sc-subspace}
A subspace   $F$ of a sc-Banach space  $E$ is called a {\bf sc-subspace}\index{D- Sc-subspace}, provided $F$ is closed
and the sequence $(F_m)_{m\geq 0}$ given by  $F_m=F\cap E_m$ defines  a sc-structure on $F$.
\end{definition} 

A sc-subspace $F$ of the sc-Banach space $E$ has a sc-complement provided there exists an algebraic complement $G$ of $F$ so that $G_m=E_m\cap G$
defines an sc-structure on $G$, such  that on every level $m$ we have a topological direct sum 
$$E_m=F_m\oplus G_m.$$
Such a splitting $E=F\oplus G$\index{$E\oplus F$} is called a {\bf sc-splitting}\index{Sc-splitting}. 

The following result from \cite{HWZ2}, Proposition 2.7, 
will be used frequently.
\begin{proposition}\label{prop1}\index{P- Finite-dimensional sc-subspaces}
Let $E$ be a sc-Banach space. A finite-dimensional subspace $F$ of $E$  is a sc-subspace  if and only if $F$ belongs to $E_\infty$.
A finite-dimensional sc-subspace always  has a sc-complement.
\end{proposition}
\begin{proof}
If $F\subset E_\infty$, then $F_m:=F\cap E_m=F$ and $F$ is equipped with the constant sc-structure, so that $F$ is a sc-subspace of $E$. Conversely, 
if $F$ is a sc-subspace, then, by definition,  $F_\infty:=\bigcap_{m\geq 0}F_m=F\cap E_\infty$ is dense in $F$.  Consequently, $F=F_\infty\subset E_\infty$, since $F$ and $F_\infty$ are finite dimensional.  

Next,  let $e_1,\ldots , e_k$ be a basis for a finite-dimensional sc-subspace $F$. In view of the above discussion, $e_i\in E_\infty$. By the  Hahn-Banach theorem, the dual basis  can be extended to linear functionals $\lambda_1,\ldots ,\lambda_k$, which are continuous on $E$ and hence on $E_m$ for every $m$. The map $P\colon 
E\to E$,  defined by $P(x)=\sum_{1\leq i\leq k}\lambda_i(x)e_i$,  has its image in $F\subset E_\infty$. It induces a continuous map from $E_m$ to $E_m$,  and since $P\circ P=P$, it is a sc-projection. Introduce the closed subspace $G:=(\mathbbm{1}-P)(E)$ and let $G_m=G\cap E_m$. Then $E=F\oplus G$ and $G_{m+1}\subset G_m$.  The set $G_\infty:=\bigcap_{m\geq 0}G_m=G\cap E_\infty$ is also dense in $G_m$. Indeed, if $g\in G_m$, we find a sequence $f_n\in E_\infty$, such that $f_n\to g$ in $E_m$. Setting $g_n:=(\mathbbm{1}-P)(f_n)\in G_\infty$ we have $g_n\to (\mathbbm{1}-P)(g)=g$, as claimed.  Consequently, $(G_m)_{m\geq 0}$ defines a sc-structure on $G$ and the proof of the proposition is complete.
\end{proof}

Next we describe  the quotient construction in the sc-framework. 
\begin{proposition}\label{thm_quotient}\index{P- Sc-quotients}
Assume that $E$ is a sc-Banach space and $A\subset E$ a sc-subspace. Then $E/A$ equipped with the 
filtration $E_m/A_m$ is a sc-Banach space. Note that $E_m/A_m= \{ (x+ A)\cap E_m \,  \vert \, x\in E\}$.
\end{proposition}
\begin{proof}
By definition of a sc-subspace,  the filtration on $A$ is given by $A_m=A\cap E_m$ and $A_m\subset E_m$ is a closed subspace.
Hence 
$$
E_m/A_m = E_m/(A\cap E_m) = \{ (x + A)\cap E_m\, \vert \, x\in E_m\}.
$$
We identify an element $x+A_m$ in $E_m/A_m$ with the  element $x+A$ of  $E/A$, so that algebraically we can view
$E_m/A_m\subset E/A$.
The inclusion $E_{m+1}\rightarrow E_m$ is compact,  implying that the quotient map  $E_{m+1}\rightarrow E_m/A_m$ is compact.  We claim that the inclusion map 
$E_{m+1}/A_{m+1}\rightarrow E_m/A_m$ is compact. In order to show this, we take a sequence $(x_k+A_{m+1})\subset E_{m+1}/A_{m+1}$ satisfying 
$$
\norm{x_k+A_{m+1}}_{m+1}:=\inf \{ \abs{x+a|_{m+1}}\, \vert \, a\in A_{m+1}\}\leq 1, 
$$
and choose a sequence $(a_{k})$ in $A_{m+1}$ such that  $\abs{x_k+a_k}_{m+1}\leq 2$. 
We may assume, after taking a subsequence,  that $x_k+a_k\rightarrow x\in E_{m}$. 
The image of the sequence $(x_k+A_{m+1})$
under the inclusion $E_{m+1}/A_{m+1}\rightarrow E_m/A_m$ is the sequence $(x_k+a_k+A_m)$. Then
\begin{equation*}
\begin{split}
&\norm{(x_k+a_k+A_{m})-(x+A_m)}_m\\
&\quad =\norm{(x_k+a_k-x)+A_m}_m\\
&\quad =\inf\{\abs{x_k+a_k-x)+a}_m\, \vert \,  a\in A_m\}\\
&\quad\leq \abs{x_k+a_k-x}_m\to  0,
\end{split}
\end{equation*}
showing  that  the inclusion  $E_{m+1}/A_{m+1}\rightarrow E_m/A_m$ is compact. Finally let us show that $\bigcap_{j\geq 0} (E_j/A_j)$ is dense in every $E_m/A_m$.
Let us first note that $(E/A)_\infty:=\bigcap_{j\geq 0} (E_j/A_j)$ consists of all elements of the form $x+A_\infty$ with $x\in E_\infty$.
Since $E_\infty$ is dense in $E_m$,  the image under the continuous quotient map $E_m\rightarrow E_m/A_m$ is dense. This completes the proof of 
Proposition \ref{thm_quotient} 
\end{proof}

A distinguished class of sc-operators is the class of sc-Fredholm operators.
\begin{definition}
A sc-operator $T\colon 
E\rightarrow F$  between sc-Banach spaces is called  {\bf sc-Fredholm}\index{D- Sc-Fredholm operator},  provided there exist sc-splittings $E=K\oplus X$ and $F=C\oplus Y$, having the following properties.
\begin{itemize}
\item[(1)] $K$ is the kernel of $T$ and is finite-dimensional.
\item[(2)] $C$ is finite-dimensional and  $Y$ is the image of $T$.
\item[(3)] $T\colon 
X\rightarrow Y$ is a sc-isomorphism.
\end{itemize}
\end{definition}
In view of Proposition \ref{prop1}  the definition implies that the kernel $K$ consists of smooth points and $T(X_m)=Y_m$ for all $m\geq 0$.  In particular,  we have the topological direct sums
$$
E_m=K\oplus X_m\quad  \text{and}\quad \ F=C\oplus T(E_m).
$$
The index of a sc-Fredholm operator $T$,  denoted by $\ind(T)$, is  as usual defined  by
$$
\ind (T)=\dim(K)-\dim(C).\index{$\ind (T)$}
$$
Sc-Fredholm operators have the following regularizing property.
\begin{proposition}\label{regular}\index{P- Regularizing property}
A sc-Fredholm operator  $T\colon 
E\rightarrow F$ is regularizing, i.e. if $e\in E$ satisfies $T(e)\in F_m$,  then $e\in E_m$.
\end{proposition}
\begin{proof}
By assumption,  $T(e)\in F_m=C\oplus T(X_m)$, so that $T(e)=T(x)+c$ for some $x\in X_m$ and $c\in C$.
Then $T(e-x)=c$ and $e-x\in E_0$. Since $T(E_0)\oplus C=F_0$ it follows that $c=0$ and $e-x\in  K\subset E_\infty$, implying that $x$ and $e$ are on the same level $m$. 
\end{proof}

The following stability result will be crucial later on. The proof, reproduced in Appendix \ref{A1.0},  is taken from \cite{HWZ2},  Proposition 2.1.
\begin{proposition}[Compact Perturbation]\label{prop1.21}\index{P- sc$^+$-perturbations}
Let $E$ and $F$ be sc-Banach spaces. If $T\colon 
E\rightarrow F$ is a
sc-Fredholm operator and $S\colon 
E\rightarrow F$ a sc$^+$-operator, then
$T+S$ is also a  sc-Fredholm operator.
\end{proposition}

The next concept will be crucial later on for the definition of boundaries.

\begin{definition}\label{partial_quadrant}
A {\bf partial quadrant} \index{D- Partial Quadrant}  
in a sc-Banach space $E$  is  a closed convex subset $C$ of $E$, 
such that  there exists a sc-isomorphism $T\colon 
E\rightarrow {\mathbb R}^n\oplus W$ satisfying   $T(C)=[0,\infty)^n\oplus W$.
\end{definition}

We now consider tuples $(U,C,E)$\index{$(U,C,E)$}, in which  $U$ is a relatively open subset of the partial quadrant $C$ in the sc-Banach space $E$.  
\begin{definition}\label{sc_continuous}
If  $(U, C, E)$ and $(U', C', E')$ are two such  tuples, then a  map $f\colon 
U\rightarrow U'$ 
is called {\bf  $\ssc^0$}  (or of {\bf class $\ssc^0$}, or {\bf sc-continuous}),  provided $f(U_m)\subset V_m$ and the induced maps $f\colon 
U_m\rightarrow V_m$ are  continuous for all $m\geq 0$. 
\end{definition}

\begin{definition}\label{sc-tangent}\index{D- Tangent tuple}
The tangent $T(U,C,E)$ of the tuple  $(U,C,E)$ is defined as the tuple 
$$
T(U,C,E)=(TU,TC,TE)\index{$T(U,C,E)$}
 $$where
$$
TU=U^1\oplus E,\quad  TC=C^1\oplus E, \quad  \text{and}\quad  TE=E^1\oplus E.
$$
\end{definition}
Note that $T(U,C,E)$ is a tuple  consisting again  of a relatively open subset $TU$ in the  partial quadrant $TC$ in  the sc-Banach space $TE$.

Sc-differentiability is a  new notion of differentiability in sc-Banach spaces, which is considerably weaker than
the familiar  notion of Fr\'echet differentiability. The new notion of differentiability is the following.

\begin{definition}\label{scx}\index{D- Sc-differentiability}
We consider  two tuples $(U,C,E)$ and $(U',C',E')$  and a map $f\colon U\to U'$. The map $f$ is called $\pmb{\ssc^{1}}$ (or of {\bf class} $\pmb{\ssc^{1}}$)  provided
the following conditions are satisfied.
\begin{itemize}
\item[(1)]  The map $f$ is $\ssc^0$.
\item[(2)] For every $x\in U_1$ there exists a bounded linear operator $Df(x)\colon E_0\to F_0$ such  that
for $h\in E_1$ satisfying  $x+h\in U_1$,
$$
\lim_{\abs{h}_1\rightarrow 0} \frac{ \abs{f(x+h)-f(x)-Df(x)h}_0}{\abs{h}_1}=0.
$$
\item[(3)] The map $Tf\colon TU\to TU'$,  defined by 
$$Tf(x,h)=(f(x),Df(x)h), \quad \text{$x\in U^1$ and $h\in E$}, $$
 is a  $\ssc^0$-map.
The map $Tf\colon TU\to TU'$ is called the {\bf tangent map} of $f$.
\end{itemize}

\end{definition}

In general,  the map $U_1\to   \mathscr{L}(E_0,F_0)$, defined by $x\to  Df(x)$,  will not(!) be continuous, if the space of bounded linear operators is equipped with the operator norm. However,  if we equip it with the compact open topology it will be continuous. The $\ssc^1$-maps between finite dimensional Banach spaces are the familiar $C^1$-maps.

Proceeding inductively,  we define  what it means for  the map $f$  to be  $\ssc^k$ or $\ssc^\infty$. Namely, a $\ssc^0$--map $f$ is said to be a $\ssc^2$--map, if  it is $\ssc^1$ and if  its tangent map $Tf\colon 
TU\to TV$ is  $\ssc^1$. By Definition \ref{scx} and Definition \ref{sc-tangent},  the  tangent map of $Tf$,
$$T^2f\colon 
=T(Tf)\colon 
T^2(U)=T(TU)\to T^2(V)=T(TV),$$
is of class $\ssc^0$.
If the tangent map $T^2f$  is $\ssc^1$, then $f$ is said to be $\ssc^3$,  and so on. The map $f$ is 
{\bf $\ssc^{\pmb{\infty}}$} or {\bf $\ssc$-smooth}, if it is $\ssc^k$ for all $k\geq 0$.
\begin{remark}
The above consideration can be generalized. Instead of taking a partial quadrant $C$ one might take a closed convex partial cone $P$ with nonempty interior. This means $P\subset E$ is a closed subset, so that for a real number $\lambda\geq 0$ it holds $\lambda\cdot P\subset P$. Moreover $P$ is convex and has a nonempty interior. If then $U$ is an open subset of $P$ we can define the notion of being sc$^1$ as in the previous definition. The tangent map $Tf$ is then defined on $TP=P^1\oplus E$. We note that $TP$ is a closed cone with nonempty interior. In this context one should be able to deal with 
sc-Fedholm problems, on M-polyfolds with `polytopal' boundaries and even more general situations.
Many of the results in this book can be carried over to this generality as well, but one should check carefully which arguments carry over.
\end{remark}

\subsection{Properties of Sc-Differentiability}

In this section we  shall discuss the relationship between the classical smoothness in the Fr\'echet sense and the sc-smoothness. The proofs of the following results can be found  in  \cite{HWZ8.7}.

\begin{proposition}[{\bf Proposition 2.1, \cite{HWZ8.7}}] \label{x1}\index{P- Sc-differentiability}
Let $U$ be a relatively open subset of a partial  quadrant in  a sc-Banach space $E$ and let $F$ be another $\ssc$-Banach space. 
Then a  $\ssc^0$-map $f\colon 
U\to F$ is of class $\ssc^1$  if and only if  the following conditions hold true.
\begin{itemize}\label{sc-100}
\item[\em(1)] For every $m\geq 1$,  the induced map
$$
f\colon 
U_m\to  F_{m-1}
$$
is of class $C^1$. In particular,  the derivative  
$$
df\colon 
 U_m\to  \mathscr{L}(E_m,F_{m-1}),\quad x\mapsto df(x)
$$
is a  continuous map.
\item[\em(2)] For every  $m\geq 1$ and every $x\in U_m$,  the bounded  linear operator
$df(x)\colon 
 E_m\to F_{m-1}$ has an extension to a bounded  linear operator $Df(x)\colon 
E_{m-1}\to F_{m-1}$. In addition, the map
\begin{equation*}
U_m\oplus  E_{m-1}\to  F_{m-1}, \quad  (x,h)\mapsto  Df(x)h
\end{equation*}
is continuous.
\end{itemize}
\end{proposition}

In particular, if $x\in U_\infty$ is  a smooth point in $U$  and $f\colon 
U\to F$ is a $\ssc^1$-map, then the linearization 
$$Df(x)\colon E\to F$$
is a sc-operator.

A  consequence of Proposition \ref{x1} is the following result about lifting the indices.
\begin{proposition}[{\bf Proposition 2.2, \cite{HWZ8.7}}] \label{sc_up}\index{P- Sc-differentiability under lifts}
Let  $U$  and $V$ be  relatively open subsets of partial quadrants in sc-Banach spaces, and let $f\colon 
U\rightarrow V$ be $\ssc^k$. 
Then $f\colon 
U^1\rightarrow V^1$ is also $\ssc^k$.
\end{proposition}

\begin{proposition}[{\bf Proposition 2.3, \cite{HWZ8.7}}] \label{lower}\index{P- Sc-smoothness versus $C^k$}
Let $U$ and $V$ be relatively open subsets of partial quadrants in sc-Banach spaces.
If $f\colon 
U\to  V$ is $\ssc^k$, then for every $m\geq 0$,  the map $f\colon 
U_{m+k}\to  V_m$ is of class $C^k$.  Moreover,  $f\colon 
U_{m+l}\to  V_m$ is of class $C^l$ for every $0\leq l\leq k$.
\end{proposition}
The next result is very useful in proving that a given map between sc-Banach spaces is sc-smooth.
\begin{proposition}[{\bf Proposition 2.4, \cite{HWZ8.7}}] \label{ABC-x}
Let $U$ be a relatively open subset of a partial  quadrant in  a sc-Banach space $E$ and let $F$ be another $\ssc$-Banach space. 
Assume that for every $m\geq 0$ and $0\leq l\leq k$,  the map $f\colon 
U\rightarrow V$  induces  a map
$$
f\colon 
U_{m+l}\rightarrow  F_m,
$$
which is of  class $C^{l+1}$. Then $f$ is $\ssc^{k+1}.$
\end{proposition}
In the case that the target  space $F= \R^N$, Proposition \ref{ABC-x} takes the following form.
\begin{corollary}[{\bf Corollary 2.5, \cite{HWZ8.7}}] \label{ABC-y}
Let $U$ be a  relatively open subset of a partial quadrant in a sc-Banach space and  $f\colon 
U\to  \R^N$. If 
for some $k$ and all $0\leq l \leq k$ the map $f\colon 
U_l \to  \R^N$ belongs to  $C^{l +1}$, then $f$ is $\ssc^{k+1}$.
\end{corollary}


\subsection{The Chain Rule and Boundary Recognition}\label{subsection_boundary_recognition}
The cornerstone of the sc-calculus, the chain rule, holds true.
\begin{theorem}[{\bf Chain Rule}]\label{sccomp}\index{T- Chain rule}
Let $U\subset C\subset E$ and $V\subset D\subset F$ and $W\subset Q\subset G$ be relatively open subsets
of partial quadrants in sc-Banach spaces and let $f\colon 
U\rightarrow V$ and $g\colon 
V\rightarrow W$ 
be $\ssc^1$-maps.
Then the composition $g\circ f\colon 
U\rightarrow W$ is also $\ssc^1$ and
$$
T(g\circ f) =(Tg)\circ (Tf).
$$
\end{theorem}
The result is  proved in \cite{HWZ2} as  Theorem 2.16 for open sets $U$, $V$ and $W$.
We give the adaptation to the somewhat more general setting here. The proof, which is given in Appendix \ref{A1.1},  is very close 
to the one in \cite{HWZ2} and does not require any new ideas.

\begin{remark}
The result is somewhat surprising, since differentiability is only guaranteed
under the loss of one level $U_1\rightarrow F_0$, so that one might expect
 for a composition a loss of two levels. However, one is saved
by the compactness of the embeddings $E_{m+1}\rightarrow E_m$.
\end{remark}

The next basic result concerns  the boundary recognition. Let $C$ be a partial quadrant in the  sc-Banach space  $E$. 
We choose a linear $\ssc$-isomorphism $T\colon 
E\to  \R^k\oplus W$ satisfying $T(C)=[0,\infty)^k\oplus W$. 
If $x\in C$,  then 
$$T(x)=(a_1, \ldots, a_k, w)\in [0,\infty )^k\oplus W,$$
and we define the integer $d_C(x)$ by 
$$
d_C(x)\colon 
=\#\{i\in \{1,\ldots, k\}\vert \, a_i=0\}.$$
\begin{definition}\label{degeneracy_index_1}\index{D- Degeneracy index}
The map $d_C\colon 
C\rightarrow {\mathbb N}_0$ is called the {\bf degeneracy index}.
\end{definition}

Points $x\in C$ satisfying $d_C(x)=0$\index{$d_C$} are interior points of $C$, the  points satisfying  $d_C(x)=1$ are honest boundary points, and the  points with $d_C(x)\geq 2$ are corner points. The size of the index gives the complexity of the corner.

It is not difficult to see that this definition is independent of  the choice of a sc-linear isomorphism $T$.
\begin{lemma}\label{corner_recognition_linear}
The  map  $d_C$  does not depend on the choice of a linear sc-isomorphism $T$.
\end{lemma}
The proof is given in Appendix \ref{A1.2}.

\begin{theorem}\label{ppp22}\index{T- Corner recognition}
Given the tuples $(U, C, E)$ and $(U', C', E')$ and a germ  $f\colon 
(U,x)\rightarrow (U',x')$ of a local sc-diffeomorphism satisfying $x'=f(x)$, then
$$d_C(x)=d_{C'}(f(x)).$$
\end{theorem}

The proof of a more general result is given in \cite{HWZ2}, Theorem 1.19. 
We  note that sc-smooth diffeomorphisms recognize  boundary points and corners, whereas
homeomorphisms recognize only  boundaries, but no corners.
\subsection{Appendix}
\subsubsection{Proof of the sc-Fredholm Stability Result}\label{A1.0}
\begin{proof}[Proof of Proposition \ref{prop1.21}]
Since $S\colon 
E_m\rightarrow F_m$ is compact for every level,  the operator 
$T+S\colon 
E_m\rightarrow F_m$ is Fredholm for every $m$. Denoting  by $K_m$ the
kernel of this operator,  we claim that $K_{m}=K_{m+1}$ for every $m\geq 0$. Clearly, $K_{m+1}\subseteq K_{m}$. If  $x\in K_m$,  then  $Tx=-Sx\in F_{m+1}$ and hence, by Proposition \ref{regular},
$x\in E_{m+1}$. Thus $x\in K_{m+1}$, implying  $K_m\subseteq K_{m+1}$, so that  $K_m=K_{m+1}$. Set $K=K_0$.  By
Proposition \ref{prop1}, $K$ splits the sc-space $E$,  since it is a
finite dimensional subset of $E_{\infty}$. Therefore, we have the
sc-splitting $E=K\oplus X$  for a suitable sc-subspace  $X$ of $E$.

Next define $Y=(T+S)(E)=(T+S)(X)$ and $Y_m\colon 
=Y\cap F_m$ for $m\geq 0$.  Since $T+S\colon 
E\to F$ is Fredholm,  its  image $Y$ is closed.  We claim, that the set  $Y_\infty$, defined by $Y_\infty:=\bigcap_{m\geq 0}Y_m=Y\cap F_\infty$,  is dense in  $Y_m$ for every $m\geq 0$. Indeed, if $y\in Y_m$, then $y=(T+S)(e)$ for some $e\in E$. Since  by Proposition \ref{regular}, $T$ is regularizing, we conclude that $e\in E_m$. Since $E_\infty$ is dense in $E_m$,  there exists a sequence $(e_n)\subset E_\infty$ such that $e_n\to e$ in $E_m$. Then, setting $y_n:=(T+S)(e_n)\in Y_\infty$, we conclude that $y_n\to (T+S)(e)=y$, which shows that $Y_\infty$ is dense in $Y_m$.  Finally, we consider the projection map $p\colon 
F\to  F/Y$. Since $p$  is a continuous surjection onto a finite-dimensional space and $F_\infty$ is dense in $F$, it follows that $\ov{p(F_\infty)}=F/Y$.  Hence, we can choose  a basis in $F/Y$ whose representatives belong to $F_\infty$.  Denoting by  $C$ the span of these representatives, we obtain  $F_m=Y_m\oplus C$.   Consequently, we have a sc-splitting $F=Y\oplus C$ and the proof of Proposition \ref{prop1.21} is complete.
\end{proof}

\subsubsection{Proof of the Chain Rule}\label{A1.1}
\begin{proof}[Proof of Theorem \ref{sccomp}]
The maps 
 $f\colon 
U_1\to F$ and $g\colon 
V_1\to G$  are of class $C^1$ in view of Proposition \ref{lower}. Moreover, $Dg(f(x))\circ Df(x)\in \mathscr{L}(E, G)$, if
$x\in U_1$. Fix $x\in U_1$ and 
 $h\in E_1$ sufficiently small satisfying $x+h\in U_1$ and  $f(x+h)\in V_1$.
Then, using  the postulated
properties of $f$ and $g$,
\begin{equation*}
\begin{split}
&g(f(x+h))-g(f(x))-Dg(f(x))\circ Df(x)h\\
&=\int_{0}^{1} Dg(tf(x+h)+(1-t)f(x))\ [f(x+h)-f(x)-Df(x)h] dt\\
&\phantom{=}+\int_{0}^{1}\bigl( [Dg(tf(x+h)+(1-t)f(x))-Dg(f(x))] \circ Df(x)h \bigr)dt.
\end{split}
\end{equation*}
The first integral after dividing by $\abs{h}_1$ takes the form
\begin{equation}\label{integral1}
\int_{0}^{1} Dg(tf(x+h)+(1-t)f(x)) \frac{f(x+h)-f(x)-Df(x)h}{\abs{h}_1}\  dt.
\end{equation}
For  $h\in E_1$ such that  $x+h\in U_1$ and $\abs{h}_1$ is small, we have that $tf(x+h)+(1-t)f(x)\in V_1$. Moreover, the maps
$
[0,1]\rightarrow F_1$ defined by $t\mapsto tf(x+h)+(1-t)f(x)
$
are continuous and converge in $C^0([0,1],F_1)$
to the constant map
$t\mapsto  f(x)$ as  $\abs{h}_1\rightarrow 0$.
Since $f$ is  of class $\ssc^1$, the quotient
 $$
a(h)\colon 
=\frac{ f(x+h)-f(x)-Df(x)h}{\abs{h}_1}$$
converges  to $0$ in $F_0$  as $\abs{h}_1\to 0$. Therefore, by the continuity assumption (3) in Definition \ref{scx},
 the map $$
(t,h)\mapsto  Dg(tf(x+h)+(1-t)f(x)) [a(h)],
$$
as a map from $[0,1]\times E_1$ into $G_{0}$, converges to $0$ as
$\abs{h}_1\to 0$  uniformly in $t$.
Consequently,   the expression in \eqref{integral1}  converges to $0$ in $G_0$ as $\abs{h}_1\to 0$, if  $x+h\in U_1$.
Next we 
consider the integral
\begin{equation}\label{integral2}
\int_{0}^{1} \bigl[  Dg(tf(x+h)+(1-t)f(x))-Dg(f(x))\bigr] \circ Df(x)
\frac{h}{\abs{h}_{1}}\ dt.
\end{equation}
By Definition \ref{sc-structure},  the inclusion operator $E_1\to E_0$ is compact,  so that 
the set of all $\frac{h}{\abs{h}_{1}}\in E_1$ has  a compact closure in $E_{0}$.  Therefore, since $Df(x)\in \mathscr{L}(E_{0}, F_{0})$ is a continuous map  by Definition \ref{scx},  the closure of
the set of all
$$
Df(x) \frac{h}{ \abs{h}_{1}}
$$
is compact  in $F_{0}$. Consequently, again by  Definition \ref{scx}, every sequence $(h_n)\subset E_1$, satisfying $x+h_n\in U_1$ and $\abs{h_n}_1\to 0$, possesses a subsequence having the property that the integrand of the integral in \eqref{integral2} converges to $0$ in $G_{0}$ uniformly in $t$. Hence, the integral \eqref{integral2} also  converges to $0$ in $G_0$, as $\abs{h}_1\to 0$ and  $x+h\in U_1$.  We
have proved  that
$$
\frac{\abs{g(f(x+h))-g(f(x))-Dg(f(x))\circ Df(x)h}_{0}}{\abs{h}_{1}}\rightarrow 0
$$
as $\abs{h}_1\to 0$  and  $x+h\in U_1$.   Consequently,  condition (2) of Definition \ref{scx} is satisfied for the composition $g\circ f$ with the linear operator
$$D(g\circ f)(x)=Dg(f(x))\circ Df(x)\in  \mathscr{L}(E_0, G_0),
$$
where $x\in U_1$. We conclude that the tangent map
$T(g\circ f)\colon 
TU\to TG$,
$$(x, h)\mapsto (\ g\circ f(x), D(g\circ f)(x)h\ )$$
is $\ssc$-continuous and, moreover,
$T(g\circ f)=Tg\circ Tf$. The proof of Theorem  \ref{sccomp} is complete.
\end{proof}

\subsubsection{Proof Lemma \ref{corner_recognition_linear} }\label{A1.2}
\begin{proof}[Proof of Lemma \ref{corner_recognition_linear}]
We take a second sc-isomorphism  $T'\colon 
E\to \R^{k'}\oplus W'$,   satisfying $T'(C)=[0,\infty)^{k'}\oplus W'$ and first show  that $k=k'$. 
We consider the composition $S=T'\circ T^{-1}\colon 
\R^k\oplus W\to \R^{k'}\oplus W'$, where  $S$ is a sc-isomorphism, mapping 
$[0,\infty )^k\oplus W\to [0,\infty )^{k'}\oplus W'$. We claim that  $S(\{0\}^k\oplus W)=\{0\}^{k'}\oplus W'$.  Indeed,  suppose $S(0, w)=(a, w')$ for some $(a, w')\in [0,\infty)^{k'}\oplus W'$.
Then,  we conclude 
$$S(0, tw)= tS(0, w)=t(a, w')=(ta, tw')\in  [0,\infty)^{k'}\oplus W'$$
for all $t\in \R$.  Consequently, $a=0$. Hence, $S(\{0\}^k\oplus W)\subset \{0\}^{k'}\oplus W'$ and since $S$ is an isomorphim, we obtain equality of these sets.  
The set  $\{0\}^k\oplus W\equiv W$ has codimension $k$ in $\R^k\oplus W$, and since $S$ is an isomorphism  we have 
$$S(\R^k\oplus W)=S(\R^k\oplus \{0\})\oplus S(\{0\}^k\oplus W)=S(\R^k\oplus \{0\})\oplus ( \{0\}^{k'}\oplus W'),$$
so that the codimension of $ \{0\}^{k'}\oplus W'$ in $\R^{k'}\oplus W'$ is equal to $k$. On the other hand, this codimension is equal to $k'$. Hence $k'=k$, as claimed.
To prove our result, it now suffices  to show that if $x=(a, w)\in [0,\infty )^k\oplus W$ and $y=S(x)=(a', w')$, then 
\begin{equation}\label{independence_0}
\# I(x)=\# I' (y)
\end{equation}
where $I(x)$ is the set of indices $1\leq i\leq k$ for which $x_i=0$. The set $I'(y)$ is defined similarly. With $I(x)$ we associate the subspace $E_x$ of $\R^k\oplus W$, 
\begin{align*}
E_x&=\{(b, v)\in R^k\oplus W\vert \,\text{$b_i=0$ for all $i\in I(x)$}\}.
\end{align*}
The set $E'_{y}$ is defined analogously. Observe that $\# I(x)$ is equal to the codimension of $E_x$ in $\R^k\oplus W$.  In order to prove \eqref{independence_0}, it is enough to show that 
$S(E_x)\subset E'_y$. If this is the case, then,  since $S$ is an isomorphism, we obtain $S(E_x)=E'_y$,  which shows that the codimensions of $E_x$ and $E'_y$ in $\R^k\oplus W$ and $\R^{k'}\oplus W'$, respectively, are the same. Now, to see that $S(E_x)\subset E'_y$, we take $u=(b, v)\in E_x$ and note that $x+tu\in [0,\infty)\oplus W$ for $\abs{t}$ small. Then, 
$S(x+tu)=S(x)+tS(u)=y+tS(u).$ If $i\in I'(y)$, then $(S(x+tu))_i=y_i+t(S(u))_i=t(S(u))_i\in [0,\infty)^k\oplus W'$, which is only possible if  $(S(u))_i=0$.  Hence, $S(E_x)\subset E'_y$ and the proof is complete.
\end{proof}

\pagebreak
\section{Differential Geometry Based on Retracts}

The crucial  concept of this book is that of a sc-smooth retraction, which we  are going to introduce next.
\subsection{Retractions and Retracts}\label{section2.1}
In this subsection we introduce several basic  concepts.
\begin{definition} \label{sc-smooth_retraction}\index{D- Sc-smooth retraction}
We consider a tuple $(U, C, E)$ in which $U$ is a relatively open subset of the partial quadrant $C$ in the sc-Banach space $E$. A sc-smooth map $r\colon U\to U$ is called 
a {\bf sc-smooth retraction}  on $U$, if 
$$
r\circ r=r.
$$
\end{definition}

As a side remark, we observe that  if $U$ is an open subset of a Banach space $E$ and $r\colon U\rightarrow U$ is a $C^\infty$-map (in the classical sense) satisfying $r\circ r=r$, then the image of $r$ is a smooth submanifold of the Banach space $E$, as the following propositions shows.

\begin{proposition}\label{smooth_retract}\index{Cartan's theorem}
Let $r\colon 
U\to U$ be a $C^\infty$-retraction defined on an open subset $U$ of a Banach space $E$. Then $O:=r(U)$ is a $C^\infty$-submanifold of $E$. More precisely, for  every point $x\in O$,  there exist an open neighborhood  $V$ of $x$ and  an open neighborhood $W$  of $0$ in $E$, a splitting $E =R \oplus N$, and a smooth diffeomorphism 
$\psi \colon 
 V\to  W$  satisfying  $\psi (0)=x$ and 
$$\psi (O\cap V ) = R \cap W.$$
\end{proposition}
An elegant proof is due to Henri Cartan in \cite{H.Cartan}\footnote{The first author would like to thank E. Ghys for pointing out this reference during his visit at IAS in 2012. It seems that this is H. Cartan's last mathematical paper.}

In sharp contrast to the conclusion of this proposition from classical differential geometry,  there are $\ssc^\infty$-retractions which have,  for example, locally varying finite dimensions, as we shall see later on. 
\begin{definition}\label{sc_smooth_retract}\index{D- Smooth retract}
A tuple $(O,C,E)$ is called a {\bf sc-smooth retract}, if there exists a relatively open subset $U$ of the partial quadrant $C$ and a  sc-smooth retraction $r\colon 
U\to U$ satisfying 
$$r(U)=O.$$
\end{definition}

In case $x\in U^1$ we deduce  from $r\circ r=r$ by the chain rule, 
$ Dr (r(x))\circ Dr(x)=Dr(x)$. Hence if $x\in O^1$, then $r(x)=x$ so that  $Dr (x)\circ Dr(x)=Dr(x)$ and $Dr(x)\colon E\to E$ is a projection.

\begin{proposition}\label{independent_of_retractions}\index{P- Intrinsic geometry of  retracts}
Let  $(O, C, E)$ be  a sc-smooth  retract and assume  that $r\colon 
U\to U$ and $r'\colon 
U'\to U'$ are two sc-smooth retractions defined on relatively open subsets $U$ and $U'$ of $C$  and satisfying $r(U)=r'(U')=O$. Then
$$Tr (TU)=Tr'(TU').$$
\end{proposition}
\begin{proof}
If $y\in U$, then there exists $y'\in U'$, so that $r(y)=r'(y')$.  Consequently, 
$r'\circ r(y)=r'\circ r'(y')=r'(y')=r(y)$, and hence $r'\circ r=r$.  Similarly, one sees that $r\circ r'=r'$.
If  $(x, h)\in Tr (TU)$,  then $(x, h)=Tr(y, k)$ for a pair  $(y, k)\in TU$.
Moreover,  $x\in O_1\subset U'_1$, so that $(x, h)\in TU'$.  From $r'\circ r=r$ it follows,  using the chain rule,  that
$$Tr'(x, h)=Tr' \circ Tr(y,k)=T(r'\circ r)(y, k)=Tr(y, k)=(x, h),$$
implying $Tr (TU)\subset Tr'(TU')$.  Similarly,  one shows that $Tr' (TU')\subset Tr(TU)$, and the proof of the proposition is complete.
\end{proof}

Proposition \ref{independent_of_retractions} allows us  to  define the tangent of a sc-smooth retract  $(O,C,E)$\index{$(O,C,E)$} as follows.
\begin{definition}\index{D- Tangent of a retract}
The {\bf tangent of the sc-smooth retract}  $(O,C,E)$,  denoted by $T(O,C,E)$,  is defined as the triple
$$
T(O,C,E)=(TO,TC,TE),\index{$T(O,C,E)$}
$$
in which $TC=C^1\oplus E$ is the tangent of the partial quadrant $C$, $TE=E^1\oplus E$,  and  $TO:= 
Tr(TU)$, where  $r\colon U\to U$ is any sc-smooth retraction onto $O$.
\end{definition}

We recall that, explicitly, 
$$
TO=Tr(TU)=Tr(U^1\oplus E)=\{(r(x), Dr(x)h)\, \vert \, x\in U^1, h\in E\}.
$$

Starting with a relatively open subset $U$ of a partial quadrant $C$ in $E$,  the tangent $TC$ of $C$  is a partial quadrant in $TE$, and $TU$ is a relatively open subset of $TC$.

In the following we shall quite often just write $O$ instead of $(O,C,E)$
and  $TO$ instead of $(TO,TC,TE)$. However, we would like to point out that the ``reference'' $(C,E)$ is important, since 
it is possible that $O$ is a sc-smooth retract with respect to some nontrivial $C$, but would not be a sc-smooth retract for $C=E$.

\begin{proposition}\label{preparation_retraction}\index{P- Tangent map}
Let $(O,C,E)$ and $(O',C',E')$  be sc-smooth retracts and let $f\colon 
O\rightarrow O'$  be a map. If  $r\colon 
U\to U$ and $s\colon 
V\to V$  are sc-smooth retractions onto $O$ for the triple $(O, C, E)$, then the following holds.
\begin{itemize}
\item[\em(1)]  If $f\circ r\colon 
U\rightarrow E'$ is sc-smooth, then $f\circ s\colon 
V\rightarrow E'$ is sc-smooth, and vice versa.
\item[\em(2)] If  $f\circ r$ is sc-smooth, then  
$T(f\circ r)|TO=T(f\circ s)|TO$.
\item[\em(3)] The  tangent map $T(f\circ r)\vert TO$ maps $TO$ into $TO'$.
\end{itemize}
\end{proposition}
\begin{proof}
We assume that $f\circ r\colon 
U\to  E'$ is $\ssc^\infty$. Since  $s\colon 
V\to  U\cap V$ is $\ssc^\infty$, the chain rule implies that the composition $f\circ r\circ s\colon 
V\to  F$ is $\ssc^\infty$.
Using  the identity  $f\circ r\circ s=f\circ s$, we conclude that $f\circ s$ is $\ssc^\infty$. Interchanging the role of $r$ and $s$,  the first part of the lemma is proved.
If $(x, h)\in TO$, then $(x, h)=Ts (x, h)$ and using the identity $f\circ r\circ s=f\circ s$ and the chain rule, we conclude
\begin{equation*}
T(f\circ r)(x,h)=T(f\circ r)(Ts)(x,h)=T(f\circ r\circ s)(x,h)=T(f\circ s)(x,h)
\end{equation*}
Now we take any  sc--smooth retraction $\rho\colon 
U'\to  U'$ defined on a  relatively open subset $U$ of the the partial quadrant $C'$ in $E'$  satisfying $\rho(U')=O'$. Then $\rho\circ f= f$ so that $\rho\circ f\circ r=f\circ r$. Application of  the chain rule  yields the identity
$$
T(f\circ r)(x,h)=T(\rho\circ f\circ r)(x,h)=T\rho\circ T(f\circ r)(x,h)
$$
for all $(x, h)\in Tr(TU).$
Consequently, $T(f\circ r)\vert Tr(TU)\colon 
TO\to TO'$ and this map is  independent of the choice of a sc-smooth retraction onto $O$.
\end{proof}

In view of  Proposition  \ref{preparation_retraction},  we can define the sc-smoothness of a map between sc-smooth retracts as follows.
\begin{definition}\label{tangent_retract}\index{D- Tangent map}
Let  $f\colon 
O\rightarrow O'$  be a map between sc-smooth retracts $(O,C,E)$ and $(O',C',E')$, and let  $r\colon 
U\to U$ be a sc-smooth retraction for   $(O, C, E)$. Then {\bf the  map $f$ 
is sc-smooth},  if  the  composition
$$
U\rightarrow E', \quad x\mapsto f\circ r(x)
$$
is sc-smooth.  In this case the {\bf tangent map} $Tf\colon 
TO\rightarrow TO'$ is defined by
$$
Tf\colon 
=T(f\circ r)\vert TO.\index{$Tf$}
$$
\end{definition}
That the image of $Tf$ lies in $TO'$ follows from
$r'\circ f \circ r = f\circ r$. Indeed, differentiating and applying the definition of $Tf$
we deduce that $Tr'\circ Tf =Tf$, which implies that the image of $Tf$ is in $TO'$.

Proposition  \ref{preparation_retraction} shows that   the definition does not depend on the choice of  the sc-smooth retraction $r$, as long as it retracts onto  $O$. With Definition \ref{tangent_retract}, the chain rule for sc-smooth maps between sc-smooth retracts  follows from Theorem \ref{sccomp}.
\begin{theorem}
Let $(O,C,E)$, $(O',C',E')$ and $(O'',C',E'')$ be two sc-smooth retracts and let $f\colon 
O\rightarrow O'$ and $g\colon 
O'\rightarrow O''$ be sc-smooth. Then $g\circ f\colon 
O\rightarrow O''$ is sc-smooth and $T(g\circ f)=(Tg)\circ (Tf)$.
\end{theorem}

If $f\colon O\to O'$ is a sc-smooth map between sc-smooth retracts as in Definition \ref{tangent_retract}, we abbreviate the linearization of $f$ at the point $o\in O^1=r(U^1)$ (on level $1$) by 
$$
Tf(o)=T(f\circ r)(o)\vert T_oO,
$$
so that 
$$Tf(o)\colon T_oO\to T_{f(o)}O'$$
is a continuous linear  map between the tangent spaces 
$T_oO=Dr(o)E$ and $T_{f(o)}O'=Dr'(f(o))E'$.

At this point, we have defined a new {\bf category ${\mathcal R}$}\index{${\mathcal R}$} whose  objects  are  sc-smooth
retracts $(O,C,E)$ and whose morphisms are sc-smooth maps 
$$
f\colon 
(O,C,E)\rightarrow (O',C',E')
$$
 between sc-smooth retracts.
The map $f$ is only defined between $O$ and $O'$, but the other data $(O,C,E)$ are  needed to define the differential geometric properties of $O$.

We also have a  well-defined functor, namely the {\bf tangent functor} \index{Tangent functor}
$$
T\colon 
{\mathcal R}\rightarrow {\mathcal R}.\index{$T\colon 
{\mathcal R}\rightarrow {\mathcal R}$}
$$
It maps  the object $(O,C,E)$ into its tangent $(TO,TC,TE)$ and the morphism
$f\colon 
(O,C,E)\rightarrow (O',C',E')$ into its tangent map $Tf$,
\begin{gather*}
Tf\colon 
(TO,TC,TE)\rightarrow (TO',TC',TE')\\
Tf(x,h)=(f(x),Df(x)h)\quad  \text{for all $(x,h)\in TO$.}
\end{gather*}
 The chain rule guarantees the functorial property 
 $$T(g\circ f)=(Tg)\circ (Tf).$$
 Clearly, this is enough to build a differential geometry  whose  local models
are sc-smooth retracts. The details will be carried out  in the next subsection. Let us note that $T(O,C,E)$ has evidently
more structure than $(O,C,E)$, for example $TO\rightarrow O^1$ seems to have some kind of bundle structure.
We shall discuss this briefly and 
 introduce another category of retracts. 
 
 We consider tuples $(U, C, E)$, where $U$ is a relatively open subset of the partial quadrant $C$ in the sc-Banach space $E$, and let $F$ be another sc-Banach space.
\begin{definition}\index{D- Bundle retraction}
Let $p\colon 
U\oplus F\rightarrow U$ be a trivial sc-bundle  defined by the projection $p\colon 
U\oplus F\rightarrow U$ onto $U$, and let $R\colon 
U\oplus F\rightarrow U\oplus F$ be  a sc-smooth map of the form $R(u,v)=(r(u),\rho(u,v))$, satisfying $R\circ R=R$, and where $\rho(u,v)$ is linear in $v$. We call $R$ a {\bf sc-smooth  bundle retraction}  (covering the sc-smooth retraction $r$)
and $B=R(U\oplus F)$ the associated sc-smooth bundle retract, and denote by $p\colon 
B\rightarrow O$ the induced projection onto $O=r(U)$. 
\end{definition}
Given two sc-smooth bundle retractions $p\colon 
B\rightarrow O$ and $p'\colon 
B'\rightarrow O'$,  a sc-smooth map $\Phi\colon 
B\rightarrow B'$
of the form $\Phi(u,v)=(a(u),\phi(u,v))$, where $\phi(u,v)$ is linear in $v$ and $p'\circ \Phi=p$ is called a {\bf sc-smooth (local) bundle map}.\index{D- Bundle map}

By $\mathcal{BR}$\index{$\mathcal{BR}$} we denote the category whose objects are sc-smooth bundle retracts and whose  morphisms are the sc-smooth bundle maps.
There is  a natural forgetful functor $\mathcal{BR}\rightarrow \mathcal{R}$, which on objects associates with the
bundle $B\xrightarrow{p}O$, the total space  $B$ and which views a bundle map $\Phi$ just as a sc-smooth map.

If $r\colon 
U\rightarrow U$ is a sc-smooth  retraction, then its  tangent map $Tr\colon 
TU\rightarrow TU$ is a sc-smooth bundle retraction. Moreover, the tangent map is a sc-smooth bundle map. Consequently,  the tangent functor can be  viewed as the functor 
 $$
 T\colon 
\mathcal{R }\rightarrow \mathcal{BR}.\index{$ T\colon 
\mathcal{R }\rightarrow \mathcal{BR}$}
 $$
The functor $T$  associates with  the object $(O,C,E)$\index{$(O,C,E)$} the triple $(TO,TC,TE)$\index{$T(O,C,E)$}, in which we  view $TE$  as the bundle $TE\rightarrow E^1$,
$TC$ as the bundle $TC\rightarrow C^1$,  and $TO$ as the bundle $TO\rightarrow O^1$. With a  sc-smooth map $f\colon 
(O,C,E)\rightarrow (O',C',E')$, the functor $T$  associates the 
sc-smooth bundle map $Tf$.

There is another class of retractions, called strong bundle retractions, which will be introduced in the later parts of Section \ref{SEC2.2}.

\subsection{M-Polyfolds and Sub-M-Polyfolds}\label{SEC2.2}
We start with the following observation about sc-retractions. 
\begin{proposition}\label{new_retract_1}\index{P- Restrictions of retracts}
Let $(O, C, E)$ be a sc-smooth retract.
\begin{itemize}
\item[\em(1)] If $O'$ is an open subset of $O$, then $(O',C,E)$ is a sc-smooth retract.
\item[\em(2)] Let  $V$ be  an open subset of $O$ and $s\colon 
V\rightarrow V$  a sc-smooth map, satisfying  $s\circ s=s$.  If $O'=s(V)$, then $(O',C,E)$ is a sc-smooth retract.
\end{itemize}
\end{proposition}
\begin{proof}\mbox{}\\
(1)\, 
By assumption, there exists a sc-smooth retraction $r\colon 
U\to U$ defined on a relatively open subset $U$ of  the partial quadrant $C\subset E$, whose image is $O=r(U)$.  Since the map $r\colon 
U\rightarrow U$ is continuous,  the set $U':=r^{-1}(O')$ is an open subset of $U$ and, therefore, a relatively open subset of $C$.  Clearly,  $O'\subset U'$ 
and the restriction  $r'=r\vert U'$ defines a sc-smooth retraction $r'\colon 
U'\rightarrow U'$ onto  $r'(U')=O'$.
Consequently, the triple  $(O',C,E)$ is a sc-smooth retract, as claimed.

\noindent (2)\, The triple $(V, C, E)$ is, in view of (1),  a sc-smooth retract. Hence, there exists a sc-smooth retraction $r\colon U\to U$ onto $r(U)=V$, where $U\subset C\subset E$ 
We define the map $\rho\colon U\to U$ by 
$$\rho=s\circ r.$$

Then $\rho$ is sc-smooth and $\rho\circ \rho=(s\circ r)\circ (s\circ r).$ If $x\in U$, then $r(x)\in V$, hence $s(r(x))\in V$ and consequently, $r(s(r(x))=s(r(x))$, so that 
$$r\circ \rho=\rho.$$
From $s\circ s=s$, we conclude that  $\rho\circ \rho=\rho$. Hence $\rho$ is a sc-smooth retraction onto the subset 
$$\rho(U)=s\circ r(U)=s(V)=O'.$$
We see that $(O', C, E)$ is a sc-smooth retract, as claimed.
\end{proof}

In the following every sc-smooth map $r\colon 
O\rightarrow O$, defined on a sc-smooth retract  $O$ and satisfying $r\circ r=r$ will also be called a {\bf sc-smooth retraction}.
\begin{definition}\label{sc-charts}\index{D- Sc-charts}
Let $X$ be a topological space and $x\in X$.
{\bf A chart around $x$}  is a tuple $(V,\phi,(O,C,E))$,\index{$(V,\phi,(O,C,E))$} in which  $V\subset X$ is an open neighborhood of  $x$ in $X$, 
and $\phi\colon 
V\rightarrow O$ is a homeomorphism. Moreover, $(O,C,E)$ is a sc-smooth retract.
Two charts $(V,\phi,(O,C,E))$ and $(V',\psi,(O',C',E'))$ are called {\bf sc-smoothly compatible},  if  the transition maps 
$$\psi\circ\phi^{-1}\colon 
\phi(V\cap V')\rightarrow \psi(V\cap V')\quad \text{and}\quad \phi\circ \psi^{-1}\colon 
\psi(V\cap V')\rightarrow \phi(V\cap V')$$
are sc-smooth maps (in the sense of Definition \ref{tangent_retract}).
\end{definition}
Let us observe that $\phi(V\cap V')$ is an open subset of $O$ and so, in view of part (1) of Proposition  \ref{new_retract_1}, the tuple $(\phi(V\cap V'), C, E)$ is a sc-smooth retract.  So,  the above  transition maps are defined on  sc-smooth retracts.
\begin{definition}\label{sc_atlas}\index{D- Sc-smooth atlas}
{\bf A sc-smooth atlas} on the topological space $X$ consists
of a set of charts 
$$
(V,\varphi,(O,C,E)),
$$ 
such  that any two of them are sc-smoothly compatible and the open sets $V$ cover $X$.
Two {\bf sc-smooth atlases} on $X$  are said to be {\bf equivalent}, if their  union is again a  sc-smooth atlas.
\end{definition}

\begin{definition} \index{D- M-polyfold}
{\bf A M-polyfold} $X$  is a Hausdorff paracompact  topological
space equipped with an equivalence class of sc-smooth atlases.
\end{definition}

Analogous to the smoothness of maps between manifolds we shall define sc-smoothness of maps between M-polyfolds.

\begin{definition}
A {\bf map $f\colon 
X\to Y$ between two M-polyfolds is called sc-smooth} \index{D- Sc-smooth map}
if its coordinate representations are sc-smooth. In detail, this requires the following. If $f(x)=y$ and $(V,\varphi, (O, C, E))$ is a chart around $x$ belonging to the atlas of $X$  and 
$(V',\varphi', (O', C', E'))$ is a chart around $y$ belonging to the atlas of $Y$, so that $f(V)\subset V'$, then the map 
$$\varphi'\circ f\circ \varphi^{-1}\colon 
O\to O'$$
is a sc-smooth map between the sc-smooth retracts in the sense of Definition \ref{tangent_retract}.
\end{definition}

The definition does not depend on the choice of sc-smoothly compatible charts.

We recall that a Hausdorff topological space is paracompact, provided  every open cover of $X$ has an open locally finite refinement.  
It is a well-known fact,  that, given an open cover $\mathscr{U}={(U_i)}_{i\in I}$,  one can find a refinement $\mathscr{V}=(V_i)_{i\in I}$ satisfying  $V_i\subset U_i$  for all $i\in I$ (some  of the sets $V_i$ might be  empty).

Given  a  M-polyfold $X$, we say that a point $x\in X$  is on the level $m$, if  there exists a chart $(V, \varphi, (O, C, E))$ around $x$, such that $\varphi(x)\in O_m$. Of course, this definition is independent of  the choice of a chart around $x$. We denote the collection of all points on the level $m$ by $X_m$.   The topology on $X_m$ is defined as follows.  We abbreviate by  $\mathscr{B}$ the collection of  all sets  $\varphi^{-1}(W)$, where $W$ is an open subset of $O_m:=O\cap E_m$ in the chart $(V, \varphi, (O, C, E))$ on $X$.  Then $\mathscr{B}$ is a basis for a topology on $X_m$.   With this topology, the set $V_m:=V\cap X_m$  is an open subset of $X_m$ and  $\varphi\colon 
V_m\to O_m$ is a homeomorphism, so that  the tuple $(V_m, \varphi, (O_m, C_m, E_m))$ is a chart on $X_m$. 
Any two such  charts are sc-smoothly compatible and  the collection of such charts is an atlas on $X_m$.
In this way the M-polyfold $X$ inherits from the charts the  filtration 
$$X=X_0\supset X_1\supset \cdots \supset X_\infty=\bigcap_{i\geq 0}X_i.$$

\begin{lemma}\label{inculsion_continuous}
The inclusion map $i:X_{m+1}\to X_m$ is continuous for all $m\geq 0$.
\end{lemma}
\begin{proof} 
In view of the definition of topologies on $X_{m}$ and $X_{m+1}$, it suffices to show, that if $(V, \varphi, (O, C, E))$ is a chart on $X$ and $W$ is an open subset of $O_m$, then $\varphi^{-1}(W)\cap X_{m+1}$ is open in $X_{m+1}$. If $x\in \varphi^{-1}(W)\cap X_{m+1}$, then $\varphi (x)\in W\cap O_{m+1}\subset W\cap E_{m+1}$. Since $W$ is open in $O_{m}$, there exists an open set  $W'$ in $E_m$, so that $W=W'\cap O$. Hence, $\varphi (x)\in  (W'\cap E_{m+1})\cap O=W''\cap O$, where $W''=W'\cap E_{m+1}$ is an open subset of $E_{m+1}$. Hence $W''\cap O$ is an open subset of $O_{m+1}$ and $ \varphi^{-1}(W)\cap X_{m+1}=\varphi^{-1}(W''\cap O)$,  proving our claim.
\end{proof}

\begin{theorem}\label{X_m_paracompact}\index{T- Paracompactness of $X_m$}
Let $X$ be a M-polyfold. For every $m\geq 0$, the space $X_m$ is metrizable and, in particular, paracompact. In addition, the space $X_\infty$ is metrizable.
\end{theorem}

The proof is  postponed to  Appendix \ref{A2.2}.

In order to define the tangent $TX$ of the M-polyfold $X$, we start with its  local description in a chart.
\begin{definition}[{\bf Tangent space $T_xO$}]\index{D- Tangent space}
Let $(O, C, E)$ be a retract and $T(O,C,E)=(TO,TC,TE)$ its tangent,  so that $p\colon 
TO\rightarrow O^1$ is the tangent bundle over $O$.
The {\bf tangent space} $T_xO$ at a point $x\in O^1$ is the pre-image $p^{-1}(x)$, which is a Banach space. Note that only in the case that $x$ is  a smooth point, the tangent space  $T_xO$ has a natural sc-structure.
\end{definition} 
In the case that $x\in O^1$,  the tangent space $T_xO$\index{$T_xO$} is the image of the projection $Dr(x)\colon 
E\rightarrow E$, where $r\colon 
U\rightarrow U$ is any sc-smooth retraction
associated with  $(O,C,E)$ satisfying $r(U)=O$.

Next we consider tuples  $(x, V, \varphi,(O,C,E),h)$, in which $x\in X_1$ is a point in the M-polyfold $X$ on  level $1$ and  $(V,\varphi,(O,C,E))$ is a chart around the point $x$. Moreover,  $h\in T_{\varphi (x)}O$. 
Two such tuples, 
$$
(x, V, \varphi,(O,C,E),h)\quad \text{and} \quad (x',V', \psi, (O',C',E'),h'),
$$
are called equivalent, if  
$$x=x'\quad  \text{and}\quad T(\varphi\circ\psi^{-1})(\psi (x))h'=h.$$ 
\begin{definition}\label{tangent_space}\index{D- Tangent space}
The {\bf tangent space $TX$  of $X$}  as a set, is the collection of all equivalence classes $[(x,V, \varphi,(O,C,E),h)]$.  
\end{definition}

If $x\in X_1$ is fixed, the tangent space $T_xX$ of the M-polyfold $X$ at the point $x \in X$ is the subset of equivalence classes
$$T_xX=\{[(x, \varphi, V, (O, C, E), h)]\ \vert \,  h\in T_{\varphi (x)}O\}.$$
It has the structure of a vector space defined by 
$\lambda\cdot [(x, \varphi, V, (O, C, E), h_1)]+
\mu\cdot [(x, \varphi, V, (O, C, E), h_2)]=
[(x, \varphi, V, (O, C, E), (\lambda h_1+\mu h_2)]$ where $\lambda,\mu\in \R$ and $h_1,h_2\in T_{\varphi(x)}O$.
Clearly,
\begin{equation}\label{equation_tangent_bundle}
TX=\bigcup_{x\in X_1}\{x\}\times T_xX.
\end{equation}

In order to  define a topology on the tangent space $TX$,  we first fix a chart $(V,\varphi,  (O, C, E))$ on $X$ and associate with it a subset $TV$ of $TX$,  defined as   
$$TV:
=\{[(x, \varphi, V, (O, C, E), h)]\ \vert \, x\in V,\, h\in T_{\varphi (x)}O\}.$$

We introduce the {\bf tangent map}  $T\varphi\colon TV\to TO$\index{Tangent map}  by 
$$T\varphi ([(x, \varphi, V, (O, C, E), h)])=(\varphi (x), h),$$
where $h\in T_{\varphi(x)}O$. If  $x\in X_1$ is fixed, then the map
$$T\varphi\colon 
T_xV\to T_{\varphi (x)}O$$
is a linear isomorphism.
Therefore, $T_xX$, the tangent space at $x \in X_1$ inherits the Banach space structure from $T_{\varphi(x)} O \subset E = E_0$. If $x \in X_\infty$, the tangent space $T_xX$ is a sc-Banach space, because the tangent space at the smooth point $\varphi(x) \in O_\infty$ is a sc-Banach space.

If $W$ is an an open subset of $TO$,  we define  the subset $\wt{W}\subset TV$ by $\wt{W}:=(T\varphi )^{-1}(W).$  
We denote by  $\mathscr{B}$ the collection of all such sets  $\wt{W}$ obtained by taking all the charts  $(\varphi, V, (O, C, E))$  of the atlas and all open subsets $W$ of the corresponding tangents $TO$. 

\begin{proposition}\label{op}\mbox{}\index{P- Properties of $TX$}
\begin{itemize}
\item[\em(1)] The collection $\mathscr{B}$ defines a basis for a Hausdorff  topology on $TX$.  
\item[\em(2)] The projection $p\colon 
TX\to X^1$ is a continuous and an open map.
\item[\em(3)] With the topology defined by $\mathscr{B}$,  the tangent space $TX$ of the M-polyfold $X$  is metrizable and hence, in particular,  paracompact.
\end{itemize}
\end{proposition}
The proof is postponed to  Appendix \ref{A2.11}.
In view of our definition of the topology on $TX$, the map 
$T\varphi\colon 
 TV\to TO$, associated with the chart $(\varphi, V, (O, C, E))$,  is a homeomorphism. Moreover, given two such maps 
$$T\varphi \colon 
TV\to TO\quad \text{and}\quad T\varphi'\colon 
TV'\to TO', $$ the composition $T\varphi' \circ (T\varphi)^{-1}\colon 
 T\varphi (TV\cap TV')\to T\varphi' (TV\cap TV')$ is explicitly of the form 
\begin{equation}\label{transition_1}
T\varphi' \circ (T\varphi)^{-1}(a, h)=(\varphi'\circ \varphi^{-1}(a), 
D(\varphi'\circ \varphi^{-1})(a)\cdot h)=T(\varphi'\circ\varphi^{-1})(a,h).
\end{equation}
Since the transition map  $\varphi' \circ  \varphi^{-1}$ between sc-retracts is sc-smooth, the composition in \eqref{transition_1} is also sc-smooth. 
In addition, $(TO, TC, TE)$ is a sc-smooth retract. Consequently, the tuples $(TV, T\varphi,  (TO, TC, TE))$  define a sc-smooth atlas on $TX$. Since, as we have proved above, $TX$ is paracompact,  the tangent space $TX$ of the M-polyfold $X$ is also  a M-polyfold.  The projection map $p\colon 
TX\to X^1$  is locally built on the bundle retractions $TO\rightarrow O^1$,  and the transition maps of the charts
are sc-smooth bundle maps. Therefore, 
$$p\colon 
TX\rightarrow X^1$$ 
is a sc-smooth M-polyfold bundle.

The  M-polyfold is the notion of a smooth manifold in our extended universe.  If 
$X$ is a  M-polyfold,  which consists entirely of smooth points, then it  has a tangent space at every point. 
There are finite-dimensional examples. For example,  the  chap depicted in Figure \ref{fig:pict1} has a M-polyfold structure,
for which $X=X_\infty$. It illustrates, in particular, that M-polyfolds allow to describe in a smooth way geometric objects having locally varying dimensions.
 For details in the construction of the chap and further illustrations, we refer to \cite{HWZ8.7}, Section 1, in particular, Example 1.22.

\begin{figure}[htb]
\centering
\def\svgwidth{30ex}
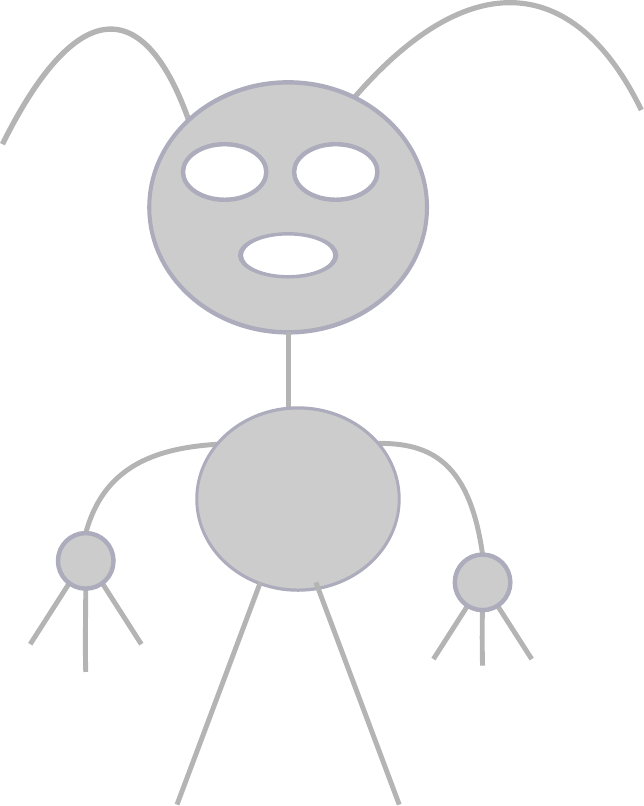
\caption{This chap has a  M-polyfold structure for which every point is smooth. }\label{fig:pict1}
\end{figure}

Next we introduce the notion of a sub-M-polyfold.

\begin{definition}\label{def_sc_smooth_sub_M_polyfold}\index{D- Sub-M-polyfold}
Let $X$ be a M-polyfold and let $A$ be a subset of $X$. The subset  $A$ is called a {\bf sub-M-polyfold} of $X$, 
if every $a\in A$ possesses  an open neighborhood $V$ and a sc-smooth retraction
$r\colon 
V\rightarrow V$, such that 
$$r(V)=A\cap V.$$
\end{definition}

\begin{proposition}\label{sc_structure_sub_M_polyfold}\index{P- Properties of sub-M-polyfolds}
A sub-M-polyfold $A$ of a  M-polyfold $X$ has, in a natural way, the structure of a  M-polyfold,
for which  the following holds. 
\begin{itemize}
\item[{\em (1)}] The inclusion map $i\colon A\rightarrow X$ is sc-smooth and a homeomorphism onto its image.
\item[{\em (2)}] For every $a\in A$ and every sc-smooth  retraction $r\colon V\rightarrow V$ satisfying  $r(V)=A\cap V$ and $a\in V$,
the map $i^{-1}\circ r\colon V\rightarrow A$ is sc-smooth.
\item[{\em (3)}] The tangent space $T_aA$ for a smooth $a\in A$ has a sc-complement in $T_aX$.
\item[{\em (4)}] If $a$ is a smooth point and $s\colon 
W\rightarrow W
$ is a sc-smooth retraction satisfying  $s(W)=W\cap A$ and $a\in W$, then the induced map
$W\rightarrow A$ is sc-smooth and $Ts(a)T_aX=T_aA$.
\end{itemize}
\end{proposition}
\begin{proof}
We  first define a sc-smooth atlas for $A$. We choose a point $a\in A$ and let $( \varphi, V, (O, C, E))$ be a chart of the M-polyfold $X$ around the point $a$.  By definition of a sub-M-polyfold,  there exists an open neighborhood $U$ of $a$ in $X$ and a sc-smooth retraction $r\colon 
U\to U$ satisfying  $r(U)=U\cap A$.  The set $W\subset X$, defined by  $W:=r^{-1}(U\cap V)\cap (U\cap V)$, is  open in $X$ and satisfies  $r(W)\subset W$ and $r(W)=W\cap A$. Hence, $\varphi (W)$ is an open subset of $O$, so that, in view of of part (1) 
Proposition  \ref{new_retract_1}, the tuple $(\varphi (W), C, E)$ is a sc-retract.  We may therefore  assume without loss of generality that $U=V=W$.  We define the sc-smooth map   $\rho\colon O\to O$ by 
$$\rho=\varphi\circ r\circ \varphi^{-1}.$$
From $r\circ r=r$ we deduce  $\rho\circ \rho=\rho$, so that $\rho$ is a sc-smooth retraction onto $\rho(O)=O'$. By the statement (2) in  Proposition  \ref{new_retract_1},  the triple $(O', C, E)$ is a sc-retract. Therefore there exists a relatively open subset $U'$ of the partial quadrant $C$ in $E$ and a sc-smooth retraction $r'\colon 
U'\to U'$ onto  $r'(U')=O'$.  Restricting the map $\varphi$ to $W\cap A$,  we set $\psi:=\varphi\vert W\cap A$ and compute,
$$\psi (W\cap A)=\varphi (W\cap A)=\varphi \circ r(W)=\varphi\circ r\circ \varphi^{-1}(O)=\rho(O)=O'.$$
Consequently,  $\psi\colon 
 W\cap A\to O'$ is a homeomorphism and the triple  $( \psi, W\cap A, , (O', C, E))$ is a chart on $A$. 

In order to consider the chart transformation we take a second compatible chart $( \varphi', V', (\wh{O}, C', E'))$ of the M-polyfold $X$ around the point $a\in A$ and use it two construct the second chart $( \psi', W'\cap A, , (O'', C', E'))$ of $A$. 
We shall show that the second chart is compatible with the already constructed chart  $( \psi, W\cap A, , (O', C, E))$. 
The domain $\psi ((W\cap A)\cap (W'\cap A))$ of  the transition map 
\begin{equation}\label{transition_map_0}
\psi'\circ \psi^{-1}\colon 
\psi ((W\cap A)\cap (W'\cap A))\to \psi' ((W'\cap A)\cap (W\cap A))
\end{equation}
 is an open subset of $O'$, so that, in view of  (2) of Proposition \eqref{new_retract_1}, there exists a relatively open subset $U''\subset C$ and a sc-smooth retraction $s''\colon 
U''\to U''$ onto  $s''(U'')=\psi ((W\cap A)\cap (W'\cap A))$.  By construction,
 $$(\psi'\circ \psi^{-1})\circ s''=(\phi'\circ \phi^{-1})\circ s''.$$
 The chart transformation $\varphi'\circ \varphi^{-1}\colon 
O\to \wh{O}$ is sc-smooth so that by the chain rule 
the right-hand side is sc-smooth. Therefore, also the left-hand  is a sc-smooth map. So, in view of Definition \ref{tangent_retract} of a sc-smooth map between retracts, the transition map $\psi'\circ \psi^{-1}$ is sc-smooth.
  
We have shown that  the collection of charts $(\psi, W\cap A, (\psi (W\cap A), C, E))$  defines a sc-smooth atlas for $A$.  The sc-smooth structure on $A$  is defined by its equivalence class.

In order to prove the statement (2) in 
Proposition \ref{sc_structure_sub_M_polyfold} we use the above local coordinates, assuming, as above, that $U=V=W$. 
The inclusion map
$i\colon 
A\rightarrow X$ is, in the local coordinates, the inclusion 
$j\colon 
O'=r'(U')\to O$ which  is a  sc-smooth map since $r'\colon 
U'\to U'$ is a  sc-retraction. Conversely, the relations $\varphi\circ r\circ \varphi^{-1}(O)=O'$ and $U=\varphi^{-1}(O)$ show that the map
$i^{-1}\circ r\colon 
U\to A$ is sc-smooth because the retraction $r\colon 
U\to U$ is, by assumption, sc-smooth. This  proves  statement (2) of the proposition. 


In order to prove the statement (3) we work in local coordinates,  and assume that  $X$ is given by the triple $(O,C,E)$ in which $O=r(U)$ and $r\colon 
U\to U$ is a retraction of the relatively open subset $U$ of $C$ in $E$. Then  $A$ is a subset of $O$ having the property that every point $a\in A$ possesses  an open neighborhood $V$ in $O$ and a sc-smooth retraction $s\colon 
V\rightarrow V$ onto $s(V)=A\cap V$. We now assume that $a\in A$ is  a smooth point and introduce  the map $t=s\circ r\colon 
U\to U$. Then 
$t\circ t=t$ and $t$ is a sc-smooth retraction onto the set $V\cap A$. Hence the tangent space $T_aA$ is defined by 
$$T_aA=Dt(a)E=Ds(a)\circ Dr(a)E=Ds(a)T_aO.$$
From $r\circ t=r\circ s\circ r=s\circ r=t$ we conclude 
$Dr(a)\circ Dt(a)E=Dt(a)E$ and hence $Dt(a)E\subset T_aO.$
Therefore, 
\begin{equation*}
\begin{split}
T_aO=Dr(a)E&=Ds(a)\circ Dr (a)+(I-Ds(a))Dr(a)E\\
&=
(T_aA)\oplus (I-Ds(a))(T_aO).
\end{split}
\end{equation*}
This proves the statement (3).  Using the same arguments,  the statement (4) follows and the proof of Proposition \ref{sc_structure_sub_M_polyfold} is complete.

\end{proof}

\subsection{Degeneracy Index and Boundary Geometry}

On the   M-polyfold $X$ we shall introduce the map $d_X\colon 
X\rightarrow {\mathbb N}_0$, called degeneracy index,  as follows.
We first take a smooth chart  $(V,\phi,(O,C,E))$ around the point $x$ 
and define the integer 
$$
d(x, V,\phi,(O,C,E))= d_C(\phi(x)),\index{$d(x, V,\phi,(O,C,E))$}
$$
where $d_C$ is  the index defined in 
Section \ref{subsection_boundary_recognition}.
In  other words, we record how many vanishing coordinates the image point $\phi(x)$ has in the partial quadrant $C$.
\begin{definition}\label{M_polyfold_degeneracy _index}\index{D- Degeneracy index $d_X$}\index{$d_X$}
The {\bf degeneracy} $d_X(x)$ at the point $x\in X$ is the minimum of all numbers 
$d(x, V, \phi,(O,C,E))$, where $(V,\phi,(O,C,E))$ varies over all smooth charts around  the point $x$.
The {\bf degeneracy index}  of the M-polyfold $X$ is the map $d_X\colon 
X\rightarrow {\mathbb N}_0$.
\end{definition}

The next lemma is evident.
\begin{lemma}\index{L- Local property of $d_X$}
Every point $x\in X$ possesses an open neighborhood $U(x)$, such  that $d_X(y)\leq d_X(x)$ for all $y\in U(x)$. 
\end{lemma}

From the definitions one deduces immediately the following result.
\begin{proposition}\label{newprop2.24}\index{P- Diffeomorphism invariance of $d_X$}
If $X$ and $Y$ are M-polyfolds and if $ f\colon (U,x)\rightarrow (V,f(x))$ is a germ of sc-diffeomorphisms
around the points $x\in X$ and $f(x)\in Y$,  then $$d_X(x)=d_Y(f(x)).$$
\end{proposition}

The index $d_X$ quantifies to which extend a point $x$ has to be seen as a boundary point. A more degenerate point has a higher index.

\begin{definition}\label{boundary_M_polyfold}\index{D- Boundary of an M-polyfold}
The subset  $\partial X=\{x\in X\ |\ d_X(x)\geq 1\}$ of $X$  is called the {\bf boundary of $X$}.
A  M-polyfold $X$ for which $d_X\equiv 0$ is called
a  {\bf M-polyfold without boundary}. 
\end{definition}

The relationship between $d_A$ and $d_X$  where  $A$ is a sub-M-polyfold of $X$ is described in the following lemma.
\begin{lemma}\label{k^0}\index{L- $d_A$ versus $d_X$}
If $X$ is  a M-polyfold and $A\subset X$ a sub-M-polyfold of $X$,   then
$$
d_A(a)\leq d_X(a)
$$
for all $a\in A$.
\end{lemma}
\begin{proof}
We take a point $a\in A$ and choose a chart $(\varphi, V, (O, C, E))$ around the point $a$, belonging to the atlas for $X$ and satisfying 
$$d_X(a)=d(a, \varphi , V, (O, C, E)).$$
The integer on the right-hand side remains unchanged, if we take a smaller domain, still containing $a$,  and replace $O$ by its image. Then, arguing as in 
Proposition \ref{sc_structure_sub_M_polyfold}, we may assume, that the  corresponding chart of the atlas  for $A$ is $(\psi, V', (O', C, E))$, where we have abbreviated $V'=V\cap A$, $\psi=\varphi\vert {V'}$, and $O'=\psi (V')$.  Then, 
$$d(a, \psi, V', (O', C, E))=d(a, \varphi, V, (O, C, E))=d_X(a),$$
and hence, taking the minimum on the left-hand side,
$d_A(a)\leq d_X(a),$ as claimed in the lemma.
\end{proof}

If $E={\mathbb R}^k\oplus W$ is the sc-Banach space and $C=[0,\infty)^k\oplus W$ the partial quadrant in $E$,  we   define the linear subspace $E_i$ of $E$ by
$$
E_i=\{(a_1,\ldots, a_k,w)\in {\mathbb R}^k\oplus W\ \vert \ a_i=0\}.\index{$E_i$}
$$
With a subset $I\subset \{1,\ldots ,k\}$,  we associate the subspace 
$$
E_I=\bigcap_{i\in I} E_i.\index{$E_I$}
$$

In particular, $E_{\emptyset}=E$, $E_{\{i\}}=E_i$, and $E_{\{1,\ldots , k\}}=\{0\}^k\oplus W\equiv W$.  If $x\in C$, we denote by $I(x)$ the  set 
of indices $ i\in \{1,\ldots ,k\}$ for which $x\in E_i.$ We 
abbreviate
$$
E_x:=E_{I(x)}.
$$
Associated with  $E_i$ we have the closed half space $H_i$,  consisting of all elements $(a,w)$ in $ {\mathbb R}^k\oplus W$, satisfying  $a_i\geq 0$, 
$$
H_i=\{(a_1,\ldots, a_k,w)\in E\ |\ a_i\geq 0\}.\index{$H_i$}
$$
If $x\in C$, we define  the partial cone $C_x$ in $E$ as 
$$
C_x=C_{I(x)}:=\bigcap_{i\in I(x)}H_i.\index{$C_x$}
$$

As an illustration we take  the standard quadrant $C\subset {\mathbb R}^2$ consisting of all $(x,y)$ with $x,y\geq 0$.
Then $C_{(0,0)}=C$, $C_{(1,0)}=\{(x,y)\ |\ y\geq 0\}$, $C_{(0,1)}=\{(x,y)\ |\ x\geq 0\}$ and $C_{(1,1)}={\mathbb R}^2$.
If $z=(x,y)$ is at least on level $1$,  one should view  $C_z$ as a partial quadrant in the tangent space $T_zC= E$. 

We shall put some additional structure on a sc-smooth retract $(O,C,E)$, which turns out to be useful.

We call a subset $C$ of a Banach space (or sc-Banach space) 
a {\bf cone} \index{Cone}provided it is closed, convex, and satisfies ${\mathbb R}^+C=C$ and $C\cap (-C)=\{0\}$.
If all properties except the last one hold,  we call $C$  a {\bf  partial cone}.\index{Partial cone}

\begin{definition}\label{reduced_cone_tangent}\index{D- Reduced tangent space}\index{D- Partial cone}
Let $(O,C,E)$ be a sc-smooth retract and let $x\in O_\infty$  be a smooth point in $O$. The {\bf partial cone} $C_xO$  at  $x$ is defined as the following subset of the tangent space at $x$, 
$$
C_xO:=T_xO\cap C_x,\index{$C_xO$}
$$
where $C_x=\bigcap_{i\in I(x)}H_i$.
The {\bf reduced tangent space}  $T^{\textrm{R}}_xO$ is defined as the following subset of the tangent space at $x$, 
$$
T_x^{\textrm{R}} O =T_xO\cap E_x.\index{$T_x^{\textrm{R}}O$}
$$
\end{definition}
\begin{remark}
We have the inclusions
$$
T^R_xO\subset C_xO\subset T_xO, 
$$
and naively one might expect that $C_xO$ is a partial quadrant in $T_xO$. However, this is in general not the case.
\end{remark}
The reduced tangent space and the partial cone are  characterized in the next lemma.
\begin{lemma}\label{characterization_reduced_tangent}\index{L- Reduced tangent}
Let $(O,C,E)$  be  a sc-smooth retract  and let  $x\in O_\infty$ be a smooth point in $O$. Then the following holds.
\begin{itemize}
\item[\em(1)] $T^{\textrm{R}}_xO=\textrm{cl}( \{ \text{$\dot{\alpha}(0)\vert \, \alpha\colon 
(-\varepsilon, \varepsilon)\to O$ is sc-smooth and $ \alpha(0)=x$}\} )$.
\item[\em(2)] $C_xO=\text{cl}( \{\text{$\dot{\alpha}(0)\vert \,  \alpha\colon 
[0, \varepsilon)\rightarrow O$ is sc-smooth and $ \alpha(0)=x$}\})$.
\item[\em(3)] $T_xO =C_xO-C_xO$.
\end{itemize}
Here $\dot{\alpha}(0)=\frac{d}{dt}\alpha (t)_{\vert t=0}$ stands for the derivative of the sc-smooth path $\alpha$ in the parameter $t$ varying in $(-\varepsilon, \varepsilon)$ resp. in $[0, \varepsilon)$.
\end{lemma}

\begin{proof} We may assume that $O$ is a sc-smooth retract in the partial quadrant $C=[0,\infty)^k\oplus W\subset E=\R^k\oplus W$ and that  $O=r(U)$, where $r\colon 
U\to U$ is a sc-smooth retraction on the  relatively open subset $U$ of $C$.  Let us denote for $x\in {\mathbb R}^k\oplus W$
by $x_1\ldots,x_k$ its  coordinates in ${\mathbb R}^k$.\\[0.6ex]  
(1)\,  In order to prove (1), we first  introduce the set  
\begin{equation}\label{set_gamma}
\Gamma= \{ \text{$\dot{\alpha}(0)\vert \, \alpha\colon 
(-\varepsilon, \varepsilon)\to O$ is sc-smooth and $ \alpha(0)=x$}\},
\end{equation}
and observe, that $\alpha ((-\varepsilon, \varepsilon))$ is contained in $ O_\infty$, and   $\dot{\alpha}(0)\in E_\infty$.
Since the closure of $(T^{\textrm{R}}_xO)_\infty$ is  equal to $T^{\textrm{R}}_xO$, it is enough to prove that  $(T^{\textrm{R}}_xO)_\infty\subset \Gamma$ and $\Gamma\subset (T^{\textrm{R}}_xO)_\infty$.  As for the  first inclusion,  we take $v\in (T^{\textrm{R}}_xO)_\infty$. Hence  $x+tv\in U$ for $\abs{t}$ small. This follows from the fact that for $i\in I(x)$,  we have $a_i=v_i=0$ and for $i\not \in I(X)$, we have $x_i>0$. Since $x$ and $v$ are smooth points, $x+tv\in U_\infty$. Then $\alpha (t)=r(x+tv)$ is defined for $\abs{t}$ small, takes values in $O_\infty$,  and $\alpha (0)=x$.  By the chain rule, 
$$\dot{\alpha}(0)=Dr(x)v=v,$$
since $v\in T_xO=\text{image of $Dr(x)$}$. Hence, $v\in \Gamma$ and $(T^{\textrm{R}}_xO)_\infty\subset \Gamma$, as claimed. Conversely,  if $\alpha\colon 
(-\varepsilon, \varepsilon)\to O$ is a  sc-smooth path  satisfying  $\alpha (0)=x$,   then $r(\alpha (t))=\alpha (t)$, so that, applying the chain rule, we find 
$$Dr(x)\dot{\alpha}(0)=\dot{\alpha} (0).$$
This shows that  $\dot{\alpha} (0)$ is a smooth point belonging to $T_xO$. If $i\in I(x)$, then $\alpha_i(0)=x_i=0$ and since 
$\alpha_i(t)\geq 0$ for all $t\in (-\varepsilon, \varepsilon)$, we conclude  that $\dot{\alpha}_i(0)=0$, so that 
$\dot{\alpha}(0)\in E_x$.  
Hence, 
$\dot{\alpha}(0)\in T_xO\cap E_x=T_x^{\textrm{R}}O$  
and since 
$\dot{\alpha}(0)$ 
is a smooth point, 
$\dot{\alpha}(0)\in (T_x^{\textrm{R}}O)_\infty$.  
Therefore,
 $\Gamma \subset (T_x^{\textrm{R}}O)_\infty$,  and the proof of (1) is complete.\\[0.6ex]
(2)\, The proof of (2) is along the same line, except that considering a sc-smooth path $\alpha\colon 
[0,\infty)\to O,$ we conclude that $Dr(x)\dot{\alpha}(0)\in C_x$ and since $Dr(x)\dot{\alpha}(0)=\dot{\alpha}(0)$, we find that $C_xO$ is a subset of the right-hand side of (2). Conversely, we take a smooth point $v\in C_xO$ and consider the path $\alpha (t)=r(x+tv)$ defined for $t\geq 0$ small. Then, $\dot{\alpha}(0)=Dr(x)v$ belongs to the right-hand side of (2) and since $C_xO$ is closed, the result follows.\\[0.8ex]
(3) 
Clearly, $C_xO-C_xO\subset T_xO$. Conversely, let $h\in T_xO$. Then $h=Dr(x)k$, where $k=(a, w)\in E.$  
If $i\in I(x)$, we set $$a_i^\pm=\frac{\abs{a_i}\pm a_i}{2}.$$ 

Then, $a_i^\pm\geq 0$ and $a_i=a^+_i-a^-_i$. Now we define elements $k^\pm\in E$ as follows.  First, $k^+=(b, w),$ where $b_i=a_i$ if $i\not \in I(x)$ and $b_i=a^+_i$ if $i\in I(x)$. The element $k^-$ is defined as $k^-=(c, 0),$ where $c_i=0$ if $i\not \in I(x)$ and $c_i=a_i^-$ if $i\in I(x)$. Then $k=k^+-k^-$ and if $h^\pm =Dr(x)k^\pm$, then, by (2), we have $h^\pm\in  C_xO$ and 
$h=h^+-h^-\in C_xO-C_xO$. The proof of (3) and hence the proof of Lemma \ref{characterization_reduced_tangent}  is complete.

\end{proof}

From the characterization of $T_xO$ and $C_xO$ in 
Lemma \ref{characterization_reduced_tangent} we  deduce immediately the next  proposition.
\begin{proposition}\label{reduced_tangent_under_sc_diff}
Let $(O,C,E)$ and $(O',C',E')$ be sc-smooth retracts,   and let $x\in O_\infty$. If $f\colon 
(O,x)\rightarrow (O',f(x))$
is  a germ of a sc-diffeomorphism mapping $x\in O_\infty$ onto $f(x)=y\in O'_\infty$, then 
$$
Tf(x)T^{\textrm{R}}_xO= T^{\textrm{R}}_{y}O'\quad  \text{and}\quad Tf(x)C_xO= C_yO'.
$$
\end{proposition}

A sc-smooth retract $O$ associated with a triple 
$(O, C, E)$ is a M-polyfold and, recalling Definition \ref{M_polyfold_degeneracy _index}, its degeneracy index $d_O(x)$ at the point $x$ is the integer 
$$d_O(x)=\min d_{C'}(\varphi (x))$$
where the minimum is taken over all germs of sc-diffeomorphisms $\varphi\colon (O, x)\to (O', \varphi (x))$ into sc-smooth retracts $O'$ associated with $(O', C', E')$. The integer $d_{C'}(\varphi (x))$ is introduced in 
Section \ref{subsection_boundary_recognition}.

\begin{theorem}\label{hofer}\index{T- Basic properties of $d_O$}
Let  $(O,C,E)$ be a smooth retract and let $d_O$ be the degeneracy index of $O$. 
If $x\in O$ is a smooth point, we have the 
inequality
$$
\dim (T_xO/T^R_xO)\leq  d_O(x).
$$
Moreover, if
$\dim(T_xO/T^R_xO)=d_O(x)$, then 
$C_xO$ is a partial quadrant in $T_xO$. 
\end{theorem}
\begin{proof}
Let $x$ be a smooth point of the retract $O$. 
By 
Proposition \ref{reduced_tangent_under_sc_diff} 
the dimension of $T_xO/T^R_xO$ is preserved 
under germs of sc-diffeomorphisms.  Hence, in view of the definition of $d_O(x)$, we may assume, without loss of generality,  that
$$d_O(x)=d_C(x)\equiv d$$ 
Moreover, without loss of generality, we may  assume that $E=\R^k\oplus W$, 
$C=[0,\infty)^k\oplus W$
and $x=(0,\ldots ,0,x_{d+1},\ldots , x_k, w)$, where $x_i>0$ for $d+1\leq i\leq k$ and 
$w\in W$. We recall that if $v=(a, b, w)\in T_x^RO\subset \R^d\oplus R^{k-d}\oplus W$, then $a=0$,  and if $v=(a, b, w)\in C_xO$, then $a_i\geq 0$ for $1\leq i\leq d$.

In order to prove the first statement, we choose smooth vectors $v^1,\ldots v^l$ in $T_xO$ such that $(v^j+T_x^RO)_{1\leq j\leq l}$ are linearly independent in the vector space $T_xO/T_x^RO$. 
Representing $v^j=(a^j,b^j, w^j)\in \R^d\oplus R^{k-d}\oplus W$, we claim that the vectors $(a^j)_{1\leq j\leq l}$ are linearly independent in $\R^d$. 
Indeed, assuming  that  
$\sum_{j=1}^l\lambda_ja^j=0$, we have  
$$\sum_{j=1}^l\lambda_jv^j=\bigl(0, \sum_{j=1}^l\lambda_jb^j, \sum_{j=1}^l\lambda_jw^j\bigr)\in T^R_xO,$$
hence 
$$\sum_{j=1}^l\lambda_j\bigl(v^j+T_x^RO\bigr)=\bigl(\sum_{j=1}^l\lambda_jv^j\bigr)+T_x^RO=T_x^RO.$$
Since $(v^j+T_x^RO)_{1\leq j\leq l}$ 
are linearly independent  in $T_xO/T_x^RO$,  we conclude that $\lambda_1=\ldots =\lambda_l=0$, proving our claim. This implies that the vectors $a^1,\ldots ,a^l$ are linearly independent in $\R^d$. Therefore, $l\leq d$ and hence $\dim(T_xO/T^R_xO)\leq  d=d_O(x)$, proving the first statement of the theorem.

In order to prove the second statement, 
we assume that $\dim(T_xO/T^R_xO)=d_O(x)$. 
If now $(v^j+T^R_xO)_{1\leq j\leq d}$ is a basis of $T_xO/T^R_xO$, then  representing  $v^j=(a^j,b^j, w^j)\in T_xO\subset   \R^d\oplus \R^{k-d}\oplus W$ and arguing as above,   the vectors $a^j$ for $1\leq j\leq d$ form a basis of  $\R^d$. 
Consequently, the map 
$\Phi\colon 
T_xO/T^R_xO\to \R^d$, defined by 
$$\Phi (v+T^R_xO)=\Phi ((a, b, w)+T^R_xO)=a,$$
is a linear  isomorphism. Moreover, if $v=(a, b, w)\in C_xO$  so that $a_j\geq 0$ for $1\leq j\leq d$, then 
$$\Phi (v+T^R_xO)=\Phi ((a, b, w)+T^R_xO)\in [0,\infty)^d.$$

Denoting by $e^j$ for $1\leq j\leq d$ the standard basis of $\R^d$, we introduce $\Phi^{-1}(e^j)=\wh{v}^j+T^R_xO$. By definition, the vectors $v^j$ are of the form 
$\wh{v}^j=(e^j, \wh{b}^j, \wh{w}^j)$ and  are linearly independent in $\R^d\oplus \R^{k-d}\oplus W$.   
If now $v=(a, b, w)\in T_xO$, we have the decomposition 
\begin{equation*}
\begin{split}
v=(a, b, w)&=\sum_{j=1}^da_j\wh{v}^j+\bigl(v-\sum_{j=1}^da_j\wh{v}^j\bigr)
\end{split}
\end{equation*}
where  $a=(a_1,\ldots ,a_d)\in \R^d$. Since the second term  on the right-hand side belongs to $T_x^RO$,  we  have the following decompositions of the tangent space $T_xO$ and of $C_xO$,
\begin{align*}
T_xO=\R \wh{v}^1\oplus \ldots \oplus \R \wh{v}^d\oplus T_x^RO\quad \text{and}\quad 
C_xO=\R^+\wh{v}^1\oplus \ldots \oplus \R^+\wh{v}^d\oplus T_x^RO.
\end{align*}
Therefore, the map $T\colon T_xO\to \R^d\oplus T_xO$, defined  by 
$$T(\lambda_1\wh{v}^1,\ldots ,\lambda_d\wh{v}^d, w)=(\lambda_1, \ldots ,\lambda_d, w)$$
is a sc-linear isomorphism
satisfying $T(C_xO)=[0,\infty )^d\oplus T_x^RO$.
 Hence $C_xO$ is a partial quadrant in $T_xO$ and the second statement of Theorem \ref{hofer} is proved.  
\end{proof}

The  partial quadrant $C$ in $E$ is the  image of the sc-smooth retraction $r=\mathbbm{1}_{C}\colon C\to C$. In particular, $C$ is a M-polyfold which we denote by $X_C$. 
In Section \ref{subsection_boundary_recognition}, we have defined the map $d_C\colon 
C\to \N_0$, which associates with a point $x\in C$ its number of vanishing coordinates. Above, we have defined the degeneracy index $d_{X_C}\colon 
X_C\to \N_0$ of the M-polyfold $X_C$.   By definition, 
$d_{X_C}\leq d_C$ and we shall prove that $d_{X_C}=d_C$.  We may assume  without loss of generality
that $C=[0,\infty)^k\oplus W$ and $E={\mathbb R}^k\oplus W$.  For a smooth point $x\in C$ we have $T_xC=E$ and $T_x^RC=E_x$,  implying 
$\dim(T_xC/T_x^RC)=d_C(x)$. This, of course, also holds for any partial quadrant $C\subset E$. Hence, 
$$
d_C(x) = \dim(T_xX_C/T_x^RX_C)
$$
for a smooth point $x\in C$. From this we deduce  the following corollary of Theorem \ref{hofer}
\begin{corollary}\label{equality_of_d}\index{C- Computation of $d_{X_C}$}
Let $C$ be a partial quadrant in a sc-Banach space. Considering $C$ as a M-polyfold,  denoted by $X_C$,  we have the equality 
$$
d_C=d_{X_C}.
$$
\end{corollary}
\begin{proof}
Let us first take  a smooth point $x\in C$. In view of Theorem \ref{hofer},  $\dim(T_x X_C/T_x^R X_C)\leq d_{X_C}(x)$.
Since $d_C(x)=\dim(T_x X_C/T_x^R X_C)$ as we have just seen,  it follows that $d_C(x)\leq d_{X_C}(x)$. By definition of  $d_{X_C}(x)$,  we always have the 
inequality $d_{X_C}(x)\leq d_C(x)$. Consequently,  
$$
d_{X_C}(x)=d_C(x)\quad  \text{if $x\in C_\infty$.}
$$
If $x\in C$ is arbitrary, we take  a sequence of smooth points $x_k\in C$ converging to $x$ and satisfying $d_C(x_k)=d_C(x)$. Hence, 
$$
d_C(x)=d_C(x_k) =d_{X_C}(x_k).
$$
In view of the definition of $d_{X_C}$ we  find a sc-diffeomorphism $f\colon 
U(x)\rightarrow O'$, where $U(x)\subset C$ is relatively open,  and $(O',C',E')$ is a sc-smooth retract, so that
$$
d_{X_C}(x)=d_{C'}(f(x)).
$$
Then $f(x_k)\rightarrow f(x)$ and trivially $d_{C'}(f(x_k))\leq d_{C'}(f(x))$ for large $k$. Hence, for large $k$, 
\begin{equation*}
\begin{split}
d_C(x)&=d_C(x_k)=d_{X_C}(x_k)\\
&=d_{X_{C'}}(f(x_k))\leq  d_{C'}(f(x_k))\leq d_{C'}(f(x)) =d_{X_C}(x).
\end{split}
\end{equation*}
Since $d_{X_C}(x)\leq d_C(x)$,  we  conclude 
$d_{X_C}(x)=d_C(x)$ and the proof of 
Corollary \ref{equality_of_d} is complete.
\end{proof}
From now on we do not have to distinguish between  the index $d_C$  defined for partial quadrants and the degeneracy index $d_{X_C}$, 
where we view $C$ as a  M-polyfold.

\subsection{Tame M-polyfolds}\label {subsec_tame_m_polyfolds}

In order to define spaces whose boundaries have more structure, we introduce the notion of a tame M-polyfold
and of tame retractions and tame retracts. We start with some basic geometry.

Let $C\subset E$ be a partial quadrant in a sc-Banach space $E$. We begin with the particular case
$E={\mathbb R}^k\oplus W$ and $C=[0,\infty)^k\oplus W$.  We recall  the linear sc-subspace $E_i$ of codimension $1$ defined as  
\begin{equation}\label{k^4}
E_i=\{(a_1,\ldots, a_k,w)\in {\mathbb R}^k\oplus W\ \vert \,  a_i=0\}.
\end{equation}
Associated with  $E_i$, there is  the closed half space $H_i$  consisting of all elements $(a,w)$ in $ {\mathbb R}^k\oplus W$ satisfying  $a_i\geq 0$, 
\begin{equation}\label{k^5}
H_i=\{(a_1,\ldots, a_k,w)\in E\ |\ a_i\geq 0\}.
\end{equation}
 In the general case of  a partial quadrant $C$ in $E$,  we can describe the above definitions in a more  intrinsic way as follows.
 
We  consider the set $\{e\in C\,\vert \, d_C(e)=1\}$ of boundary points. This set has exactly $k$ connected components, which we denote by $A_1,\ldots ,A_k$. Each component $A_i$ lies in the smallest subspace $f_i$ of $E$ containing $A_i$. We call $f_i$ an {\bf extended face} and denote by ${\mathcal F}$ the set of all extended faces. The set ${\mathcal F}$ contains exactly $k$ extended faces. 
 Given an extended face $f\in {\mathcal F}$,  we denote  by $H_f$ the closed half subspace of $E$ which contains $C$.  
 
 In the special case $C=[0,\infty)^k\oplus W\subset \R^k\oplus W$, 
 the  extended  faces  $f_i$ are the subspaces $E_i$, and the half spaces $H_{f_i}$ are the half subspaces $H_i$. 
 
 If  $e\in C$,  we introduce the set of all extended faces containing $e$ by 
 $$ {\mathcal  F}(e)=\{f\in {\mathcal F} \vert \, e\in f\}.$$
Clearly, 
 $$d_C(e)=\# {\mathcal F}(e).$$
 \begin{definition}\label{new_def_2.33}
The {\bf partial quadrant $C_e$}  associated  with  $e\in C$ is defined as 
$$
C_e:= \bigcap_{f\in {\mathcal F}(e)} H_f.\index{$C_e$}
$$
The  minimal linear subspace associated with  $e\in C$ is defined by
$$
E_e:=\bigcap_{f\in {\mathcal F}(e)} f.\index{$E_e$}
$$
\end{definition}
Clearly, the following inclusions hold,
$$
E_e\subset C_e\subset E.
$$
For an interior point $x\in C$, i.e. a point satisfying $d_C(x)=0$, we set $E_e=C_e=E$.  The codimension of $E_e$  in $E$ is precisely $d_C(e)$.
The maximal value $d_{C_e}$ attains  is $d_C(e)$.

Next we introduce a special class of sc-smooth retracts. 
\begin{definition}[{\bf Tame sc-retraction}] \label{tame_retarctions}\index{D- Tame sc-smooth retraction}
Let $r\colon U\to U$ be a sc-smooth retraction defined on a relatively open subset $U$ of a partial quadrant $C$ in the sc-Banach space $E$.
The sc-smooth retraction  $r$ is  called a {\bf tame sc-retraction}, if  the following two conditions are satisfied.
\begin{itemize}
\item[(1)] \text{ $d_C(r(x))=d_C(x)$ for all  $x$ in $U$.}
\item[(2)] At every smooth point $x$ in $O:=r(U)$, there exists a sc-subspace $A\subset E$, such that $E=T_xO\oplus A$ and $A\subset E_x$. 
\end{itemize}
A sc-smooth retract $(O,C,E)$ is called a {\bf tame sc-smooth retract},  if $O$ is the image of a sc-smooth  tame
retraction.
\end{definition}

Let $x$ be a smooth point in the tame sc-retract $O$ and let $A\subset E_x$ be a sc-complement of the tangent space $T_xO$ as guaranteed by condition (2) in Definition \ref{tame_retarctions}, so that 
\begin{equation}\label{eq_T_oplus_A}
E=T_xO\oplus A.
\end{equation}
We recall that $T_xO=Dr(x)E$ and hence 
\begin{equation}\label{eq_T_oplus_A_1}
\begin{split}
E&=Dr(x)E+(\mathbbm{1}-Dr(x))E\\
&=T_xO+(\mathbbm{1}-Dr(x))E.
\end{split}
\end{equation}
Applying  the projection $\mathbbm{1}-Dr(x)$ to the equation \eqref{eq_T_oplus_A} and using that $(\mathbbm{1}-Dr(x))T_xO=0$, we obtain 

\begin{equation}\label{eq_T_oplus_A_2}
(\mathbbm{1}-Dr(x))E=(\mathbbm{1}-Dr(x))A.
\end{equation}

We claim that $(\mathbbm{1}-Dr(x))E\subset E_x$. In order to prove this claim we recall that $A\subset E_x$, so that, in view of \eqref{eq_T_oplus_A_2}, 
it is sufficient to prove that $Dr(x)E_x\subset E_x$.  We may assume that $E=\R^k\oplus W$ and $x=(0,\ldots ,0,x_{d+1},\ldots,x_k,w)$ with 
$x_{d+1},\ldots,x_k>0$. We choose a smooth point $y\in E_x$ so that $y=(0,\ldots,0,y_{d+1},\ldots,y_k,v)$. If $\abs{\tau}$ is small, then 
$x_\tau:=x+\tau y$ belongs to  the partial quadrant $C$ and has 
 the first $d$ coordinates vanishing. By condition (1) in Definition \ref{tame_retarctions} of a tame sc-retraction $r$, we have  $d_C(r(x+\tau y))=d_C(x+\tau y).$  Hence 
 the first 
 $d$ coordinates of $r(x+\tau y)$ 
 vanish, and from
 $$
 \dfrac{d}{d\tau}r(x+\tau y)\bigl\lvert_{\tau=0}=Dr(x)y
 $$
 we conclude that the first $d$ coordinates of $Dr(x)y$ 
 vanish.  The same is true if $y$ is on level $0$ in $E_x$. So, $Dr(x)E_x\subset E_x$ and hence, in view of \eqref{eq_T_oplus_A_2}, we have verified that $(\mathbbm{1}-Dr(x))E\subset E_x$, as claimed.

In view of  \eqref{eq_T_oplus_A_1}, we can therefore always assume, without loss of generality,  that in condition (2) of Definition \ref{tame_retarctions} the complement $A$ of $T_xO$ is equal to $A=(\mathbbm{1}-Dr(x))E$. 

Summarizing the discussion we have established the following proposition. 
\begin{proposition}\label{IAS-x}\index{P- Properties of tame retractions}
Let $U\subset C\subset E$ be a relatively open subset of a partial quadrant in a sc-Banach space and let $r\colon 
U\rightarrow U$  be 
a sc-smooth tame retraction. Then,  for every smooth point $x\in O=r(U)$,  the sc-subspace $(\mathbbm{1}-Dr(x))E$ is a subspace of $E_x$, so that in condition (2) of Definition \ref{tame_retarctions} we can take $A=(\mathbbm{1}-Dr(x))E$.
\end{proposition}

Let us discuss the tame sc-smooth retracts in more detail. 
\begin{proposition}\label{tame_equality}\index{P- Basic equality for tame retracts}
Let $(O,C,E)$ be a {\bf tame sc-smooth retract},  and let $x\in O_\infty$ be a smooth point of $O$.
Then, $T_x^{\textrm{R}}O$ is a Banach space of codimension $d_O(x)$ in $T_xO$, so that 
\begin{equation}\label{opp}
\dim(T_xO/T^R_xO)=d_O(x). 
\end{equation}
In particular,   $C_xO$ is a partial quadrant in the the tangent space $T_xO$.
In addition,  for every point $x\in O$, we have the equality
\begin{equation}\label{equality_d_O_and_d_C_1}
d_O(x)=d_C(x).
\end{equation}
\end{proposition}
\begin{proof}
We assume that $O=r(U)$ is the retract in $U\subset C\subset E$ of the tame sc-retraction $r\colon 
U\to U$. Moreover, we may assume that $C=[0,\infty )^k\oplus W\subset E=\R^k\oplus W$. 

Then the  condition (2) of Definition \ref{tame_retarctions}  says that,  for every $x\in O_\infty$ there exists a sc-subspace $A$ of $E$, satisfying  $E=\text{im}\ Dr(x)\oplus A =T_xO\oplus A$ and $A\subset E_x$, where $E_x=E_{I(x)}$.  
Therefore, 
\begin{equation}\label{direct_sum_1}
E_x=(T_xO\oplus A)\cap E_x= (T_xO\cap E_x)\oplus (A\cap E_x)=T_x^{\textrm{R}}O\oplus A, 
\end{equation}
since  $A\subset E_x$. The space $E_x$ has codimension $d_C(x)$ in $E$, so that 
\begin{equation}\label{direct_sum_2}
d_C(x)=\dim\bigl(E/E_x \bigr)= \dim \bigl(T_xO\oplus A /T_x^{\textrm{R}}O\oplus A \bigr)=  \dim \bigl(T_xO /T_x^{\textrm{R}}O \bigr).
\end{equation}
Employing Theorem \ref{hofer},  we conclude from 
\eqref{direct_sum_2} that
$d_C(x)\leq d_O(x) $ for all $x\in O_\infty$.
By definition,  $d_O(x)\leq d_C(x)$ for all $x\in O$,  and hence, 
\begin{equation}\label{eq_d_C=d_O}
d_C(x)=d_O(x)\quad \text{if  $x\in O_\infty$.}
\end{equation}
Consequently, $\dim T_xO/T_x^RO=d_O(x)$. Employing 
Theorem \ref{hofer} once more, $C_xO$ is a partial quadrant in the tangent space $T_xO$.

It remains to prove that $d_O(x)=d_C(x)$ for all (not necessarily smooth) points $x\in O$. If $x=(a, w)\in O$,  we take a sequence $(w_j)\subset W_\infty$  converging to $w$  in $W$.
Then the sequence $x_j\in O$, defined by $x_j=r(a, w_j)$, consists of smooth points and converges to $x$ in $E$.
Since, by assumption, the  retraction $r$ is tame, 
$$d_C(x_j)=d_C(r(a, w_j))=d_C(a, w_j)=d_C(a, w)=d_C(x)$$
and, using  \eqref{eq_d_C=d_O}, we conclude 
that $d_O(x_j)=d_C(x_j)=d_C(x).$
By definition of $d_O$, we find an open neighborhood $V'\subset O$ around $x$, such that 
$d_O(y)\leq d_O(x)$ for all $y\in V'$. Consequently, for large $j$,
$$d_O(x)\geq \lim_{j\to \infty}d_O(x_j)=d_C(x).$$
In view of $d_O(x)\leq d_C(x)$, we conclude 
$d_O(x)=d_C(x)$, and the proof of Proposition \ref{tame_equality} is complete.
\end{proof}

Among all M-polyfolds, there is a distinguished class of M-polyfolds which are modeled on tame retracts.
These turn out to have an interesting and useful boundary geometry.

\begin{definition}\label{def_tame_m-polyfold}\index{D- Tame M-polyfold}
A  {\bf tame M-polyfold} $X$ is a M-polyfold which possesses an equivalent sc-smooth atlas whose charts are all modeled on {\bf tame} sc-smooth retracts.
\end{definition}

By Proposition \ref{reduced_tangent_under_sc_diff}, the following concepts for any M-polyfold  are well-defined and independent of the choice of the charts.


\begin{definition}\label{def_partial_cone_reduced_tangent}\index{D- Reduced tangent space}
Let $X$ be a  M-polyfold. For a smooth point $x\in X$,  the {\bf reduced tangent space $T^R_xX$}\index{$T^R_xX$}  is, by definition, the sc-subspace of the tangent space
$T_xX$   which, by definition,   is the preimage of $T_o^RO$ under any  chart $\psi\colon 
(V,x)\rightarrow (O,o)$, so that 
$$ 
T\psi(x) (T_x^RX) = T_o^RO = T_oO \cap E_o.
$$

For a smooth point $x\in X$,  the {\bf partial cone} $C_xX$\index{Partial cone}\index{$C_xX$} is the closed convex subset of $T_xX$ which under a sc-smooth chart $\psi$ as above is mapped onto $C_oO$, i.e.
$$
T\psi(x)(C_xX) =  T_oO \cap C_o.
$$
\end{definition}

From Proposition \ref{tame_equality} we conclude, if  
the M-polyfold $X$ is tame and if $x\in X$ is a smooth point, that $C_xX$ is a partial quadrant in the tangent space $T_xO$, and  we have the identity
$$
d_X(x)=\dim(T_xX/T_x^RX) = d_{C_xX}(0_x),
$$
where $0_x$ is the zero vector in $T_xX$.  

We are going to show that the boundary of a tame M-polyfold has an additional structure. 
\begin{definition} \index{D- Face of an M-polyfold}\index{D- Face}
Let $X$ be a tame M-polyfold.
A {\bf face} $F$ of $X$  is the closure of a connected component in the subset $\{x\in X\ |\ d(x)=1\}$.
The  M-polyfold $X$ is called {\bf face-structured},  if every point $x\in X$ lies in exactly $d_X(x)$ many faces.
\end{definition}

Before we study faces in more detail we have a look at the local situation.
\begin{lemma}\label{ert}
Let $(O,C,E)$ be a sc-smooth, tame retract and let ${\mathcal F}$ be the collection of extended faces,
associated with the partial quadrant  $C$.  Then the faces $F$  in $O$ are the connected components of the sets $f\cap O$, where 
$f\in {\mathcal F}$ are extended faces. The connected components of $f\cap O$ and $f'\cap O$,
containing a point $x \in O$, are equal if and  only if  $f=f'$.
\end{lemma}
\begin{proof}
By definition of $O$, the retract $O$ itself  is a tame M-polyfold. We denote by $U\subset  C\subset E$
the  relatively open subset in the  partial quadrant $C$, on which a tame sc-smooth retraction satisfies 
$r(U)=O$. Let $\wh{F}$ be a connected component
of the subset $\{x\in O\, \vert \, d_O(x)=1\}$. By Proposition \ref{tame_equality}, $d_O=d_C$ so that this set is  the same as a connected component
of $\{x\in O\, \vert \, d_C(x)=1\}$. If we look at  the isomorphic case $E={\mathbb R}^k\oplus W$, we see immediately, that there
exists an index  $i$, such  that $\wh{F}$
must lie in the subset of $C$, consisting of points $(a,w)$, for which $a_i=0$. We conclude that there exists
an extended face $f\in {\mathcal F}$, such that $\wh{F}\subset f$.  Therefore, $F$ is contained in a connected subset of $O\cap f$.

Next we consider  a connected component $Q$ of $O\cap f$, where  $f$ is an extended face of $C$.
Let $e\in Q$. By assumption, $e\in f$ and we can take a vector $h\in f$ so small, that $e+th\in U\cap f$ for $t\in (0,1]$,
and such that  for $t\in (0,1]$ the points $e+th$ do not belong to any extended face other than $f$. This can be constructed explicitly, using the model $E={\mathbb R}^k\oplus W$ for $E$. Since the retraction $r$ is tame,  we have $d_{C}(r(e+th))=1$ for $t\in (0,1]$, and 
for $t\in (0,1]$ the points $r(e+th)$ belong
to the same connected component of $\{x\in O\ |\ d_O(x)=1\}$. Using  this argument repeatedly,
 we can connect any two points in $Q$ by a continuous path $\gamma\colon 
[0,1]\rightarrow Q$,
such that  the points $y\in \gamma(0,1)$ satisfy  $d_C(y)=1$. This implies that $Q$ is contained in the closure of the connected component $\wh{F}$ in $\{x\in X\ |\ d_O(x)=1\}$. \end{proof}

\begin{corollary}\label{cor_2.42}\index{C- Characterization of $d_O$ for tame $O$}
If  $(O,C,E)$ is a tame sc-smooth retract, then every point $x\in O$ lies in precisely $d_O(x)$ many faces.
 \end{corollary}
 \begin{proof}
 In view of Proposition \ref{tame_equality}, $d_O(x)=d_C(x)$, so that $x$ belongs to precisely $d:=d_C(x)$ extended faces $f\in {\mathcal F}(x)$, let us say, it belongs to the faces $f_1,\ldots ,f_d$ in the sc-Banach space $E=\R^k\oplus W$, defined as the subspaces $f_i=\{(a_1,\ldots, a_k, w)\vert \, a_i=0\}$, $1\leq i\leq d$. The point $x$ is represented  by $x=(0,\ldots ,0, a_{d+1}, \ldots, a_k, w)$, where 
 $a_j\neq 0$ for $d+1\leq j\leq k$. Let $r\colon 
U\to U$ be the tame  retraction onto $O=r(U)$. Then,  the paths $\gamma_i(t)$ in $O$ for $1\leq i\leq d$, starting at $\gamma_i(0)=x$ are  defined by $$
 \gamma_i(t) = r(t,\ldots ,t, a_i=0, t,\ldots, t, a_{d+1},\ldots ,a_k,w).
 $$
Because the retraction $r$ is tame we have, $d_C(\gamma_i(t))=1$ if $t>0$ and  the points $\gamma_i(t)$ belong to $f_i\cap O$, but not to any other face  $f_j\cap O$ with $j\neq i$. Hence, the connected components containing $x$ of the sets $O\cap f_j$ for $1\leq j\leq d$ are all different.
 This proves the last assertion.
 \end{proof}
Corollary \ref{cor_2.42} implies immediately the following result.
\begin{proposition}\index{P- Number of local faces}
If  $X$ is  a tame M-polyfold, then every point $x\in X$ has an open neighborhood
$V$, so that $y\in V$ lies in precisely $d_X(y)$ many faces of $V$ and  $d_V(y)=d_X(y)$ for all $y\in V$.
In particular, globally, a point $x\in X$ lies in at most $d_X(x)$-many faces.
\end{proposition}

The following technical result turns out to be useful later on.  
 \begin{proposition}[Properties of faces]\label{FACE_X}\index{P- Properties of faces}
Let $(O,C,E)$ be a tame sc-smooth retract associated with the tame sc-smooth retraction $r\colon 
U\to U$ onto $O=r(U)$, and let $F$ be a face of $O$.  Then, there exists an open neighborhood $V^\ast$ of $F$ in $U$ and a sc-smooth retraction $s\colon 
V^\ast\rightarrow V^\ast$ onto  $s(V^\ast)=F$. Moreover, defining $V\subset O$
by $V=O\cap V^\ast$, the restriction $s\colon 
V\rightarrow V$ is a sc-smooth retraction onto  $s(V)=F$, so that $F$ is a sc-smooth  sub-M-polyfold of $O$.  Further, $F$ is tame, i.e. it admits a compatible sc-smooth atlas consisting of  tame local models.
In addition, 
$$
d_F(x)=d_O(x)-1\quad  \text{for all $x\in F$}.
$$
\end{proposition}
\begin{proof}
We may assume that $E={\mathbb R}^k\oplus W$ and $C=[0,\infty)^k\oplus W$. Let $U\subset C$ be a relatively open subset of the partial quadrant $C$ and let
$r\colon 
U\rightarrow U$ be a tame sc-smooth retraction onto $r(U)=O$. In view of Lemma \ref{ert}
we may assume, that our face $F$ is a connected component of $O\cap f_1$, where $f_1$ is the extended face consisting of those
points whose first coordinate vanishes. Consequently,  
 for every  $(a,w)\in F$ we find an open  neighborhood
$U_{(a,w)}\subset U$ in $C$, such  that, if $(b, v)\in U_{(a, w)}$, then $r(0,b_2,\ldots ,b_k,v)$  has the first coordinate vanishing.
Taking the union of these neighborhoods, we find an open neighborhood $U^\ast$ of $F$ which is contained in $U$,
such that for every $(b,v)\in U^\ast$,  the point $(0,b_2,\ldots ,b_k,v)$ belongs to $ U$ and $r(0,b_2,\ldots ,b_k,v)$ has the first coordinate vanishing
and belongs to $F$.
Hence, we can define the sc-smooth map
$$
s\colon 
U^\ast\rightarrow U\quad \text{by}\quad s(b,v):=r(0,b_2,\ldots,b_k,v),
$$
which has its image in $F$, so that,  in view of $F\subset U$,  we may assume that
$$
s\colon 
U^\ast\rightarrow U^\ast.
$$
It follows that the face $F$ is a sc-smooth  sub-M-polyfold of $O$. 

We know that $F$ is a connected component of $f_1\cap O$.  We find an open neighborhood $V$
of $O$ in $f_1\cap C\subset f_1$, so that $r(V)=F$. Clearly,  $s=r\vert V\colon 
V\rightarrow V$ preserves $d_{C\cap f_1}$, since it preserves $d_C$ and $d_C=d_{C\cap f_1}+1$ on $C\cap f_1$.

If $x\in F$ is a smooth point, then $T_xF\subset T_xO$ is of codimension $1$ and hence has a complement
$P$ of dimension $1$, so that 
$$
T_xO=T_xF\oplus P.
$$
We know from Definition \ref{tame_retarctions}  of a tame retract that $T_xO$ has a sc-complement $A\subset E_x$ in $E$. Hence, 
$$
E=T_xO\oplus A = T_xF\oplus P\oplus A.
$$

In view of $P\cap f_1 = \{0\}$ and $T_xF \cap f_1 = T_xF$, we deduce
$$
f_1 = E \cap f_1 = (T_xF\oplus P \oplus A) \cap f_1 = T_xF \oplus (f_1 \cap A).
$$
Noting that $(f_1 \cap A) \subset (f_1 \cap E_x) = f_1$, the sc-smooth retraction $s \colon 
 V \rightarrow V$ onto $s(V) = F$ is tame. Since the last statement of Proposition \ref{FACE_X} is obvious, the  proof  of Proposition \ref{FACE_X} is complete.
\end{proof}
In order to formulate another result along the same lines we first recall  that, given any point $x$ in a tame 
M-polyfold $X$,  we find an open neighborhood $U=U(x)$ such  that $U\cap X$ has precisely $d_X(x)$ many faces containing $x$. Since
we can choose $U$ as small was we wish,  it makes sense to introduce the notion of a {\bf face germ}\index{Face germ}. Two local faces $F$ and $F'$ at $x$ have the same
germ if there exists an open neighborhood $U=U(x)$ such  that $U\cap F=U\cap F'$. We shall write $[F,x]$ for a face germ. 
Sometimes we shall call  them {\bf codimension one face germs}\index{Codimension one face germ}. If we refer to a face germ we always mean the latter.
The collection of all codimension one face
germs at $x$  is denoted by ${\mathcal F}_x$.  For a subset $\sigma\subset {\mathcal F}_x$ we denote by
$F^\sigma_x$ the intersection
$$
F^\sigma_x = \bigcap_{[F,x]\in\sigma} F,\index{$F^\sigma_x$}
$$
and call it {\bf codimension $\#\sigma$ face germ}\index{Codimension $\#\sigma$ face}. 
\begin{proposition}
Given a tame M-polyfold $X$ every point $x$ has $d_X(x)$ many local faces containing $x$. The local faces are tame M-polyfolds
with $d_F(x)=d_X(x)-1$.  The intersection
of local faces $F_\sigma$ associated to a subset $\sigma$ of ${\mathcal F}_x$ is a tame M-polyfold of codimension $\#\sigma$
and $d_{F_\sigma}(x)=d_X(x)-\#\sigma$.
\end{proposition}

We sum up the discussion about faces in the following theorem.

\begin{theorem}\index{T- Faces as sub-M-polyfolds}
The interior of a face $F$ in a tame M-polyfold $X$ is a sc-smooth sub-M-polyfold of $X$.  If $X$ is face-structured, then every face $F$ is a sub-M-polyfold  and the induced M-polyfold structure is tame.
The inclusion map $i\colon F\rightarrow X$ satisfies $d_X(i(x)) = d_F(x)+1$. 
Every point $x\in X$ has an open neighborhood $U$ such that every point $y\in U$
lies in precisely $d_X(y)$-many faces of $U$. If $X$ is face-structured, then every point $x \in X$ lies in precisely 
$d_X(x)$-many global faces of $X$.
\end{theorem}

Later on we shall frequently use the following proposition about degeneracy index in fibered products.

\begin{proposition}\label{fibered-x}\index{P- Fibered products}
Let $f\colon  X\rightarrow Z$ be a local sc-diffeomorphism  and $g\colon 
Y\rightarrow Z$ a sc-smooth map.
Then the fibered product  $X{_{f}\times_g}Y$,  defined by
$$
X{_{f}\times_g}Y=\{(x,y)\in X\times Y \vert \, f(x)=g(y)\}, 
$$
is a sub M-polyfold of $X\times Y$. If $Y$ is tame, then also $X{_{f}\times_g}Y$ is tame and 
$$
d_{X{_{f}\times_g}Y}(x,y)=d_Y(y).
$$
If both $X$ and $Y$ are tame and both $f$ and $g$ are local sc-diffeomorphism, then 
$$
d_{X{_{f}\times_g}Y}(x,y)=\frac{1}{2} d_{X\times Y}(x,y)=\frac{1}{2}[d_X(x)+d_Y(y)] = d_X(x)=d_Y(y).
$$
\end{proposition}
\begin{proof}
Clearly the product $X\times Y$ is an M-polyfold,  and if both $X$ and $Y$ are tame, then  also $X\times Y$ is tame. 
To see that $X{_{f}\times_g}Y$ is a sub-M-polyfold of $X\times Y$, we take a point $(x,y)\in X{_{f}\times_g}Y$ and  fix  open neighborhoods $U$ of $x$ in $X$ and $W$ of  $f(x)$ in $Z$ so that $f\vert U\colon 
 U\rightarrow W$ is a sc-diffeomorphism. Next we choose an open neighborhood $V$ of $y$ in $Z$ such that $g(V)\subset W$. We define the map 
$R\colon 
U\times V\rightarrow U\times V$ by
$$
R(x',y') = ((f\vert U)^{-1}\circ g (y'),y'), 
$$
It is sc-smooth and satisfies $R\circ R =R$ and $R(U\times V) = (U\times V)\cap (X{_{f}\times_g}Y).$
Consequently, $X{_{f}\times_g}Y$ is a sub M-polyfold of $X\times Y$.   Next we assume, in addition,  that $Y$ is tame.  With the  point $(x,y)\in X{_{f}\times_g}Y$ and open sets $U$ and $V$ as above, we assume that $(V, \varphi', (O', C', E'))$  is a chart around $x$ such that $ (O', C', E')$ is a tame sc-smooth retract.
Then we define a map  $\Phi\colon 
(U\times V)\cap X{_{f}\times_g}Y\to O'$  by setting 
$$\Phi (x', y')=\varphi' (y').$$
The map $\Phi$ is injective since given a pair $(x', y')\in  X{_{f}\times_g}Y$, we have $x'=(f\vert U)^{-1}\circ g (y')$ and hence  $\Phi$ is homeomorphism onto $O'$.
So, the tuple $((U\times V)\cap X{_{f}\times_g}Y, \Phi, (O', C', E'))$ is a chart on $X{_{f}\times_g}Y$ and any two such charts are sc-smooth compatible.  Consequently, 
 $ X{_{f}\times_g}Y$ is tame. 
To prove the formula for the degeneracy index, we 
consider the projection $\pi_2\colon 
 X{_{f}\times_g}Y\to X$, $(x, y)\to y$, onto the second  component.   In view of the above discussion, 
$\pi_2$ is a local sc-diffeomorphism and consequently preserves the degeneracy index.  
Hence  $d_{X{_{f}\times_g}Y}(x,y)=d_Y(y)$. 

If both $f$ and $g$ are sc-diffeomorphisms and $X$ and $Y$ are both tame, then it follows from the previous case that  $d_{X{_{f}\times_g}Y}(x,y)=d_X(x)$, so that 
$$d_{X{_{f}\times_g}Y}(x,y)=d_X(x)=d_Y(y).$$
Finally, since $X\times Y$ is tame, then  $d_{X\times Y}(x,y)=d_X(x)+d_Y(y)$, so that 
$$d_{X\times Y}(x,y)=d_X(x)+d_Y(y)=2d_X(x)=2d_Y(y)=2d_{X{_{f}\times_g}Y}(x,y),$$
for every $(x, y)\in X{_{f}\times_g}Y$.  The proof of Proposition \ref{fibered-x} is complete.
\end{proof} 

\subsection{Strong Bundles}\label{section2.5_sb}
As a preparation for the study of sc-Fredholm sections in the next chapter we shall introduce in this section the notion of a strong bundle over a M-polyfold. 

As usual we shall first describe the new notion in local charts 
of a strong bundle and consider $U\subset C\subset E$, where $U$ is a relatively open subset of the partial quadrant $C$ in the sc-Banach space $E$. If $F$ is another sc-Banach space, we define the non- symmetric product
 $$
U\triangleleft F\index{$U\triangleleft F$}
$$
 as follows. As a set, $U\triangleleft F$ is the product $U\times F$, but it possesses  a double filtration
$(m,k)$ for $0\leq k\leq m+1$, defined by
$$
(U\triangleleft F)_{m,k}:=U_m\oplus F_k.
$$
We view $U\triangleleft F\rightarrow U$ as a trivial bundle. We define for $i=0,1$ 
the sc-manifolds $(U\triangleleft F)[i]$\index{$(U\triangleleft F)[i]$} by their filtrations
$$
((U\triangleleft F)[i])_{m}:= U_{m}\oplus F_{m+i},\quad m\geq 0.
$$
\begin{definition}\index{D-Strong bundle map}
A {\bf strong bundle map}
$$
\Phi\colon 
U\triangleleft F\rightarrow U'\triangleleft F'
$$
is a map which preserves the double filtration and is of the form
$$
\Phi(u,h)=(\varphi(u),\Gamma(u,h)),
$$
where the map $\Gamma$ is linear in $h$. In addition,  for $i=0,1$,  the maps
$$
\Phi\colon 
(U\triangleleft F)[i]\rightarrow (U'\triangleleft F')[i]
$$
are  sc-smooth.
\mbox{}\\

A {\bf strong bundle isomorphism}\index{D- Strong bundle isomorhism} is an  invertible strong bundle map whose inverse is also a strong bundle map.
\mbox{}\\

A {\bf  strong bundle retraction}  \index{strong bundle retraction} is a strong  bundle map
$$
R\colon 
U\triangleleft F\rightarrow U\triangleleft F
$$
satisfying, in addition,  $R\circ R=R$.  The map  $R$ has the form
$$
R(u,h)=(r(u),\Gamma(u,h)),
$$
where $r\colon 
U\rightarrow U$ is a sc-smooth retraction.  
A {\bf tame strong bundle retraction}\index{D- Tame strong bundle retraction}  is one for which the retraction $r$ is tame.  
\end{definition}

The condition $R\circ R=R$ of the retraction $R$ requires that 
$$\bigl( r(r(u)), \Gamma (r(u), \Gamma (u, h))\bigr)=
\bigl( r(u), \Gamma (u, h)\bigr).$$
Hence, if $r(u)=u$, then $\Gamma \bigl(u, \Gamma (u, h)\bigr)=\Gamma (u, h)$, and the bounded linear operator $h\mapsto \Gamma (u, h)\colon  F\to F$ is a projection. If $r(u)=u$ is, in addition, a smooth point, then the projection is a sc-operator.
\mbox{}\\

We continue to denote by $U\subset C\subset E$ a relatively open subset of the partial quadrant $C$ in the sc-Banach space $E$ and let $F$ be another sc-Banach space.
\begin{definition}\label{def_loc_strong_b_retract}\index{D- Strong local bundle}
A {\bf local  strong bundle retract}, denoted by 
$$
(K,C\triangleleft F, E\triangleleft F),\index{$(K,C\triangleleft F, E\triangleleft F)$}
$$
consists of a subset $K\subset C\triangleleft F$, which is the image,  
$$
K=R(U\triangleleft F), 
$$
of a strong bundle retraction
$$
R\colon 
U\triangleleft F\rightarrow U\triangleleft F
$$
of the form $R(u, h)=(r(u),\Gamma (u, h)).$
Here, $r\colon 
U\to U$ is a sc-smooth retraction onto $r(U)=O$.

The local strong bundle retract $(K,C\triangleleft F, E\triangleleft F)$ will sometimes be abbreviated by 
$$
p\colon 
K\to O\index{$p\colon K\to O$}
$$
where $p$ is the map induced by the projection onto the first factor. If the strong bundle retraction $R$  is tame, the  local strong bundle retract is called a {\bf tame local strong bundle retract}. \index{D- Tame local strong bundle retract}
\end{definition}

The retract  $K$ inherits the double filtration $K_{m,k}$\index{$K_{m,k}$} for $m\geq 0$ and $0\leq k\leq m+1$, defined by 
\begin{equation*}
\begin{split}
K_{m,k}&=K\cap (U_m \oplus F_k)\\
&=\{(u, h)\in U_m \oplus F_k\vert \, R(u, h)=(u, h)\}\\
&=\{(u, h)\in O_m \oplus F_k\vert \, \Gamma (u, h)=h\}.
\end{split}
\end{equation*}
The associated spaces $K[i]$, defined by 
$$
K[i]=K_{0,i},\quad i=0,1,\index{$K[i]$}
$$
are equipped with the filtrations
$$
K[i]_m=K_{m,m+i}
\quad \text{for all $m\geq 0$.}$$
The projection maps 
$$p\colon 
K[i]\to O$$
are sc-smooth for $i=0,1$.
They are, in fact, sc-smooth bundle maps as introduced earlier.
\begin{definition}\label{def_section_loc_strong_bundle}\index{D- Section}
A {\bf section} of the local strong bundle retract $p\colon 
K\to O$ is a map $f\colon 
O\to K$ satisfying $p\circ f=\mathbbm{1}_O$.
The section $f$ is called {\bf sc-smooth}\index{Sc-smooth section}, if $f$ is a section of the bundle 
$$p\colon 
K(0)\to O.$$
The section $f$ is called  {\bf $\ssc^{\pmb{+}}$-smooth}, \index{D- Sc$^+$-section} if $f$ is a sc-smooth section of the bundle 
$$p\colon 
K(1)\to O.$$
\end{definition}

A section $f\colon 
O\to K$ is of the form $f(x)=(x, {\bf f}(x))\in O\times F$ and the map ${\bf f}\colon 
O\to F$ is called {\bf principal part of the section}\index{Principal part of a section}. We shall usually denote the principal part with the same letter as the section.\\[0.5ex]

At this point, we can introduce the {\bf category $\mathcal{SBR}$}.\index{$\mathcal{SBR}$} Its  objects are the local,  strong bundle retractions $(K,C\triangleleft F,E\triangleleft F)$. The morphisms of the category   are maps
$\Phi\colon 
K\rightarrow K'$ between local strong bundle retracts, which are linear in the fibers and preserve the double filtrations. Moreover,  the induced maps 
$\Phi[i]\colon 
K[i]\rightarrow K'[i]$ are sc-smooth for $i=0,1$.

We recall that, by definition, {\bf the map $\Phi[i]\colon 
K[i]\rightarrow K'[i]$ between retracts is sc-smooth}, if the composition with the retraction,
$$(\Phi \circ R) \colon 
(U\oplus F)[i]\to (E'\oplus F')[i],$$
is sc-smooth.
There are  two forget functors 
$$
\text{forget}[i]\colon 
\mathcal{SBR}\rightarrow \mathcal{BR},
$$
into the category of  $\mathcal{BR}$ of sc-smooth bundle retractions introduced in Section \ref{section2.1}. They are defined by 
\begin{equation*}
\text{forget}[i](K)=K[i]
\end{equation*}
on the objects $K$ of the category, and 

\begin{equation*}
\text{forget}[i](\Phi)=\Phi[i]
\end{equation*}
on the morphisms between the objects.

We are in a position to introduce the notion of a strong bundle over a M-polyfold $X$. 
We consider a continuous surjective map
$$P\colon 
Y\to X$$
from the paracompact Hausdorff space $Y$ onto the M-polyfold $X$.  We assume, for every $x\in X$,  that  the fiber $P^{-1}(x)=Y_x$
has the structure of a Banachable  space.

\begin{definition}\label{def_strong_bundle_chart}\index{D- Strong bundle chart}
A {\bf strong bundle chart} for the bundle $P\colon 
Y\to X$ is a tuple
$$
(\Phi, P^{-1}(V), K, U\triangleleft F),\index{$(\Phi, P^{-1}(V), K, U\triangleleft F)$}
$$
in which $U\subset C\subset E$ is an open subset of the partial quadrant $C$ in the sc-Banach space $E$, and $F$ is a sc-Banach space. Moreover, 
$$K=R(U\triangleleft F)$$
is the image of a strong bundle retraction $R\colon 
U\triangleleft F\to U\triangleleft F$ of the form $R(u, h)=(r(u), \Gamma (u,h)),$ where $r\colon 
U\to U$ is a sc-smooth retraction onto the retract $O=r(U)$ of 
$(O, C, E)$ and $\Gamma$ is linear in $h$. In addition, $V\subset X$ is an open subset of $X$, homeomorphic to the retract $O$ by a  homeomorphism $\varphi\colon 
V\to O$. In addition, 
$$\Phi\colon 
P^{-1}(V)\to K$$
is a homeomorphism from $P^{-1}(V)\subset Y$ onto the retract $K$, covering the homeomorphism $\varphi\colon 
V\to O$, so that the diagram 
\begin{equation*}
\begin{CD}
P^{-1}(V)@>\Phi>>K\\
@VPVV @VVpV \\
V@>\varphi>>O
\end{CD}
\end{equation*}
commutes. The map $\Phi$ has the property that,  in the fibers over $x\in V$,  the map $\Phi\colon 
P^{-1}(x)\to p^{-1}(\varphi (x))$ is a 
bounded linear operator between Banach spaces.

{\bf Two strong bundle charts} $\Phi\colon 
P^{-1}(V)\to K$ and 
$\Phi'\colon 
P^{-1}(V')\to K'$,  satisfying $V\cap V'\neq \emptyset$ are 
{\bf compatible},\index{Compatibility of strong bundle charts} if the transition maps
$$
\Phi'\circ \Phi^{-1}[i]\colon 
\Phi (P^{-1}(V\cap V'))[i]\to 
\Phi' (P^{-1}(V\cap V'))[i]
$$
are sc-smooth diffeomorphisms for $i=0,1.$
\end{definition}

As usual, one now proceeds to define a strong  bundle atlas, consisting of compatible strong bundle charts covering $Y$ and the equivalence between two such atlases.
\begin{definition}\label{def_strong_bundle} \index{D- Strong bundle}
The continuous surjection 
$$
P\colon 
Y\to X
$$
from the paracompact Hausdorff space $Y$ onto the M-Polyfold $X$, equipped with an equivalence class of strong bundle atlases is called a {\bf strong bundle over the M-polyfold $X$}.
\end{definition}

Induced by the strong bundle charts, the M-polyfold $Y$ is equipped with a natural double filtration into subsets $Y_{m,k}$, $m\geq 0$ and $0\leq k\leq m+1$. Therefore, we can distinguish the underlying M-Polyfolds $Y[i]$ for $i = 0,1$ with the filtrations
$$
Y[i]_m = Y_{m,m+i}  ,  m \geq 0.
$$
The projections
$$
P[i] \colon 
 Y[i] \rightarrow X
$$
are sc-smooth maps between M-Polyfolds.

Correspondingly, we distinguish two types of sections of the strong bundle $P \colon 
 Y \rightarrow X$.

\begin{definition}\label{def_sc_inft_sections}
A section of the strong bundle $P \colon 
 Y \rightarrow X$ is a map $f \colon 
 X \rightarrow Y,$ satisfying $P\circ f = {\mathbbm 1}_X.$

The section $f$ is called a {\bf sc-smooth section}\index{D- Sc-smooth section}, if $f$ is an sc-smooth section of the bundle
$$
P[0] \colon 
 Y[0] \rightarrow X.
$$
The section $f$ is called a ${\bf sc^+}$-{\bf section} of $P \colon 
 Y \rightarrow X$, \index{D- Sc$^+$-section} if $f$ is an sc-smooth section of the bundle
$$ 
P[1] \colon 
 Y[1] \rightarrow X.
$$
\end{definition}
If we say $f$ is an sc-section of $P:Y\rightarrow X$ we mean that it is an sc-smooth section of $Y[0]\rightarrow X$.
An sc$^+$-section or sc$^+$-smooth section of $P:Y\rightarrow X$ is an sc-smooth section of $Y[1]\rightarrow X$.

\begin{definition}[{\bf Pull-back bundle}]
Let $P\colon Y\to X$ be a strong bundle over the M-polyfold $X$ and let  $f\colon Z\to X$  be an sc-smooth map  from the M-polyfold $Z$ into $X$. The pull-back bundle of $f$,
$$P_f\colon f^\ast Y\to Z,$$ 
is defined by the set 
$f^\ast Y=\{(z, y)\in Z\times Y\, \vert \, P(y)=f(z)\}$
and the projection $P_f (z, y)=z$, so that with the projection 
$P'\colon f^\ast Y\to Y$, defined by $P'(z, y)=y$, the diagram 
\begin{equation*}
\begin{CD}
f^\ast Y@>P'>>Y\\
@VP_fVV @VVPV \\
Z@>f>>X
\end{CD}
\end{equation*}
commutes.
\end{definition}

As already shown in \cite{HWZ2}, Proposition 4.11, the strong M-polyfold structure $P$ induces a natural strong $M$- polyfold structure of the pull-back bundle $P_f$.
\begin{proposition}\label{pull_back_strong_bundle}\index{P- Pull-back of strong bundles}
The pull-back bundle $P_f\colon f^\ast Y\to Z$ 
carries a natural structure of a strong M-polyfold bundle over the  M-polyfold $Z$.
\end{proposition}
The easy proof is left to the reader.  
\subsection{Appendix}

\subsubsection{Proof of Proposition  \ref{smooth_retract}}\label{A2.1}
\begin{proof}[Proof of Proposition \ref{smooth_retract}]
The $C^\infty$-retraction $r$ satisfies $r\circ r=r$ and hence,  by the chain rule,  
$$Dr(r(x))\circ Dr(x)=Dr (x)\quad \text{for every $x\in U$}.$$
Therefore, if $r(x)=x$, then 
 \begin{equation}\label{c_infty_retract_eq0}
 Dr (x)\circ Dr (x)=Dr(x)
  \end{equation}
and hence the linear operator  $Dr(x)\in \mathscr{L}(E, E)$ is a projection at every point $x\in O$.

Now we take a point $x\in O=r(U)$ and,  for simplicity,  assume that $x=0$.  In view of \eqref{c_infty_retract_eq0}, the Banach space $E$ splits into $E=R\oplus N$, where 
\begin{align*}
R&=\text{range}\ Dr(0)=\text{ker}\ (\mathbbm{1}-Dr (0))\\
N&=\text{ker}\ Dr(0)=\text{range}\ (\mathbbm{1}-Dr (0)).
\end{align*}
According to the splitting $E=R\oplus N$, we use the equivalent norm $\abs{(a, b)}=\max\{\norm{a}, \norm{b}\}$. By $B(\varepsilon)$ we denote an open ball of radius $\varepsilon$ (with respect to the norm $\abs{\cdot }$)centered at the origin.

Now we introduce  the map $f\colon 
 (R\oplus  N)\cap U\to E$, defined   by 
$$f(a, b)=r(a)+(\mathbbm{1}-Dr(a))b.$$
At $(a, b)=(0, 0)$, we have $f(0, 0)=r(0)=0$ and 
$$Df(0, 0)[h, k]=h+k$$
for all $(h, k)\in R\oplus N.$ Consequently, in view of the inverse function theorem, $f$ is a local $C^\infty$-diffeomorphism, and we assume without loss of generality that $f$ is a diffeomorphism on $U$. 

We claim that there exist positive numbers $\delta$,  such that, if 
 $f(a, b)\in O=r(U)$ for $(a, b)\in B(\delta)$, then $b=0$.   The proposition then follows by setting  $W=B(\delta)$, $V=f(W)$,  and defining  the map $\psi\colon 
W\to V$  by $\psi=(f\vert V)^{-1}$.

It remains to  prove the  claim that $b=0$.  Since  $r$ is smooth, we find a constant $\varepsilon>0$, such that  
 $B(2\varepsilon)\subset U$ and 
 \begin{equation}\label{c_infty_retract_eq1}
 \abs{Dr(x)-Dr(0)}<\frac{1}{3}\quad \text{for all $\abs{x}<\varepsilon$}.
 \end{equation}
Moreover,  since $r(0)=0$, there exists a constant $0<\delta <\frac{3}{4}\varepsilon$, such that 
  \begin{equation}\label{c_infty_retract_eq2}
 \abs{r(x)}<\varepsilon \quad \text{for all $\abs{x}<\delta$}.
 \end{equation}
If $x\in B(\delta)$, then for $ \abs{h}<\varepsilon$, 
 \begin{equation}\label{c_infty_retract_eq3}
  r(r(x)+h)=r(r(x))+Dr(r(x))h+o(h)=r(x)+Dr(r(x))h+o(h),
 \end{equation}
where $o(h)$  is an $E$-valued map satisfying $\frac{o(h)}{\abs{h}}\to 0$, as $\abs{h}\to 0$.
Taking a smaller $\varepsilon$, we may assume that 
 \begin{equation}\label{c_infty_retract_eq3a}
 \abs{o(h)}<\frac{\abs{h}}{2}\quad  \text{ for all $\abs{h}<\varepsilon. $}
 \end{equation}
Let  $x=(a, b)\in B(\delta)$  and $f(a, b)\in O$. This means that  
 \begin{equation}\label{c_infty_retract_eq4}
r\bigl(r(a)+(\mathbbm{1}-Dr(a))b\bigr)=r(a)+(\mathbbm{1}-Dr(a))b.
 \end{equation}
By \eqref{c_infty_retract_eq1} and the fact that, $Dr(0)b=0$ for  $b\in N$, the norm of  $(\mathbbm{1}-Dr(a))b$ for $\abs{a}<\delta$ can be estimated as 
\begin{equation}\label{c_infty_retract_eq5}
\abs{b-Dr(a))b}=\abs{b-\bigl(Dr(a))-Dr(0)\bigr)b} \leq \frac{4}{3}\abs{b}.
\end{equation}
Inserting $x=a$ and $h=(\mathbbm{1}-Dr(a))b$ into  \eqref{c_infty_retract_eq3} and using  the identity  
\eqref{c_infty_retract_eq4}, we obtain
\begin{equation}\label{c_infinity_retraction_eq7} 
Dr(r(a))h=h+o(h).
\end{equation}
The left-hand side is equal to 
\begin{equation*}\label{c_infinity_retraction_eq8} 
Dr(r(a))h=Dr(r(a))b-Dr(r(a))Dr(a)b=Dr(r(a))b-Dr(a)b,
\end{equation*}
and the right-hand side is equal to $h+o(h)=b-Dr(a))b+o(h)$,
%
so that 
\eqref{c_infinity_retraction_eq7}  takes the form 
\begin{equation}\label{c_infty_retract_eq10}
Dr(r(a))b=b+o(h).
\end{equation} 
Arguing by contradiction, we assume  that $b\neq 0$. Since   $\abs{r(a)}< \varepsilon$,  the norm of the left-hand side is, in view of  \eqref{c_infty_retract_eq1},  bounded from above by 
\begin{equation}\label{c_infty_retract_eq11}
\abs{Dr(r(a))b}=\abs{Dr(r(a))b-Dr(0)b}<\frac{1}{3}\abs{b}.
\end{equation} 
On the other hand, using \eqref{c_infty_retract_eq3a} and \eqref{c_infty_retract_eq5}, we estimate  the norm of the right-hand side  of \eqref{c_infty_retract_eq10} from below as  
\begin{equation}\label{c_infty_retract_eq12}
\abs{b+o(h)}\geq \abs{b}-\frac{1}{2}\abs{h} \geq   \abs{b}-\frac{2}{3}\abs{b}=\frac{1}{3}\abs{b}. 
\end{equation} 
Consequently, $\frac{1}{3}\abs{b}>\frac{1}{3}\abs{b}$ 
which is absurd. Therefore,  $b=0$ and the proof is complete.
\end{proof}

\subsubsection{Proof of Theorem \ref{X_m_paracompact}}\label{A2.2}
We  make use of the following theorem of metrizability due to Yu. M. Smirnov \cite{Smirnov}.
\begin{theorem}[{\bf {Yu. M. Smirnov}}]\label{Smirnov}\index{T- Smirnov's metrizability theorem}
Let $X$ be a space that is paracompact Hausdorff, and assume that  every point has an open  neighborhood in $X$ that is metrizable. Then $X$ is metrizable.
\end{theorem}
The following result is an equivalent definition of paracompactness.
\begin{lemma}[\cite{Michael}, Lemma 1] \label{equivalent_definitions_paracompact}\index{L- Characterization of paracompactness}
A  regular Hausdorff space $X$ is paracompact, if and only if every open cover of $X$ has a locally finite refinement consisting of closed sets.
\end{lemma}

We shall use the above lemma in the proof of the following result.

\begin{proposition}\label{union_of_paracompact}\index{P- Paracompactness}
Let $Y$ be  a regular topological space, and let $(Y_i)_{i\in I}$ be a locally  finite family of closed subspaces  of $Y$, so that  $Y=\bigcup_{i\in I}Y_i$. If every subspace $Y_i$ is paracompact, then 
$Y$ is  paracompact.
\end{proposition}
\begin{proof}
Given an open cover  $\mathscr{U}=(U_j)_{j\in J}$ of $Y$,  it is, in view of 
Lemma \ref{equivalent_definitions_paracompact}, enough to show, that there exists a closed locally finite refinement 
$\mathscr{C}=(C_j)_{j\in J}$ of $\mathscr{U}$.  In order to prove this, we  consider for every $i$  the cover  $\mathscr{Y}_i=(Y_i\cap U_j)_{j\in J}$ of $Y_i$, consisting of open sets in $Y_i$.  Since $Y_i$ is paracompact, Lemma \ref{equivalent_definitions_paracompact} implies that there exists a closed locally finite refinement  $(C_j^i)_{j\in J}$ of  $\mathscr{Y}_i$.  The sets $C_j^i$ are closed in $Y$ and by  definition of refinement satisfy 
\begin{equation}\label{eq1}
C_j^i\subset Y_i\cap U_j\quad \text{and}\quad \bigcup_{j\in J}C^i_j=Y_i.
\end{equation}  
Now, for every $j\in J$,  we define the set
$$C_j:=\bigcup_{i\in I}C_j^i$$
and claim that the family $\mathscr{C}=(C_j)_{j\in J}$ is a closed locally finite refinement of $\mathscr{U}$.  We start with  showing that  $C_j$ is closed. Let 
 $x\in X\setminus C_j=\bigcap_{i\in I}(X\setminus C^i_j)$.   Since $(Y_i)_{i\in I}$ is locally finite,  we find  an open neighborhood  $V(x)$ of $x$ in $Y$, which  intersects $Y_i$ for at most finitely many indices $i$, say for $i$ belonging to the  finite subset $I'\subset I$. 
If $i\not \in I'$, then, by \eqref{eq1},  $V(x)\subset  Y\setminus Y_i\subset Y\setminus C_j^i$,  so that $V(x)\subset \bigcap_{i\not \in  I'}(Y\setminus Y_i).$  
Then the set $U(x)=V(x)\cap  \bigcap_{l\in I'}(Y\setminus C_j^i)$ is an open neighborhood of $x$ in $Y$ contained in $Y\setminus C_j$. This shows that  $C_j$ is closed as claimed.
That  $\mathscr{C}$ is a cover of $X$ and a refinement of $\mathscr{U}$ follows from  \eqref{eq1}, 
\begin{gather*}
\bigcup_{j\in J}C_j=\bigcup_{j\in J}\bigcup_{i\in I}C_j^i=\bigcup_{i\in I}\bigcup_{j\in J}C_j^i=\bigcup_{i\in I}Y_i=X\\
\intertext{and}
V_j=\bigcup_{i\in I}C_j^i\subset \bigcup_{i\in I}Y_i\cap U_j=U_j. 
\end{gather*}
It remains to show that $\mathscr{C}$ is locally finite.  Local finiteness of $(Y_i)_{i\in I}$ implies that a given point $x\in Y$ belongs to finitely many $Y_i$'s, say $x\in Y_i$, if and only if $i$ belongs to the finite subset  $I'\subset I$.  Moreover, there exists an open neighborhood $V(x)$ of $x$ in $Y$ intersecting at most finitely many $Y_i$'s. 
 Since $Y$ is regular, we can replace  $V(x)$  by  a smaller  open set intersecting only those $Y_i$'s whose indices $i$ belong to $ I'$.  Also,  $(C_j^i)_{j\in J}$ is a locally finite family in $Y_i$, so that  there exists an open neighborhood $V_i(x)$ of $x$ in $Y_i$ intersecting $C_j^i$  for, at most, finitely many indices $j$,  which belong to the finite subset $J_i\subset J$. Each $V_i(x)$ is of the form 
$V_i(x)=U_i(x)\cap Y_i$ for some open neighborhood $U_i(x)$ of $x$ in $Y$. Then $U(x)=V(x)\cap \bigcap_{i\in I'}U_i(x)$ is an open neighborhood of $x$ in $Y$  intersecting 
$Y_i$ only for $i\in I'$ and the set $U(x)\cap Y_i$ intersects  $C_j^i$ only for  $j\in J_i$.   This implies that  $U(x)$ can have a nonempty intersection with $C_j$ only for $j\in \bigcup_{i\in I'}J_i$.  Hence,  the family $\mathscr{C}=(C_j)_{j\in J}$ is locally finite,  and the proof is complete.  

\end{proof}
Now we are ready to prove the theorem.
\begin{proof}[Proof of Theorem \ref{X_m_paracompact}]
We start with $m=0$ and take an atlas 
$$
(U^j, \varphi^j, (O^j, C^j, E^j))_{j\in J}.
$$  
Since $\varphi^j\colon 
U^j\to O^j$
 is a homeomorphism and $O^j$ is a metric space, $U^j$ is metrizable. 
Hence, $X$ is locally metrizable. Since by assumption, $X$ is Hausdorff and paracompact,  the Smirnov metrizability theorem implies that $X$ is metrizable. 

In order to prove that $X_m$ is metrizable for $m\geq 1$,  we fix $m\geq 1$  and consider the topological  space $X_m$.  Since by assumption, $X$ is Hausdorff,  given two distinct points $x$ and $x'$ in $X_m$, there exist  two disjoint open neighborhoods $U$ and $U'$ of $x$ and $x'$ in $X$.  Hence, the sets $U_m$ and $U'_m$ are disjoint open neighborhoods of $x$ and $x'$ in $X_m$, so that also $X_m$ is a  Hausdorff space. 
Moreover, the maps $\varphi\colon 
U_m\to O^j_m$ are homeomorphisms and since $O^j_m$ is a metric space, $U^j_m$ are metrizable. So,  $X_m$ is locally metrizable. To prove that $X_m$ is metrizable, it suffices to show, that, in view of the Smirnov metrizability theorem, $X_m$ is paracompact.  Using  the paracompactness  of $X$ and 
Theorem \ref{equivalent_definitions_paracompact}, we find a closed locally finite refinement $\mathscr{C}=(C_j)_{j\in J}$. In particular, $C_j\subset U_j$, so that $C^j_m\subset U^j_m$. Since $U^j_m$ is metrizable, it is also paracompact, so that $C^j_m$ is paracompact as a closed subspace of $U^j_m$.   Hence, $\mathscr{C}_m=(C^j_m)_{j\in J}$ is a locally  finite family of closed subsets of $X_m$ and each $C^j_m$ is paracompact. Thus,  by Proposition \ref{union_of_paracompact}, the space $X_m$ is paracompact, and since it is Hausdorff and locally metrizable, it is metrizable.

Finally, choosing a metric $d_m$ defining the topology on $X_m$ , we set
$$
d(x,y)=\sum_{m=0}^{\infty} \frac{1}{2^m}\cdot  \frac{d_m(x,y)}{1+d_m(x,y)}.
$$
The metric $d$ defines the topology on $X_\infty$.
\end{proof}

\subsubsection{Proof of Proposition \ref{op}}\label{A2.11}
\begin{proof}[Proof of Proposition \ref{op}]
(1)\, In order to prove that $\mathscr{B}$ defines a basis for a topology on $TX$, we take two  sets   $\wt{W}_1$ and $\wt{W}_2$ in  $\mathscr{B}$ and assume that $\alpha=[(x, \varphi, V, (O, C, E), h)]\in \wt{W}_1\cap \wt{W}_2$. We claim that  there exists a set $\wt{W}\in \mathscr{B}$, satisfying  $\alpha\in \wt{W}\subset \wt{W}_1\cap \wt{W}_2.$ By definition,  $\wt{W}_i=(T\varphi_i)^{-1}(W_i)$, where $(\varphi_i, V_i, (O_i, C_i, E_i))$ is a chart and $W_i$ is an open subset of $TO_i$, containing 
$$
\alpha=[(x_i, \varphi_i, V_i, (O_i, C_i, E_i), h_i)]\quad \text{for $i=1,2$}.
$$
This means that 
\begin{equation}\label{basis_top_1}
x_i=x\quad \text{and}\quad h_i=T(\varphi_i \circ \varphi^{-1})(\varphi (x))h\quad \text{for $i=1,2$},
\end{equation}
and, moreover, $(\varphi_i (x_i), h_i)=(\varphi_i (x), h_i)\in W_i$. We define  $W_i'= [ T(\varphi_i \circ \varphi^{-1})(\varphi (x))]^{-1}(W_i)$, and observe that the $W_i$'s  are open subsets of $TO_i$. By \eqref{basis_top_1},  $(\varphi (x), h)\in W_1'\cap W_2'$ and if  $W=W'_1\cap W_2'$ and $\wt{W}=(T\varphi )^{-1}(W)$, then 
$$[(x, \varphi, V, (O, C, E), h)]\in \wt{W}\subset \wt{W}_1\cap \wt{W}_2.$$
Consequently, $\mathscr{B}$ defines a topology on $TX$. To prove that this topology is Hausdorff, we take two distinct elements 
$$\alpha=[(x, \varphi, V, (O, C, E), h)]\quad \text{and}\quad \alpha'=[(x', \varphi', V', (O', C', E'), h')]$$
in the tangent space $TX$.  Since $\alpha\neq \alpha'$, either $x\neq x'$ or if $x=x'$, then $h'\neq T(\varphi'\circ \varphi)^{-1}(\varphi (x))h.$ In the first case $x\neq x'$ we may replace $V$ and $V'$ by smaller open neighborhoods of $x$ and $x'$, so that $V\cap V'=\emptyset$, and then replace $(O, C, E)$, resp. $(O, C', E')$,  by the retracts $(\varphi (V), C, E)$, resp. $(\varphi'(V'), C', E')$. If $W$ (resp. $W'$)  is an open neighborhood of  $(\varphi (x), h)$ (resp. $(\varphi'(x'), h')$)  in $TO$ (resp. $TO'$), then $\wt{W}=(T\varphi)^{-1}(W)$ (resp. $\wt{W}'=(T\varphi')^{-1}(W')$ is an open neighborhood  
of $\alpha$ (resp. $\alpha'$) in $TX$ and  $\wt{W}\cap \wt{W'}=\emptyset$.  In the second case $x=x'$ and $h'\neq T(\varphi'\circ \varphi)^{-1}(\varphi (x))h$. We choose an open neighborhood 
 $W$ of $(\varphi' (x), h')$ in $TO'$ and an open neighborhood $W$ of $(\varphi (x), h)$ in $TO$, so that $W'\cap T(\varphi')^{-1}(W)=\emptyset$. 
  Then $\alpha\in \wt{W}=(T\varphi)^{-1}(W)$ and $\alpha'\in \wt{W'}= (T\varphi')^{-1}(W')$. Moreover, both sets are open  and their intersection is empty. Consequently, the topology defined by $\mathscr{B}$ is Hausdorff.\\[0.7ex]
\noindent (2)\,   We start by proving that  $p\colon 
TX\to X^1$ is an open map. It suffices to show that $p(\wt{W})$ is open in $X^1$ for every element
 $\wt{W}\in\mathscr{B}.$
Let $(\varphi, V, (O, C, E))$ be a chart on $X$ and let $T\varphi\colon 
TV\to TO$ be the associated map defined above and introduce $\wt{W}=(T\varphi )^{-1}(W)$ for the open subset $W$ of $TO$. 
We denote by $\pi\colon 
TO\to O_1$ the projection onto the first factor. This map is continuous and open. Moreover,  
$$p(\wt{W})=\varphi^{-1}\circ \pi\circ (T\varphi)(\wt{W}).$$
Since, by construction, the map $T\varphi\colon 
TV\to TO$ is open and $\varphi\colon 
V\to O$ is a homeomorphism, the composition on the right hand side is an open subset of $X^1$. Hence, $p(\wt{W})$ is open in $X^1$ as claimed.

To show that the projection map  $p\colon TX\to X^1$ is continuous, it suffices to show, that given a chart $(V, \varphi, (O, C, E))$ on $X$ and an open subset $U$ of $X^1$, satisfying $U\subset V_1$, the preimage $p^{-1}(U)$ is open. For such a chart and open set $U$ we have 
$$p^{-1}(U)=(T\varphi )^{-1}\bigl((\varphi (U)\times E)\cap TO\bigr).$$
Since   $(\varphi (U)\times E)\cap TO$ is open in $TO$, the set on the left-hand side belongs to $\mathscr{B}$. Hence, the set $p^{-1}(U)$ is open, and the projection $p$ is continuous as claimed. \\[0.7ex]

\noindent (3)\,  
We start with an atlas $\mathscr{V}=(V^j, \varphi^j, (O^j, C^j, E^j))_{j\in J}$, such  that the family $\mathscr{V}=(V^j)_{j\in J}$ of domains is an open,  locally finite cover of $X$.  The associated maps 
$T\varphi^j\colon 
TV^j\to TO^j$ are homeomorphisms, and since $TO^j$ is metrizable, the same holds for the open sets $TV^j$.  Hence, $TX$ is locally metrizable.  To show that $TX$ is metrizable, it remains  to show that $TX$ is paracompact.  By Theorem \ref{equivalent_definitions_paracompact}, there exists a closed, locally finite refinement $\mathscr{C}=(C^j)_{j\in J}$ of $\mathscr{V}$, so that $\mathscr{C}_1=(C^j_1)_{j\in J}$ is a closed, locally finite refinement of $\mathscr{V}_1=(V^j_1)_{j\in J}$.  The sets $K^j=TO\vert \varphi(C^j_1):=\bigcup_{x\in \varphi (C^j_1)}T_xO$ are closed in $TO$, so that 
the sets $\wt{K^j}:=(T\varphi^j)^{-1}(K^j)$ are closed subsets of $TV$. In particular, each $\wt{K^j}$ is paracompact as a closed subset of metrizable space.  Also, the family $(\wt{K^j})_{j\in J}$ is locally finite. Indeed, let $\alpha=[(x, V, \varphi, (O, C, E), h)]\in TX$. Then there exists an open neighborhood $U(x)$ of $x$ in $X_1$ intersecting at most finitely many $C^j_1$'s, say with indices $j$, belonging to a finite subset $J'\subset J$. Moreover, since $X_1$ is regular, $U(x)$ can be taken so small that also $x\in C^j_1$ for $j\in J'$. Then, setting $W(x):=TO\vert \varphi (U (x)):=\bigcup_{y\in \varphi (U(x))}T_yO$, the set $W(x)$ is an open subset of $TO$ and $\wt{W(x)}:=(T\varphi)^{-1}(W(x))$ intersects only those $\wt{K^j}$ whose indices $j$ belong to $J'$. 
Now applying  Proposition \ref{union_of_paracompact}, we conclude, that the tangent space $TX$ is paracompact and hence metrizable, in view of the Smirnov metrizability  theorem. This finishes the proof of the proposition.
\end{proof}

\pagebreak
\section{Basic Sc-Fredholm Theory}
In this chapter  we start with  the Fredholm theory in the sc-framework. Since sc-maps are more flabby than $C^\infty$-maps, 
we do not have an implicit function theorem for all sc-smooth maps.  However, for a restricted class, which occurs in the applications 
of the theory,  such a theorem is available. We start with a figure  illustrating  a finite-dimensional M-polyfold showing  what to expect from the Fredholm theory. Of course,  our interest is in the infinite-dimensional case.

\begin{figure}[htb]
\centering
\def\svgwidth{50ex}
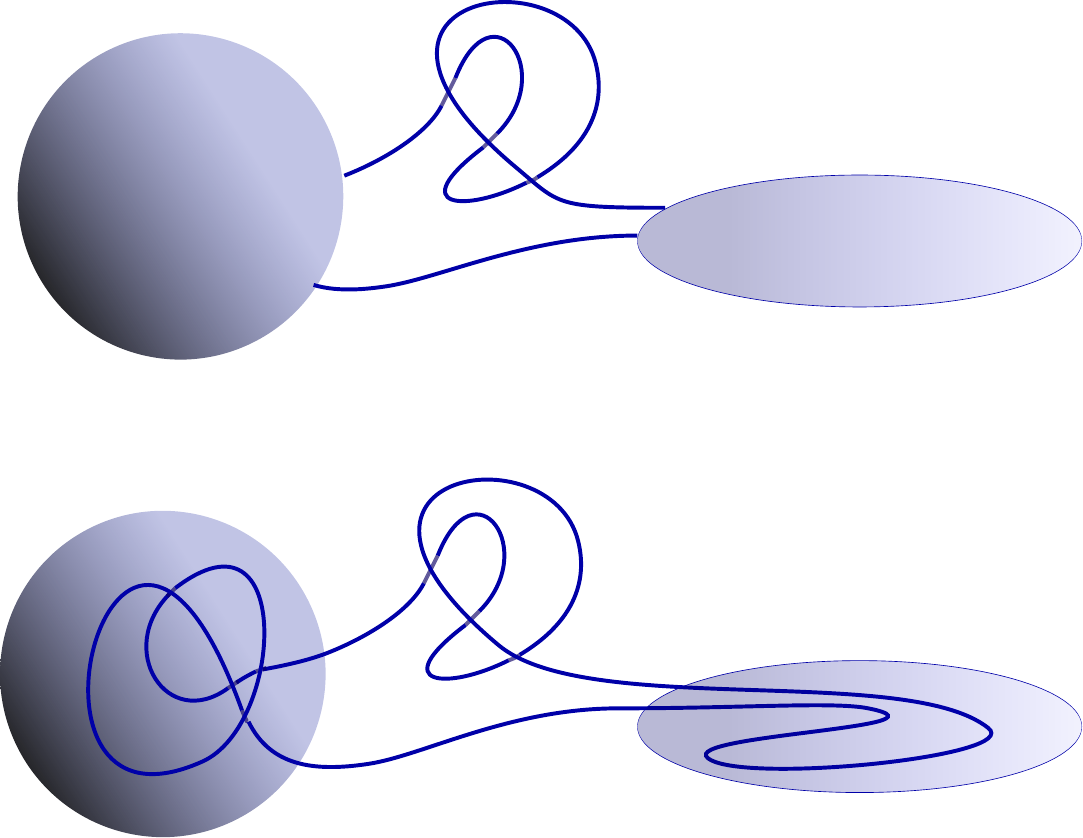
\caption{The top part of the figure shows a finite-dimensional M-polyfold, which is a smooth space.  Assume that we have a bundle in this extended category which 
has  jumps of dimensions coordinated with the base. Then a smooth generic section would produce a solution set which looks like the one depicted in the lower part of the figure.}\label{fig:pict2}
\end{figure}

\subsection{Sc-Fredholm Sections and Some of the Main Results}
The section is devoted to the basic notions and the description of the results leading  to  implicit function theorems. Our overall goal is the notion of a sc-Fredholm section  of a tame strong bundle $P\colon Y\rightarrow X$, as defined in
 Definition \ref{def_strong_bundle}. 
Tameness of the bundle requires that  $X$ is a tame M-polyfold as defined in  
Definition \ref{def_tame_m-polyfold}. The section will end with some useful implicit function theorems.

The more sophisticated perturbation and transversality results are described in a later section.
\mbox{}\\

We start by introducing various types of germs in the sc-context. As usual we denote  by $E$ be a sc-Banach space and by $C\subset E$ a partial  quadrant of $E$.  The sc-Banach space $E$ is equipped with the filtration 
$$E_0=E\supset E_1\supset\cdots \supset E_\infty =\bigcap_{m\geq 0}E_m.$$

\begin{definition}\label{def_germ_nbgh}\index{D- Sc-germ of neighborhoods}
A  {\bf sc-germ of neighborhoods around $0\in C$}, 
denoted by ${\mathcal O}(C,0)$,  is 
a decreasing sequence 
$$
U_0\supset U_1\supset U_2\supset \cdots 
$$
where $U_m$ is a relatively open neighborhoods of $0$ in 
$C\cap E_m$. 

The {\bf tangent of ${\mathcal O}(C,0)$},  denoted by $T{\mathcal O}(C,0)$, is  the decreasing sequence 
$$
U_1\oplus E_0\supset U_2\oplus E_1\supset\cdots \, .
$$
\end{definition}

A special example of a sc-germ of neighborhoods around $0\in C\subset E$ is  a relatively open neighborhood $U$ of $C$ containing $0$ which is equipped with the induced sc-structure defined by the filtration $U_m=U\cap E_m$, $m\geq 0$.  If $E$ is infinite dimensional, the sets $U_m$ is this example 
are not bounded in $E_m$, since the inclusions $E_m\to E_0$ are compact operators. 
Another example is the decreasing sequence $U_m=U
\cap B_{1/m}^{E_m}(0)\cap E_m$, where $B_{1/m}^{E_m}(0)$ is the open ball in $E_m$ centered at $0$ and radius $1/m$, presents a sc-germ of neighborhoods. Here the sets $U_m$ are bounded in $E_m$ for $m>0$. We point out that the size of the sets  $U_m$ in Definition \ref{def_germ_nbgh} does not matter. In the applications the size of $U_m$ quite often decreases rapidly.

\begin{definition}\label{germy}\index{D- Sc-mapping germ}
A  {\bf $\ssc^{\pmb{0}}$-germ} $f\colon {\mathcal O}(C,0)\rightarrow F$\index{$f\colon {\mathcal O}(C,0)\rightarrow F$} into the sc-Banach space $F$, is a continuous map
$f\colon U_0\rightarrow F$ such that $f(U_m)\subset F_m$ and
$f\colon U_m\rightarrow F_m$ is continuous.
A  {\bf $\ssc^{\pmb{1}}$-germ} $f\colon {\mathcal O}(C,0)\rightarrow F$ is a  $\ssc^0$-germ which is of class  $\ssc^1$ in the sense, that there exists 
for every $x\in U_1$  a bounded linear operator $Df(x)\in L(E_0,F_0)$ such that for $h\in U_1$ with $x+h\in U_1$, 
$$
\lim_{\abs{h}_1\rightarrow 0} \frac{\abs{f(x+h)-f(x)-Df(x)h}_0 }{\abs{h}_1}=0.
$$
Moreover,  $Tf\colon U_1\oplus E_0\rightarrow TF$,  defined by $Tf(x,h)=(f(x),Df(x)h)$,  satisfies $Tf(U_{m+1}\oplus E_m)\subset F_{m+1}\oplus F_m$ and
$$
Tf\colon T{\mathcal O}(C,0)\rightarrow TF\index{$Tf\colon T{\mathcal O}(C,0)\rightarrow TF$}
$$
is a $\ssc^0$-germ. We say $f$ is a $\ssc^2$-germ provided $Tf$ is  $\ssc^1$, etc.  If the the germ $f$ is a $\ssc^k$-germ for every $k$ we call it a sc-smooth germ.  
If we write $f\colon {\mathcal O}(O,0)\rightarrow (F,0)$ we indicate
that $f(0)=0$.
\end{definition}
We shall be mostly interested in $\ssc^\infty$-germs $f\colon {\mathcal O}(C,0)\rightarrow F$. \index{Sc$^\infty$-germs}

From Definition \ref{def_strong_bundle_chart} we recall the strong bundle chart $(\Phi, P^{-1}(V), K, U\triangleleft F)$ of a strong bundle $P\colon Y\to X$ over the M-polyfold $X$, illustrated by the commutative diagram
\begin{equation*}
\begin{CD}
P^{-1}(V)@>\Phi>>K\\
@VPVV @VVpV \\
V@>\varphi>>O .
\end{CD}\,
\end{equation*}
In the diagram, $V\subset X$ is an open set and the maps
$\Phi$ and $\varphi$  are homeomorphisms. Moreover, 
$K=R(U\triangleleft F)$ is the image of the strong bundle retraction $R$, and $O=r(U)$ is the image of the sc-smooth retraction $r\colon U\to U$ of the relatively open set $U$ of the the partial quadrant $C$ in the sc-Banach space $E$. 

\begin{definition}\index{D- Sc-section germ}
A {\bf sc-smooth section germ} $(f,x_0)$ of the strong bundle $P\colon Y\to X$ is a continuous section $f\colon V\rightarrow P^{-1}(V)$ on the open neighborhood $V$ of the smooth point $x_0$,  for which the following holds.

There exists a strong bundle chart 
$(\Phi,  P^{-1}(V), K, U\triangleleft F)$  satisfying $\varphi (x_0)=0\in O$ in which the principal part $\wh{\bf f}$ of the local continuous section $\wh{f}=\Phi \circ f\circ \varphi^{-1}\colon O\to K\subset U\triangleleft F$ has the property that the composition 
$$\wh{\bf f}\circ r\colon {\mathcal O}(C, 0)\to F$$
is a sc-smooth germ as defined in 
Definition \ref{germy}.
\end{definition}

To recall, the section $\wh{f}\colon  O\to K\subset U\triangleleft F$ is  of the form $\wh{f}(p)=(p, \wh{\bf f}(p))$ for $p\in O$ and $\wh{\bf f}\colon O\to F$ is called its {\bf principal part}.  By abuse of the notation we shall often use the same letter for the principal part as for the section.

In the next step we introduce the useful  notion of a  filling of a  sc-smooth section germ $(f,0)$  of  a tame strong local bundle $K\rightarrow O$ near the given smooth point  $0$.  The notion of  a filling is a new concept specific to the world of retracts. In all known applications it deals successfully with bubbling-off phenomena and similar singular phenomena.

\begin{definition}[{\bf Filling}]\label{x-filling}\index{D- Filling}
We consider a tame strong local bundle $K\to O$. We recall that $K=R(U\triangleleft F)$  where  $U\subset C\subset E$ is a relatively open neighborhood  of $0$ in the partial quadrant $C$ of the  sc-Banach space $E$ and  $F$ is a sc-Banach space. Moreover,  $R$ is a strong bundle retraction of the form 
$$R(u, h)=(r(u), \rho(u)(h)), $$
covering the tame retraction $r\colon U\to U$ onto $O=r(U),$
and $\rho (u)\colon F\to F$ is a bounded linear operator.  
We  assume that $r(0)=0$. 

A sc-smooth section germ 
$(f, 0)$ of the bundle $K\to O$ possesses a  {\bf filling}\index{D- Filling of a sc-smooth section germ}
if there exists a  sc-smooth section germ $(g, 0)$ of the bundle $U\triangleleft F\to U$ extending $f$ and having  the following properties.
\begin{itemize}
\item[(1)]  $f(x)=g(x)$  for $x\in O$ close to $0$.
\item[(2)] If $g(y)=\rho (r(y))g(y)$ for a point $y\in U$ near $0$, then $y\in O$.
\item[(3)] The linearization of the map
$
y\mapsto  [\mathbbm{1} -\rho(r(y))]\cdot g(y)
$
at the point $0$, restricted to $\ker(Dr(0))$, defines a topological linear  isomorphism
$$
\ker(Dr(0))\rightarrow \ker(\rho (0)).
$$
\end{itemize}

\end{definition}

The crucial property of a filler is the fact that the 
solution sets $\{y\in O\, \vert \, f(y)=0\}$ and $\{y\in U\, \vert \, g(y)=0\}$ coincide near $y=0$. 
Indeed, if $y\in U$ is a solution of the filled section $g$ so that $g(y)=0$, then it follows from (2) that $y\in O$ and from (1) that $f(y)=0$. The section $g$ is, however, much  easier to analyze than  the section $f$,  whose domain of definition has a rather complicated structure. It turns out that in the applications these extensions $g$ are surprisingly easy to  detect. In the Gromov-Witten theory and the SFT they  seem almost canonical.

 The condition (3) plays a role in the comparison of the linearizations $Df(0)$ and $Dg(0)$, assuming that $f(0)=0=g(0)$, as we are going to explain next.
 
 It  follows from the definition of a retract that
$\rho (r(y))\circ \rho (r(y))=\rho (r(y)).$ Hence, since $y=0\in O$ we have $r(0)=0$ and $\rho(0)\circ \rho(0)=\rho (0)$ so that $\rho (0)$ is a linear sc-projection in $F$ and we obtain the sc-splitting
$$
F=\rho (0)F\oplus (\mathbbm{1} -\rho (0))F.
$$
Similarly, it follows from $r(r(y))=r(y)$ for $y\in U$ that $Dr (0)\circ Dr (0)=Dr(0)$ so that $Dr (0)$ is a linear sc-projection in $E$ which gives rise to the sc-splitting
$$
\alpha \oplus \beta \in  E=Dr (0)E\oplus (\mathbbm{1} -Dr (0))E.
$$
We recall that the linearization $Tf(0)$  of the section $f\colon O\to K$ at $y=0=r(0)$ is defined as the restriction of the derivative $D(f\circ r)(0)$ of the map $f\circ r\colon U\to F$ to $T_0O$. From $\rho (r(y))f(r(y))=f(r(y))$ for $y\in U$  close to $0$,  we obtain, using $f(0)=0$,  by linearization 
at $y=0$ the relation $\rho (0)Tf(0)=Tf(0)$. From $g(r(y))=f(r(y))$ for $y\in U$ near $0$ we deduce $Tg(0)=Tf(0)$ on $T_0O$. 
 From the identity
 \begin{gather*}
(\mathbbm{1} -\rho (r(y))g(r(y))=0\quad \text{for $y\in U$ near $0$},
\end{gather*}
we deduce, using $g(0)=0$, the relation
$(\mathbbm{1} -\rho (0))Dg (0)\circ Dr(0)=0$. 
Hence  the matrix representation of $D g (0)\colon E\to F$ with respect to the above splittings of $E$ and $F$ looks as follows,$$
Dg(0)\begin{bmatrix}\alpha \\ \beta
\end{bmatrix}=
\begin{bmatrix}Tf (0)&\rho (0)Dg(0)({\mathbbm 1}-Dr(0))\\0&(\mathbbm{1} -\rho (0)) Dg (0)({\mathbbm 1}-Dr(0))\end{bmatrix}\cdot
\begin{bmatrix}\alpha\\ \beta
\end{bmatrix}.
$$
In view of property (3), the linear map $\beta \mapsto (\mathbbm{1} -\rho (0))\circ Dg (0)({\mathbbm 1}-Dr(0))\beta$ from $(\mathbbm{1} -Dr (0))E$ to $(\mathbbm{1} -D\rho(0))F$ is an isomorphism of Banach spaces. Therefore, 
$$
\text{kernel}\ Dg (0)=(\text{kernel}\ Tf (0))\oplus \{0\}.$$
Moreover the filler has the following additional properties.

\begin{proposition}[{\bf Filler}]\label{filler_new_1}\mbox{}\index{P- Properties of a filler} Assume $f$ has the filling $g$ and $f(0)=0$.
\begin{itemize}
\item[{\em (1)}] The operator $Tf(0)\colon Dr(0)E\to \rho (0)F$ is surjective if and only if the operator $Dg(0)\colon E\to F$ is surjective.
\item[{\em (2)}] $Tf(0)$ is a Fredholm operator (in the classical sense) if and only if $Dg(0)$ is a Fredholm operator and $\ind Tf(0)=\ind Dg(0)$.
\end{itemize}
\end{proposition}
\begin{proof}
(2)\, To simplify the notation we abbreviate the above matrix representing $Dg(0)$ by 
$$
Dg(0)=\begin{bmatrix}A&B\\0&C\end{bmatrix}
$$
and abbreviate the above splittings by $E=E_0\oplus E_1$ and $F=F_0\oplus F_1$.
The operators in the matrix are bounded between corresponding Banach spaces and $C\colon E_1\to F_1$ is an isomorphism of Banach spaces. Therefore, if $B=0$, the operator $A=Df(0)\colon E_0\to F_0$ is Fredholm if and only if the operator 
$$\begin{bmatrix}
A&0\\0&C\end{bmatrix}\colon E\to F
$$
is Fredholm in which case their indices agree. The statement (2) now follows from the composition formula
$$
\begin{bmatrix}
{\mathbbm 1}&BC^{-1}\\
0&{\mathbbm 1}
\end{bmatrix}
\begin{bmatrix}
A&0\\
0&C
\end{bmatrix}=
\begin{bmatrix}
A&B\\
0&C
\end{bmatrix}
$$
since the first factor is an isomorphism from $F$ to $F$, and hence has index equal to $0$,  and the Fredholm indices of a composition are additive. \\[0.5ex]
(1)\, The statement (1) is an immediate consequence of our assumption that $C$ is an isomorphism.
\end{proof}


To sum up the role of a filler, instead of studying the solution set of the section $f\colon O\to K$ we can as well study the solution set of the filled section $g\colon U \to U\triangleleft F$, which is defined on the relatively open set $U$ of the partial quadrant $C$ in the sc-space $E$ and which, therefore,  is easier to analyze.

\begin{definition}[{\bf Filled version}]\label{filled_version_def}\index{D- Filled version}

If $f$ is a sc-smooth section of the tame strong bundle $P\colon Y\to X$ and $x_0\in X$ a smooth point, we 
 say that {\bf section germ $(f, x_0)$ has a filling}, if there exists a strong bundle chart  as defined in 
 Definition \ref{def_strong_bundle_chart}, 
$$
\Phi\colon \Phi^{-1}(V)\to  K\quad  \text{covering $\varphi\colon (V,x_0)\mapsto  (O,0)$,}
$$
 where $K\rightarrow O$ is a tame strong local bundle  containing $0\in O\subset U$,  such that the section germ $\Phi\circ  f\circ \varphi^{-1}\colon O\to K\subset (U\triangleleft F)$ has a filling 
 $g\colon U\to U\triangleleft F$ near $0$. 
 
We shall refer to the section germ $(g, 0)$ as a 
{\bf filled version of $(f, x_0)$}.

\end{definition}

The next concept is that of a   basic germ.
\begin{definition}[{\bf Basic germ}]\label{BG-00x}\index{D- Basic germ}
Let $W$ be a sc-Banach space and $C=[0,\infty )^k\oplus \R^{n-k}\oplus W$ a partial quadrant. Then 
a {\bf basic germ} is a sc-smooth germ 
$$f\colon {\mathcal O}(C, 0)\rightarrow \R^N\oplus W,$$
satisfying $f(0)=0$ and having the following property.
If  $P\colon  \R^N\oplus W\to  W$ denotes the projection, 
the germ $P\circ f\colon {\mathcal O}(C,0)\rightarrow (W,0)$ has the form
$$
P\circ f(a,w)=w-B(a,w) 
$$
for $(a,w)\in ([0,\infty )^k\oplus \R^{n-k})\oplus W$ where $B(0,0)=0$. 
Moreover, for every $\varepsilon>0$ and  every integer $m\geq 0$, the estimate
$$
\abs{B(a,w)-B(a,w')}_m\leq \varepsilon\cdot \abs{w-w'}_m
$$
holds, if $(a, w)$ and $(a, w')$ are close enough to $(0,0)$ on level $m$.
\end{definition}

\begin{remark}
The notion of basic class was introduced in \cite{HWZ3} where, however, we did not require 
$f(0)=0$. Instead we required that $P\circ (f-f(0))$ has  a form as described in Definition \ref{oi}. The later developments convinced us that it is more convenient to require that $f(0)=0$. 
\end{remark}

\begin{lemma}\label{new_Lemma3.9}\index{L- Derivative of a basic germ}
Let $B\colon [0,\infty)^k\oplus \R^{n-k}\oplus W\to W$ be a sc-smooth germ around $0$ satisfying the properties described in Definition \ref{BG-00x}. Then 
for every $\varepsilon>0$ and $m\geq 0$,
\begin{equation}\label{est_second_der_B}
\abs{D_2B(a, w)\zeta}_m\leq \varepsilon\abs{\zeta}_m
\end{equation}
for all $\zeta\in W_m$, if $(a, w)\in E_{m+1}$ is close enough to $(0, 0)$ in $E_m$. In particular, 
$$D_2B(0, 0)=0.$$
\end{lemma}
\begin{proof}
Since $W_{m+1}\subset W_m$ is dense, it is sufficient to verify the estimate for $\zeta\in W_{m+1}$ satisfying $\abs{\zeta}_{m+1}=1$. For such a $\zeta$ we know  from the definition of the linearization,  recalling Proposition \ref{sc_up},  that $B(a, w)-B(a, w+t\zeta)-D_2B(a, w)(t\zeta)=o(t)$, where $w+t\zeta\in C$,  $\abs{o(t)/t}_m\to 0$ as $t\to 0$. Therefore, 
\begin{equation*}
\begin{split}
\abs{D_2B(a, w)\zeta}_m&=\abs{\dfrac{1}{t}D_2B(a, w)(t\zeta)}_m\\
&\leq \dfrac{1}{\abs{t}}\abs{B(a, w)-B(a, w+t\zeta)}_m+ \dfrac{1}{\abs{t}}\abs{o(t)}_m.
\end{split}
\end{equation*}
The first  term on the right hand side is estimated by $\varepsilon\abs{\zeta}_m$ if $(a, w)$ and $(a, w+t\zeta)$ are sufficiently small in $E_m$. Therefore, the estimate \eqref{est_second_der_B} follows as $t\to 0$.
\end{proof}

We will see that basic germs have special properties as already the following application of Lemma \ref{new_Lemma3.9} demonstrates.

\begin{proposition}\label{Newprop_3.9}\index{P- Fredholm property of basic germs}
Let $f\colon [0,\infty)^k\oplus \R^{n-k}\oplus W=E \to \R^N\oplus W$ be a sc-smooth germ around $f(0)=0$ of the form $f=h+s$ where $h$ is a basic  germ and $s$ a $\ssc^+$-germ. Then 
$$Df(0)\colon \R^n\oplus W\to \R^N\oplus W$$
is a sc-Fredholm operator and its index is equal to 
$$\ind Df(0)=n-N.$$
Moreover, for every $m\geq 0$,
$$Df(a, w)\colon \R^n\oplus W_m\to \R^N\oplus W_m$$ is a Fredholm operator having index $n-N$,  if $(a, w)\in E_{m+1}$ is sufficiently small in $E_m$.
\end{proposition}

\begin{proof}

With the sc-projection $P\colon \R^N\oplus W\to W$, the linearization of $f$ at the smooth point $0$, 
$Df(0)=P\circ Df(0)+({\mathbbm 1}-P)\circ Df(0)$, is explicitly given by the formula
\begin{equation*}
\begin{split}
Df(0)(\delta a, \delta w)&=\delta w-D_2B(0)\delta w-D_1B(0)\delta a\\
&\phantom{=}+({\mathbbm 1}-P)\circ Df(0)(\delta a, \delta w) +Ds(0)(\delta a, \delta w).
\end{split}
\end{equation*}
By Lemma \ref{new_Lemma3.9}, $D_2B(0)=0$. Therefore, the operator $Df(0)$ is a $\ssc^+$-perturbation of the operator 
\begin{equation}\label{new_equation_sc-pert}
\R^n\oplus W\to \R^N\oplus W, \quad (\delta a, \delta w)
\mapsto (0,\delta w).
\end{equation}
The operator \eqref{new_equation_sc-pert} is a sc-Fredholm operator whose 
kernel is equal to $\R^n$ and whose cokernel is $\R^N$, so that its Fredholm index is equal to $n-N$. 
Since $Df(0)$ is a $\ssc^+$-perturbation of a sc-Fredholm operator, it is also a   sc-Fredholm operator by Proposition \ref{prop1.21}. Because $\ssc^+$-operators are compact, if considered on the same level, the Fredholm index is unchanged and so, 
$\ind Df(0)=n-N$. 

The second statement follows from the fact that the linear operator 
$Df(a, w)\colon \R^n\oplus W_m\to \R^N\oplus W_m$ is a compact perturbation of the operator 
\begin{equation}\label{new_equation_fredholm_pert}
(\delta a, \delta w)
\mapsto (0,({\mathbbm 1}-D_2B(a, w))\delta w).
\end{equation}
Choosing $0<\varepsilon <1$ in Lemma \ref{new_Lemma3.9}, the operator 
${\mathbbm 1}-D_2B(a, w)\colon W_m\to W_m$ is an isomorphism of Banach spaces if $(a, w)\in E_{m+1}$ is sufficiently small in $E_m$. Hence the operator \eqref{new_equation_fredholm_pert} is a Fredholm operator of index $n-N$ and the proposition follows.

\end{proof}

Finally,  we are in a position to introduce the  sc-Fredholm germs.

\begin{definition}[{\bf sc-Fredholm germ}]\label{oi}\index{D- Sc-Fredholm germ}
Let $f$ be a sc-smooth section of the strong bundle $P\colon Y\to X$ over the tame M-polyfold $X$, and let $x_0\in X$ be a smooth point.  Then $(f,x_0)$ is a {\bf sc-Fredholm germ} provided it possesses  a filled version $(g,0)\colon U\to U\triangleleft F$ according to Definition \ref{filled_version_def} and having the following property. 
There exists a local $\ssc^+$-section $s\colon U\to U\triangleleft F$ satisfying $s(0)=g(0)$ such that  the germ 
$(g-s,0)\colon U\to U\triangleleft F$
is conjugated to a basic germ. 

The last condition requires the existence of a strong bundle isomorphism $\Psi\colon U\triangleleft F\to U'\triangleleft F'$ covering the sc-diffeomorphism 
$\psi\colon U\to U'$ such that the push-forward section 
$$\Psi \circ (g-s)\circ \psi^{-1}\colon U'\to U'\triangleleft F'$$
is a basic germ.
\end{definition}

From Proposition \ref{Newprop_3.9} we deduce that the linearization $D(g-s)(0)$ at the point $0$ is a sc-Fredholm operator. Consequently,  $Dg(0)$ is a sc-Fredholm operator by Proposition \ref{prop1.21}, and so, the tangent map 
$Tf(x_0)\colon T_{x_0}X\to Y_{f(x_0)}Y$ is a linear Fredholm operator having the same index as $Dg(0)$, namely $\ind Tf(x_0)=n-N$, in view of the properties of a filler in Proposition \ref{filler_new_1}.

\mbox{}\\

The above definition of a sc-Fredholm germ looks very complicated; one first has to find a filled version, which then, after some correction by a $\ssc^+$-section, is conjugated to a basic germ.
It turns out that the definition of a sc-Fredholm germ is extremely practicable in the applications we have in mind. 
By experience one may say that the fillings, which are usually only needed
near data describing bubbling-off situations seem almost natural, i.e. ``if one sees one example one has seen them all''. 
Examples of fillings  in the  Gromov-Witten, SFT and Floer Theory can be found in  \cite{HWZ5,HWZ5.5}.
The subtraction of a suitable $\ssc^+$-section is in applications essentially the removing of lower order terms of a nonlinear differential operator
and therefore allows tremendous simplifications of the expressions before one tries to conjugate them to a basic germ.
 One also has to keep in mind that the sc-Fredholm theory is designed to cope with spaces whose tangent spaces have locally varying dimensions, on which, on the analytical side, one studies systems of partial differential equations on varying domains into varying codomains.  Later on 
we shall give criteria (which in practice are easy to check)  to verify  that a section is conjugated to a basic germ.

Sc-Fredholm germs possess a useful local compactness property.
\begin{theorem}[{\bf Local Compactness for sc-Fredholm Germs}] \label{compact-x}\index{T- Local compactness}
Let $f$ be a sc-smooth section of the tame strong bundle $P\colon Y\rightarrow X$, and $x_0\in X$  a smooth point. We assume that $(f,x_0)$ is a sc-Fredholm germ satisfying $f(x_0)=0$. Then there exist a nested sequence
of open neighborhood ${\mathcal O}(i)$ of $x_0$ in $X_0$,  for $i\geq 0$, 
$$
{\mathcal O}(0)\supset  {\mathcal O}(1)\supset \cdots  \supset {\mathcal O}(i)\supset   {\mathcal O}(i+1) \supset \cdots  \, ,
$$
such  that for every $i$ the $X_0$-closure 
 $\cl_{X_0}(\{x\in {\mathcal O}(i)\,  \vert \,  f(x)=0\})$ is a compact subset of $ X_i$.
\end{theorem}
We emphasize that the ${\mathcal O}(i)$ are open neighborhoods in $X$, i.e. on level $0$.

The result is an immediate consequence of 
Theorem \ref{save}, which will be introduced later,
and has the following corollary.
\begin{corollary}\index{C- Local compactness and regularity}
Let $f$ be a sc-smooth section of the tame strong bundle $P\colon Y\rightarrow X$, and $x_0\in X$  a smooth point. We assume that $(f,x_0)$ is a sc-Fredholm germ satisfying $f(x_0)=0$. 
If $(x_k)\subset X$ is a sequence satisfying   $f(x_k)=0$ and $x_k\rightarrow x$ in $X_0$,  then it follows, for every  given any $m\geq 0$, that 
$x_k\in X_m$ for $k$ large and $x_k\rightarrow x$ in $X_m$.
\end{corollary}

\begin{definition}[{\bf Sc-Fredholm section}]\index{D- Sc-Fredholm section}
A section  $f$ of the tame strong bundle $P\colon Y\rightarrow X$ over the M-polyfod $X$ is called  {\bf sc-Fredholm section}, if it has 
the following three properties.
\begin{itemize}
\item[(1)] $f$ is sc-smooth.
\item[(2)] $f$ is regularizing, i.e.,  if $x\in X_m$ and  $f(x)\in Y_{m,m+1}$, then $x\in X_{m+1}$.\index{D- Regularizing section}
\item[(3)] The germ $(f,x)$ is a sc-Fredholm germ at  every smooth point $x\in X$.
\end{itemize}
\end{definition}

The  implicit function theorem,  introduced later on,  is applicable to sc-Fredholm sections and will lead to the following local  result near a smooth interior point $x_0\in X$ ( i.e.,  $d_X(x_0)=0$). We assume that  $f(x_0)=0$. Then the linearization $f'(x_0)\colon T_{x_0}X\rightarrow Y_{x_0}$ 
is a sc-Fredholm operator. Moreover, if  $f'(x_0)$ is surjective, then the solution set
$\{x\in X\, \vert \, f(x)=0\}$ near $x_0$ has the structure of a finite dimensional smooth manifold ( in the classical sense) whose dimension agrees with the Fredholm index.  Its smooth structure is in a canonical way  induced 
from the M-polyfold structure of $X$.

In case that $x$ is a boundary point, so that  $d_X(x)\geq 1$, we can only expect the solution set to be reasonable provided 
the kernel of $f'(x)$ lies in good position to the boundary of $X$ and the boundary $\partial X$ is sufficiently well-behaved.
In order that $\partial X$ is regular enough we have required 
that $X$ is a tame M-polyfold so that we can ask  $\ker(f'(x))$ to be in good position to the partial quadrant $C_xX$ in $T_xX$, 
a  notion which we shall introduce later on.

\mbox{}\\
\mbox{}

 If $P\colon Y\to X$ is a tame strong bundle,  we denote by $\Gamma(P)$\index{$\Gamma(P)$} the vector space of sc-smooth sections; by  $\text{Fred}(P)$\index{$\text{Fred}(P)$} we denote the subset of $\Gamma(P)$ consisting of sc-Fredholm sections. Finally, by  
$\Gamma^+(P)$\index{$\Gamma^+(P)$} we denote  
the vector space of $\ssc^+$-sections as introduced in Definition 
\ref{def_sc_inft_sections}.

\mbox{}\\

The following stability property   of a sc-Fredholm section will be crucial for the transversality theory.

\begin{theorem}[{\bf Stability under $\ssc^{\pmb{+}}$-perturbations}]\label{stabxx}  \index{T- Stability of sc-Fredholm sections}
Let $P\colon Y\rightarrow X$ be a  strong bundle over the tame M-polyfold $X$.
If  $f\in \Fred (P)$ and $s\in \Gamma^+(P)$, then $f+s\in \Fred (P)$.
\end{theorem}

In order to prove the theorem we need  two lemmata for local strong bundles. We recall  the local strong bundle retract  $(K, C\triangleleft F, E\triangleleft F)$ from  Definition \ref{def_loc_strong_b_retract}, consisting of the retract $K=R(U\triangleleft F)$, where $R\colon U\triangleleft F\to U\triangleleft F$ is a strong bundle retraction of the form 
$$R(u, h)=(r(u), \rho (u)h), $$
in which $r\colon U\to U$ is a smooth retraction onto $O=r(U)\subset  U$. We shall denote the principal parts  of the sections  of the 
bundles $U\triangleleft F\to U$ and $K\to O$  by bold letters.

\begin{lemma}\label{rio}
A $\ssc^+$-section $s\colon O\to K$ (as defined in Definition \ref{def_section_loc_strong_bundle}) possesses an extension to a $\ssc^+$-section 
$\wt{s}\colon U\to U\triangleleft F$, 
$\wt{s}(u)=(u, \wt{\bf s}(u))$, 
having the following properties.
\begin{itemize}
\item[{\em (1)}]  $\wt{s}(u)=s(u)$ if $u\in O$.
\item[{\em (2)}]  $R(r(u),\wt{\bf s}(u))=s(r(u))$ if $u\in U$.
\end{itemize} 
\end{lemma}
\begin{proof}
If $s\colon O\to K$  is given by $s(u)=(u, {\bf s}(u))$, $u\in O$, we define the section $\wt{s}\colon U\to U\triangleleft F$ by 
$$\wt{s}(u)=(u, \wt{\bf s}(u))=\bigl(u, {\bf s}(r(u))\bigr),\quad u\in U.$$
Clearly, $\wt{s}$ is a $\ssc^+$-section  of the bundle 
$U\triangleleft F\to U$ and we claim that its restriction to $O$ agrees with the section $s$. Indeed, if $u\in O$, 
then $r(u)=u$, implying $\wt{s}(u)=\wt{s}(r(u))=\bigl( r(u), {\bf s}(r\circ r(u))\bigr)=\bigl( r(u), {\bf s}(r(u))\bigr)=
\bigl(u, {\bf s}(u)\bigr)=s(u)$ as claimed.  Moreover, using that $s(u)=R(s(u))$ if $u\in O$,
$$R\bigl(r(u), \wt{\bf s}(u)\bigr)
=R\bigl(r(u), {\bf s}(r(u))\bigr)=R\bigl({ s}(r(u))\bigr)=s(r(u))$$
for $u\in U$.
\end{proof}
\begin{lemma}\label{filler_extension_s}
Let $f\colon O\to K$ be a sc-smooth section of the (previous) local strong bundle retract, and let $s\colon O\to K$ be a $\ssc^+$-section. 
If $f$ possesses the filler $g\colon U\to U\triangleleft F$, then $f+s$ has the filler $g+\wt{s}\colon U\to U\triangleleft F$, where $\wt{s}$ is the extension of $s$ constructed in the previous lemma.
\end{lemma}
\begin{proof}
We have to verify that the section $g+\wt{s}$ meets the three conditions in Definition \ref{x-filling}. The properties  (1) and (2) for $g+\wt{s}$ follow immediately from the properties (1) and (2)   for the filler $g$ and the properties (1) and (2) for the section $\wt{s}$ in Lemma \ref{rio}. In order to verify property (3) of a filler we have to linearize the map 
$$u\mapsto [{\mathbbm 1}-\rho (r(u))](g(u)+\wt{s}(u))$$
at the point $u=0$. Since 
$\bigl({\mathbbm 1}-\rho (r(u))\bigr) \wt{s}(u)=
\bigl({\mathbbm 1}-\rho (r(u))\bigr) s(r(u))=0$ by property (2) of Lemma \ref{rio}, the linearization agrees with the linearization of the map $\bigl({\mathbbm 1}-\rho (r(u))\bigr) g(u)$ which satisfies the required property (3), since $g$ is a filler. The proof of Lemma \ref{filler_extension_s} is finished.
\end{proof}

\begin{proof}[{\bf Proof of Theorem \ref{stabxx}}]
Let $f$ be a sc-Fredholm section of the tame strong bundle $P\colon Y\to X$ and let $s\colon X\to Y$ be a $\ssc^+$-section of $P$. Then $f+s$ is a sc-smooth section which is also regularizing. It remains to verify that $(f+s, x)$ is a sc-Fredholm germ for every smooth point $x\in X$. By definition of sc-Fredholm,
$(f, x)$ is a sc-Fredholm germ at the smooth point $x$. Therefore, there exists an open neighborhood $V$ of $x$ and a strong bundle chart $(V, P^{-1}(V), K, U\triangleleft F)$ as defined in Definition
 \ref{def_strong_bundle_chart} and satisfying $\varphi (x)=0\in O$, such that the local representation $\wt{f}=\Phi_\ast (f)=\Phi\circ f\circ \varphi^{-1}\colon O\to K$ of the section $f$ possesses a filled version $g\colon U\to 
U\triangleleft F$ around $0$,  which after subtraction of a suitable $\ssc^+$-section,  is conjugated to a basic germ around $0$. Define $t=\Phi_\ast (s)$. Then $t\colon O\to K$ is a $\ssc^+$-section. By Lemma \ref{rio}  
there is a particular $\ssc^+$-section $\wt{t}\colon U\to U\triangleleft F$ extending $t$. By Lemma \ref{filler_extension_s}, the section $g+\wt{t}\colon U\to U\triangleleft F$ is a filling of $\wt{f}+t$.
In view of the sc-Fredholm germ property, there exists a $\ssc^+$-section $t'$ satisfying $t'(0)=g(0)$ and such that $g-t'$ is conjugated to a basic germ. Now taking the $\ssc^+$-section $\wt{t}+t'\colon U\to U\triangleleft F$, we have $(g+\wt{t})(0)=(\wt{t}+t)(0)$. Moreover, $(g+\wt{t})-(\wt{t}+t')=g-t'$ which, as we already know, is conjugated to a basic germ. To sum up, we have verified that $(f+s, x)$ is a sc-Fredholm germ. This holds true for every smooth point $x\in X$. Consequently, the section $f+s$ is a Fredholm section and the proof of Theorem \ref{stabxx} is complete.
\end{proof}

In order to formulate  a parametrized version of Proposition \ref{stabxx} 
we assume that $P\colon Y\rightarrow X$ is a  strong bundle and $f$ a sc-Fredholm section. The sc-smooth projection 
$$
\pi\colon \R^n\times X\rightarrow X,\quad  (r,x)\mapsto  x,
$$
pulls back the bundle $P$ to the  strong bundle   $\pi^\ast(P)\colon \pi^\ast Y\rightarrow \R^n\times X$.
The section $\wt{f}$ of $\pi^\ast(P)$,  defined by
$$
\wt{f}(r,x)=((r,x),f(x)), 
$$
 is a sc-Fredholm section as is readily verified.  
If $s_1,\ldots ,s_n$ are $\ssc^+$-sections of $P$, then $\wt{s}(r,x)\colon =\bigl( (r,x),\sum_{i=1}^n r_i\cdot s_i(x)\bigr)$ is a $\ssc^+$-section
 of  the pull back bundle $\pi^\ast(P)$ and,  by the stability 
 Theorem \ref{stabxx}, the section 
 $$
 (r,x)\mapsto \wt{f}(r,x)+\wt{s}(r,x)
 $$
 is a sc-Fredholm section of $\pi^\ast(P)$. Hence we have proved the following stability result.
 \begin{theorem}[{\bf Parameterized Perturbations}]\label{corro}\index{T- Parameterized perturbations}
 Let $P\colon Y\rightarrow X$ be a tame strong bundle and $f$ a sc-Fredholm section. If  $s_1,\ldots ,s_n\in\Gamma^+(P)$, 
then the map
 $$
 \R^n\times X\rightarrow Y,\quad (r,x)\mapsto f(x)+\sum_{i=1}^n r_i\cdot s_i(x)
 $$
 defines a sc-Fredholm section of the bundle $\pi^\ast(P)\colon \pi^\ast Y\rightarrow \R^n\times X$.
 \end{theorem}
This theorem and refined versions of the theorem play a role  in the perturbation and transversality theory. As already pointed out,  
the distinguished class of sc-Fredholm sections allows  to apply an implicit function theorem in the usual sense. 
We first formulate the  implicit function theorem at an interior point. Note that a M-polyfold $X$ with an identically vanishing degeneracy map $d_X$ is tame.
\begin{theorem}[{\bf Implicit Function Theorem: Interior Case}]\label{implicit-x}\index{T- Implicit function theorem}
Assume that $P\colon Y\rightarrow X$ is a strong bundle over the  M-polyfold $X$ satisfying $d_X\equiv 0$,  and $f$ a sc-Fredholm section.
Suppose that $x_0\in X$ is a smooth point in $X$,  such  that $f(x_0)=0$. Then the linearization  $f'(x_0)\colon T_{x_0}X\rightarrow Y_{x_0}$  is a sc-Fredholm operator. If $f'(x_0)$  surjective,  
then there exists an open neighborhood $U$ of $x_0\in X$ such  that the solution set $S(f,U)=\{x\in U\, \vert \,  f(x)=0\}$ in $U$ has in a natural way the structure
of a smooth finite dimensional manifold whose dimension agrees with the Fredholm index.  In addition,  $U$ can be chosen  in such a way that the linearization $f'(y)\colon T_yX\rightarrow Y_y$ for $y\in S(f,U)$ is surjective  and  $\ker(f'(y))=T_yS(f,U)$ is the tangent space.
\end{theorem}
Theorem \ref{implicit-x} is an immediate consequence of Theorem \ref{IMPLICIT0} in Section \ref{ssec3.4}.
As we shall see, the smooth manifold structure on $S(f,U)$ is induced from the M-polyfold structure of $X$.

Considering a sc-Fredholm section $f$ of the 
strong bundle $Y\rightarrow X$ over the  M-polyfold $X$ having no boundary (i.e., $d_X\equiv 0$),
we assume, in addition,  that the M-polyfold $X$ admits a sc-smooth partition of unity.
Then there exists, for two given smooth point $x\in X$ and $e\in Y_x$, a $\ssc^+$-section $s$ supported near $x$ and satisfying $s(x)=e$.  For the easy proof we refer to \cite{HWZ3}. As we shall see later it suffices
to assume the existence of sc-smooth bump functions instead of sc-smooth partitions of  unity.

If $f(x_0)=0$ and $f'(x_0)$ is not surjective, we find finitely many smooth elements $e_1,\ldots ,e_k\in Y_{x_0}$ satisfying $R(f'(x_0))\oplus \R e_1\oplus \cdots \oplus \R e_k=Y_{x_0}$. Taking the $\ssc^1$-sections $s_i$ satisfying $s_i(x_0)=e_i$, we define the map $\wt{f}\colon \R^k\oplus X\to Y$  by 
$$\wt{f}(r, x)=f(x)+\sum_{i=1}^k r_is_i(x).$$
The linearization of $\wt{f}$ at  the distinguished point $(0,x_0)\in  \R^k\oplus X$ is the continuous linear map 
$$
\wt{f}'(0, x_0)(h,u)= f'(x_0)u +\sum_{i=1}^k h_i e_i,
$$
which is surjective. By Theorem \ref{corro}  and Theorem \ref{IMPLICIT0} we find an open neighborhood $U$ of $(0,x_0)\in \R^k\oplus X$
such  that the solution set of $S(\wt{f}, U)=\{(r,x)\in U\, \vert \,  f(x)+\sum_{i=1}^k r_i\cdot s_i(x)=0\}$ is a smooth finite-dimensional manifold.
The trivial bundle $S(\wt{f},U)\times \R^k\rightarrow S(\wt{f},U)$ has the canonical section $(r,y)\mapsto r$. 
The zero set of this section is precisely the unperturbed solution set of $f(y)=0$ for $y\in U$.
If the solution set of $f(y)=0$, $y\in X$, is compact we can carry out the previous construction globally, which gives rise to  global finite-dimensional reduction. This will be discussed later on.


\begin{remark}\label{remark.3_20}

We should point out that in the original proof of Theorem \ref{IMPLICIT0} in \cite{HWZ3} (Theorem 4.6 and Proposition 4.7) the sc-Fredholm section is defined slightly differently, namely as follows. In \cite{HWZ3}
a sc-smooth section $f$ of the strong bundle $Y\to X$ is called sc-Fredholm, if it possesses around all smooth points of $X$ a filled version 
 $(g,0)$ such that $g-g(0)$ near $0$ is conjugated to a basic germ. It is, in this case, not true that $f+s$ is sc-Fredholm for $\ssc^+$-section $s$. However, by a nontrivial theorem in \cite{HWZ3} (Theorem 3.9), increasing the level, the section 
$(f+s)^1$ of the the strong bundle $Y^1\rightarrow X^1$ is sc-Fredholm. Although not harmful in practice,  this looks unsatisfactory.

This is why we have introduced the new definition (Definition \ref{oi}) of sc-Fredholm, where we require for the filled version $(g, 0)$ that there exists a local $\ssc^+$-section $s\colon U\to U\triangleleft F$ such that $s(0)=g(0)$ and $g-s$ is locally conjugated to a basic germ.
If now $t$ is a $\ssc^+$-section, then the section $f+t$ of the bundle $Y\to X$ is automatically sc-Fredholm in view of Theorem \ref{stabxx}. 

\end{remark}

The difficulty of the nontrivial theorem is now hidden in the proof of the implicit function theorem, which has to incorporate
the arguments of the nontrivial theorem. With the new definition, even if we want to study $f$ only, we have
only normal forms for the perturbed expression, which might be unrelated to our problem.  However, writing $f=(f-s) + s$, we know how $f-s$ looks like, and we know that $s$ is a compact perturbation of $f-s$.  We  combine  these facts to gain sufficient information about $f^1$ to determine, in view the regularizing property of sc-Fredholm sections, the solution set of $f$.

Let us briefly sketch the construction of the manifold structure of the solution set.
Theorem \ref{implicit-x}  is a local result so that we can work without loss of generality with a section $f$ in a strong local bundle $K\rightarrow O$ whose local model is 
$(O,E,E)$ containing  $0\in O$. By the Fredholm requirement we may assume without loss of generality, 
by passing to a suitable filling at $0$,  that we have a section $g$ of $V\triangleleft F\rightarrow V$ which has the following properties.
We have that $f=g$ on $O$ and the germ $(g,0)$ is conjugated to  a basic germ.   Note that for $x\neq 0$ near $0$ the germ $(g,x)$
might not have this property. However if $x\in O_\infty$, then $(f,x)$, after a possible coordinate change, has a filling (which might look quite different than the one we used at $0$),
which is conjugated to a basic germ (after correction by a $\ssc^+$-section).

We point out  that  different  fillings are needed in the proof in order to get the smoothness of $S(f,U)$ at points other than $0$. The  proof of the theorem shows that locally the solution set
of $f=0$ near $0$ can be written as follows. If  $N=\ker(f'(x))=\ker(g'(0))$ and if $Y$ is a sc-complement in $E$ so that $E=N\oplus Y$, then there exists
a sc-smooth map $\delta\colon N\supset {\mathcal U}(0)\rightarrow Y$ satisfying  $\delta(0)=0$ and $D\delta(0)=0$, such  that
$$
\text{graph}(\delta)\colon {\mathcal U}(0)\rightarrow E,\quad  a\rightarrow a+\delta(a)
$$
parametrizes all solutions of $g=0$ and $f=0$ in a suitable neighborhood  of $0$. In particular, by the definition of a filler,  the image of the map $\text{graph}(\delta)$ lies in  $O$.
 The inverse of the map $\text{graph}(\delta)$ is the restriction of the projection $N\oplus Y\to N$  to the solution set,  which is sc-smooth.

In order to formulate the boundary version of the implicit function theorem we start with some preparation.

\begin{definition}[{\bf In good position}] \label{mission1}\index{D- Good position}\index{D- Good complement}
Let $C\subset E$ be a partial quadrant in the sc-Banach space $E$ and $N\subset E$ be a finite-dimensional sc-subspace of $E$.  The subspace $N$ is in  {\bf good position to the partial quadrant $C$}, if the interior of $N\cap C$ in $N$  is non-empty, and if $N$ possesses a  sc-complement $P$, so that $E=N\oplus P$, having the following property. There exists  $\varepsilon>0$, such that 
for pairs $(n, p)\in N\times P$ satisfying $\abs{p}_{0}\leq \varepsilon\abs{n}_{0}$ the statements 
$n\in C$ and $n+p\in C$ are equivalent. We call such a sc-complement $P$ a {\bf good complement}.
\end{definition}

The choice of the right complement $P$ is important. One cannot take a random sc-complement of $N$, in general,  as Lemma 
\ref{cone1} demonstrates.

In view of Proposition \ref{prop1} the finite dimensional sunspace $N$ in Definition \ref{mission1}, possessing the sc-complement $P$, is necessarily a smooth subspace. A finite dimensional subspace $N$ which is not necessarily smooth is called in good position to the partial quadrant $C$ in the sc-Banach space $E$ if there exists a (merely) topological complement $P$ in $E$ satisfying the requirements of Definition \ref{mission1} for some $\varepsilon>0$.

\begin{figure}[htb]
\centering
\def\svgwidth{85ex}
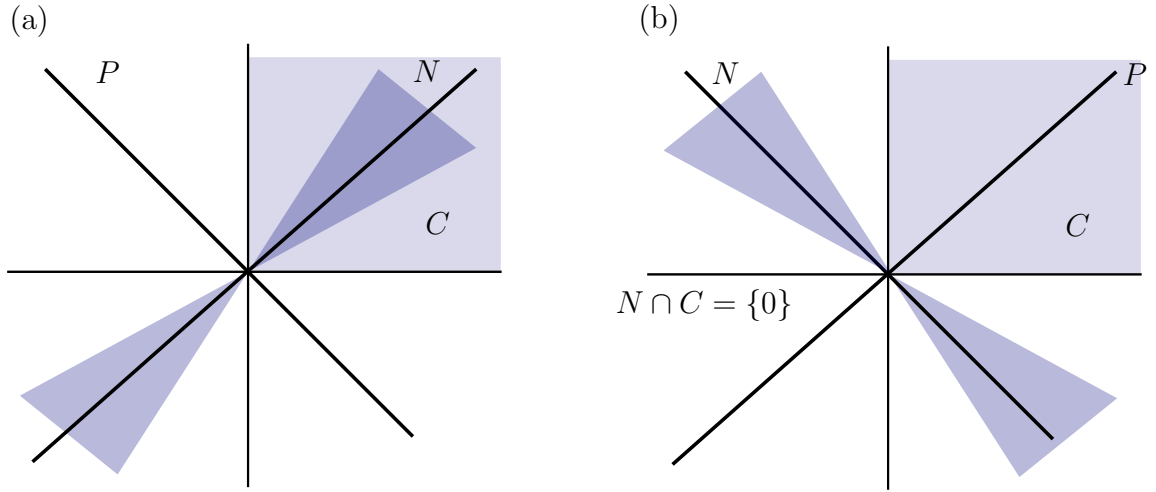
\caption{In figure (a) $N$ is in good position to $C$ while in figure (b) $N$ is not in good position to $C$.}\label{fig:pict3}
\end{figure}

The following result  is taken from \cite{HWZ3}, Proposition 6.1. Its proof is reproduced in Appendix \ref{pretzel-A}
\begin{proposition}\label{pretzel}\index{P- Good position and partial quadrants}
If  $N$ is a finite-dimensional sc-subspace in good position to the partial quadrant  $C$ in $E$, then
$N\cap C$ is a partial quadrant in $N$.
\end{proposition}

The boundary version of the implicit function theorem is formulated in the next theorem. The proof is again given later.

\begin{theorem}[{\bf Implicit Function Theorem: Boundary Case}]\label{bound}\index{T- Implicit function theorem {II}}
We assume that $P\colon Y\rightarrow X$ is a strong bundle over the tame M-polyfold $X$, and $f$ is a sc-Fredholm section.
Suppose that $x\in X$ satisfies $f(x)=0$ and the following two properties.
\begin{itemize}
\item[{\em (1)}] The linearisation $f'(x)\colon T_xX\rightarrow Y_x$ is surjective.
\item[{\em (2)}] The kernel $N$ of $f'(x)$ is in good position to the boundary of $X$, i.e. $N$ is in good position to the partial quadrant $C_xX$ in the tangent space $T_xX$.
\end{itemize}
Then there exists an open neighborhood $U$ of $x$ such  that the following holds.
\begin{itemize}
\item[{\em (1)}] The local solution set $S(f,U):=\{y\in U\ |\ f(y)=0\}$, which consists of smooth points, is a tame sub-M-polyfold of $X$. 
\item[{\em (2)}]  The tame sub M-polyfold $S(f,u)$ admits a uniquely determined structure as a smooth manifold with boundary with corners. This  
M-polyfold structure on $S(f,U)$ is the same as the one defined in 
Proposition \ref{sc_structure_sub_M_polyfold}
\end{itemize}
\end{theorem}
Theorem \ref{bound} is a consequence of Theorem \ref{IMPLICIT0} in Section \ref{ssec3.4}.

In the proof of Theorem \ref{bound} we shall describe the manifold structure on the solution space in  detail.
Here we just indicate  how it looks like.  Since $S(f,U)\subset X_\infty$ is a tame M-polyfold, we can take for a point $y\in S(f,U)$
a sc-diffeomorphism $\Psi\colon U(y)\rightarrow O=O(0)$, where $(O,C,E)$ is a tame retract. Then if $t\colon V\rightarrow V$ satisfies
$O=t(V)$ we have the splitting $E=T_0O\oplus (({\mathbbm 1}-Dt(0))E)$, with $Y=({\mathbbm 1}-Dt(0))E$ contained in $T_0^RC$. 
Also the proof will show that $T_0O$ is finite-dimensional. Let $p=Dt(0)$ be the projection onto $T_0O$.
Then it will be shown that near $0$ the projection $P\colon {\mathcal U}'(0)\rightarrow {\mathcal V}(0)$ is a sc-diffeomorphism,
where ${\mathcal U}'$ is an open neighborhood of $0$ in $O$ and ${\mathcal V}$ is an open neighborhood of $0$ in $C_0O$.
Then, for ${\mathcal U}=\Psi^{-1}({\mathcal U}')$ the map
$$
{\mathcal U}\rightarrow {\mathcal V}, \quad  y\mapsto  p\circ \Psi(y)
$$
is a sc-diffeomorphism and its mage lies in an relatively open neighborhood of $0$ in the partial quadrant $C_0O$ in $T_0O$.
The associated transition maps for any two such sc-diffeomorphisms are (trivially) sc-diffeomorphisms between relatively open subsets in partial quadrants 
of finite-dimensional vector spaces. Hence they are classically smooth. This shows that the system of such sc-diffeomorphisms
defines a smooth atlas for the structure of a manifold with boundary with corners and by construction this structure is compatible with the existing 
M-polyfold structure on $S(f,U)$.

As a corollary of Theorem \ref{bound} we shall obtain the following result.
\begin{corollary}\index{C- Global implicit function theorem}
Let $P\colon Y\rightarrow X$ be  a strong bundle over the tame M-polyfold $X$ and $f$ be a sc-Fredholm section.
Suppose that for every $x\in X$ satisfying  $f(x)=0$ the linearisation $f'(x)\colon T_xX\rightarrow Y_x$ is surjective 
and the kernel $\ker(f'(x))$ is in good position to the boundary of $X$ (the latter being an empty condition if $d_X(x)=0$). 
Then the solution set $M:=\{x\in X\ |\ f(x)=0\}\subset X$
is a  sub-M-polyfold of $X$ for which the induced structure is tame, and which, moreover,  admits  a sc-smoothly  equivalent structure as a smooth manifold with boundary with corners.
\end{corollary}

The remaining subsections are devoted to the proof of the above  results. Since the Fredholm theory is one of the main parts 
of the polyfold theory and draws heavily on the possibilities offered  in the sc-smooth theory we shall carry out  the constructions in great details.

\subsection{Subsets with Tangent Structure}
The solution sets of sc-Fredholm sections will come with a certain structure, which in the generic case will induce a natural smooth manifold 
on the solution set. This subsection studies this structure.  Recall the definition of a smooth finite-dimensional subspace $N$ in good position to the partial quadrant $C$ (Definition \ref{mission1}). For such a subspace, $N\cap C$ is a partial quadrant in $N$ (Proposition \ref{pretzel}).

\begin{definition}[{\bf $n$-dimensional tangent germ property}]\label{toast}\index{D- Tangent germ property}
We consider a tame M-polyfold $X$ and a subset $M\subset X$ of  $X$. The subset $M$ has the {\bf $\boldsymbol{n}$-dimensional tangent germ property}
provided the following holds.
\begin{itemize}
\item[(1)]  $M\subset X_\infty$.
\item[(2)]  Every point $x\in M$  lies in an open neighborhood $U\subset X$ of $x$ such that there exists a sc-smooth chart
$\varphi\colon (U,x)\rightarrow (O,0)$ onto a tame  retract  $(O,C,E)$. Moreover, there exists  a $n$-dimensional smooth subspace $N\subset E$ in good position to the partial quadrant $C$, which possesses a good complement $Y$ so that $E=N\oplus Y$.  In addition,  there exists a relatively open neighborhood $V$ of $0$ in the partial quadrant $N\cap C$  and a continuous map
$\delta\colon V\rightarrow Y$ having the following properties.
\begin{itemize}
\item[(i)] $\varphi(M\cap U) =\{v+\delta(v)\, \vert \, v\in V\}\subset N\oplus Y$.
\item[(ii)] $\delta\colon {\mathcal O}(N\cap C,0)\rightarrow (Y,0)$ is a $\ssc^\infty$-germ satisfying  $\delta (0)=0$ and $D\delta(0)=0$.
\end{itemize}
\end{itemize}
 \end{definition}

\begin{figure}[htb]
\begin{centering}
\def\svgwidth{50ex}
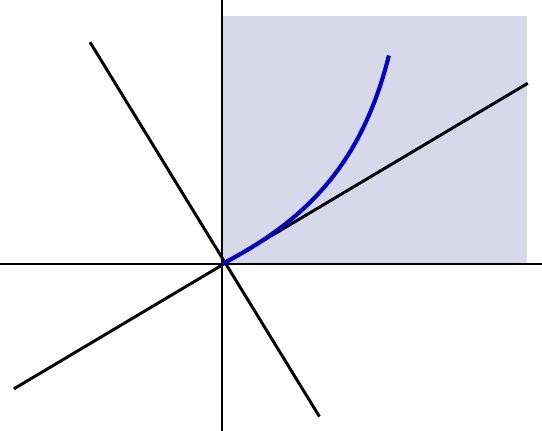
\caption{}\label{fig:pict4}
\end{centering}
\end{figure}

Recalling  the definition of a $\ssc^\infty$-germ (Definition \ref{germy}) we note that here $V$ is a relatively open neighborhood of $0$ in the partial quadrant $N\cap C$ where $N$ is a smooth  finite-dimensional space.
There exists a nested sequence $(V_m)$ of relatively open neighborhoods of $0$ in $N\cap C$, say $V=V_0\supset V_1\supset V_2\supset \ldots$,  such  that
$\delta(V_m)\subset Y_m$ and $\delta\colon V_m\rightarrow Y_m$ is continuous.  Denoting this sequence of neighborhoods by ${\mathcal O}(N\cap C,0)$, its tangent is the nested sequence 
$TV_1\supset TV_2\supset TV_2 \ldots $ denoted by $T{\mathcal O}(N\cap C,0)$.  If $x\in V_1$, then the map $D\delta(x)$ is defined, and since  $\delta$ is $\ssc^\infty$-germ, the tangent map 
 $T\delta\colon T{\mathcal O}(N\cap C,0)\rightarrow TY$  is again of class $\ssc^1$. Iteratively it follows that $T\delta$ is a $\ssc^\infty$-germ.

\begin{proposition}\index{P- Invariance of tangent germ property}
For a pair $(X,M)$ in which  $X$ is a tame M-polyfold and $M$ a subset of $X$,  the property that $M$ has the n-dimensional tangent germ property, 
is a sc-diffeomorphism invariant. More precisely, if $(X',M')$ is a second pair in which  $X'$ is a tame M-polyfold and $M'$ a subset of $X'$ and if $\psi\colon X\rightarrow X'$ is a sc-diffeomorphism satisfying $\psi(M)=M'$,  then $M'$ has the n-dimensional tangent germ property if and only if
$M$  has the n-dimensional tangent germ property.
\end{proposition}
\begin{proof}
From  $X'=\psi(X)$ we conclude  that $X'$ is tame. We show that if $M\subset X$ has the n-dimensional tangent germ property, then $M'\subset X'$ has this property  too.
Since $M\subset X_\infty$ and $\psi$ is sc-smooth,  we see that $M'=\psi(M)\subset X'_\infty$. Let $m'\in M'$ and  choose a point  $m\in M$ satisfying $\psi(m)=m'$.
By assumption there exists a sc-smooth chart $\varphi\colon (U,m)\rightarrow (O,0)$, where $(O,C,E)$ is a tame retract. By assumption there exists
a smooth n-dimensional linear subspace $N$ in good position to $C$ with sc-complement $Y$ and a continuous map
$\delta\colon V\rightarrow Y$, where $V$ is a relatively open neighborhood of $0$ in $C_N: =C\cap N$, which satisfies
\begin{itemize}
\item[(i)]   $\varphi(M\cap U)=\{v+\delta(v)\ |\ v\in V\}\subset N\oplus Y$.
\item[(ii)] $\delta\colon {\mathcal O}(C_N,0)\rightarrow (Y,0)$ is a $\ssc^\infty$-germ satisfying $\delta (0)=0$ and $D\delta(0)=0$.
\end{itemize}
To deduce the corresponding construction for $(X',M')$ we define the open neighborhood $U'\subset X'$ of $m'=\psi (m)$ by $U'=\psi (U)$ and the sc-smooth chart by $\varphi'=\varphi\circ \psi^{-1}\colon (U', m')\to (O, 0)$. Then  
$$
\varphi'(U'\cap M') = \varphi(U\cap M) =\{v+\delta(v)\, \vert \, v\in V\}
$$
and the lemma follows.
\end{proof}

The important aspect of the $n$-dimensional tangent germ property of a subset is the following result. 

\begin{theorem}\label{HKL}\index{T- Characterization of subsets with tangent germ property}
If  $X$ is a tame M-polyfold and $M\subset X$ a subset possessing the $n$-dimensional tangent germ property, then the following holds.
\begin{itemize}
\item[{\em (1)}] $M$ is a  sub-M-polyfold of $X$ whose  induced M-polyfold structure is tame.  Moreover,  the induced M-polyfold structure on M is sc-smoothly equivalent to a smooth structure  of a manifold with boundary with corners on $M$. 
\item[{\em (2)}] If  $x\in M$ is given, we denote by $U$,  $\varphi$, $N$, $V \subset N$, and $\delta\colon V\rightarrow Y$ the data described in condition (2) of Definition \ref{toast}.
Denoting  by $\pi\colon N\oplus Y\rightarrow N$ the sc-projection, the map $U\cap M\rightarrow V$, given by $y\mapsto  \pi\circ \varphi(y)$,  defines a smooth chart  on $M$ around the point $x$. 
\end{itemize}
\end{theorem}

\begin{proof}

We choose a point $x\in M$ and  find an open neighborhood $U\subset X$ of $x$ and a M-polyfold chart
$\varphi\colon (U,x)\rightarrow (O,0)$ onto the retract $(O, C,E)$  so that the set $M\cap U$ is represented as 
$$
\varphi(M\cap U)=\{v+\delta(v)\, \vert \,   v\in V\}\subset N\oplus Y.
$$
The map $\delta\colon V\to Y$ possesses all the properties listed in condition (2) of Definition \ref{toast}.  The map $\varphi\colon M\cap U\to V$  is of the form $\varphi (y)=v(y)+\delta (v(y))$.
With the sc-projection $\pi\colon N\oplus Y\to N$ onto $N$, the map 
\begin{equation}\label{edi}
\pi\circ \varphi\colon M\cap U\to V,\quad \pi\circ \varphi (y)=v(y)
\end{equation}
is continuous  and bijective onto $V$. It is the restriction of the sc-smooth map $\pi\circ \varphi\colon U\to N$, which maps the point $x\in M$ onto $0\in N$. Then the  inverse $\gamma$  of \eqref{edi},
$$
\gamma\colon V\rightarrow M\cap U,\quad  \gamma(v)=\varphi^{-1}(v+\delta(v))
$$
has its image in $X_\infty$ and,  as a map into any $X_m$,  has arbitrarily high regularity if only $v$ is close enough to $0$, depending on $m$. 

Next we shall show that the map $\gamma$ is $\ssc^\infty$ on all of $V$ and not only at the point $0\in V$.

To this aim we choose a $v_0\in V$ and put  $x_0=\varphi^{-1}(v_0+\delta(v_0))$. By construction,  $x_0\in M$,  and by our assumption there is a sc-smooth chart $\psi\colon (U',x_0)\to (O', 0)$ satisfying 
$$\psi(M\cap U')=\{w+\tau (w)\, \vert\, w\in V'\}$$
where $V'$ is a relatively open neighborhood of $0$ in the smooth $n$-dimensional subspace $N'\subset E'$ possessing the tangent germ property. The map $\tau\colon V'\to Y'$ possesses the properties listed in condition (2) of Definition \ref{toast}. In particular, $\tau$ is a $\ssc^\infty$-germ at the point $w=0$.

For $v\in V$ near $v_0$ and $w\in V'$ near $0$ we consider the equation
\begin{equation}\label{edi1}
v=\pi\circ\varphi\circ \psi^{-1}(w+\tau(w)).
\end{equation}

If $v=v_0$, we have the solution $w=0$. Near $0$ the map $\tau$ possesses arbitrary high classical differentiability into any level. Linearizing the right-hand side of the equation at the point $0$, and recalling that $D\tau (0)=0$, we obtain the linear isomorphism 
$$h\mapsto \pi\circ T(\varphi\circ \psi^{-1})(0)h$$
from $N'$ onto $N$.  By the classically implicit function theorem we obtain a germ $v\mapsto w(v)$ for $v$ close to $v_0$ satisfying  $w(v_0)=0$ and solving the equation \eqref{edi1}. The germ has arbitrary high classical differentiability once we are close enough to $v_0$. Now consider the map 
\begin{equation}\label{edi2}
v\mapsto  \varphi\circ \psi^{-1}(w(v)+\tau(w(v))
\end{equation}
for $v$ near $v_0$. Since $v\mapsto w(v)$ has arbitrarily high differentiability at $v_0$ and $\tau$ is $\ssc^\infty$-germ near  $w=0$, we see that the map\eqref{edi2} has, into  any given level,  arbitrarily high differentiability for $v$ near $v_0$.  Consequently, the map is a $\ssc^\infty$-germ near $w=0$. The image of the map  lies in the infinity level. Applying the sc-projection  ${\mathbbm 1}-\pi\colon N\oplus Y\to Y$, we obtain the identity 
$$
\delta(v) =({\mathbbm 1}-\pi)\circ \varphi\circ \psi^{-1}(w(v)+\tau(w(v))
$$
which implies that $\delta$ is a $\ssc^\infty$-germ near $v_0$.
Since $v_0$ is arbitrary in $V$ we see that
$v\mapsto \delta(v)$
is a $\ssc^\infty$-germ around every $v_0\in V$ as we wanted to show. 

Moreover, we conclude that the map
$$
V\rightarrow X,\quad  v\rightarrow \varphi^{-1}(v+\delta(v))
$$
is an injective sc-smooth map whose image is equal to  $M\cap U$.

Next we shall verify that the set $M$ is a sub-M-polyfold of $X$ according to Definition \ref{sc_structure_sub_M_polyfold}.  By construction, we have, so far, at every point $x\in M$ an open neighborhood $U=U(x)\subset X$  and a sc-smooth chart $\varphi\colon U\rightarrow O$,
satisfying  $\varphi(x)=0$, where $(O,C,E)$ is a tame sc-smooth retract. Moreover,  recalling the sc-splitting 
 $
 E=N\oplus Y
 $
 there is  a relatively open neighborhood $V$ of $0$ in the partial quadrant $N\cap C$ of $N$ and a sc-smooth map
 $$
\text{$ \delta \colon V\rightarrow Y$ satisfying $\delta (0)=0$ and
 $D\delta (0)=0$}, 
 $$
 such that  $\varphi(M\cap U) =\{v+\varphi(v)\,  \vert \, v\in V\}$. The map $V\rightarrow U$, 
 $$
 v\mapsto \varphi^{-1}(v+\delta(v)), 
 $$
 is sc-smooth and injective.  The subset $\Sigma\subset C$, defined by 
 $$
 \Sigma=\{v+y\in C\, \vert \, v\in V,\ y\in Y\},
 $$
is relatively open in $C$ and contains $0$. Since $O$ is a tame retract,  there exist a relatively open subset $W$ of $C$ and a tame sc-smooth retraction $r\colon W\rightarrow  W$ onto $O=r(W)$. Consequently, in view of 
 $$
 r\colon r^{-1}(\Sigma\cap O)\rightarrow r^{-1}(\Sigma\cap O),
 $$
the subset $\Sigma\cap O$ is also a sc-smooth retract.

By construction, $v+\delta (v)\in O$ and also $v+\delta (v)\in \Sigma$ and we define the map $t\colon \Sigma\cap O\rightarrow \Sigma\cap O$  by
 $$
 t(v+w)=v+\delta(v).
 $$
The map $t$ is sc-smooth 
 and satisfies $t\circ t=t$, so that  $t$ is a sc-smooth retraction defined on a relatively  open neighborhood of $0$ in $C$ 
 and $t(\Sigma\cap O) = \varphi(M\cap U).$ Therefore, the composition  $s=\varphi^{-1}\circ t\circ\varphi$ defines a sc-smooth retraction
\begin{equation}\label{edi3}
 s\colon U\rightarrow U
 \end{equation}
onto $M\cap U=s(U)$, proving that the subset $M$ is a sc-smooth sub-M-polyfold of $X$.

 The map $u\mapsto C\cap N$, $u\mapsto \pi\circ \varphi (s(u))$ is sc-smooth. Therefore, the map 
 $$M\cap U\to V,\quad m\mapsto \pi\circ \varphi (m)$$
 is a sc-smooth M-polyfold chart on $M$ for the induced M-polyfold structure. The image of the chart is the local model 
 $(V, N\cap C, N)$  so that the transition maps are classically smooth maps and define on $M$  the structure of a manifold with boundary with  corners. The proof of Theorem \ref{HKL} is complete.
 \end{proof}

 As an aside we mention that, in general,  we can not find a local retraction $s$ in \eqref{edi3} which is  tame, as the 
example $X=[0,\infty)^2$ and $M=\{(x,x)\,  \vert , \ x\geq 0\}$ shows.
 
The strength  of the theorem stems  from the fact that in our sc-Fredholm theory the machinery produces subsets $M\subset X$, which have the $n$-dimensional tangent germ property.

\subsection{Contraction  Germs}
The notion of a contraction germ is a slight modification of a basic germ. These germs are convenient for the proof of the implicit function theorem (Theorem \ref{newthm5.4}),  which is the main result of this section. It turns out that the local geometry of sc-Fredholm  germs are intimately related to contraction germs. In the generic case they are used to prove that the zero set of a sc-Fredholm section 
must have the n-dimensional tangent germ property. It follows that  the zero set  is  in a natural  way a smooth manifold with boundary with corners.

\mbox{}\\

In the following we abbreviate by $\wt{C}$ the partial quadrant $\wt{C}=[0,\infty)^{k}\oplus \R^{n-k}$ in $\R^n$ so that 
$C=\wt{C}\oplus W$ is a partial quadrant in the sc-Banach space $E=\R^n\oplus W$.

 
\begin{definition}[$\ssc^0$-{\bf contraction germ} ] \label{BG}\index{D- Contraction germs}
A $\ssc^0$-germ $f\colon {\mathcal O}(C,0)\rightarrow ( W,0)$  is called  a 
$\ssc^0$-{\bf contraction germ} if the following holds.
The germ $ f\colon {\mathcal O}(C,0)\rightarrow (W,0)$
has the form
$$
 f(a,w)=w-B(a,w)
$$
for $(a,w)$ close to $(0,0)\in C$. Moreover, for every $\varepsilon>0$ and $m\geq 0$,  the estimate
$$
\abs{B(a,w)-B(a,w')}_m\leq \varepsilon\cdot \abs{w-w'}_m
$$
holds for all $(a,w),(a,w')$ on level $m$ sufficiently close to $(0, 0)$, depending on $\varepsilon$ and $m$.
\end{definition}


More precisely, the  $\ssc^0$-contraction germ requires for given $\varepsilon>0$,  that we can choose a perhaps smaller germ ${\mathcal O}(C,0)$ of neighborhoods 
$U_0\supset U_1\supset U_2\supset U_3\supset \ldots $ of the point $(0, 0)$ in $[0,\infty)^k\oplus \R^{n-k}\oplus W$ such that 
$
\abs{B(a,u)-B(a,v)}_m\leq \varepsilon \abs{u-v}_m
$
holds if $(a, u), (a, v)\in U_m$.

Starting on level $0$, 
the  parametrized version of Banach  fixed point theorem together with  $B(0,0)=0$, guarantee the existence of relatively open and connected neighborhood $V=V_0$ of $0$ in $[0,\infty)^k\oplus \R^{n-k}$  and  
a uniquely determined continuous map $\delta\colon V\rightarrow W_0$ satisfying  $\delta(0)=0$ and solving the equation
$$
\delta(a)=B(a,\delta(a))\quad \text{for all $a\in V$}.
$$
Going to level $1$ we find, again using the fixed point theorem, an 
 open neighborhood $V_1\subset V_0$ of $0$  and a continuous map $\delta_1\colon V_1\to W_1$ satisfying $\delta (0)=0$ and solving 
 the equation on level $1$. 
From the uniqueness of the solutions of the Banach fixed point problem we conclude that $\delta_1 =\delta\vert V_1$. Continuing  this way, we obtain a decreasing sequence 
of relatively open neighborhoods of $0$ in $[0,\infty)^{k}\times \R^{n-k}$, 
$$
V=V_0\supset V_1\supset V_2\supset \ldots 
$$
such that the continuous solution $\delta\colon V\to W$  satisfies 
$\delta(0)=0$ and $\delta (V_m)\subset W_m$ and $\delta\colon V_m\to W_m$ is continuous. 
In  other words,  we obtain  a $\ssc^0$-solution germ $\delta\colon {\mathcal O}([0,\infty)^k\oplus \R^{n-k},0)\rightarrow (W,0)$.

Summarizing  the discussion we have proved the  following theorem from \cite{HWZ3}, Theorem 2.2.

\begin{theorem}[{\bf Existence}]\label{thm5.2}\index{T- Local solution germ}
A $\ssc^0$-contraction germ $f\colon {\mathcal O}(\wt{C}\oplus W,0)\rightarrow (W,0)$ admits a uniquely determined 
sc$^0$-solution germ 
$$
\delta\colon {\mathcal O}(\wt{C},0)\rightarrow (W,0)
$$
solving 
$$
f\circ \gr (\delta)=0.
$$
Here $\gr (\delta)$ is the associated graph germ $a\mapsto(a,\delta(a))$.
\end{theorem}

\mbox{}\\

Our next aim is the regularity of the unique continuous solution germ $\delta$ of the equation $f(v, \delta (v))=0$ guaranteed by Theorem  \ref{thm5.2}, and we are going to prove that the solution germ $\delta$ is of class $\ssc^k$ if the given germ $f$ is of class $\ssc^k$. 
 By a somewhat tricky induction it turns out that we actually  only have to know that if  $f$ is  $\ssc^1$,  then the solution germ $\delta$ is 
 $\ssc^1$ as well. Here we shall make use of the following regularity result from \cite{HWZ3}, theorem 2.3, which is the hard part of the regularity theory.

\begin{theorem}\label{thm5.3}\index{T- Regularity of solution germ}
If the $\ssc^0$-contraction germ $f\colon \mo(\wt{C}\oplus  W,0)\rightarrow (
W,0)$ is of class $\ssc^1$, then the solution germ $\delta\colon \mo(\wt{C},
0)\to (W, 0)$ in Theorem \ref{thm5.2} is also of class  $\ssc^1$. 
\end{theorem}

Theorem \ref{thm5.3} shows that a $\ssc^0$-contraction germ $f$
of class $\ssc^1$ has a solution germ $\delta$ satisfying $f(v, \delta (v))=0$ which is 
also of class $\ssc^1$. We shall use this to verify by induction that $\delta$ is of class $\ssc^k$ if $f$ is of class $\ssc^k$.
We start with the following lemma. 

\begin{lemma}\label{newlemma5.4}\index{L- Higher regularity of solution germ}
Let $f\colon {\mathcal O}(\wt{C}\oplus W,0)\rightarrow (W,0)$ be 
a $\ssc^0$-contraction germ of class $\ssc^k$ where  $k\geq 1$.
Moreover, we assume that the solution germ $\delta$ is of class $\ssc^j$.
(By Theorem \ref{thm5.3}, $\delta$ is  at least of class $\ssc^1$.)
We define the germ $f^{(1)}$ by
\begin{equation*}
f^{(1)}\colon {\mathcal O}(T\wt{C}\oplus TW,0)\rightarrow TW,
\end{equation*}
\begin{equation}\label{germequation}
\begin{split}
f^{(1)}(v,b,u,w)&=\left(u-B(v,u),w-DB(v,\delta(v))\left(b,w\right)\right)\\
&=(u,w)-B^{(1)}(v,b,u,w),
\end{split}
\end{equation}
where the last line defines the map $B^{(1)}$. Then $f^{(1)}$ is an
$\ssc^0$-contraction germ and of class $\ssc^{\min\{k-1,j\}}$.
\end{lemma}
\begin{proof}
For $v$ small, the map $B^{(1)}$ has the contraction property with
respect to $(u, w)$.
 Indeed,  on the $m$-level of $(TW)_m=W_{m+1}\oplus W_m$, i.e.,
 for $(u, w)\in W_{m+1}\oplus W_m$,  we  estimate for given $\varepsilon>0$ and $v$ sufficiently small,
  \begin{equation*}
\begin{split}
&|B^{(1)}(v,b,u',w')-B^{(1)}(v,b,u,w)|_m\\
&\phantom{===}=|B(v,u')-B(v,u) |_{m+1}\\
&\phantom{=====}+
|DB(v,\delta(v))(b,w')-DB(v,\delta(v))(b,w)| _{m}\\
&\phantom{===}\leq \varepsilon |u'-u|_{m+1} + |D_2B(v,\delta(v))[w'-w] |_m,
\end{split}
\end{equation*}
which, using the estimate $\norm{D_2B(v, \delta (v))}_m\leq \varepsilon$ for the operator norm from Lemma \ref{new_Lemma3.9},  is estimated by
\begin{equation*}
\phantom{===}\leq \varepsilon \cdot \bigl(|
u'-u|_{m+1}+|w'-w|_m\bigr)= \varepsilon \cdot | (u', w')-(u,
w)|_m.
\end{equation*}
Consequently,  the  germ $f^{(1)}$ is an  $\ssc^0$-contraction germ.
If  now $f$ is of class $\ssc^k$ and $\delta$ of class $\ssc^j$,
then the germ $f^{(1)}$ is of class $\ssc^{\min\{k-1,j\}}$, as one
verifies by comparing the tangent map $Tf$ with the map $f^{(1)}$
and using the fact that the solution $\delta$ is of class $\ssc^j$.
By Theorem \ref{thm5.2}, the solution germ $\delta^{(1)}$ of
$f^{(1)}$ is at least of class $\ssc^0$. It solves the equation
\begin{equation}\label{chap5eq9}
f^{(1)}(v,b, \delta^{(1)}(v, b))=0.
\end{equation}
But also the tangent germ $T\delta$,  defined by $T\delta
(v,b)=(\delta (v), D\delta (v)b)$,  is a solution of \eqref{chap5eq9}.
From the uniqueness we conclude  that $\delta^{(1)}=T\delta $.
\end{proof}

To prove higher regularity we will also make use of  the next lemma.
\begin{lemma}\label{newcontrlem}
Assume we are given a $\ssc^0$-contraction germ  $f$ of class
$\text{sc}^k$ and a solution germ $\delta$ of class $\ssc^j$
with $j\leq k$. Then  there exists a $\ssc^0$-contraction germ
$f^{(j)}$ of class $\ssc^{\min\{k-j, 1\}}$ having
$\delta^{(j)}:=T^j\delta$ as the solution  germ.
\end{lemma}
\begin{proof}
We prove the lemma by induction with respect to $j$. If $j=0$ and
$f$ is a $\ssc^0$-contraction germ of class $\ssc^k$, $k\geq 0$,
then we  set $f^{(0)}=f$ and $\delta^{(0)}=\delta$.  Hence the
result holds true if $j=0$. Assuming the result has been proved for
$j$,  we show it is true for $j+1$. Since $j+1\geq 1$ and $k\geq
j+1$,   the map $f^{(1)}$,  defined by \eqref{germequation},  is of
class $\ssc^{\min \{k-1, j+1\}}$ in view of Lemma \ref{newlemma5.4}.
Moreover, the solution germ $\delta^{(1)}=T\delta$  satisfies
$$
f^{(1)}\circ \gr (\delta^{(1)})=0,
$$
and is of class $\ssc^j$.  Since $\min \{k-1, j+1\}\geq j$, by  the
induction hypothesis there exists  a map
${(f^{(1)})}^{(j)}=:f^{(j+1)}$ of regularity class
$\min\{\min\{k-1,j+1\}-j,1\}=\min\{k-(j+1),1\}$ so that
$$
f^{(j+1)}\circ\gr ({(\delta^{(1)})}^{(j)})=0.
$$
Setting  $\delta^{(j+1)}={(\delta^{(1)})}^{(j)}=T^{j}(T\delta
)=T^{j+1}\delta$,   Lemma \ref{newcontrlem} follows.
\end{proof}

The main result of this section is the following germ-implicit functions theorem.

\begin{theorem}[{\bf Germ-Implicit Function Theorem}] \label{newthm5.4}\index{T- Germ implicit function theorem}
If $f\colon \mo(\wt{C}\oplus {W},0)\to  ({ W},0)$ is  an
$\text{sc}^0$-contraction germ which is, in addition, of class
$\ssc^k$, then the solution germ
$$\delta\colon  \mo(\wt{C},0)\rightarrow ({ W},0)$$
satisfying
$$f(v,\delta (v))=0$$
is also of class $\ssc^k$.
\end{theorem}

From Theorem \ref{newthm5.4}, using  Proposition \ref{sc_up} and \ref{lower}
and Proposition \ref{save}  we deduce the following properties of the section germ $\delta$ under the additional assumptions, that  $f$ is a sc-smooth germ.

\begin{corollary}\label{new_cor_3.33}\index{C- Classical smoothness properties of solution germs}

If the $\ssc^0$-contraction germ $f$ is a sc-smooth germ, there exists for every $m\geq 0$ and $k\geq 0$ a relatively open neighborhoods $V_{m, k}$ of $0$ in $\wt{C}$ such that 
\begin{itemize}
\item[{\em (1)}\ ] $\delta (V_{m, k})\subset W_m$.
\item[{\em (2)}\ ] $\delta\colon  V_{m, k}\to W_m$ is of class $C^k$.
\end{itemize}

In particular, the solution germ $\delta$ is sc-smooth at the smooth point $0$.
\end{corollary}

Theorem \ref{newthm5.4} will be one of the 
building blocks for all future versions of implicit function
theorems,  as well as for the transversality theory.

\begin{proof}[Proof of Theorem \ref{newthm5.4}]
Arguing by contradiction we assume that the solution germ $\delta$ is
of class $\text{sc}^j$ for some $j<k$ but not of class $\text{sc}^{j+1}$. In view of Lemma \ref{newcontrlem}, there exists an
$\text{sc}^0$-contraction germ $f^{(j)}$ of class $\text{sc}^{\min
\{k-j, 1\}}$ such  that  $\delta^{(j)}=T^j\delta$ satisfies
$$
f^{(j)}\circ \gr (\delta^{(j)})=0.
$$
Since also $k-j\geq 1$, it follows that $f^{(j)}$ is at least of
class $\text{sc}^1$. Consequently,  the solution germ $\delta^{(j)}$
is at least of class $\ssc^{1}$. Since
$\delta^{(j)}=T^j\delta$, we conclude that $\delta$ is at least of
class $\text{sc}^{j+1}$ contradicting our assumption. The proof of
Theorem \ref{newthm5.4} is complete.
\end{proof}

The same discussion applies to germs $f$ defined
on $\wt{C}\oplus W$ where $\wt{C}$ is any  finite-dimensional partial quadrant in $\R^n$ leading to the following theorem from \cite{HWZ3}, 
Theorem 2.7. 
\begin{theorem}\label{newthmboundary5.4}\index{T- Germ implicit function theorem {II}}
Let $\wt{C}$ be a finite-dimensional partial quadrant in $\R^n$. If $f\colon \mo(\wt{C}\oplus
{W},0)\to (W,0)$ is a $\ssc^0$-contraction germ which is,
in addition, of class $\ssc^k$, then the solution germ
$$\delta\colon  \mo(\wt{C},0)\rightarrow (W,0)$$
satisfying $f\circ \gr (\delta)=0$
is also of class $\ssc^k$. In particular,  if $f$ is a sc-smooth germ so is its solution germ $\delta$.
\end{theorem}

\begin{remark}\label{hofer-rem}

For  later use we reformulate Corollary \ref{new_cor_3.33} in quantitative terms. If $f$ is a $\ssc^0$-contraction germ which, in addition, is a sc-smooth germ satisfying $f(0)=0$. then the solution germ $\delta$ possesses the following properties of existence, uniqueness, and regularity. 

There exist monotone decreasing sequences  $(\varepsilon_i)$ for $i\geq 0$ and $(\tau_i)$ for $i\geq 0$ such that 
\begin{itemize}
\item[(1)] $\delta\colon \{a\in [0,\infty)^k\oplus \R^{n-k}\, \vert \, |a|_0\leq \varepsilon_0\}\rightarrow
\{w\in W\ |\ |w|_0\leq \tau_0\}$
is a continuous solution of $f(a, \delta (a))=0$ satisfying $\delta (0)=0$. 
 \item[(2)] If the solution  $f(a,w)=0$ satisfies  $|a|_0\leq \varepsilon_0$ and $|w|_0\leq \tau_0$, then $w=\delta(a)$.
\item[(3)] If  $|a|_0\leq \varepsilon_i$, then  $\delta(a)\in W_i$ and $|\delta(a)|_i\leq \tau_i$ for every  $i\geq 0$.
\item[(4)] The germ  $\delta\colon \{a\in [0,\infty)^k\oplus \R^{n-k} \, \vert \,  |a|_0\leq \varepsilon_i\}\rightarrow W_i$ is of class $C^i$,  for every $i\geq 0$.
\end{itemize}

\end{remark}

\subsection{Stability  of Basic Germs}
All the maps considered in the section are sc-smooth maps. 
Let us recall (from Definition \ref{BG-00x}) the notion of  a basic germ 
\begin{definition}[{\bf The basic class $\mathfrak{C}_{basic}$}]\index{D- Basic class}\index{$\mathfrak{C}_{basic}$} Let $W$ be a sc-Banach space.  
A {\bf basic germ} $f\colon {\mathcal O}([0,\infty)^k\oplus\R^{n-k}\oplus W,0)\rightarrow (\R^N\oplus W,0)$
is a sc-smooth germ having the property that the germ $P\circ f$ is a $\ssc^0$-contraction  germ, where 
$P\colon \R^N\oplus W\rightarrow W$ is the sc-projection. 
We denote the class of all basic germs by $\mathfrak{C}_{basic}$. 
\end{definition}
In view of Definition \ref{oi}, the basic germs are the local models for the germs of sc-Fredholm sections.


 \begin{theorem}[{\bf Weak Stability of Basic Germs}]\label{arbarello}\index{T- Weak stability of basic germs}
We consider a basic germ 
$$
f\colon {\mathcal O}(([0,\infty)^k\oplus \R^{n-k})\oplus W,0)\rightarrow (\R^N\oplus W,0), 
$$
which we can view as the principal part of a sc-smooth section of the obvious strong bundle. We assume that
$s$ is the principal part of a $\ssc^+$-section of the same bundle satisfying $s(0)=0$.  Then there exists a strong bundle isomorphism
$$
\Phi\colon U\triangleleft (\R^N\oplus W)\rightarrow U'\triangleleft (\R^{N'}\oplus W'),
$$
where $U$ is an open neighborhood of $0$ in $[0,\infty)^k\oplus \R^{n-k}\oplus W$, and $U'$ is an open neighborhood
of $0$ in $[0,\infty)^k\oplus \R^{n'-k}\oplus W'$, 
covering the sc-diffeomorphism $\varphi\colon (U,0)\rightarrow (V,0)$, 
so that ${(\Phi_\ast(f+s))}^1$ is a basic germ. 

Here ${(\Phi_\ast(f+s))}^1$  is the germ $\Phi_\ast(f+s)\colon V^1\to (\R^{N'}\oplus W')^1$, where the levels are raised by $1$.
\end{theorem}

Recalling  the Fredholm index of a basic germ in Proposition \ref{Newprop_3.9}, we conclude that $n-N=n'-N'$, because the Fredholm index is invariant under strong bundle isomorphisms. The integer $k$ is the  degeneracy index $k=d_C(0)$ of the point $0$ which is, in view of Proposition \ref{newprop2.24} and 
Corollary \ref{equality_of_d}, preserved under the sc-diffeomorphism $\varphi$ satisfying $\varphi (0)=0$.
Although Theorem \ref{arbarello}  was not explicitly formulated  in \cite{HWZ3}, it follows from the proof of Theorem 3.9  in \cite{HWZ3}.

\begin{proof}
Denoting by  $P\colon \R^N\oplus W\rightarrow W$ the sc-projection, the  composition $P\circ f$ is, by definition, of  the form
$$
 P\circ f(a,w)=w-B(a,w), 
$$
and has the property that for every $\varepsilon>0$ the estimate $\abs{B(a,w)-B(a,w')}_m\leq \varepsilon\cdot |w-w'|_m$ holds,   if $(a, w)$ and $w'$ are sufficiently small on level $m$.

Linearizing the  $\ssc^+$-section  $s$ with respect to the variable $w\in W$ at the point $0$, we introduce the sc-operator 
$$A:=P\circ D_2s(0)\colon W\to W.$$
Since $s$ is a
$\ssc^+$-section and $0$ is smooth point, the operator $A\colon W\rightarrow W$ is a $\ssc^+$-operator. Therefore,  the operator ${\mathbbm 1}+A\colon  W\to W$ is a  $\ssc^+$-perturbation of the identity and hence a sc-Fredholm operator by Proposition \ref{prop1.21}.  Because $A$ is level wise compact, 
the index $\ind ({\mathbbm 1}+A)$ is equal to $0$. 
The associated sc-decompositions of the sc-Banach space $W$ are the following,
$$
{\mathbbm 1}+A\colon W=C\oplus X\to W=R\oplus Z,
$$
where $C=\text{ker}({\mathbbm 1}+A)$ and $R=\text{range}\ ({\mathbbm 1}+A)$ and
$\text{dim}\ C=\text{dim}\ Z<\infty$.

Since $s$ is a $\ssc^+$-section, we conclude from Proposition \ref{lower} that the restriction  $s\colon U_m\to \R^N\oplus W_m$ is of class $C^1$, for every 
$m\geq 1$.  
From the identity 
$P\circ s(a, w)=
P\circ D_2 s(0)w+(P\circ s(a, w)-P\circ D_2 s(0)w)$,
one  deduces the following representation for 
$P\circ s$, on every level $m\geq 1$, 
$$
\text{$P\circ s(a,w)=Aw+S(a, w)$\quad  and \quad 
 $D_2S(0,0)=0$}.$$
Therefore,  $S$ is, with respect to the second
variable $w$, a arbitrary small contraction on every level $m\geq 1$, if $a$ and $w$ are sufficiently small  depending 
on the level $m$ and the contraction constant.  
We can make the arguments which follow only on the levels $m\geq 1$. This explains  the reason for the index  raise by $1$ in the theorem.

We can write
\begin{equation*}
\begin{split}
P\circ (f+s)(a, w)&=w-B(a, w)+Aw+S(a, w)\\
&=({\mathbbm 1}+A)w-[ B(a, w)-S(a, w) ]\\
&=({\mathbbm 1}+A)w-\ov{B}(a, w),
\end{split}
\end{equation*}
where we have abbreviated
$$\ov{B}(a, w)=B(a, w)-S(a, w).$$
By assumption,  the map $B$ belongs to the $\ssc^0$-contraction germ and hence the map $\ov{B}$ is a contraction in the second
variable on every level $m\geq 1$ with arbitrary small contraction
constant $\varepsilon>0$ if $a$ and $w$ are sufficiently small
depending on the level $m$ and the contraction constant $\varepsilon$.  Introducing  the canonical projections by
\begin{align*}
P_1&\colon W=C\oplus X\to X\\
P_2&\colon W=R\oplus Z\to R,
\end{align*}
we abbreviate 
\begin{equation*}
\begin{split}
\varphi (a, w)&:=P_2\circ P\circ (f+s)(a,w)\\
&=P_2  [({\mathbbm 1}+A)w-\ov{B}(a, w)]\\
&=P_2[ ({\mathbbm 1}+A)P_1 w- \ov{B}(a, w)].
\end{split}
\end{equation*}
We have used the relation  $({\mathbbm 1}+A)({\mathbbm 1}-P_1)=0$. The operator
$L:=({\mathbbm 1}+A)\vert X\colon X\to R$ is a sc-isomorphism. In view of
$L^{-1}\circ P_2\circ ({\mathbbm 1}+A)P_1w=P_1w$, we obtain the formula
$$
L^{-1}\circ \varphi (a, w)=P_1w-L^{-1}\circ P_2\circ \ov{B}(a, w).$$

\noindent Writing $w=({\mathbbm 1}-P_1)w\oplus P_1w$,  we  shall consider $(a,
({\mathbbm 1}-P_1)w)$ as our  new finite parameter,  and
correspondingly  define the map $\wh{B}$ by
$$
\wh{B}((a, (1-P_1)w),
P_1w)=L^{-1}\circ P_2\circ  \ov{B}(a, ({\mathbbm 1}-P_1)w+P_1w).
$$
Since $\ov{B}(a, w)$ is a contraction in the second variable on every level $m\geq 1$ with arbitrary small contraction constant if $a$ and $w$ are sufficiently small depending on the level $m$ and the contraction constant,  the right hand side of
$$
L^{-1}\circ \varphi (a, ({\mathbbm 1}-P_1)w+P_1w)=P_1w-\wh{B}(a, ({\mathbbm 1}-P_1)w, P_1w)
$$
possesses the required  contraction normal
form with respect to the variable $P_1w$ on all levels $m\geq 1$,
again if $a$ and $w$ are small enough depending on $m$ and the contraction constant.

 It remains
to prove that the above normal form is the result of an admissible
coordinate transformation of the perturbed section $f+s$. Choosing  a
linear isomorphism $\tau \colon Z\to C$, we  define the fiber
transformation $\Psi\colon \R^N\oplus W\to \R^N\oplus X\oplus C$ by
\begin{equation*}
\begin{split}
\Psi (\delta a\oplus \delta w):= \delta a \oplus L^{-1}\circ P_2 \cdot \delta w \oplus  \tau  \circ ({\mathbbm 1}-P_2)\cdot \delta w.
\end{split}
\end{equation*}
We shall view $\Psi$ as a strong bundle map covering the
sc-diffeomorphism $\psi\colon V\oplus W\to V\oplus C\oplus X$ defined by $\psi (a,
w)=(a, (1-P_1)w, P_1w)$ where $V=[0,\infty )^k\oplus \R^{n-k}.$ With
the canonical projection
\begin{gather*}
\ov{P}\colon  (\R^N\oplus C)\oplus X\to X\\
\ov{P}(a\oplus ({\mathbbm 1}-P_1)w\oplus P_1w)=P_1w,
\end{gather*}
and the relation  $\ov{P}\circ \Psi \circ ({\mathbbm 1}-P)=0$,  we obtain the
desired formula
\begin{gather*}
\ov{P}\circ \Psi \circ(f+s)\circ \psi^{-1} (a, ({\mathbbm 1}-P_1)w, P_1w)\\
=P_1w-\wh{B}(a, ({\mathbbm 1}-P_1)w, P_1w).
\end{gather*}
The proof of Theorem \ref{arbarello}  is complete.

\end{proof}

The theorem has the following corollary, where we use the standard notations, denoting, as usual,  by $C$  a partial quadrant in a sc-Banach space $E$.
We also use a second sc-Banach space $F$. 

\begin{corollary}\label{op-perp}\index{C- Weak stability of basic germs}
We assume that the sc-germ  $g\colon {\mathcal O}(C,0)\rightarrow (F,0)$ is equivalent by a strong bundle isomorphism  $\Phi$ to the basic germ $\Phi_\ast g$, and assume that  $s\colon {\mathcal O}(C,0)\rightarrow (F,0)$ is a $\ssc^+$-germ. Then there exists a strong bundle map $\Psi$ such  that 
 $((\Psi\circ \Phi)_\ast(g+s))^1$ is a basic germ.
\end{corollary}

\begin{proof}
By assumption there exist open neighborhoods $U$ of $0$ in $C$ and $V$ of $0$ in $[0,\infty)^k\oplus \R^{n-k}\oplus W$
and a sc-diffeomorphism $\varphi\colon (U,0)\rightarrow (V,0)$ which is covered by a strong bundle isomorphism
$\Phi\colon U\triangleleft F\rightarrow V\triangleleft(\R^N\oplus W)$ such  that $\Phi_\ast g$ is a basic germ $h$.
Then $t=\Phi_\ast s$ defines a $\ssc^+$-section satisfying  $t(0)=0$. Clearly $\Phi_\ast(g+s)=h+t$,  and applying Theorem \ref{arbarello}, we find a  second strong bundle map $\Psi$ such  that $(\Psi_\ast(h+t))^1$ is a basic germ.  Taking the composition $\Gamma =\Psi\circ \Phi$, we conclude that 
$\Gamma_\ast (g+s)^1$ is a basic germ.  This completes the proof of Corollary \ref{op-perp}.
\end{proof}

In order to illustrate the corollary, we now consider the sc-smooth germ  $h\colon {\mathcal O}(C,0)\rightarrow F$ for which we know that there exists a $\ssc^+$-germ $s$
satisfying  $s(0)=h(0)$,  and assume that the germ $h-s$ around $0$  is equivalent to the  basic germ $g=\Phi_\ast(h-s)$.
We observe that  $h-h(0)=(h-s) +(s-h(0))$, where  $s-h(0)$ is a $\ssc^+$-section. Then $t=\Phi_\ast(s-h(0))$ is a $\ssc^+$-section and $g+t$ is a perturbation by a $\ssc^+$-section of a basic germ. By the previous corollary we find 
a strong bundle coordinate change such  that $(\Psi_\ast(g-s))^1$ is a basic germ, or in  other words, 
$((\Psi\circ\Phi)_\ast( h-h(0)))^1$ is a basic germ.

Note that for the implicit function theorem it does not matter whether  we work with $f$, or $f^1$, ore even $f^{(501)}$.
It  matters that our coordinate change is compatible with the original sc-structure. 

We also  point out  that a strong bundle coordinate change 
for $h^1$ is not the same as  a strong bundle coordinate change for $h$ followed by a subsequent raise of the index.

\subsection{Geometry of Basic Germs}
In this section we shall study in detail sc-smooth germs 
\begin{equation}\label{poi1}
f\colon {\mathcal O}(([0,\infty)^k\oplus \R^{n-k})\oplus W,0)\rightarrow (\R^N\oplus W,0)
\end{equation}
around $0$
of the form
\begin{equation}\label{pi2}
f=h+s
\end{equation}
where $h$ is a basic germ and $s$ is a $\ssc^+$-germ satisfying $s(0)=0$. 
\mbox{}\\

We already know from Corollary \ref{Newprop_3.9} that $Df(0)\colon \R^n\oplus W\to \R^N\oplus W$ is a sc-Fredholm operator of index $\ind Df(0)=n-N$.

\mbox{}\\

In the following we abbreviate $E=\R^n\oplus W$,  $C=([0,\infty)^k\oplus \R^{n-k})\oplus W$, and  $F=\R^N\oplus W$  and by  $P\colon \R^N\oplus W\to W$ the sc-projection.

\begin{theorem}[{\bf Local Regularity and Compactness}]\label{save}\index{T- Local regularity and compactness}
Let $U$ be a relatively open neighborhood of $0$ in $C$.  We  assume that $f\colon U\rightarrow F$ is a sc-smooth map satisfying 
$f(0)=0$ and of the form $f=h+s$ where $h$ is a basic germ and $s$ is a  $\ssc^+$-germ  satisfying $s(0)=0$. We denote by $S=\{(a,w)\in U \, \vert \, f(a,w)=0\}$ the solution set of $f$ in $U$.  Then there exists 
a nested sequence 
$$U\supset {\mathcal O}(0)\supset {\mathcal O}(1)\supset {\mathcal O}(2)\supset \ldots $$
of relatively open neighborhoods of $0$ in $C$ on level $0$ such that  for every $m\geq 0$, the closure of $S\cap {\mathcal O}(m)$ in $C\cap E_0=C_0$ is contained in $C\cap E_m=C_m$, i.e., 
$$
\cl_{C_0}(S\cap {\mathcal O}(m))\subset C_m.
$$
\end{theorem}

Theorem \ref{save} says, in particular, that 
${\mathcal O}(m)\cap S\subset C_m$ for all $m\geq 0$.  Therefore,  the regularity of solutions $(a, w)$ of the equation $f(a, w)=0$ is the higher, the closer to $0$  they are on the level $0$.  Moreover, the solution set on level $m$ sufficiently close to $0$ on level $0$ has a closure on level $0$, which still belongs to level $m$. 
Moreover, the solution set on level $m$, sufficiently close to $0$ on level $0$, has a closure on level $0$, which belongs to level $m$.

\begin{proof}

We construct the sets ${\mathcal O}(m)$ inductively  by showing that there exists a decreasing sequence $(\tau_m)_{m\geq 0}$ of positive numbers such that the sets 
$$
{\mathcal O}(m)=\{(a, w)\in C\, \vert \, \text{$\abs{a}_0<\tau_m$ and $\abs{w}_0<\tau_m$}\}
$$
have the desired properties. We begin with the construction of ${\mathcal O}(0)$. 

By definition of a basic germ,  the composition $P\circ h$ is of the form 
$$P\circ h(a, w)=w-B(a, w),$$ 
where $B(0, 0)=0$ and $B$ is a contraction in $w$ locally near $(0, 0)$. 
Moreover, $s$ is a $\ssc^+$-germ satisfying $s(0)=0$.

We choose $\tau_0'>0$ such that the closed set $\{(a, w)\in C\, \vert \, \abs{a}_0\leq  \tau_0',\abs{w}_0\leq \tau_0'\}$ in  $E_0$ is contained in $U$  and such that, in addition, 
\begin{itemize}
\item[$(0_1)$\, ] $\abs{B(a, w)-B(a, w')}_0\leq \dfrac{1}{4}\abs{w-w'}_0$.
\end{itemize}
for all $a$, $w$, and $w'\in W_0$ satisfying $\abs{a}_0\leq \tau_0'$, 
$\abs{w}_0\leq \tau_0'$, and $\abs{w'}_0\leq \tau_0'$.
Using $B(0)=0$ and $s(0)=0$,  we can choose $0<\tau_0<\tau_0'$ such that 
\begin{itemize}
\item[$(0_2)$\, ] $\abs{B(a, 0)}_0\leq \dfrac{1}{4}\tau_0'$ for all 
$\abs{a}_0\leq \tau_0.$
\item[$(0_3)$\, ] $\abs{P\circ s(a, w)}_0\leq \dfrac{1}{4}\tau_0'$ for all 
$\abs{a}_0\leq \tau_0$ and $\abs{w}_0\leq \tau_0$.
\end{itemize}
For these choices of the constants $\tau_0'$ and $\tau_0$, we introduce the closed set 
$$\Sigma_0=\{(a, z)\in C\, \vert \, \abs{a}_0\leq \tau_0,\, \abs{z}_0\leq \tau_0'/4\}, $$ and denote by $\ov{B}_0(\tau_0')\subset W_0$ the  closed ball in $W_0$ centered  at $0$ and having radius $\tau_0'$.
We define the map 
$F_0\colon \Sigma_0\times \ov{B}_0(\tau_0')\to  W_0$  by 
$$F_0(a, z, w)=B(a, w)-z.$$
If $(a, z)\in \Sigma_0$ and $w, w'\in \ov{B}_0(\tau_0')$, we estimate  using $(0_1)$ and $(0_2)$, 
\begin{equation*}
\begin{split}
\abs{F_0(a, z, w)}_0&=\abs{B(a, w)-z}_0\leq \abs{B(a, w)-B(a, 0)}_0+\abs{B(a, 0)}_0+\abs{z}_0\\
&\leq \dfrac{1}{4}\tau_0'+\dfrac{1}{4}\tau_0'+\dfrac{1}{4}\tau_0'=\dfrac{3}{4}\tau_0'<\tau_0',
\end{split}
\end{equation*}
and
\begin{equation*}
\abs{F_0(a, z, w)-F_0(a, z, w')}_0\leq \dfrac{1}{4}\abs{w-w'}_0.
\end{equation*}
Hence $F_0(a, z,\cdot )\colon \ov{B}_0(\tau_0')\to \ov{B}_0(\tau_0')$ is a contraction,  uniform in  $(a, z)\in \Sigma_0$. 
Therefore, by the parametrized version of Banach's  fixed point theorem there exists a unique continuous function 
$\delta_0\colon \Sigma_0\to \ov{B}_0(\tau_0')$ solving the 
equation 
$$\delta_0(a, z)=B(a, \delta_0 (a, z))-z$$ 
for all  $(a, z)\in \Sigma_0$.
Now we define the open neighborhood ${\mathcal O}(0)$ by 
$${\mathcal O}(0)=\{(a, w)\in C\, \vert \, \text{$\abs{a}_0<\tau_0$ and $\abs{w}_0<\tau_0$}\}.$$
Clearly, the set ${\mathcal O}(0)$ satisfies $\text{cl}_{C_0}(S\cap {\mathcal O}(0))\subset C_0$.

We observe that if $(a, w)\in {\mathcal O}(0)$, then $\abs{P\circ s(a, w)}_0\leq {\tau_0'}/4$ by $(0_3)$ so that $\delta_0(a, P\circ s(a, w))$ is defined. If, in addition, $f(a, w)=0$, then  $P\circ f(a, w)=0$  so that $w=B(a, w)-P\circ s(a, w)$ and we claim that  
\begin{equation}\label{new_number_1}
w=\delta_0(a, P\circ s(a, w))\quad \text{for all $(a, w)\in {\mathcal O}(0)$.}
\end{equation} 
Indeed, since $\delta_0(a, P\circ s(a, w))=B(a, \delta_0(a, P\circ s(a, w)))-P\circ s(a, w)$, we estimate, using $(0_1)$,  
\begin{equation*}
\begin{split}
\abs{w-\delta_0(a, P\circ s(a, w))}_0&=\abs{B(a, w)-B(a, \delta_0(a, P\circ s(a, w)))}_0\\
&\leq 
\dfrac{1}{4}\abs{w-\delta_0(a, P\circ s(a, w))}_0,
\end{split}
\end{equation*}
implying $w=\delta_0(a, P\circ s(a, w))$ as claimed.


We next construct the set ${\mathcal O}(1)\subset {\mathcal O}(0)$.  Since the embedding $W_1\to W_0$ is continuous, there is a constant $c_1>0$ such that 
$\abs{\cdot }_0\leq c_1\abs{\cdot}_1$. With the constant $\tau_0$ defined above, we choose $0<\tau_1'<\min \{\tau_0, \tau_0/c_1\}$ such that the following holds. 
The set $\{(a, w)\in C\, \vert \,\text{$\abs{a}_0\leq \tau_1'$ and $\abs{w}_1\leq \tau_1'$}\}$ is contained in $U$,  and 
\begin{itemize}
\item[$(1_1)$\, ] $\abs{B(a, w)-B(a, w')}_1\leq \dfrac{1}{4}\abs{w-w'}_1$.
\end{itemize}
for all $a$, $w, w'\in W_1$ satisfying $\abs{a}_0\leq \tau_1'$, 
$\abs{w}_1\leq \tau_1'$, and $\abs{w'}_1\leq \tau_1'$.
We choose $0<\tau_1<\tau_1'$ such that 
\begin{itemize}
\item[$(1_2)$\, ] $\abs{B(a, 0)}_1\leq \dfrac{1}{4}\tau_1'$ for all 
$\abs{a}_0\leq \tau_1$
\item[$(1_3)$\, ] $\abs{P\circ s(a, w)}_1\leq \dfrac{1}{4}\tau_1'$ for all 
$\abs{a}_0\leq \tau_1$ and $\abs{w}_0\leq \tau_1$.
\end{itemize}
Proceeding as in the construction of ${\mathcal O}(0)$, we introduce the closed set $\Sigma_1$ in $E_1$ by 
$$\Sigma_1=\{(a, z)\in C_1\,\vert \, \abs{a}_0\leq  \tau_1,\, \abs{z}_1\leq \tau_1'/4\}, $$
and abbreviate by $\ov{B}_1(\tau_1')$ the closed ball in $W_1$ having its center at $0$ and radius $\tau_1'$.  We define the map 
$F_1\colon \Sigma_1\times \ov{B}_1(\tau_1')\to W_1$ by 
$F_1(a, z, w)=B(a, w)-z$. By  $(1_1)$ and $(1_3)$, the map 
$F_1\colon \Sigma_1\times \ov{B}_1(\tau_1')\to \ov{B}_1(\tau_1')$ is a contraction, uniform in $(a, z)\in \Sigma_1$. 
Again using the Banach fixed point theorem, we find a unique continuous map
$\delta_1\colon \Sigma_1\to \ov{B}_1(\tau_1')$  solving the equation $\delta_1(a, z)=B(a, \delta_1(a, z))-z$ for all $(a, z)\in \Sigma_1$.  
Now we define the open neighborhood ${\mathcal O}(1)$ as 
$${\mathcal O}(1)=\{(a, w)\in C_0\, \vert \, \abs{a}_0<\tau_1,\, \abs{w}_0< \tau_1\}.$$
By our definition of $\tau_1$ we have $\tau_1\leq \tau_1'<\tau_0$ so that  
${\mathcal O}(1)\subset {\mathcal O}(0)$.

We next claim that 
\begin{equation}\label{delta_0_equals_delta_1}
\delta_0(a, P\circ s(a, w))=\delta_1(a, P\circ s(a, w))\quad \text{for all $(a, w)\in {\mathcal O}(1)$}.
\end{equation}
To verify the claim,  we note that if $(a, w)\in {\mathcal O}(1)$, then, by $(1_3)$,  $\abs{P\circ s(a, w)}_1\leq \tau'_1/4$. Hence $\delta_1(a,  P\circ s(a, w))$ is defined  and its norm  satisfies $\abs{\delta_1(a, P\circ s(a, w))}_1\leq \tau_1'$ because $\delta_1$ takes its values in the ball $\ov{B}_1(\tau_1')$. 
This implies, recalling that $\abs{\cdot}_0\leq c_1\abs{\cdot}_1$ and 
$\tau_1'\leq \tau_0/c_1$,  the estimate
$$\abs{\delta_1(a, P\circ s(a, w))}_0\leq c_1\abs{\delta_1(a, P\circ s(a, w))}_1\leq c_1\tau_1'\leq \tau_0\leq \tau'_0.$$
Therefore, by construction, the map 
$(a, w)\mapsto \delta_1(a, P\circ s(a, w))$ solves the equation $\delta_1(a, P\circ s(a, w))=B(a, \delta_1(a, P\circ s(a, w))-P\circ s(a, w)$ for all $(a, w)\in {\mathcal O}(1)$.  On the other hand, it follows from ${\mathcal O}(1)\subset {\mathcal O}(0)$ and $(0_3)$ that $\abs{P\circ s (a, w)}_0\leq \tau_0'/4$ and hence the map $(a, w)\mapsto \delta_0(a, P\circ s(a, w))$ solves, by construction, the same equation $\delta_0(a, P\circ s(a, w))=B(a, \delta_0(a, P\circ s(a, w))-P\circ s(a, w)$ for all $(a, w)\in {\mathcal O}(1)$. The claim \eqref{delta_0_equals_delta_1} now follows from the uniqueness of the Banach fixed point theorem on the level $0$.

If $(a, w)\in {\mathcal O}(1)$ satisfies, in addition, $f(a, w)=0$, we deduce from \eqref{new_number_1} and 
\eqref{delta_0_equals_delta_1} that 
\begin{equation}\label{w_equals_delta_1}
w=\delta_1(a, P\circ s(a, w))\in W_1\quad \text{for all $(a, w)\in {\mathcal O}(1)$}.
\end{equation}

In order to verify the desired  property of ${\mathcal O}(1)$ we fix 
$(a, w)\in \text{cl}_0(S\cap {\mathcal O}(1))$. Then there exists a sequence  
$(a_n, w_n)\in S\cap {\mathcal O}(1)$ such that $(a_n, w_n)\to (a, w)$ on level $0$. 
From \eqref{w_equals_delta_1} it follows that  $w_n=\delta_1(a_n, P\circ s(a_n, w_n))\in W_1$  for all $n$.   Since $s$ is $\ssc^+$,  we know that $P\circ s(a_n, w_n)\to P\circ s(a, w)$ on level $1$. From the continuity of $\delta_1$ we conclude the convergence 
$w_n=\delta_1(a_n, P\circ s(a_n, w_n))\to \delta_1(a, P\circ s(a, w))=w$ on level $1$.  Consequently, $(a, w)\in C_1$ as desired.

The induction step is now clear and the further details are left to the reader.

\end{proof}

The previous result has a useful corollary.

\begin{corollary}\label{corex1}\index{C- Stability of surjectivity}
We assume that $U$ is a relatively open neighborhood of $0$ in a partial quadrant $C=[0,\infty)^n\oplus W$ in a sc-Banach space $E=\R^n\oplus W$ and let $F=\R^N\oplus W$.  Let $f\colon U\to F$ be a sc-smooth map satisfying $f(0)=0$ and  admitting the decomposition
$f=h+s$ where $h\in \mathfrak{C}_{basic}$ and $s$  is a $\ssc^+$-map satisfying $s(0)=0$. We assume, in addition, that 
$Df(0,0)$ is surjective.  Then there exists a 
 relatively open neighborhood $U'\subset U$ on  level $0$ such that the following holds.
\begin{itemize} 
\item[{\em (1)}] If $(a,w)\in U'$ satisfies  $f(a,w)=0$, then $(a,w)$ is on level $1$.
\item[{\em (2)}] If $(a,w)\in U'$ and $f(a, w)=0$, then $Df(a,w)\colon E\rightarrow F$ is a surjective Fredholm operator of index $n-N$.
\end{itemize}
\end{corollary}

\begin{proof}

In view of Theorem \ref{save}  we know that if $f(a,w)=0$ and $(a,w)$ is sufficiently close to $(0,0)$ on level $0$, then $(a, w)\in E_1$. Hence the linearization $Df(a,w)$ is a well-defined as a bounded linear operator from $E_0$ to $F_0$. By Proposition \ref{Newprop_3.9}, the linearization $Df(0,0)\colon E\to F$ is a  Fredholm operator whose index is equal to $\ind Df(0)=n-N$. 

By assumption, the Fredholm operator $Df(0,0)\colon E\to F$ is surjective. Denoting by $K$ its kernel, we have the splitting $E=K\oplus N$ and conclude that the restriction $Df(0)\vert N\colon N\to F$ is an isomorphism of 
Banach spaces and hence a Fredholm operator of index $0$. To see this, we observe that 
$$
Df(a, w)\vert N=Df(0, 0)\vert N+(Df(a, w)\vert N-Df(0, 0)\vert N), 
$$
and the  second term is a admissible perturbation which does not affect the Fredholm character nor the index. Indeed, 
\begin{equation*}
\begin{split}
\bigl( Df(a, w)-Df(0, 0)\bigr)(\alpha, \zeta)=D_2B(a, w)\zeta+C(a, w)(\alpha, \zeta),
\end{split}
\end{equation*}
where $C(a, w)\colon E\to F$ is a compact operator, and 
$\norm{D_2B(a, w)}\leq \varepsilon$ for every $\varepsilon>0$ if $(a, w)$ sufficiently small in $E_0$ depending on $\varepsilon$,  in view of 
Lemma \ref{new_Lemma3.9}.

It follows that $Df(a, w)\vert N\colon N\to F$ is a Fredholm operator of index $0$, if $(a, w)\in E_1$ is sufficiently small in $E_0$.

\mbox{}\\

We finally show that $Df(a, w)\vert N\colon N\to F$ is a surjective operator if $(a, w)\in E_1$ solves, in addition, $f(a, w)=0$.
Since the index is equal to $0$, the kernel of $Df(a, w)\vert N$ has the same dimension as the cokernel of $Df(a, w)\vert N$ in $F$. Hence we have to prove that the kernel of $Df(a, w)\vert N$ is equal to $\{0\}$, if $(a, w)\in E$ is close to $(0, 0)$ in $E_0$ and solves $f(a, w)=0$.

Arguing by contradiction we assume that  there exist a sequence $(a_k, w_k)\in C\cap E_1$ satisfying $f(a_k, w_k)=0$ and $\abs{(a_k, w_k)}_0\to 0$ as $k\to \infty$. Moreover, there exists a sequence 
$(\alpha_k, \zeta_k)\in (\R^n\oplus W)\cap N$  satisfying $\abs{(\alpha_k, \zeta_k)}_0=1$ and $Df(a_k, w_k)(\alpha_k, \zeta_k)=0$.  Consequently, 
\begin{equation}\label{eq_surj_1}
 \zeta_k-D_2B(a_k,w_k) \zeta_k = D_1B(a_k,w_k)\alpha_k - PDs(a_k,w_k)(\alpha_k,  \zeta_k).
\end{equation}

Without loss of generality we  may assume that $\alpha_k\rightarrow \alpha$.  In view of the proof of the previous theorem, for large values of $k$, the sequence $(a_k, w_k)$ is bounded on level $2$. Consequently, since the embedding $W_2\to W_1$ is compact, we may assume that $(a_k, w_k)\to (0, 0)$ on level $1$. Therefore, 
$D_1B(a_k,w_k)\alpha_k \rightarrow D_1B(0,0)\alpha$ in $E_0$. 
In addition, since $s$ is a $\ssc^+$-operator, the map $E_1\oplus E_0\to F_1$,  defined  by $(x, h)\mapsto PDs(x)h$,  is continuous. Hence,  there exists $\rho>0$ such that $\abs{PDs(x)h}_1\leq 1$ if $\abs{x}_1\leq \rho$ and $\abs{h}_0\leq \rho$. This implies that there is a constant $c>0$ such that $\abs{PDs(x)h}_1\leq c$ for all $\abs{x}_1\leq \rho$ and $\abs{h}_0\leq 1$.  From this estimate, we conclude that the sequence 
$PDs(a_k,w_k)(\alpha_k, \zeta_k)$ is bounded in $W_1$. Since the embedding $W_1\to W_0$ is compact, we may assume that the sequence $PDs(a_k,w_k)(\alpha_k, \zeta_k)$ converges to some point $w_0$  in $W_0$. Denoting by $z_k$ the right-hand side of \eqref{eq_surj_1}, we have proved that $z_k\to z_0=D_1B(0,0)\alpha-w_0$ in $W_0$.  Choosing $0<\varepsilon<1$ in Lemma \ref{new_Lemma3.9}, the  operators ${\mathbbm 1}-D_2B(a_k, w_k)$ have a bounded inverse for large $k$ and we obtain from \eqref{eq_surj_1} that 
\begin{equation*}
\zeta_k=\bigl({\mathbbm 1}-D_2B(a_k, w_k)\bigr)^{-1}z_k=\sum_{l\geq 0}\bigl(D_2B(a_k, w_k)\bigr)^{l}z_k
\end{equation*}
for large $k$. We claim that the sequence $\zeta_k$ converges to $z_0$ in $W_0$. Indeed, take  $\rho>0$. From $\norm{D_2B(a_k, w_k)}\leq \varepsilon$ for $k$ large and the fact that the sequence  $(z_k)$ is bounded in $W_0$, it follows that there exists $l_0$ such that 
$$\abs{\sum_{l\geq l_0}\bigl(D_2B(a_k, w_k)\bigr)^{l}z_k}_0\leq \rho/2.$$
Hence 
\begin{equation*}
\begin{split}
\abs{\zeta_k-z_0}_0&\leq \abs{z_k-z_0}_0+\sum_{l=1}^{l_0}\abs{\bigl(D_2B(a_k, w_k)\bigr)^{l}z_k}_0+\abs{\sum_{l>l_0}\bigl(D_2B(a_k, w_k)\bigr)^{l}z_k}_0\\
&\leq \abs{z_k-z_0}_0+\sum_{l=1}^{l_0}\abs{\bigl(D_2B(a_k, w_k)\bigr)^{l}z_k}_0+\rho/2.
\end{split}
\end{equation*}
From  $\abs{z_k-z_0}_0\to 0$ and $D_2B(a_k, w_k)z_k\to D_2B(0, 0)z_0=0$, it follows that $\limsup_{k\to \infty}\abs{\zeta_k-z_0}_0\leq \rho/2$ and,  since $\rho$ was arbitrary,  that $\zeta_k\to z_0$ in $W_0$.
We have proved that $(\alpha_k, \zeta_k)\to (\alpha, z_0)$ in $E_0$.  Hence $(\alpha, z_0)\in N$ and 
$\abs{(\alpha, z_0)}_0=1$. On the other hand,
\begin{equation*}
0=\lim_{k\to \infty}Df(a_k, w_k)(\alpha_k, \zeta_k)=Df(0)(\alpha, z_0).
\end{equation*}
Since $Df(0, 0)\vert N$ is an isomorphism, $(\alpha, z_0)=(0, 0)$, in  contradiction to $\abs{(\alpha, z_0)}_0=1$. The proof of 
Corollary \ref{corex1} is complete.

\end{proof}

In the following theorem we denote, as usual, by $C$ the  partial quadrant $[0,\infty )^k\oplus \R^{n-k}\oplus W$ in the  sc-Banach space $E=\R^n\oplus W$ and by $U$ a relatively open neighborhood of $0$ in $C$. Moreover, $F$ is another sc-Banach space of the form $F=\R^N\oplus W$. 

We consider a sc-smooth germ $f\colon U\to F$ satisfying $f(0)=0$ of the form $f=h+s$ where $h$ is a basic germ and $s$ is a $\ssc^+$-germ satisfying $s(0)=0$.

By Theorem \ref{arbarello} there exists a strong bundle isomorphism 
$$\Phi\colon U\triangleleft (\R^N\oplus W)\to U'\triangleleft (\R^{N'}\oplus W').
$$
where $U'$ is a relatively open neighborhood of $0$ in the partial quadrant $C'=[0,\infty)^k\oplus \R^{n'-k}\oplus W'$, covering the sc-diffeomorphism 
$\varphi\colon (U, 0)\to (U',0)$ such that the section $g=\Phi\circ f\circ \varphi^{-1}$ has the property that $g^1=(\Phi_\ast (h+s))^1\colon (U')^1\to (\R^{N'}\oplus W')^1$ is a basic germ. Clearly, $\varphi \bigl(\{x\in U\, \vert \, f(x)=0\}\bigr) =\{x'\in U'\, \vert \, g(x')=0\}.$

{We abbreviate $F'=\R^{N'}\oplus W'$ and denote by $P'$ the sc-projection $P'\colon \R^{N'}\oplus W'\to W'$.

By definition of a basic germ, the composition $P'\circ g^1$ is a $\ssc^0$-contraction germ which is sc-smooth. Therefore, in view of Remark \ref{hofer-rem} applied to the levels $m\geq 1$,
there are monotone decreasing sequences $(\varepsilon_i')$ for $i\geq 1$ and 
$(\tau_i')$ for $i\geq i$ such that, abbreviating by $\ov{B}_i(\tau_i')$ the closed ball $W'_i$ of center $0$ and radius $\tau_i'$, and by $U'_i$ the neighborhood $U_i'=\{a\in [0,\infty )^k\oplus \R^{n'-k}\, \vert \, \abs{a}_0\leq \tau_i'\}$, the following statements (1)-(4) hold. 
\begin{itemize}
\item[(1)] There exists a unique continuous map $\delta\colon U'_1\rightarrow \ov{B}_1(\tau_1') $  satisfying  $P'\circ g(a,\delta(a))=0$ and $\delta(0)=0$.
\item[(2)] If $(a, w)\in U_1'\oplus \ov{B}_1(\tau_1')$ solves the equation $P'\circ g(a,w)=0$, then $w=\delta(a)$.
\item[(3)] If $a\in U_i'$, then $\delta (a)\in W_{i}'$ and $|\delta(a)|_{i}\leq \tau_i'$ for every 
$i\geq 1$.
\item[(4)] $\delta\colon U_i'\to W_i'$ is of class $C^{i-1}$ for all $i\geq 1$.
\end{itemize}

\begin{lemma}\label{new_lemmat_level_0}
There exists $\varepsilon_0'$ and $\tau_0'$ having the following properties. If $(a', w')$ solves the equation $g(a', w')=0$ and satisfies $\abs{a'}_0\leq \varepsilon'_0$ and $\abs{w'}_0\leq \tau_0'$, then $w'=\delta (a')$.
Moreover, if $\abs{a'}_0\leq \varepsilon'_0$, then $\abs{\delta (a')}_0\leq \tau_0'$.
\end{lemma}

\begin{proof}

The solutions of $f(a)=0$ and $g(x')=0$ are related by the sc-diffeomorphism 
$\varphi$ via $x'=\varphi (x)$. Applying the proof of Theorem \ref{save} to the sc-germ $f=h+s$ we find,  for every $\sigma>0$,  a constant $\tau>0$ such that if $x\in C$ is a solution of $f(x)=0$ satisfying $\abs{x}_1<\tau$, then $\abs{x}_1<\sigma$. Using the continuity of $\varphi$ on the level $1$, we choose now  $\sigma>0$ such that if $\abs{x}_1<\sigma$, then $x'=\varphi (x)=(a', w')$ satisfies $\abs{a'}_0\leq \varepsilon_1'$ and $\abs{w'}_1\leq \tau_1'$. Then we conclude from $\abs{x}_0<\tau$ that $\abs{x}_1<\sigma$. Using the continuity of $\varphi^{-1}$ on level $0$ we next choose the desired constants $\varepsilon_0'$ and $\tau_0'$ such that the estimates $\abs{a'}_0\leq \varepsilon_0'$ and $\abs{w'}_0\leq \tau_0'$ imply that $x=\varphi^{-1}(a, w)$ satisfies $\abs{x}_0<\tau$. Assuming now that  $x'=(a', w')$ is a solution of $g(a', w')=0$ satisfying $\abs{a'}_0<\varepsilon_0'$ and $\abs{w'}_0< \tau_0'$, we conclude that $\abs{\varphi^{-1}(x')}_0<\tau$. It follows that 
$\abs{\varphi^{-1}(x')}_1<\sigma$, which implies $\abs{a'}_0<\varepsilon_1'$ and 
$\abs{w'}_1< \tau_1'$. From property (2) we conclude that $w'=\delta (a')$ proving the first statement of the lemma.

Using that the embedding $W_1\to W_0$ is continuous, we find a constant $c>0$ such that 
$\abs{\delta (a')}_0\leq c\abs{\delta (a')}_1$. Taking $\varepsilon_0'$ and $\tau_0'$ smaller we can achieve that 
 $\abs{\delta (a')}_0\leq \tau_0'$ if $\abs{a'}_0\leq \varepsilon_0'$ and the proof of the lemma is complete.
 
 \end{proof}
 
\begin{theorem}[{\bf Local Germ-Solvability I}]\label{LGS}\index{T- Local germ-solvability {I}}

Let  $f\colon U\rightarrow F$ be  a sc-smooth  germ  satisfying $f(0)=0$ and of the form $f=h+s$,  where $h$ is a basic germ 
and $s$ is a $\ssc^+$-section satisfying $s(0)=0$. We assume 
that the linearization $Df(0)\colon E\to F$ is surjective and the  kernel $K=\ker Df(0)$ is in good position to the partial quadrant $C$. 
Let 
$$\Phi\colon U\triangleleft (\R^N\oplus W)\to U'\triangleleft (\R^{N'}\oplus W')$$
be the strong bundle isomorphism covering the sc-diffeomorphism $\varphi\colon (U,0)\rightarrow (U',0)$ guaranteed by Theorem \ref{arbarello}. 
Here $U'$ is a relatively open neighborhood of $0$ 
in the  partial quadrant $C'=[0,\infty )^k\oplus \R^{n'-k}\oplus W'$ sitting in   the sc-Banach space $E'=\R^{n'}\oplus W'$.   

Then, denoting by $g=\Phi_\ast (f)$ the push-forward section, the following holds.
The kernel  $K'=\ker Dg(0)=T\varphi(0)K$  is in good position to the partial quadrant $C'$ and there is a good complement $Y'$  of $K'$ in $E'=K'\oplus Y'$,  and a $C^1$-map $\tau\colon V\rightarrow Y_1'$, defined on the  relatively  open neighborhood $V$ 
of $0$ in $K'\cap C'$  such that 
\begin{itemize}
\item[{\em (1)}] $\tau\colon {\mathcal O}(K'\cap C',0)\rightarrow (Y',0)$ is a sc-smooth germ.
\item[{\em (2)}] $\tau (0)=0$ and $D\tau(0)=0$.
\item[{\em (3)}] After perhaps suitably shrinking $U$ it holds
$$\varphi(\{x\in U\ |\ f(x)=0\}) =
\{x'\in U'\, \vert \, g(x')=0\}=\{y+\tau(y)\,  \vert \,  y\in V\}.
$$
\end{itemize}

\end{theorem}

\begin{proof}

The linearization $Dg(0)\colon \R^{n'}\oplus W'\to \R^{N'}\oplus W'$ is a surjective map because, by assumption, $Df(0)\colon \R^{n}\oplus W\to \R^{N}\oplus W$ is surjective. Moreover, $\ker Dg(0)=T\varphi (0)\ker Df(0)$. In view of Proposition \ref{newprop2.24} and Corollary \ref{equality_of_d} we also have $T\varphi (0)C=C'$.
We also  recall that (above Lemma \ref{new_lemmat_level_0}) we have introduced the sc-projection $P'\colon \R^{N'}\oplus W'\to W'$, the neighborhood $U'_i=\{a\in [0,\infty)^k\oplus \R^{n'-k}\, \vert \, \abs{a}_0\leq \tau'_i\}$ of $0$ in $\R^{n'}$, and the map $\delta\colon U_i'\to W_i'$ satisfying $P'\circ g(a, \delta (a))=0$ and $\delta (a)=0$.

We introduce the map $H\colon U_1'\to \R^{N'}$ by 
$$H(a)=({\mathbbm 1}-P')\circ g(a, \delta (a)).$$
It satisfies $H(0)=0$ since $g(0)=0$.  In the following we need the map $H$ to be of class $C^1$. In view of property (4), the map $H$ restricted to $U_i'$ is of class $C^{i-1}$. Since, by Proposition \ref{lower}, the sc-smooth map $g$ induces a $C^1$-map $g\colon E_{m+1}\to E_m$ for every $m\geq 0$, we shall restrict the domain of $H$ to the set $U'_3$. In order to prove the theorem we shall first relate the solution set $\{H=0\}$ to the solution set $\{g=0\}$, and start with the following lemma.

\begin{lemma}\label{new_lemma_relation}
The linear map $\alpha \mapsto (\alpha , D\delta (0)\alpha)$ from $\R^{n'}$ into $E'$ induces a linear isomorphism 
$$\wt{K}:=\ker DH(0)\to \ker Dg(0)=:K'.$$
\end{lemma}

\begin{proof}[Proof of the lemma]
We first claim that  
$$\ker Dg(0)=\{(\alpha ,D\delta (0)\alpha)\, \vert \, \alpha \in \ker DH(0)\}.$$  Indeed, if $(\alpha, \zeta)\in \ker Dg(0)$, then  $P'\circ Dg(0)(\alpha, \zeta)=0$ and  $({\mathbbm 1}-P')\circ Dg(0)(\alpha, \zeta)=0$. By assumption, $P'\circ g(a, w)=w-B(a, w)$. Hence,   differentiating at the point $0$ in the direction of $(\alpha, \zeta)$ and  recalling from Lemma \ref{new_Lemma3.9} that $D_2B(0)=0$, we obtain 
$$0=P'\circ Dg(0)(\alpha, \zeta)= \zeta-D_1B(0)\alpha-D_2B(0)\zeta= \zeta-D_1B(0)\alpha.$$
 On the other hand, differentiating  the identity $P'\circ g(a, \delta (a))=0$ at $a=0$ and evaluating the derivative at $\alpha$, we find,  in view of previous equation,  
$$0=P'\circ g(0)(\alpha, D\delta (0)\alpha)=D\delta (0)\alpha-D_1B(0)\alpha.$$
Consequently,  $\zeta=D\delta (0)\alpha$ as claimed.

From  
$$DH(0)\alpha =({\mathbbm 1}-P')\circ Dg(0)(\alpha , D\delta (0)\alpha )$$
 it follows, if $\alpha \in \ker DH(0)$, in view of 
 $P'\circ Dg(0)(\alpha ,D\delta (0)\alpha)=0$ for all $\alpha$, that 
 $(\alpha, D\delta (0)\alpha)\in \ker Dg(0)$.
 
Conversely, if $(\alpha, \zeta)\in \ker Dg(0)$, then $0=P'\circ Dg(0)(\alpha, \zeta)=
\zeta-D_1B(0)\alpha$ and $({\mathbbm 1}-P')\circ Dg(0)(\alpha, \zeta)=0$. Since 
$P'\circ Dg(0)(\alpha, D\delta (0)\alpha)=0$, we conclude that $\zeta=D\delta (0)\alpha$ and hence $({\mathbbm 1}-P')\circ Dg(0)(\alpha, D\delta (0)\alpha)=DH(0)\alpha=0$, so that $\alpha \in \ker DH(0)$. The proof of 
Lemma \ref{new_lemma_relation} is complete.

\end{proof}

By assumption, the kernel  $K=\ker Df(0)$ is in good position to the partial quadrant $C$. We recall that this requires  that 
$K\cap C$ has a nonempty interior in $K$ and there exists a complement of $K$, denoted by $K^\perp$, so that $K\oplus K^\perp=E$, having the following property. There exists 
$\varepsilon>0$ such that 
if $k+k^\perp\in K\oplus K^\perp$  satisfies  $\abs{k^\perp}_0\leq \varepsilon\abs{k}_0$, then 
\begin{equation}\label{eq_good_position}
\text{$k+k^\perp\in C$ \quad if and only if \quad $k\in C$}.
\end{equation}

The kernel $K'=\ker Dg(0)=T\varphi (0)K$ has a complement in $E'$,  denoted by $(K')^\perp=T\varphi (0)(K^\perp),$ so that 
$K'\oplus (K')^\perp =E'$.
\begin{lemma} 
The complement $(K')^\perp$ is a good complement of $K'$ in $E'$. 
\end{lemma}
\begin{proof}
Since $T\varphi (0)\colon \R^n\oplus W\to \R^{n'}\oplus W'$ is a topological isomorphism and $T\varphi (0)C=C'$ and $K'=T\varphi (0)K$, it follows that $K'\cap C'$ has a nonempty interior in $K'$. Next we choose 
$\varepsilon'>0$ satisfying $\varepsilon'\norm{T\varphi (0)}\norm{(T\varphi (0))^{-1}}\leq \varepsilon$ and take 
$x'\in K'$ and $y'\in (K')^\perp$ satisfying $\abs{y'}_0\leq \varepsilon'\abs{x'}_0$. 
If  $x=(T\varphi (0))^{-1}x'$ and $y=(T\varphi (0))^{-1}y'$, then
\begin{equation*}
\begin{split}
\abs{y}_0&
=\abs{(T\varphi (0))^{-1}y'}_0\leq \norm{(T\varphi (0))^{-1}}\abs{y'}_0\leq \varepsilon' \norm{(T\varphi (0))^{-1}}\abs{x'}_0\\
&=\varepsilon' \norm{(T\varphi (0))^{-1}}\abs{T\varphi (0)}_0\leq 
\varepsilon' \norm{(T\varphi (0))^{-1}}\norm{T\varphi (0)}\abs{x}_0\leq \varepsilon\abs{x}_0.
\end{split}
\end{equation*}
By \eqref{eq_good_position}, $x+y\in C$ if and only if $x\in C$ and since  $T\varphi (0)C=C'$,  we conclude that $x'+y'\in C'$ if and only if $x'\in C'$.  
\end{proof}
The next lemma is proved in Appendix \ref{pretzel-B}.
\begin{lemma}\label{new_lemma_Z}
The kernel  $\wt{K}:=\ker DH(0)\subset \R^{n'}$ of the linearization $DH(0)$ is in good position to the partial quadrant $\wt{C}=[0,\infty)^{k}\oplus \R^{n'-k}$ in $\R^{n'}$. Moreover, there exists a good complement $Z$ of $\wt{K}$ in $\R^{n'}$, hence 
$\wt{K}\oplus Z =\R^{n'}$, having the property that $Z\oplus W'$ is a good complement of $K'=\ker Dg(0)$ in $E'$, 
$$E'=K'\oplus ( Z\oplus W').$$
\end{lemma}

\mbox{}\\
 In view of  Lemma \ref{new_lemma_Z},  Theorem \ref{help-you} in Appendix \ref{implicit_finite_partial_quadrants} is applicable to the $C^1$-map $H\colon U_2'\to \R^{N'}$,  where 
$U_2'=\{a\in [0,\infty )^k\oplus \R^{n'-k}\, \vert \, \abs{a}_0<\tau_2'\}\subset \R^{n'}$ is introduced before the statement of Theorem \ref{LGS}. According to Theorem 
\ref{help-you}  there exists a relatively open neighborhood $\wt{V}$ of $0$ in $\wt{K}\cap \wt{C}$ and a $C^1$-map 
$$\sigma\colon \wt{V}\to Z$$
satisfying $\sigma(0)=0$ and $D\sigma(0)=0$ and solving the equation  
$$\text{$H(a+\sigma(a))=0$ for all $a\in \wt{V}$.}$$

The situation is illustrated in the following figure.

\begin{figure}[htb]
\begin{centering}
\def\svgwidth{60ex}
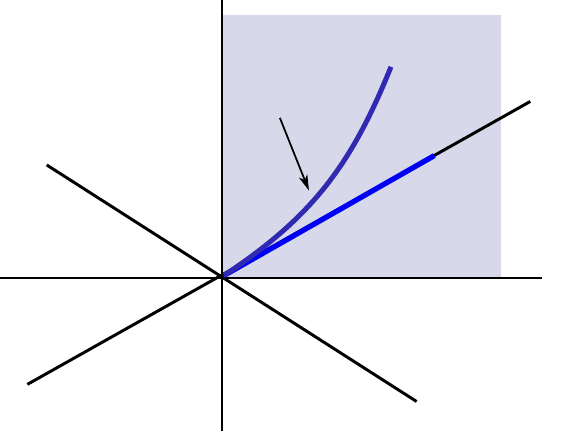
\caption{}\label{fig:pict5}
\end{centering}
\end{figure}

Since $H$ is the more regular the closer we are  to $0$,  the map  $\tau$ has the same property.   Recalling the solution germ 
$\delta\colon U_1'\subset \R^{n'}\to \ov{B}_1' \subset W_1'$ of 
$g(a, \delta (a))=0$, we define the map
$\gamma\colon \wt{V}\rightarrow Z\oplus W'_1$ by 
$$\gamma (a)=\sigma(a)+\delta(a+\sigma(a)),\quad v\in \wt{V}.
$$
By the properties of $\sigma$ and $\delta$ 
there exists a nested sequence $\wt{V}_1\supset \wt{V}_2\supset \ldots $ of relatively open neighborhoods of $0$ in $\wt{K}\cap \wt{C}$ such that the restriction
$$
\gamma \vert \wt{V}_i\colon \wt{V}_i\to Z\oplus W_i'
$$
is of class $C^i$,  for every $i\geq 1$.  Recalling from Lemma 
\ref{new_lemma_relation} that $K'=\ker Dg(0)=\{(v, D\delta (0)v)\, \vert \, v\in \wt{K}\}$, we denote by $\pi\colon K'\to\wt{K}$
the projection
$$\pi\colon (v, D\delta (0)v)\mapsto v.$$
The map $\pi$  is a sc-isomorphism. 

We next introduce the map $\tau$ from a relatively open neighborhood of $0$ in $K'\cap C'$ into $Z\oplus W'$ by 
$$\tau (y)=-D\delta (0)\circ \pi (y)+\gamma (\pi (y)),\quad y\in K'\cap C'.$$
It satisfies $\tau (0)=0$. Differentiating $\tau$  at $y=0$, and using  $D\sigma(0)=0$, we obtain
$$
D\tau(0)\eta =-D\delta(0)\circ \pi ( \eta)+D\delta(0)\circ \pi ( \eta)=0
$$
for all $\eta\in K'$ and hence $D\tau (0)=0$. 
Recalling that $y=(v, D\delta (0)v)$ where $v\in \ker DH(0)$, and $\pi (y)=v$, we  compute 
\begin{equation*}
\begin{split}
y+ \tau (y)&=(v +D\delta(0)v) +(-D\delta(0)v+\gamma (v))\\
&=v+\gamma (v)=(v+\sigma (v))+\delta (v+\sigma (v)),
\end{split}
\end{equation*}
where $v+\sigma (v)\in \R^{n'}$. Hence, by definition of the solution germ $\delta$ we conclude from Lemma \ref{new_lemmat_level_0} that 
$$g(y+\tau (y))=0$$
for all $y\in K'\cap C'$ near $0$ on level $0$. By construction, the map $\tau$ is a sc-smooth germ satisfying $\tau (0)=0$ and $D\tau (0)=0$ and there exists a nested sequence of relatively open subsets $\wt{O}_i=\pi^{-1} (V_i)$ of $0$ in $K'$, where 
$O_i=\{a\in [0,\infty)^k\oplus\R^{n'-k}\ |\ |a|_0<\varepsilon_i'\}$
such that the restrictions satisfy 
$$\tau\vert \wt{O}_i\in C^{i}(\wt{O}_i,Z\oplus W_{i}'),\quad \text{for $ i\geq 1$}.$$
Moreover, if $g(y+w)=0$ for $y\in K'\cap C'$ and $w\in Z\oplus W'$ sufficiently small on level $0$, then $w=\tau (y)$.
To sum up, there exists a relatively open subset $U$ of $0$ in $C$ diffeomorphic to  the relatively open subset $U'=\varphi (U)$ of $0$ in $C'$, and a relatively open subset $V$ of $0$ in $K'\cap C'$ such that  
$$
\varphi \bigl(\{x\in U\  |\ f(x)=0\}\bigr) =\{x'\in U'\, \vert \, g(x')=0\}=\{y+\tau(y)\, \vert \, y\in V\}.
$$
The proof of Theorem \ref{LGS} is finished.

\end{proof}

\begin{corollary}\label{LGS2}\index{C- Germ-solvability {II}}
Let  $f\colon U\rightarrow F$ be  a sc-smooth  germ  satisfying $f(0)=0$ and of the form $f=h+s$,  where $h$ is a basic germ 
and $s$ is a $\ssc^+$-section satisfying $s(0)=0$. We assume 
that the linearization $Df(0)\colon E\to F$ is surjective and the  kernel $K=\ker Df(0)$ is in good position to the partial quadrant 
$C\subset E$. 

Then there exists a good complement $Y$ of $K$ in $E$, so that $K\oplus Y=E$, and  there exists a $C^1$-map 
$\sigma\colon V\rightarrow Y_1$ defined on relatively open neighborhood $V$ of $0$ in $K\cap C$, having the following properties.  
\begin{itemize}
\item[{\em (1)}] $\sigma\colon {\mathcal O}(K\cap C,0)\rightarrow (Y,0)$ is a sc-smooth germ.
\item[{\em (2)}] $\sigma (0)=0$ and $D\sigma(0)=0$.
\item[{\em (3)}] 
$\{x\in U\, \vert \, f(x)=0\} =\{v+\sigma(v)\, \vert \, v\in V\}$ after perhaps replacing $U$ by a smaller open neighborhood.
\end{itemize}

\end{corollary}
\begin{proof}

In view of  Theorem \ref{LGS}, the sc-diffeomorphism $\varphi\colon (U, 0)\to (U', 0)$ between  relatively open subsets $U$ and $U'$ of the partial quadrants $C=[0,\infty)^k\oplus \R^{n-k}\oplus W$ and 
$C'=[0,\infty)^k\oplus \R^{n'-k}\oplus W'$ satisfies 
$T\varphi (0)K=K'=\ker Dg(0)$ and $K'$ is in good position to $C'$.   If $Y'$ is the  good sc-complement in Theorem \ref{LGS} of $K'$ in $E'$, we define $Y=T\varphi(0)^{-1}(Y')$. Then $Y$  is the desired good complemement of $K$ in $E$ with respect to the partial quadrant $C$, so that $E=K\oplus Y$. Let $\pi\colon K\oplus Y\to K$ be the sc-projection. 


Recalling  the map $\tau\colon V\subset K'\cap C'\to Y_1'$ from Theorem \ref{LGS},  we define the map 
$\psi\colon V\subset K'\cap C'\rightarrow K$  by 
$$\psi (v)=\pi\circ\varphi^{-1}(v+\tau(v)).$$
The map $\psi$ satisfies  $\psi (0)=0$ and its derivative $D\psi (0)\colon K'\to K$  is equal to 
$D\psi (0)=T\varphi(0)^{-1}\vert K'$, hence it is an isomorphism. 

We claim that the map $\psi$ preserves the degeneracy index, that is 
\begin{equation}\label{index_preserving}
d_{K'\cap C'}(v)=d_{K\cap C}(\psi (v))
\end{equation}
for $v\in V$ close to $0$.  To see this we first assume that $v\in V\subset K'\cap C'$ belongs to $K'\cap (\R^{n'-k}\oplus W')$ where $\R^{n'-k}\oplus W'$ is identified with $\{0\}^k\oplus \R^{n'-k}\oplus W'.$ Hence $v$ is of the form $v=(0, w)\in \{0\}^k\oplus \R^{n'-k}\oplus W'$. 

If $K'$ is one-dimensional, then, in view of Lemma \ref{ll1},  $v=t(a, w')\in \R^{k}\oplus \R^{n'-k}\oplus W'$ where $a=(a_1,\ldots ,a_k)$ satisfies $a_j>0$. Consequently, $t=0$, so that $v=0$. 
Since $\psi (0)=0$,  we have $d_{C'\cap K'}(0)=d_{C\cap K}(\psi (0))$.

If $\dim K'\geq 2$, then,  denoting by $\wt{K'}$ the algebraic complement of $K\cap (\R^{n'-k}\oplus W')$ in $K'$,  we may assume, after linear change of coordinates, that the following holds. 
\begin{itemize}
\item[(a)] The finite dimensional subspace $\wt{K}'$ is spanned by the vectors $e'_j=(a'_j, b'_j,  w'_j)$ for $1\leq j\leq m=\dim \wt{K}'$ in which $a_j'$ are the vectors of the standard basis of $\R^{l}$, $b_j'=(b_{j, m+1}',\ldots, b_{j,k}')$ with $b_{j,i}'>0$, and $w_j'\in \R^{n'-k}\oplus W'.$
\item[(b)] The good complement $Y'$ of $K'$ in $E'=\R^{n'}\oplus W'$ is contained in $\{0\}^l\oplus \R^{k-l}\oplus \R^{n'-k}\oplus W'$.
\end{itemize}
By adding a basis $e_j'=(0, w_j')\in \{0\}^k\oplus \R^{n'-k}\oplus W'$,  $m+1\leq j\leq \dim K'$,  of $K'\cap (\R^{n'-k}\oplus W')$ we obtain a basis of $K'$.

We have a similar statement for  the subspace $K$ with the algebraic complement $\wt{K}$ of $K\cap (\R^{n-k}\oplus W)$ in $K$, $e'_j$ replaced by $e_j$, and the good complement $Y$ of $K$ contained in $\{0\}^m\oplus \R^{k-m}\oplus \R^{n-k}\oplus W$.

Now if $v=(0, w_1')\in \{0\}^k\oplus \R^{n'-k}\oplus W'$, then, in view of (a), $d_{K'\cap C'}(v)=m=\dim \wt{K}'$.  
Recalling $K'$ is in good position to the partial quadrant $C'$ and $Y'$ is a good complement of $K'$ in $E'$. there exists a constant $\gamma$  such that if $n\in K'$ and $y\in Y'$, then $\abs{y}_0\leq \gamma\abs{n}_0$ implies that $n+y\in C$ if and only if $n\in C$.  It follows from $\tau (0)=0$ and $D\tau (0)=0$ that 
$\abs{\tau (v)}_0\leq \gamma\abs{v}_0$ and,  since $v\in C$,  we conclude that $v\pm \tau (v)\in C$. Since 
for $v$ as above $v_j=0$ for all $1\leq j\leq k$, we find that also $\tau_j (v)=0$ for all $1\leq j\leq k$. Hence 
$v+\tau (v)=(0, w_2')\in \{0\}^k\oplus \R^{n'-k}\oplus W'$.  Because   the map $\varphi$ is a sc-diffeomorphism, $\varphi^{-1}(v+\tau (v))=\varphi^{-1}((0, w_2'))=(0, w)\in  \{0\}^k\oplus \R^{n-k}\oplus W$.

We shall show that $\pi ((0, w))\in K\cap C$ and that 
$d_{K\cap C}(\pi ((0, w))=m$. We decompose $(0, w)$ according to the direct sum  $E=K\oplus Y$ as 
$(0, w)=k+y$ where $k\in K$ and $y\in Y$, and write 
$k=(\alpha_1, \ldots , \alpha_k, w_1)\in \R^k\oplus \R^{n-k}\oplus W$, $y=(y_1,\ldots , y_k, w_2)\in Y$. Then,   using 
$Y\subset \{0\}^m\oplus \R^{k-m}\oplus \R^{n-k}\oplus W$, we find that  $y_1=\ldots y_m=0$. This implies that 
also $\alpha_1=\ldots =\alpha_m=0$. Consequently, $\pi ((0, w))=k=(0, w_1)\in K\cap C$ and 
$d_{K\cap C}(\pi ((0, w))=m$. We have proved the identity 
\eqref{index_preserving} in the case $v\in K'\cap  (\R^{n'-k}\oplus W')$.

In the case $v=(a, w)\in K\cap C\setminus K\cap  (\R^{n-k}\oplus W)$, we use  Lemma \ref{big-pretzel_1a} and  compute  
$$
d_{C\cap K}(\psi(v))=d_C(\varphi^{-1}(v+\tau(v)))=d_{C'}(v+\tau(v))=d_{C'\cap K'}(v).
$$
Hence $\psi\colon V\rightarrow C$ is a $C^1$-map  satisfying 
$\psi (0)=0$, $D\psi (0)(K'\cap C')=K\cap C$, and preserving  the degeneracy index. Consequently, 
applying  the inverse function theorem for partial quadrants in $\R^n$,  Theorem \ref{QIFT},  to the map $D\psi (0)^{-1}\circ \psi$,  we find relatively open neighborhoods $V_0\subset V$ and $V_1\subset C$ of $0$ such that 
$$
\psi \colon V_0\rightarrow V_1
$$
is a $C^1$-diffeomorphism, which has higher and higher differentiability closer and closer to $0$. Considering the map
$$
w\mapsto  \varphi^{-1}(\psi^{-1}(w)+\tau\circ\psi^{-1}(w)), \quad w\in V_1,
$$
we obtain   $\pi\circ  \varphi^{-1}(\psi^{-1}(w)+\tau\circ\psi ^{-1}(w)) =w$,  and define the map $\sigma\colon V_1\to Y$ by 
$$
\sigma(w)=\varphi^{-1}\bigl(\psi^{-1}(w)+\tau\circ\psi^{-1}(w)\bigr)-w =({\mathbbm 1}-\pi)\varphi^{-1}\bigl(\psi^{-1}(w)+\tau\circ\psi^{-1}(w)\bigr).
$$
Finally,   we can now choose  open neighborhoods $V\subset K\cap C$ of $0$ and $U$ of $0$ in $C$ appropriately 
and see that the map $\sigma$ has the desired properties (1)-(3) of Corollary \ref{LGS2}

\end{proof}

\mbox{}
\subsection{Implicit Function Theorems}\label{ssec3.4}

Let us  recall from Definition \ref{x-filling}
 the concept  of a filled version.
\begin{definition}[{\bf Filling}]\index{D-Filling}
We consider a tame strong local bundle $K\to O$, where $K=R(U\triangleleft F)$, and  let  the set $U\subset C\subset E$ be  a relatively open neighborhood  of $0$ in the partial quadrant $C$ of the  sc-Banach space $E$.  Here  $F$ is a sc-Banach space and $R$ is a strong bundle retraction of the form 
$$R(u, h)=(r(u), \rho(u)(h))$$
covering the tame retraction $r\colon U\to U$ onto $O=r(U)$. We  assume that $r(0)=0$.

A sc-smooth section germ 
$(f, 0)$ of the bundle $K\to O$ possesses a  {\bf filling}
if there exists a sc-smooth section germ $(g, 0)$ of the bundle $U\triangleleft F\to U$ satisfying the following properties.
\begin{itemize}
\item[(1)]  $f(x)=g(x)$  for $x\in O$ close to $0$
\item[(2)] If $g(y)=\rho (r(y))g(y)$ for a point $y\in U$ near $0$, then $y\in O$.
\item[(3)] The linearisation of the map
$
y\mapsto  [{\mathbbm 1} -\rho(r(y))]\cdot g(y)
$
at the point $0$, restricted to $\ker(Dr(0))$, defines a topological linear  isomorphism
$$
\ker(Dr(0))\rightarrow \ker(\rho (0)).
$$
\end{itemize}

\end{definition}

\begin{remark}
By replacing $O$ by a smaller set, we may assume in (1) that $f(x)=g(x)$ for all $x\in O$ and in (2) that 
$g(y)=\rho(r(y))g(y)$ for $y\in U$ implies that $y\in O$.
\end{remark}

We also  recall the notions  sc-Fredholm germ and sc-Fredholm section.

\begin{definition}[{\bf sc-Fredholm germ}]\index{D-Sc-Fredholm germ}

Let $f$ be a sc-smooth section of the tame strong bundle 
$P\colon Y\to X$ and let $x\in X$ be a smooth point.

Then $(f, x)$ is a {\bf sc-Fredholm germ}, if there exists a strong bundle chart around $x$ (as defined in 
Definition \ref{def_strong_bundle_chart})
$$\text{$\Phi\colon \Phi^{-1}(V)\to K$ covering $\varphi \colon (V, x)\to (O, 0)$,}$$
where $K\to O$ is a tame strong local bundle containing $0\in O\subset U$, such that the local section germ $\Phi\circ f\circ \varphi^{-1}\colon O\to K$ has a filling $g\colon U\to U\triangleleft F$ near $0$ which possesses the following additional property. There exists a local $\ssc^+$-section $s\colon U\to U\triangleleft F$ satisfying $s(0)=g(0)$ such that $g-s$ is conjugated near $0$ to a basic germ.

\end{definition}
\begin{definition}[{\bf sc-Fredholm section}]\label{def_again_fedholm_sect}
A sc-smooth section $f$ of the tame strong bundle $P\colon Y\to X$ is a {\bf sc-Fredholm section}, if 
\begin{itemize}
\item[(1)] $f$ is regularizing, i.e., if $x\in X_m$ and $f(x)\in Y_{m, m+1}$, then $x\in X_{m+1}$.
\item[(2)] The germ $(f, x)$ is a sc-Fredholm germ at every smooth point $x\in X$.
\end{itemize}
\end{definition}

\mbox{}\\

From Theorem \ref{stabxx} we know that sc-Fredholm sections are stable under  
$\ssc^+$-perturbations; if $f$ is a sc-Fredholm section, then $f+s$ is also a sc-Fredholm section, for every $\ssc^+$-section $s$.

\begin{proposition}\label{prop3_52}
If $f$ is a sc-Fredholm section of the tame strong bundle $P\colon Y\to X$ and $x$ is a solution of $f(x)=0$, then $Tf(x)\colon T_xX\to Y_x$  is a sc-Fredholm operator.
\end{proposition}
\begin{proof}
Since $(f, x)$ is a sc-Fredholm germ, there exists a tame strong bundle chart $\Phi\colon \Phi^{-1}(V)\to K$ around $x$ such that the section $\wt{f}=\Phi\circ f\circ \varphi^{-1}\colon O\to K$ has a filling $g\colon U\to U\triangleleft F$ for which there exists a $\ssc^+$-section $s\colon U\to U\triangleleft F$ satisfying $s(0)=g(0)$ such that $g-s$  is conjugated near $0$ to a basic germ. Hence, by Proposition \ref{Newprop_3.9}, the linearization $D(g-s)(0)\colon E\to F$ is a sc-Fredholm operator. Since $Ds(0)$ is a $\ssc^+$-operator,  the linearization $Dg(0)$ is sc-Fredholm operator,  in view of Proposition \ref{prop1.21}. From Proposition \ref{filler_new_1} (2) about fillers it follows that $T\wt{f}(0)\colon T_0O\to K_0$ is a sc-Fredholm operator of index 
$\ind (T\wt{f}(0))=\ind (Dg(0))=\ind (D(g-s)(0))$,  and the proposition follows.
\end{proof}

We now  focus on the solution set $\{f=0\}$ of a sc-Fredholm section near a point $x\in X$ of  $f(x)=0$. In the case of a boundary,  i.e., $d_X(x)\geq 1$,  we require that $\ker f'(x)$ is in good position to the partial cone $C_xX$
in the tangent space $T_xX$, as defined in 
Definition  \ref{reduced_cone_tangent}.

\mbox{}
\begin{definition}[{\bf Good position of a sc-Fredholm  germ}]\label{new_good_position_def}\index{D-Good position}
A sc-Fredholm germ $(f, x)$ of a tame strong bundle
$Y\rightarrow X$  satisfying $f(x)=0$ is in {\bf good position},   if 
\begin{itemize}
\item[(1)] $f'(x)\colon T_xX\rightarrow Y_x$ is surjective.
\item[(2)]  If $d_X(x)\geq 1$,  then  $\ker (f'(x))\subset T_xX$ is in good position to the partial cone  $C_xX$ in the tangent space $T_xX$.
\end{itemize}
\end{definition}

The fundamental implicit function theorem is as follows.

\begin{theorem}\label{IMPLICIT0}\index{T- Implicit function theorem}
Let $f$ be a sc-Fredholm section  of a tame strong bundle
$Y\rightarrow X$ . If $f(x)=0$, and if the sc-Fredholm germ $(f, x)$ is in good position, then there exists an open neighborhood $V$ of $x\in X$ such that  
the solution set $S=\{y\in V\, \vert \, f(y)=0\}$ in $V$ has the following properties.
\begin{itemize}
\item[(1)] At  every point $y\in S$, the sc-Fredholm germ  $(f,y)$ is in good position.
\item[(2)] $S$ is a sub-M-polyfold of $X$ and  the induced  M-polyfold 
structure is equivalent to a smooth manifold structure with boundary with corners. 
\end{itemize}
\end{theorem}

In case $\partial X=\emptyset$,  the requirement that $(f,x)$ is in good position just means that $f'(x)$ is surjective.
In this case the  local solution set $S$ is a smooth manifold and the linearization $f'(y)$ is surjective at every point $y\in S$.

We note that Theorem \ref{implicit-x} and Theorem \ref{bound} are immediate consequences of Theorem  \ref{IMPLICIT0}.

\begin{proof}[{\bf Proof of Theorem \ref{IMPLICIT0}}]
The sc-smooth sc-Fredhom section $f$   of the tame strong bundle $P\colon Y\to X$ over the M-polyfold $X$ is regularizing so that the solutions $y\in X$ of $f(y)=0$ are smooth points. Moreover, at every smooth point $y\in X$, the sc-germ $(f, y)$ is a sc-Fredholm germ.

We now focus  on a neighborhood of the solution $x$ of $f(x)=0$. By assumption, the sc-Fredholm germ $(f, x)$ is in good position according to Definition \ref{new_good_position_def}, so that the kernel $N:=\ker f'(x)$ is in good position to the partial cone $C_xX\subset T_xX$ as defined in Definition \ref{def_partial_cone_reduced_tangent}, and the linearization $f'(x)\colon T_xX\to Y_x$ is surjective. In view of Proposition \ref{prop3_52}, the linear operator $f'(x)$ is a Fredholm operator. Its index is equal to $\ind (f'(x))=\dim N.$

By definition of a sc-Fredholm germ, there exists a strong bundle chart $(V, P^{-1}(V), K, U\triangleleft F)$ of the tame strong bundle $P\colon Y\to X$
\begin{equation*}
\begin{CD}
P^{-1}(V)@>\Phi>>K\\
@VPVV @VVpV \\
V@>\varphi>>O 
\end{CD}\,
\end{equation*}
covering the sc-diffeomorphism $\varphi\colon V\to O$, which is defined on the open neighborhood $V\subset X$ of the given point $x\in X$ and satisfies  $\varphi (x)=0$. The retract $K=R(U\triangleleft F)$ is the image of the strong bundle retraction $R\colon U\triangleleft F\to U\triangleleft F$ of the form 
$R(u, h)=(r(u), \rho (u)h)$  and which covers the tame sc-smooth retraction $r\colon U\to U$ onto $O=r(U)$. As usual, the set $U\subset C$ is a relatively  open subset of the partial quadrant $C=[0,\infty)^k\oplus \R^{n-k}\oplus W$ in the sc-Banach space $E=\R^n\oplus  W$.
Still by definition of a sc-Fredholm germ, the push-forward section
$$\wt{f}=\Phi_\ast (f)\colon O\to K$$
possesses a filled version $g\colon U\to U\triangleleft F$. It has, in particular, the property that 
\begin{equation}\label{g=wtilde_f}
g(u)=\wt{f}(u)\quad \text{if $u\in O$}.
\end{equation}
In addition, there exists a strong bundle isomorphism 
\begin{equation*}
\begin{CD}
U\triangleleft F@>\Psi>>U'\triangleleft F'\\
@VVV @VVV \\
U@>\psi>>U'
\end{CD}\,
\end{equation*}
where $F'=\R^{N}\oplus W$. It covers the sc-diffeomorphism $\psi\colon U\to U'$ satisfying  $\psi (0)=0$. The set $U'$ is a relatively open subset of the partial quadrant $C'=T\psi (0)C$ in the sc-space $E'=\R^N\oplus W$. In addition,  by definition of a sc-Fredholm germ, there exists a $\ssc^+$-section $s\colon U\to U\triangleleft F$ satisfying $s(0)=g(0)$ such that the push-forward section 
$$\Psi_\ast (g-s)=h\colon U'\to F'$$
is a basic germ according to Definition \ref{BG-00x}. Therefore, 
\begin{equation}\label{new_equ_50}
\Psi_\ast (g)=h+t,
\end{equation}
where $t=\Psi_\ast (s)$ is a $\ssc^+$-section $U'\to F'$.

In view of Proposition \ref{Newprop_3.9}, the linearization $D(h+t)(0)\colon E'\to F'$ is a sc-Fredholm operator, so that, by Proposition \ref{prop1.21}, the operator $Dg(0)\colon E\to F$ is also a sc-Fredholm operator.

From the postulated surjectivity of the linearization $f'(x)\colon T_xX\to Y_x$ we deduce that 
$T\wt{f}(0)\colon T_0O\to K_0$ is surjective. Hence, by Proposition \ref{filler_new_1}, the operator $Dg(0)$ is surjective. Since $\ker f'(x)$ is in good position to the partial cone $C_xX\subset T_xX$, and since $T\varphi (x)(C_xX)=C_0O=T_0O\cap C$, the kernel $\ker T\wt{f}(0)$ is in good position to the partial cone $C_0O$. The retract $O=r(U)$ is, by assumption, tame. Hence,  the tangent space $T_0O$ has, in view Proposition \ref{IAS-x}, the sc-complement $\ker Dr(0)$ in $E$ so that 
$$E=T_0O\oplus \ker Dr(0),$$
and $\ker Dr(0)\subset E_x$ at the point $x=0$ ($E_x$ is defined in Definition \ref{new_def_2.33}).
If $Z\subset T_0O$ is a good complement of $\ker T\wt{f}(0)$ in $T_0O=\ker T\wt{f}(0)\oplus Z$, the space $Z\oplus \ker Dr(0)$ is a good complement of 
$\ker T\wt{f}(0)$ in $E$, so that 
$$E=\ker T\wt{f}(0)\oplus (Z\oplus \ker Dr(0)).$$

Since, by the properties of the filler,  
$$\ker T\wt{f}(0)\oplus \{0\}=\ker Dg(0),$$
we conclude that $\ker Dg(0)$ is in good position to $C$ in $E$. Therefore, $N'=\ker D(h+t)(0)$ is in good position to $C'=T\psi (0)C$ in $E'$. Let now $Y'\subset E'$ be the  good complement of $N'$ in $E'=N'\oplus Y'$  from Theorem \ref{LGS2}. 

Then we can apply Theorem \ref{LGS2} 
about the local germ solvability and find an open neighborhood $V'\subset N'\cap C'$ of $0$,  and a map 
$$\sigma\colon V'\to (Y')_1$$ 
possessing the following properties.
\begin{itemize}
\item[(1)] $\sigma\colon V'\to (Y')_1$ is of class $C^1$ and satisfies $\sigma (0)=0$ and $D\sigma (0)=0$.
\item[(2)] $\sigma\colon {\mathcal O}(N'\cap C', 0)\to (Y', 0)$ is a sc-smooth germ.
\item[(3)] 
\begin{align*}
\psi\circ \varphi (&\{y\in V\, \vert \, f(y)=0\}&& \text{}\\
&=\psi (\{u\in O\, \vert \, \wt{f}(u)=0\})&& \text{}\\
&=\psi (\{u\in O\, \vert \, g(u)=0\})&& \text{by \eqref{g=wtilde_f}}\\
&=\{u'\in U'\, \vert \, (h+t)(u')=0\}&& \text{by \eqref{new_equ_50}}\\
&=\{v+\sigma (v)\, \vert \, v\in V'\}&& \text{by Corollary \ref{LGS2}.}
\end{align*}
\end{itemize}
Moreover,
\begin{itemize}
\item[(4)] For every $y\in V$ satisfying $f(y)=0$, so that  $\psi\circ \varphi (y)=v+\sigma (v)$, the kernel $\ker f'(y)$ is in good position to the partial cone $C_yX\subset T_yX$, and $f'(y)\colon T_yX\to Y_y$ is surjective.
\end{itemize}
 Property (4) follows from Corollary \ref{corex1} and from the following lemma, whose proof is postponed to Appendix \ref{geometric_preparation}.

\mbox{}\\
\begin{lemma}\label{good_pos}
Let $C\subset E$ be  a partial quadrant in the  sc-Banach space $E$ and $N\subset E$ a finite-dimensional smooth subspace in good position to $C$ and let $Y$ be a good complement of $N$ in $E$, so that $E=N\oplus Y$.
We assume that $V\subset N\cap C$ is a relatively open neighborhood
of $0$ and $\tau\colon V\rightarrow Y_1$ a map of class $C^1$ satisfying $\tau(0)=0$ and $D\tau(0)=0$.

\mbox{}\\

Then there exists a relatively open neighborhood $V'\subset V$ of $0$ such  that the following holds.
\begin{itemize}
\item[{\em (1)}] $v+\tau(v)\in C_1$ for $v\in V'$.
\item[{\em (2)}] For every $v\in V'$,  the linear subspace $N_v=\{n+D\tau(v)n\, \vert \, n\in N\}$ has the Banach space $Y=Y_0$ as a topological complement.
\item[{\em (3)}] For every $v\in V'$,  there exists a  constant $\gamma_v>0$ such  that  if  $n\in N_v$ and $y\in Y$ satisfy 
$\abs{y}_0\leq \gamma_v\cdot \abs{n}_0$,  the statements $n\in C_z$ and $n+y\in C_z$ are equivalent, where $z=v+\tau (v)$. 
\end{itemize}

\end{lemma}

So far we are confronted with the following situation. The open neighborhood $V\subset X$ of the smooth point $x$ is a M-polyfold and we denote the solution set of the sc-Fredholm section $f$ of the tame strong bundle $P^{-1}(V)\to V$ by $S=\{y\in V\, \vert \, f(y)=0\}$. It consists of smooth points. For every $y\in S$, the germ $(f, y)$ is a sc-Fredholm germ and the linearization $f'(y)\colon T_yV\to Y_y$ is a surjective Fredholm operator whose kernel $\ker f'(y)$ is in good position to the partial cone $C_yV\subset T_yV$, so that, proceeding as above we can construct a map $\sigma$ satisfying the above properties (1)-(3). Consequently, abbreviating $d=\dim (\ker f'(x))$, the solution set $S\subset V$ possesses the $d$-dimensional tangent germ property according to Definition \ref{toast}. Therefore we can apply Theorem \ref{HKL} to conclude that the solution set 
$S=\{y\in V\, \vert \, f(y)=0\}$ is a sub-M-polyfold, whose induced M-polyfold structure is equivalent to the structure of a smooth manifold with boundary with corners.

This completes the proof of Theorem \ref{IMPLICIT0}.

\end{proof}
Finally,  we note two immediate consequences of Theorem \ref{IMPLICIT0}.
{\begin{theorem}[{\bf Global implicit function theorem in the case $\partial X=\emptyset$}]\label{io-xx}\index{T- Global implicit function theorem}
If $P\colon Y\rightarrow X$ is a tame strong bundle over an  M-polyfold $X$ satisfying  $\partial X=\emptyset$,  and if $f$ a sc-Fredholm section having the property that at  every point $x$ in the solution set $S=\{y\in X\ |\ f(y)=0\}$,  the linearization $f'(x)\colon T_xX\rightarrow Y_x$ is surjective. Then $S$ is a sub M-polyfold
of $X$ and the induced M-polyfold structure on $S$ is equivalent to the structure of a smooth manifold  without boundary.
\end{theorem}

In a later section we shall study the question how to perturb a sc-Fredholm section to guarantee the properties required 
in the hypotheses of Theorem \ref{io-xx} and the following  boundary version.
\begin{theorem}[{\bf Global implicit function theorem: boundary case}]\label{io-xxx}\index{T- $\partial$-global implicit function theorem}
Let $P\colon Y\rightarrow X$ be  a tame strong bundle,  and $f$ a sc-Fredholm section having the property that
at  every point $x$ in the solution set $S=\{y\in X\ |\ f(y)=0\}$,  the linearization $f'(x)\colon T_xX\rightarrow Y_x$ is surjective and the kerenel $\ker(f'(x))$ is in good position
to the partial cone $C_xX\subset T_xX$.  Then $S$ is a sub-M-polyfold
of $X$ and the induced M-polyfold structure on $S$ is equivalent to the structure of a smooth manifold with boundary with corners.
\end{theorem}

\subsection{Conjugation to a Basic Germ}

A useful criterion to decide  in practice whether  a filled version is conjugated to a basic germ is given in Theorem \ref{basic_germ_criterion} below.
We would like to point out a similar result,  due to K. Wehrheim,  in  \cite{Wehr}.
 The following criterion was 
introduced in \cite{HWZ5} and employed to show that the nonlinear Cauchy-Riemann operator occurring in the Gromov-Witten theory
defines a sc-Fredholm section.
\begin{theorem}[{\bf Basic Germ Criterion}]\label{basic_germ_criterion}\index{T- Basic germ criterion}
Let  $U$ be  a relatively open neighborhood of $0$ in the partial quadrant $C$ of the sc-Banach space $E$, and let 
${\bf f}\colon U\rightarrow F$ be a sc-smooth map into the sc-Banach space $F$, which satisfies the following conditions.
\begin{itemize}
\item[{\em (1)}] At every smooth point $x\in U$ the linerarization $D{\bf f}(x)\colon E\rightarrow F$
is a sc-Fredholm operator and the index does not depend on $x$.
\item[{\em (2)}] There is a sc-splitting $E=B\oplus X$ in which $B$ is a finite-dimensional subspace of $E$ containing the kernel
of $D{\bf f}(0)$, and $X\subset C$, such  that the following holds for $b\in B\cap U$ small enough.  If $(b_j)\subset B\cap U$
is a sequence converging to $b$ and if $(\eta_j)\subset X$ is a sequence which is bounded on level $m$ and satisfying
$$
D{\bf f}(b_j)\eta_j=y_j+z_j,
$$
where $y_j\rightarrow 0$ in $F_m$ and where the sequence $(z_j)$ is bounded in $F_{m+1}$, then the sequence $(\eta_j)$ possesses a convergent
subsequence in $X_m$.
\item[{\em (3)}] For every  $m\geq 0$ and $\varepsilon>0$, the estimate 
$$
\abs{[D_2{\bf f}(b,0)-D_2{\bf f}(b,x)]h}_m\leq \varepsilon\cdot \abs{h}_m
$$
holds for all $h\in E_m$, and for  $b\in B$ sufficiently small,  and for $x\in X_{m+1}$ sufficiently small on level $m$. 
\end{itemize}
Then the section ${\bf f-s}$ is conjugated near $0$ to a basic germ, where ${\bf s}$ is the constant $\ssc^+$-section ${\bf s}(x)={\bf f}(0)$.
\end{theorem}

\begin{proof}
Abbreviating  $Y:=D{\bf f}(0)X\subset F$, the restriction 
$$
D{\bf f}(0)\vert X\colon X\rightarrow Y
$$
is an injective and surjective sc-operator. Since $D{\bf f}(0)\colon E\to F$ is a sc-Fredholm operator, we find a finite-dimensional sc-complement $A$ of $Y$ in $F$, such that 
$$
F=Y\oplus A.
$$
Denoting  by $P\colon Y\oplus A\rightarrow Y$ the sc-projection, we consider the family $b\mapsto L(b)$ of bounded linear operators, defined by  
$$L(b)=P\circ D{\bf f}(b)\vert X\colon X\to Y.$$
It is not assumed  that the operators $L(b)$ depend continuously on $b$. Since $B$ is a finite-dimensional
sc-smooth space,  the map 
$$
(B\cap U)\oplus X\rightarrow Y\colon (b,x)\mapsto  P\circ D{\bf f}(b)x
$$
is sc-smooth. If we raise the index by one, then the map  
$$
(B\cap U)\oplus X^1\rightarrow Y^1\colon (b,x)\mapsto  P\circ D{\bf f}(b)x
$$
is also sc-smooth by Proposition \ref{sc_up}.

\begin{lemma}\label{trivial_kernel}\index{L- Basic germ criterion {I}}
There exists a relatively open neighborhood $O$ of $0$ in $B\cap C$ such  that  the composition  
$$
P\circ D{\bf f}(b)\colon X\rightarrow Y,
$$
has a trivial kernel for every $b\in O$.
\end{lemma}
\begin{proof}
Assuming  that such an open set $O$ does not exist,  we find a sequence $b_j\in B\cap C$ satisfying  $b_j\rightarrow 0$ and a sequence $(h_j)\subset X_0$ satisfying $|h_j|_0=1$ such  that
$P\circ D{\bf f}(b_j)h_j=0$. Then $D{\bf f}(b_j)h_j = z_j$ is a bounded sequence in $A$,  in fact on every level since $A$ is a smooth subspace, and we consider  the level $1$ for the moment.
From property (2) we deduce that $(h_j)$ has a convergent subsequence in $X_0$. So,  without loss of generality,  
we may assume $h_j\rightarrow h$ in $X_0$ and  $|h|_0=1$. Hence
$$
P\circ D{\bf f}(0)h=0, 
$$
in contradiction to the injectivity of the map $P\circ D{\bf f}(0)\vert X$. 
The lemma is proved.
\end{proof}

From the property (1) and the fact that $P$ is a sc-Fredholm operator,  we conclude that $P\circ D{\bf f}(x)$ for smooth $x$ are all sc-Fredholm operators having  the same index. In particular, if $b\in O$, then $P\circ D{\bf f}(b)\colon X\rightarrow Y$ is an  injective sc-Fredholm operators of index $0$ in view of Lemma \ref{trivial_kernel},  and 
hence  a  sc-isomorphisms. Next we sharpen this result.

\begin{lemma}\index{L- Basic germ criterion {II}}
We take a relatively  open neighborhood $\wt{O}$ of $0$ in $C\cap B$ whose compact closure is contained in $O$. Then
for every level $m$ there exists a number $c_m>0$ such  that,  for every $b\in \wt{O}$,  we have the estimate
$$
|P\circ D{\bf f}(b)h|_m\geq c_m\cdot |h|_m\quad  \text{for all $h\in X_m$.}
$$
\end{lemma}

\begin{proof}

Arguing indirectly we find a level $m$ for which there is no such constant $c_m$. Hence there are sequences
$(b_j)\subset \wt{O}$ and $(h_j)\subset X_m$ satisfying $|h_j|_m=1$ and  
$\abs{P\circ D{\bf f}(b_j)h_j}_m\rightarrow 0$.
After perhaps taking a subsequence we may assume that $b_j\rightarrow b$ in $O$. 

From 
$$
D{\bf f}(b_j)h_j =P\circ D{\bf f}(b_j)h_j=y_j\to 0\quad \text{in $F_m$},
$$
we conclude,  in view of the property (2),  for a subsequence,  that $h_j\to h$ in $X_m$. By continuity, $P\circ D{\bf f}(b)h=0$ and $\abs{h}_m=1$, in contradiction to the fact, that $P\circ D{\bf f}(b)\colon X\to Y$ is a 
sc-isomorphism for $b\in O$.

\end{proof}

So far we have verified that the family $b\mapsto L(b)$ meets the assumptions 
of the following lemma,  taken from \cite{HWZ8.7}, Proposition 4.8.

\begin{lemma}\label{family_maps}\index{L- Basic germ criterion {III}}
We assume that $V$ is a relatively open subset in the partial quadrant of a finite-dimensional vector space $G$.
We suppose further that $E$ and $F$ are sc-Banach spaces and consider a family of linear operators $v\rightarrow L(v)$ 
having  the following properties.
\begin{itemize}
\item[{\em (i)}] For every $v\in V$, the linear operator  $L(v)\colon E\rightarrow F$ is a sc-isomorphism.
\item[{\em (ii)}] The map
$$
V\oplus E\rightarrow F,\quad (v,h)\mapsto L(v)h
$$
is sc-smooth.
\item[{\em (iii)}] For every level $m$ there exists a constant $c_m>0$ such that for $v\in V$ and all $h\in E_m$
$$
|L(v)h|_m\geq c_m\cdot |h|_m.
$$
\end{itemize}
Then the well-defined map
$$
V\oplus F\rightarrow E\colon (v,k)\mapsto  L(v)^{-1}(k)
$$
is sc-smooth.
\end{lemma}

Let us emphasize that it is  not assumed  that the operators $v\rightarrow [L(v)\colon E_m\rightarrow F_m]$ depend continuously as operators 
on $v$.


In view of Lemma \ref{family_maps}, the map 
$B\oplus Y\to Y$,
$$(b, y)\mapsto \bigl(P\circ Df(b)\vert X\bigr)^{-1}y,$$
is sc-smooth.

We may assume that the finite-dimensional space  $B$ is equal to $\R^n$ and that $E=\R^n\oplus X$ and $C=[0,\infty)^k\oplus \R^{n-k}\oplus X$ is the partial quadrant in $E$. Hence $B\cap C=[0,\infty)^k\oplus \R^{n-k}$. Moreover, we may identify the finite-dimensional subspace $A$ of $F$ with $\R^N=A$. Replacing, if necessary, the  relatively open neighborhood $U \subset C$ of $0$ by a smaller one
we may assume, in addition,  that $(b,x)\in U$ implies that $b\in \wt{O}$.

 We now define a strong bundle map
$$
\Phi\colon U\triangleleft (\R^N\oplus Y)\rightarrow U\triangleleft (\R^N\oplus Y)
$$
covering the identity $U\to U$ by
$$
\Phi((b,x),(c,y))= ((b,x), (c,[P\circ D{\bf f}(b)|X]^{-1}(y))).
$$
We define  the sc-smooth germ $({\bf h},0)$ by ${\bf h}(b,x)={\bf f}(b,x)-{\bf f}(0,0)$, where  $(b,x)\rightarrow {\bf f}(0,0)$ is a constant $\ssc^+$-section.
We shall show that the  push-forward germ ${\bf k}=\Phi_\ast({\bf h})$ is a basic germ. Using $D{\bf h}=D{\bf f}$,  we compute 
$$
{\bf k}(b,x) = (({\mathbbm 1}-P){\bf h}(b,x), [P\circ D{\bf h}(b)|X]^{-1}P{\bf h}(b,x))\in \R^N\oplus X.
$$
The germ ${\bf k}$ is  a sc-smooth germ 
$$ 
{\mathcal O}([0,\infty)^k\oplus \R^{n-k}\oplus X,0)\rightarrow (\R^N\oplus X,0).
$$
Denoting by $Q\colon \R^N\oplus X\rightarrow X$ the sc-projection, we shall verify that $Q{\bf k}$ is a $\ssc^0$-contraction germ.  We define the sc-smooth  germ $H\colon \R^n\oplus X\to X$ by 
$$
H(b,x) = x- Q{\bf k}(b,x)=x - [P\circ D{\bf h}(b)|X]^{-1}P{\bf h}(b,x).
$$
By construction,  $H(0,0)=0$.  The family $L(b)\colon X\to Y$ of bounded linear operators, defined by 
$$
L(b):=P\circ D{\bf f}(b)\vert X,\quad b\in B,
$$
satisfies the assumptions of Lemma \ref{family_maps}. Recalling now the condition (3) of Theorem \ref{basic_germ_criterion}, we choose $m\geq 0$ and $\varepsilon>0$ and accordingly take $b\in B$ small and $x, x'\in X_{m+1}$ small on level $m$. Then using the estimates in condition (3) and in Lemma  \ref{family_maps}, we estimate, recalling that $P\circ D_2h(b, x)=P\circ D_2f(b, x)$, 

\begin{equation*}
\begin{split}
&|H(b,x)-H(b,x')|_m=\abs{L(b)^{-1} [L(b)(x-x') - P{\bf h}(b,x) +P{\bf h}(b,x')}_m\\
&\quad \leq c_m^{-1}\cdot \abs{P \circ \bigl[D_2{\bf f}(b, 0)(x-x')  - {\bf h}(b,x) +{\bf h}(b,x')] }_m\\
&\quad =c_m^{-1}\cdot \abs{ 
\int_0^1 
P\circ  \bigl[D_2{\bf f}(b, 0)- D_2{\bf f}(b,tx +(1-t)x')\bigr](x-x')dt}_m\\
&\quad \leq  c_m^{-1}\cdot d_m\cdot\varepsilon\cdot |x-x'|_m.
\end{split}
\end{equation*}
The map $(b, x)\mapsto H(b, x)$ is sc-smooth. Therefore, using the density of $X_{m+1}$ in $X_m$, we conclude,  
for every $m\geq 0$ and $\varepsilon>0$,  that the estimate 
$$\abs{H(b, x)-H(b, x')}_m\leq \varepsilon\abs{x-x'}_m$$
holds, if $b\in B$ is small enough and $x, x'\in X_m$ are sufficiently small.  Having verified that the push-forward germ ${\bf k}=\Phi_\ast ({\bf h})$ is a basic germ, the proof of Theorem \ref{basic_germ_criterion} is complete.

\end{proof}
\subsection{Appendix}

\subsubsection{Proof of Proposition \ref{pretzel}}\label{pretzel-A} 
\begin{P3.22}\label{pretzel_0}
If  $N$ is a finite-dimensional sc-subspace in good position to the partial quadrant  $C$ in $E$, then
$N\cap C$ is a partial quadrant in $N$.
\end{P3.22}
As a preparation for the proof we recall some tools and results from the appendix in \cite{HWZ3} and begin with the geometry of closed convex cones and quadrants in finite dimensions.
A  {\bf closed convex cone} $P$ in a finite-dimensional vector space
 $N$ is a closed convex
subset satisfying  $P\cap (-P)=\{0\}$ and $\R^+\cdot P=P$.
 An
{\bf extreme ray} in a closed convex cone $P$ is a subset $R$ of the
form
$$
R=\R^+\cdot x, 
$$
where  $x\in P\setminus \{0\}$, having the property that  if   $y\in P$ and  $x-y\in P$, then $y\in R$. If the cone $P$ has a nonempty interior, then  it generates the vector space  $N$, that is $N=P-P$.

A {\bf quadrant} $C$ in a vector space $N$ of dimension $n$ is a closed convex cone such that there exists a linear isomorphism $T\colon N\rightarrow \R^n$ mapping $C$ onto $[0,\infty)^n$. We observe that a quadrant
in $N$ has precisely $\dim(N)$ many extreme rays.  

The following  version of the Krein-Milman theorem
is well-known, see exercise 30 on page 72 in \cite{Schaefer}.  A proof can be found in the appendix of \cite{HWZ3}, Lemma 6.3.
\begin{lemma}\label{kreinmilman}
A closed convex cone $P$ in a finite-dimensional vector space $N$ is
the  closed convex hull of its extreme rays.
\end{lemma}
 A closed convex cone $P$ is called {\bf finitely generated} provided $P$ has
finitely many extreme rays. If this is the case, then  $P$ is the convex hull of
its finitely many extreme rays.

 For example,  if $C$ is a partial
quadrant in the sc-Banach space $E$ and $N\subset E$ is a finite-dimensional subspace of
$E$ such  that $C\cap N$ is a closed convex cone, then $C\cap N$ is
finitely generated.

\begin{lemma}\label{nomer}
Let $N$ be a finite-dimensional vector space and $P\subset N$ a
closed convex cone having a nonempty interior. Then $P$ is a quadrant if
and only if it has $\dim(N)$-many extreme rays.
\end{lemma}
The proof is given in  \cite{HWZ3}, Lemma  6.4. 

We consider the sc-Banach space $E =\R^n\oplus W$  containing the partial quadrant $C =[0,\infty)^n\oplus W$. A point  $a\in C=[0,\infty )^n\oplus W\subset \R^n\oplus W$, has  the representation
$a=(a_1,\ldots ,a_n, a_{\infty})$,  where $(a_1,\ldots ,a_n)\in [0,\infty )^n$ and $a_{\infty}\in W$.
By $\sigma_a$ we shall denote the collection of all  indices $i\in \{1,\ldots ,n\}$ for which
$a_i=0$ and denote by $\sigma_a^c$  the complementary set of indices in $\{1,\ldots ,n\}$. Correspondingly, we introduce the following subspaces in $\R^n$,
\begin{align*}
\R^{\sigma_a}&=\{x\in \R^n\ \vert \ \text{$x_j=0$ for all $j\not \in \sigma_a$}\}\\
\R^{\sigma^c_a}&=\{x\in \R^n\ \vert \ \text{$x_j=0$ for all $j\not \in \sigma^c_a$}\}.
\end{align*}
The next lemma and its proof is taken from the appendix in \cite{HWZ3}. The hypothesis that $C\cap N$ is a closed convex cone is crucial.

\begin{lemma}\label{roxy}
Let $N\subset E_{\infty}$ be  a finite-dimensional smooth subspace of  $E=\R^n\oplus W$ such  that $N\cap C$ is a closed convex cone. If $a\in N\cap C$ is a generator of  an extreme ray $R=\R^+\cdot x$ in $N\cap C$, then
$$
\dim(N)-1\leq \#\sigma_a.
$$
If,  in addition, $N$ is in good position to $C$, then
$$
\dim(N)-1= \#\sigma_a.
$$
\end{lemma}
\begin{proof}
We assume $R=\R^+\cdot a$ is an extreme ray in $C\cap N$ and abbreviate $\sigma=\sigma_a$ and its complememnt in $\{1, \ldots ,n\}$ by  $\sigma^c$.  Then $R\subset N\cap C\cap (\R^{\sigma^c}\oplus W)$. Let $y\in C\cap N\cap (\R^{\sigma^c}\oplus
W)$ be a nonzero element. Since $a_i>0$ for all $i\in\sigma^c$,  there exists $\lambda>0$ so that $\lambda a-y\in N\cap C\cap
(\R^{\sigma^c}\oplus W)\subset N\cap C$. We conclude $y\in R$ because  $R$ is an extreme ray. Given any element $z\in
N\cap(\R^{\sigma^c}\oplus W)$ we find $\lambda>0$ so that
$\lambda a+z\in N\cap C\cap(\R^{\sigma^c}\oplus W) $ and infer,
by the previous argument,  that $\lambda a+z\in R$. This implies that $z\in \R \cdot a$. Hence
\begin{equation}\label{pol}
 \dim(N\cap(\R^{\sigma^c}\oplus W))=1.
\end{equation}
The projection $P\colon \R^n\oplus W=\R^{\sigma}\oplus (\R^{\sigma^c}\oplus W) \to \R^{\sigma}$ induces a linear map
\begin{equation}\label{ohx}
P\colon N\rightarrow \R^{\sigma}
\end{equation}
which by \eqref{pol} has  an one-dimensional kernel. Therefore,
$$
\#\sigma=\dim(\R^{\sigma})\geq \dim R(P)=\dim N-\dim \ker P=\dim(N)-1.
$$
Next assume $N$ is in good position to $C$. Hence there exist a constant $c>0$ and a sc-complement $N^{\perp}$ such that $N\oplus N^{\perp}=\R^n\oplus W$ and if $(n, m)\in N\oplus N^{\perp}$ satisfies $|m|_0\leq c |n|_0$, then $n+m\in C$ if and only if $n\in C$.  We claim that  $N^{\perp}\subset \R^{\sigma^c}\oplus W$. Indeed, let $m$ be any element of $N^{\perp}$.  Multiplying $m$ by a real number we may assume  $|m|_0\leq c |a|_0$ .  Then $a+m\in C$ since $a\in C$. This implies that
$m_i\geq 0$ for  all indices $i\in\sigma_a$. Replacing $m$ by $-m$,  we conclude
$m_i=0$ for all $i\in\sigma_a$.  So $N^{\perp}\subset  \R^{\sigma^c}\oplus W$ as claimed.
Take  $k\in \R^{\sigma_a}$
and write $(k, 0)=n+m\in N\oplus N^{\perp}$.  From  $N^{\perp}\subset {\mathbb
R}^{\sigma^c}\oplus W$, we conclude $P(n)=k$. Hence the map $P$  in \eqref{ohx}  is surjective  and the desired result follows.
\end{proof}


Having studied the geometry of closed convex cones and partial quadrants in finite dimensions we shall next study finite dimensional subspaces $N$ in good position to a partial quadrant $C$ in a sc-Banach space.  In this case,  $N\cap C$ can be a partial quadrant rather than a quadrant, which requires, some additional arguments.

We assume that $N$ is a smooth  finite-dimensional subspace of
$E=\R^n\oplus W$ which is in good position to the partial quadrant $C=[0,\infty)^n\oplus
W$. Thus, by definition, there is  a sc-complement,  denoted by  $N^{\perp}$,  of $N$ in $E$ and a constant $c>0$ such 
that if  $(n,m)\in N\oplus N^{\perp}$ satisfies  $|m|_0\leq c |n|_0$, then  the statements
$n\in C$ and $n+m\in C$ are equivalent.
We introduce the subset $\Sigma$ if $\{1,\ldots,n\}$ by 
$$\Sigma=\bigcup_{a\in C\cap N, a\neq 0}\sigma_{a}\subset \{1,\ldots ,n\}.$$
and denote by $\Sigma^c$ the complement $\{1,\ldots ,n\}\setminus \Sigma$. The associated subspaces of $\R^n$ are defined by
$\R^{\Sigma}=\{x\in \R^n\ \vert \ \text{$ x_j=0$ for $j\not \in \Sigma$}\}$ and
$\R^{{\Sigma}^c}=\{x\in \R^n\ \vert \ \text{$ x_j=0$ for $j\not \in \Sigma^c$}\}$.

\begin{lemma}\label{cone1}
$N^{\perp}\subset {\mathbb
R}^{\Sigma^c}\oplus W$.
\end{lemma}
\begin{proof}
Take   $m\in N^{\perp}$. We  have to show that $m_i=0$ for all $i\in\Sigma$.
So fix an index  $i\in\Sigma$ and let $a$ be a nonzero element of  $C\cap N$ such that
$i\in\sigma_a$. Multiplying $a$ by a suitable positive number we may
assume $|m|_0\leq c |a|_0$.  Since $a\in C$,   we infer that $a+m\in C$. This implies
that $a_i+m_i\geq 0$. By definition of $\sigma_a$, we have $a_i=0$  implying $m_i\geq 0$.
Replacing $m$ by $-m$ we find $m_i=0$. Hence $N^{\perp}\subset {\mathbb
R}^{\Sigma^c}\oplus W$ as claimed.
\end{proof}

Identifying $W$ with $\{0\}\oplus W$,  we take   an algebraic complement $\wt{N}$ of $N\cap W$ in $N$ so that
\begin{equation}\label{hoferx}
N=\wt{N}\oplus (N\cap W)\quad \text{and}\quad\quad E= \wt{N}\oplus(N\cap W)\oplus N^{\perp}.
\end{equation}
Let us note that the projection $\pi\colon \R^n\oplus W\rightarrow  \R^n$ restricted to $\tilde{N}$ is injective,  so that
\begin{equation}\label{hoferx0}
\dim(\pi(\tilde{N}))=\dim(\tilde{N}).
\end{equation}
\begin{lemma}\label{ntilde}
If the subspace $N$ of $E$ is in good position to the partial quadrant  $C$, then
$\wt{N}$ is also  in good position to $C$ and the  subspace $\wt{N}^{\perp}:=(N\cap W)\oplus N^{\perp}$ is a good complement of $\wt{N}$ in $E$.
\end{lemma}
\begin{proof}
We define $\abs{x}:=|x|_0$.
Since $N$ is in good position to the quadrant $C$ in $E$, there exist a constant $c>0$ and a sc-complement $N^{\perp}$ of $N$ in $E$ such that if $(n, m)\in N\oplus N^{\perp}$ satisfies $\abs{m}\leq c\abs{n}$, then
 the statements $n\in C$ and $n+m\in C$ are equivalent.  Since $E$ is a Banach space and $N$ is a finite dimensional subspace of $E$, there is a constant $c_1>0$ such that
 $\abs{n+m}\geq c_1[ \abs{n}+\abs{m}]$ for all $(n, m)\in N\oplus N^{\perp}$.  To prove  that $\wt{N}$ is in good position to $C$,  we shall show that $\wt{N}^{\perp}:=(N\cap W)\oplus N^{\perp}$ is a good complement of $\wt{N}$ in $E$. Let $(\wt{n}, \wt{m})\in \wt{N}\oplus \wt{N}^{\perp}=E$ and assume that  $\abs{\wt{m}}\leq c_1c\abs{\wt{n}}$. Write $\wt{m}=n_1+n_2\in (N\cap W)\oplus N^{\perp}$. Since $c_1[\abs{n_1}+\abs{n_2}]\leq \abs{n_1+n_2}=\abs{\wt{m}}\leq c_1c\abs{\wt{n}}$, we get $\abs{n_2}\leq c\abs{\wt{n}}$.  Note that $\wt{n}+\wt{m}=\wt{n}+n_1+n_2\in C$ if and only if $\wt{n}+n_2\in C$ since $n_1\in  \{0\}\oplus W$.  Since  $\abs{n_2}\leq c\abs{\wt{n}}$, this is equivalent to $\wt{n}\in C$. It remains to show that $\wt{N}\cap C$ has a nonempty interior. By assumption $N\cap C$ has nonempty interior. Hence there is a point $n\in N\cap C$ and $r>0$ such that the ball $B^{N}_r(n)$ in $N$ is contained in $N\cap C$. Write $n=\wt{n}+w$ where $\wt{n}\in \wt{N}$ and $w\in N\cap W$. Since $n\in C$ and  $w\in W$, we conclude that $\wt{n}\in C$. Hence $\wt{n}\in \wt{N}\cap C$.
Take $\nu \in B^{\wt{N}}_r(\wt{n})$,  the open ball in $\wt{N}$ centered at
  $\wt{n}$ and of radius $r>0$.  We want to prove that $\nu \in C$. Since $C=[0,\infty )^n\oplus W\subset \R^n\oplus W$, we have to prove for $\nu= (\nu', \nu'')\in \R^n\oplus W$ that $\nu'\in [0,\infty )^n$.  We estimate  $\abs{(\nu +w)-n}=\abs{(\nu +w)-(\wt{n}+w)}=\abs{\nu -\wt{n}}<r$ so that $\nu +w\in B^{N}_r(n)$ and hence  $\nu +w\in N\cap C$ . Having identified $W$ with $\{0\}\oplus W$, we have $w=(0, w'')\in \R^n\oplus W$.  Consequently, $\nu+w=(\nu', \nu''+w'')\in N\cap C$ implies $\nu' \in [0,\infty )^n$.  Since also $\nu \in \wt{N}$, one concludes  that $\nu \in \wt{N}\cap C$ and that $\wt{n}$ belongs to the interior of $\wt{N}\cap C$ in $\wt{N}$.  The proof of Lemma \ref{ntilde} is complete.
 \end{proof}
 Since by  Lemma \ref{cone1},  $N^{\perp}\subset \R^{\Sigma^c}\oplus W$,  the  good complement   $\wt{N}^{\perp}=(N\cap W)\oplus N^{\perp}$ satisfies
$$
\wt{N}^{\perp}\subset \R^{\Sigma^c}\oplus W.
$$

We claim that $C\cap \wt{N}$ is a closed convex cone . 
It suffices to verify that if $a\in C\cap \wt{N}$ and $-a\in C\cap \wt{N}$, then $a=0$. We write 
$a=(a', a_\infty)$ where $a'\in \R^n$ and $a_\infty\in W$. Then $a', -a'\in [0,\infty)^n$ implies that $a'=0$ so that $a=(0, a_\infty)\in \{0\}\oplus W$. Since $a\in \wt{N}$ and $\wt{N}$ is an algebraic complement of $N\cap W$ in $N$, we conclude that $a=0$.  
Moreover,  by  Lemma  \eqref{ntilde},   $\wt{N}$ is in good position to $C$. 
Hence,  recalling  Lemma \ref{roxy}, we have proved  the following result.

\begin{lemma}
The intersection $C\cap\wt{N}$ is a closed convex cone in
$\wt{N}$. If  $a$ is a generator of an extreme ray in $C\cap \wt{N}$, then  
\begin{equation}\label{Nsigma}
\dim \wt{N}-1 =\#{\sigma_a}= d_C(a).
\end{equation}
\end{lemma}

The position of $\wt{N}$ with
respect to $\R^{\Sigma^c}\oplus W$ is described in the next lemma.
\begin{lemma}\label{ll1}
Either $\wt{N}\cap (\R^{\Sigma^c}\oplus W)=\{0\}$ or
$\wt{N}\subset \R^{\Sigma^c}\oplus W$.  In the second case
$\dim \wt{N}=1$ and $\Sigma=\emptyset$.
\end{lemma}

\begin{proof}
Assume that  $\wt{N}\cap (\R^{\Sigma^c}\oplus
W)\neq\{0\}$. Take  a nonzero point $x\in \wt{N}\cap (\R^{\Sigma^c}\oplus W)$. We know that $\wt{N}\cap C$ has a
nonempty interior in $\wt{N}$ and is therefore generated as the convex hull of
its extreme rays by Lemma \ref{kreinmilman}.  Let $a\in C\cap \wt{N}$ be a generator of an
extreme ray $R$. Then $a_i>0$ for all $i\in\Sigma^c$ and  hence $\lambda a +x\in C\cap \wt{N}$ for large  $\lambda >0$. Taking another large  number  $\mu>0$, we get   $\mu a- (\lambda a+x)\in C\cap \wt{N}$. Since $R=\R^+\cdot a$ is an extreme ray, we conclude $\lambda a+x\in \R^+\cdot a$ so that $x\in \R\cdot a$. Consequently,
there is only one extreme ray in $\wt{N}\cap C$, namely
 $R=\R^+\cdot a$ with  $a\in \R^{\Sigma^c}\oplus W$.  Since $\wt{N}\cap C$ has a nonempty interior in $\wt{N}$, we conclude that $\dim \wt{N}=1$  . Hence $\wt{N}=\R\cdot a$ and
 $\wt{N}\subset  \R^{\Sigma^c}\oplus W$. From equation \eqref{Nsigma} we also conclude that $a_i>0$ for all $1\leq i\leq n$.  This in turn implies that $\Sigma=\emptyset$ since  $a\in \R^{\Sigma^c}\oplus W$. The proof of Lemma \ref{ll1} is complete.
\end{proof}

We finally come to the proof of Proposition  \ref{pretzel}.

\begin{proof}[Proof of Proposition \ref{pretzel}] 
We consider, according to Lemma \ref{ll1},  two cases.
Starting with the first case  we assume that $\wt{N}\cap (\R^{\Sigma^c}\oplus
 W)=\{0\}$. The projection $P\colon \wt{N}\oplus \wt{N}^{\perp}=\R^{\Sigma}\oplus (\R^{\Sigma^c}\oplus W)\to \R^{\Sigma}$  induces the linear map
 \begin{equation}\label{opl}
 P\colon \wt{N}\rightarrow \R^{\Sigma}.
 \end{equation}
Take $k\in\R^{\Sigma}$ and write
 $(k,0)=n+m\in\wt{N}\oplus\wt{N}^{\perp}$. Since $\wt{N}^{\perp}\subset
\R^{\Sigma^c}\oplus W$,  we conclude that
 $$
 P(n+m)=P(n)=k
 $$
so that $P$ is surjective. If $n\in \wt{N}$ and $P(n)=0$, then $n\in  \wt{N}\cap ({\mathbb
 R}^{\Sigma^c}\oplus W)=\{0\}$ by assumption. Hence the map in
 \eqref{opl} is a bijection.  By  Lemma \ref{nomer},
$C\cap\wt{N}$ is a quadrant in $\wt{N}$. We shall show that $P$ maps the quadrant $C\cap \wt{N}$ onto the standard quadrant $Q^{\Sigma}=[0,\infty )^{\Sigma}$ in $\R^{\Sigma}$.
Let $a$ be a nonzero element in $C\cap \wt{N}$ generating  an extreme ray $R=\R^+\cdot a$. Then, by Lemma \ref{roxy},
 $$
 \dim \wt{N}-1=\sharp\sigma_a, 
 $$
 and since $\sharp \Sigma=\text{dim}\ \wt{N}$ there is
exactly one index $i\in\Sigma$ for which
$a_i>0$. Further,  $a_i>0$ for all $i\in\Sigma^c$ by definition of $\Sigma$.  This implies that there can be
at most $\dim(\wt{N})$-many extreme rays. Indeed, if $a$ and $a'$
generate extreme rays and
$a_i,a_i'>0$  for some $i\in \Sigma$, then $a_k=a_k'=0$ for all
$k\in\Sigma\setminus\{i\}$.  Hence,  from  $a_j>0$ for all $j\in \Sigma^c$,   we conclude $\lambda
a-a'\in C$ for large $\lambda>0$. Therefore,
$a'\in \R^+a$ implying that $a$ and $a'$ generate the same
extreme ray. As a consequence,  $\wt{N}\cap C$ has  precisely $\dim \wt{N}$-many extreme rays because $\wt{N}\cap C$ has a nonempty interior in view of Lemma \ref{ntilde}.  Hence  the map $P$ in \eqref{opl}  induces an isomorphism
$$
(\wt{N}, \wt{N} \cap C)\rightarrow (\R^\Sigma,Q^\Sigma).
$$
This implies that $(N,C\cap N)$ is isomorphic to
$\bigl( {\R}^{\dim(N)}, [0,\infty)^{\sharp\Sigma}\oplus {\R}^{\dim(N)-\sharp\Sigma}\ \bigr).$

In the second case we assume  that $\wt{N}\subset \R^{\Sigma^c}\oplus W$.  From Lemma  \ref{ll1},  $\Sigma=\emptyset$ and $\wt{N}=\R\cdot a$ for an element $a\in
C\cap\wt{N}$  satisfying  $a_i>0$ for all $1\leq i\leq n$. Hence
$(\wt{N},\wt{N}\cap C)$ is isomorphic to
$( {\R}, {\R}^+)$ and therefore $(N,N\cap C)$ is isomorphic to
$( {\R}, \R^+)$ since in this case $N=\wt{N}$. The proof of Proposition \ref{pretzel} is complete.

\end{proof}

We would like to add two results, which makes use of the previous discussion.

\begin{proposition}\label{big-pretzel}
Let $N$ be a sc-smooth finite-dimensional subspace of $E=\R^n\oplus W$ in good position to the partial quadrant $C=[0,\infty)^n\oplus W$.
If $x\in N\cap C$, then 
\begin{itemize}
\item[{\em (1)}] $d_{N\cap C}(x)= d_C(x)$  if  $x\not \in N\cap W$.
\item[{\em (2)}] $d_{N\cap C}(x)= \dim(N)-\dim(N\cap W)$ if $x\in N\cap W$. 
\end{itemize}
Here we identify $W$ with $\{0\}^n\oplus W$.
\end{proposition}

\begin{proof}
We make use of the previous notations and distinguish the two case $\dim N=1$ and $\dim N>1$.

If $\dim N=1$, then $N$ is spanned by a 
vector $e=(a_1,\ldots ,a_n, w)\in \R^n\oplus W$ in which $a_j>0$ for all $1\leq j\leq n$. This  implies that $(N, N\cap C)$ is isomorphic to $(\R, [0,\infty))$. 
If $x\in N\cap C$, then $x=te$ for some $t\geq 0$. If, in addition, 
$x\not \in N\cap W$, then $t>0$ and hence $d_{N\cap C}(x)=0=d_C(x)$. If  however, $x\in N\cap W$, then $t=0$, implying that $d_{N \cap C}(x)=1=\dim N.$

Now we assume that $\dim N>1$ and that $\wt{N}$ is the algebraic complement of $N\cap W$ in $N$ so that 
$N=\wt{N}\oplus (N\cap W)$. We may assume, 
after a linear change of coordinates, that $\wt{N}$ is represented as follows.
\begin{itemize}
\item[(1)] The linear subspace $\wt{N}$ is spanned by the vectors $e^j\in \R^n\oplus W$,  for $1\leq j\leq m=\dim \wt{N}$ of the form 
$e^j=(a^j, b^j, w^j)\in \R^m\oplus \R^{n-m}\oplus W$ where $a^j$  is  the standard basis vector in $\R^m$, $b^j=(b^j_1, \ldots ,b^j_{n-m})\in \R^{n-m}$ satisfies  $b^j_i>0$ for all $1\leq i\leq n-m$, and $w^j\in W$. 
\end{itemize}

If $e^j=(0,w^j)\in \{0\}^n\oplus W$ for $m+1\leq j\leq k$,$k=\dim N$,  is a basis of the finite-dimensional subspace $N\cap W$, then 
the vectors $e^1,\ldots ,e^k$ form a basis of $N=\wt{N}\oplus (N\cap W)$.  It follows that 
$(N, N\cap C)$ is isomorphic to $(\R^k, [0,\infty )^m\oplus \R^{k-m})$.

Now we assume that $x\in N\cap C$. 
 If $x\in N\cap C$, then 
\begin{equation}\label{algebraic_eq}
x=\sum_{j=1}^m\lambda_je^j+\sum_{j=m+1}^k\lambda_j e^j\in \wt{N}\oplus(N\cap W)
\end{equation} 
with $\lambda_j\geq 0$ for all $1\leq j\leq m$ and $d_{N\cap C}(x)=\#\{1\leq j\leq m\ \vert \ \lambda_j=0\}$. If, in addition, $x\in N\cap W$, then $\lambda_j=0$ for all $1\leq j\leq m$, and we conclude that 
$d_{N\cap C}(x)=m=\dim N-\dim (N\cap W)$.

If $x\in ( N\cap C)\setminus (N\cap W)$, then there is at least one index $j_0$ in $1\leq j_0\leq m$  for which $\lambda_{j_0}>0$. This implies, in view of (1)  and \eqref{algebraic_eq} that 
$$x=\sum_{j=1}^me^j=(x_1, \ldots, x_n,w)\in \R^n\oplus W$$
satisfies $x_s>0$ for all $m+1\leq s\leq n$ since $b^j_i>0$ for  all $1\leq i\leq n-m$ and $m+1\leq j\leq n$. Using that $a^j$ in (1) are vectors of the standard basis in $\R^m$,  it follows  for $1\leq s\leq m$
that $x_s=0$ if and only if $\lambda_s=0.$ This shows that $d_{N\cap C}(x)=d_C(x)$ if $x\not \in N\cap W$ and completes the proof of Proposition \ref{big-pretzel}. 

\end{proof}

\begin{lemma}\label{big-pretzel_1a}
Let $N$ be a sc-smooth finite-dimensional subspace of $E=\R^n\oplus W$ in good position to the partial quadrant $C=[0,\infty)^n\oplus W$ whose good complement in $E$ is $Y$.  We assume that $\tau\colon N\cap C\to Y_1$ is  a $C^1$-map satisfying $\tau (0)=0$ and $D\tau (0)=0$.  Then 
$$d_{N\cap C}(v)=d_{C}(v+\tau (v))$$
for all $v\in N\cap C\setminus N\cap W$ close to $0$.
\end{lemma}

\begin{proof}

Since $N$ is in good position to $C$ and $Y$ is a good complement of $N$ is $E$, there exists a constant 
exists a constant  $\gamma>0$ such that if $n\in N$ and $y\in Y$ satisfy $\abs{y}_0\leq \gamma \abs{n}_0$, then 
$n\in C$ if and only if $n+y\in C$. It follows, in view of $\tau (0)=0$ and $D\tau (0)=0$, that $\abs{\tau (v)}_0\leq \gamma\abs{v}_0$ for $v\in N\cap C$ close to $0$. 

Now proceeding as in the proof of Proposition \ref{big-pretzel}, we 
distinguish the two cases $\dim N=1$ and $\dim N\geq 2$.

In the first case, $N$ is spanned by a vector $e=(a, w)\in \R^n\oplus W$ in which $a=(a_1,\ldots ,a_n)$ satisfies $a_j>0$ for all $1\leq j\leq n$. Take  $v\in N\cap C\setminus N\cap W$, then $v=t(a, w)$ for $t>0$, so that $d_{N\cap C}(v)=0$. 
We  already know that $v_j+\tau_j(v)\geq 0$ for all $1\leq j\leq n$. So, to prove the lemma, it suffices to show that 
$v_j+\tau_j(v)>0$ for $1\leq j\leq n$ and $v\in N\cap C\setminus N\cap W$ close to $0$.  Arguing by contradiction we assume that there exists a sequence $t_k\to 0$ such  that if $v^k=t_ka$, then $v^k_i+\tau_i(v^k)=0$ for some index $1\leq i\leq n$. Then 
$$0\neq \dfrac{a}{\abs{a}_0}=\dfrac{v_i^k}{\abs{v^k}_0}=-\dfrac{\tau_i(v^k)}{\abs{v^k}_0},$$
which leads to   a contradiction since the right-hand side converges to $0$ in view of $\tau (0)=0$ and $D\tau (0)=0$. Consequently, $d_C(v+\tau (v))=0$ for all $v\in  N\cap C\setminus N\cap W$ which are close to $0$.\\[0.5ex]

In the case $\dim N\geq 2$, we use the notation introduced in the proof of Proposition \ref{big-pretzel}. If $v=(v_1,\ldots ,v_k, w)\in 
 N\cap C\setminus N\cap W$, then, in view of (1) in the proof of Proposition \ref{big-pretzel}, $d_{N\cap C}(v)=\#\{1\leq i\leq k\, \vert \, v_i=0\}$ and  there exists at least one index $1\leq i\leq k$ for which  $v_i>0$. This implies,  in particular,  that 
 $v_j>0$ for all $k+1\leq j\leq n$.
 Moreover,  in view of Lemma \ref{cone1}, the good complement $Y$ of $N$ is contained in $\{0\}^m\oplus \R^{n-m}\oplus W$. From this we conclude that $v_j+\tau_j (v)=v_j$ for all $1\leq j\leq m$. Hence in order to finish the proof it suffices to show that $v_j+\tau_j(v)>0$ for all $m+1\leq j\leq n$. In order to verify  this we introduce the sc-linear map $T\colon N\to \R^k$ defined by 
 $T(e^l)=\ov{e}^l$ for $1\leq l\leq k:=\dim N$ where  $\ov{e}^1, \ldots ,\ov{e}^k$ of $\R^k$. Then $T(N\cap C)=C':=[0,\infty )^m\oplus \R^{k-m}$ and  $d_{N\cap C}(v)=d_{C'}(T(v)).$

We consider  the map $g\colon C'\to \R^n\oplus W$,  defined by 
$$g(v')=T^{-1}(v')+\tau (T^{-1}(v'))$$ 
for $v'=(v_1', \ldots ,v'_k)\in C'$ close to $0$. Since $v+\tau (v)\in C$, $g(v')\in C$ for $v'\in C'$ close to $0$. 
We prove our claim by showing that $g_j(v')>0$ for $m+1\leq j\leq n$  and nonzero $v'\in C'$ close to $0$.
Arguing by contradiction we assume that there exists a point $v'\in C'$  different from  $0$ at which $g_j(v')=0$ for some $m+1\leq j\leq n$. 
We have  $v_i'>0$ for some $1\leq i\leq m$ so that  $v'+t\ov{e}^i\in C'$ for $\abs{t}$ small.  Then we compute, 
$$\dfrac{d}{dt}g_j(v'+t\ov{e}^i)\vert_{t=0}=b^i_j+D\tau_j(v)T^{-1}\ov{e}^i=b^i_j+D\tau_j(v)e^i.$$
Since the map $\tau$ is of class $C^1$ and $D\tau (0)=0$ and $b^i_j>0$ we conclude that the derivative is positive, and hence the function $t\mapsto g_j(v'+t\ov{e}^i)$ is strictly increasing for $\abs{t}$ small.  By assumption $g_j(v')=0$, hence  $g_j(v'+t\ov{e}^i)<0$ for $t<0$ small, contradicting  $g_j(v'+t\ov{e}^i)\geq 0$ for $\abs{t}$ small.

Consequently, $v_j+\tau_j(v)>0$ for all $m+1\leq j\leq n$ and all $v\in N\cap C\setminus N\cap W$ sufficiently close to $0$ and since $v_j+\tau_j(v)=v_j$ for $1\leq j\leq m$, we conclude $d_{N\cap C}(v)=d_C(v+\tau (v))$ 
for $v\in N\cap C\setminus N\cap W$ close to $0$. The proof of Lemma \ref{big-pretzel_1a} is complete. 

\end{proof}

\subsubsection{Proof of Lemma \ref{new_lemma_Z}}\label{pretzel-B}

\begin{L3.46}\label{new_lemma_Z_1}
The kernel  $\wt{K}:=\ker DH(0)\subset \R^{n}$ of the linearization $DH(0)$ is in good position to the partial quadrant $\wt{C}=[0,\infty)^{k}\oplus \R^{n'-k}$ in $\R^{n'}$. Moreover, there exists a good complement $Z$ of $\wt{K}$ in $\R^{n'}$, so that 
$\wt{K}\oplus Z =\R^{n'}$, having the property that $Z\oplus W'$ is a good complement of $K'=\ker Dg(0)$ in $E'$, 
$$E'=K'\oplus ( Z\oplus W').$$\\[1ex]
\end{L3.46}
\begin{proof}
The space $E'=\R^{n'}\oplus W'$ is equipped with the norm $\abs{\cdot }_0$. We use the  equivalent norm defined by $\abs{(a, w)}=\max\{ \abs{a}_0,\abs{w}_0\}$ for $(a, w)\in \R^{n'}\oplus W'$. The kernel $K'=\ker Dg(0)\subset \R^{n'}\oplus W'$ is in good position to the partial quadrant $C'=[0,\infty )^k\oplus \R^{n'-k}\oplus W'$ and $Y'$ is a good complement of $K'$ in $E'$, so that $K'\oplus Y'=E'$.

We recall that by the definition of good position, there exists $\varepsilon>0$ such that,  
if  $k\in K'$ and $y\in Y'$ satisfy 
\begin{equation}\label{again_good_1}
\abs{y}\leq \varepsilon \abs{k},
\end{equation}
then 
\begin{equation}\label{again_good_2}
\text{$k+y\in C'$\quad  if and only if\quad  $k\in C'$.}
\end{equation}

 


Let $Y_0$ be an algebraic complement of $Y'\cap W'$ in  $Y'$, where we have  identified  $W'$ with $\{0\}^{n'}\oplus W'$, so that $Y'=Y_0\oplus (Y'\cap W').$

We claim that the projection $P'\colon \R^{n'}\oplus W'\to \R^{n'}$, restricted to $Y_0$,  is injection. Indeed, assume that $x=(a, w)\in Y_0\subset \R^{n'}\oplus W'$ satisfies that  $P'(x)=0$. Then $a=0$ and $x=(0, w)$. Since $x=(0, w)\in Y'\cap W'$ and $Y_0\cap (Y'\cap W')=\{0\}$, we conclude that $x=0$ so  that indeed $P'\vert Y_0$ is an injection.

Introducing the subspace  $Z'=P'(Y_0)$, we claim that $\R^{n'}=\wt{K}+Z'$.  To verify the claim, we take $a\in \R^{n'}$ and let $x=(a, 0)$. Since $\R^{n'}\oplus W'=K'\oplus Y'$, there are unique elements $k\in K'$ and $y\in Y'$ such that $x=k+y$. In view of Lemma 3.44,
the element $k\in K'$ is of the form $k=(\alpha, D\delta (0)\alpha)$ for a unique $\alpha \in \wt{K}$. The element $y\in Y'$ is of the form $y=(b, w)\in \R^{n'}\oplus W'$. Hence 
$x=(a, 0)=k+y=(\alpha , D\delta (0)\alpha)+(b, w)=(\alpha+b, D\delta (0)\alpha+w)$, showing $a=\alpha+b$. On the other hand, since $Y'=Y_0\oplus (Y'\cap W')$,  we have  $(b, w)=(b, w_1)+(0, w_2)$ for $w_1, w_2\in W'$ such that $(b, w_1)\in Y'$ and $(0, w_2)\in Y'\cap W'.$ Hence $b=P'(b, w_1)\in Z'$ and, therefore, $a=\alpha +b\in \wt{K}+Z'$ as claimed.

From this it follows that  $\dim Z'\geq n'-\dim \wt{K}$ and we  choose a subspace $Z$ of $Z'$ of dimension  $\dim Z=n'-\dim \wt{K}$, so that $\R^{n'}=\wt{K}\oplus Z$.  Recalling that  the projection $P'\colon Y_0\to Z'$ is an isomorphism, we have $(P')^{-1}(Z)\subset Y'$.

We shall prove $\wt{K}$ is in good position to the partial quadrant $\wt{C}=[0,\infty )^k\oplus R^{n'-k}$ in $\R^{n'}$ and $Z$ is a good complement of $\wt{K}$, so that 
$\wt{K}\oplus Z=\R^{n'}$.  In view of  the fact that the map $(P')^{-1}\colon Z'\to Y_0$ is an isomorphism and $Z\subset Z'$, there exists a constant $A$ such that 
\begin{equation}\label{again_good_3}
\text{$\abs{(P')^{-1}(z)}\leq A\abs{z}$\quad  for all $z\in Z$.}
\end{equation}
Moreover, by Lemma \ref{new_lemma_relation}, the map $L\colon K'\to \wt{K}$, defined by $L(\alpha , D\delta (0)\alpha)=\alpha$,  is an isomorphism, and hence there exists a constant $B$ such that 
\begin{equation}\label{again_good_4}
\text{$\abs{\alpha}=\abs{L(\alpha ,D\delta  (0)\alpha)}\leq B\abs{(\alpha ,D\delta (0)\alpha)}$ \quad  for all $\alpha\in \wt{K}$.}
\end{equation}

We choose $\varepsilon'>0$ such that $\varepsilon' \cdot A\cdot B<\varepsilon$, where 
$\varepsilon>0$ is the constant from the condition \eqref{again_good_1}. 
We claim that if  $z\in Z$ and $\alpha\in \wt{K}$ satisfy 
\begin{equation}\label{again_good_5}
\abs{z}\leq \varepsilon' \abs{\alpha}, 
\end{equation} 
then 
\begin{equation}\label{again_good_6}
\text{$z+\alpha \in \wt{C}$\quad  if and only if \quad $\alpha \in \wt{C}.$}
\end{equation}
Since $z\in Z\subset Z'$, there exists a unique $w\in W'$ such that $(P')^{-1}(z)=(z, w)\in Y_0$. 
Then we estimate,  using \eqref{again_good_3}, \eqref{again_good_5}, and then \eqref{again_good_4}, 
$$\abs{(z, w)}=\abs{(P')^{-1}(z)}\leq A\abs{z}\leq \varepsilon' A\cdot B\abs{(\alpha ,D\delta  (0)\alpha)}<\varepsilon  \abs{(\alpha ,D\delta (0)\alpha)}.$$
Thus, $(z, w)\in Y_0\subset Y'$ and $(\alpha, D\delta (0)\alpha )\in K'$ satisfy the estimate \eqref{again_good_1} and it follows that 
\begin{equation}\label{again_good_7}
\text{$(z, w)+(\alpha, D\delta (0)\alpha )\in C$ \quad if and only if\quad  $(\alpha, D\delta (0)\alpha)\in C'.$}
\end{equation}
Since $(z, w)+(\alpha, D\delta (0)\alpha)=(z+\alpha ,w+D\delta (0)\alpha)\in \R^{n'}\oplus W'$, 
\eqref{again_good_7} implies that $z+\alpha\in \wt{C}$ if and only if $\alpha \in \wt{C}$, proving our claim \eqref{again_good_6} and we see that $Z$ is a good complement of $\wt{K}$ in $\R^{n'}$.

Next we set 
$$\ov{Y}=Z\oplus W'$$
and claim that $K'\oplus \ov{Y}=\R^{n'}\oplus W'$ and that $\ov{Y}$ is a good complement of $K'$ in $E'=\R^{n'}\oplus W'.$ We first verify that $\R^{n'}\oplus W'=K'+\ov{Y}$. Take $(a, w)\in \R^{n'}\oplus W'$.  Then, since 
$\R^{n'}\oplus W'=K'\oplus Y'$, there are unique elements $(\alpha, D\delta (0)\alpha)\in K'$, where $\alpha \in \wt{K}$,   and $(b, w_1)\in Y'$ such that 
\begin{equation}\label{again_good_8}
(a, w)=(\alpha , D\delta (0)\alpha)+(b, w_1)\in K'\oplus Y'.
\end{equation}
Hence $a=\alpha +b.$ Since $\R^{n'}=\wt{K}\oplus Z$, we may decompose $b$ as $b=\alpha_1+b_1$ with $\alpha_1\in \wt{K}$ and $b_1\in Z$  and then \eqref{again_good_8} can be written as 
\begin{equation*}
\begin{split}
(a, w)&=(\alpha , D\delta (0)\alpha)+(b, w_1)=(\alpha , D\delta (0)\alpha)+(\alpha_1+b_1, w_1)\\
&=(\alpha , D\delta (0)\alpha)+(\alpha_1, D\delta (0)\alpha_1)+(b_1, w_1-D\delta (0)\alpha_1)\\
&=(\alpha+\alpha_1, D\delta (0)(\alpha+\alpha_1))+(b_1, w_1-D\delta (0)\alpha_1)
\end{split}
\end{equation*}
where  $(\alpha+\alpha_1, D\delta (0)(\alpha+\alpha_1))\in K'$ (since $\alpha+\alpha_1\in \wt{K}$) and $(b_1, w-D\delta (0)\alpha_1)\in Z\oplus W'=\ov{Y}$ (since $b_1\in Z$). Hence $\R^{n'}\oplus W'=K'+\ov{Y}$ as claimed.

If $(\alpha ,D\delta (0)\alpha)\in K'\cap \ov{Y}=K'\cap (Z\oplus W')$, then $\alpha \in \wt{K}\cap Z$, so that $\alpha=0$ and $K'\cap \ov{Y}=\{0\}$ and we  have proved that  $Z\oplus \ov{Y}=\R^{n'}\oplus W'$.

Finally, we shall show that $\ov{Y}=Z\oplus W'$ is a good complement of $K'$ in $\R^{n'}\oplus W'$.  Recalling the isomorphism $L\colon K'\to \wt{K}$, $L(\alpha, D\delta (0)\alpha)=\alpha$, there exists a constant $M$ such that $\abs{L^{-1}(\alpha)}\leq M\abs{\alpha}$, i.e.,  
$$\text{$\abs{(\alpha ,D\delta (0)\alpha)}\leq M\abs{\alpha}$\quad  for all $\alpha\in \wt{K}$.}$$ 

We choose $\varepsilon''>0$ such that $\varepsilon'' M<\varepsilon'$ where $\varepsilon'$ is defined in \eqref{again_good_5}. We shall show that if 
$(\alpha, D\delta (0)\alpha)\in K'$ and $(a, w)\in Z$
satisfy the estimate
\begin{equation}\label{again_good_9}
\abs{(a, w)}\leq \varepsilon'' \abs{(\alpha, D\delta (0)\alpha)},
\end{equation}
then 
\begin{equation}\label{again_good_10}
\text{$(a, w)+(\alpha, D\delta (0)\alpha)\in C'$\quad if and only if $(\alpha, D\delta (0)\alpha)\in C'$.}
\end{equation}
Since $\abs{(a, w)}=\max\{\abs{a}_0, \abs{w}_0\}\geq \abs{a}$, we conclude from  \eqref{again_good_10} and 
\eqref{again_good_9} the estimate
$$\abs{a}\leq \abs{(a, w)}\leq \varepsilon'' \abs{(\alpha, D\delta (0)\alpha)}\leq \varepsilon'' M\abs{\alpha}.$$
In view of  \eqref{again_good_5} and 
\eqref{again_good_6}, 
$$\text{$a+\alpha \in \wt{C}$\quad if and only if $\alpha \in \wt{C}$}.$$
This is equivalent to 
$$\text{$(a+\alpha, w+D\delta (0)\alpha)\in C'$ \quad if and only if $(\alpha, D\delta (0)\alpha)\in C'.$}$$
Hence \eqref{again_good_10} holds and the proof that $\ov{Y}=Z\oplus W'$ is a good complement of $K'$ is complete. The proof of Lemma \ref{new_lemma_Z} is finished.

\end{proof}

\subsubsection{Proof of Lemma \ref{good_pos} }\label{geometric_preparation}
\begin{L3.55}\label{good_pos_1}

Let $C\subset E$ be  a partial quadrant in the  sc-Banach space $E$ and $N\subset E$ a finite-dimensional smooth subspace in good position to $C$ and let $Y$ be a good complement of $N$ in $E$, so that $E=N\oplus Y$.
We assume that $V\subset N\cap C$ is a relatively open neighborhood
of $0$ and $\tau\colon V\rightarrow Y_1$ a map of class $C^1$ satisfying $\tau(0)=0$ and $D\tau(0)=0$.

Then there exists a relatively open neighborhood $V'\subset V$ of $0$ such that the following holds.
\begin{itemize}
\item[{\em (1)}] $v+\tau(v)\in C_1$ for $v\in V'$.
\item[{\em (2)}] For every $v\in V'$,  the Banach space $Y=Y_0$ is a topological complement of the linear subspace $N_v=\{n+D\tau(v)n\, \vert \, n\in N\}$.
\item[{\em (3)}] For every $v\in V'$,  there exists a  constant $\gamma_v>0$ such  that  if  $n\in N_v$ and $y\in Y$ satisfy 
$\abs{y}_0\leq \gamma_v\cdot \abs{n}_0$,  then  $n\in C_x$ if and only if  $n+y\in C_x$ are equivalent, where $x=v+\tau (v)$. 
\end{itemize}

\end{L3.55}

\begin{proof}[{\bf Proof of Lemma \ref{good_pos}}]
We shall use the notations of Proposition 
\ref{big-pretzel}
Since  the smooth finite dimensional subspace $N$ of $E$ is in good position to $C$,  and $Y$ is its good complement so that  $E=N\oplus Y$,  there exists a constant $\gamma>0$ such that if $n\in N$ and $y\in Y$ satisfy
\begin{subequations}
\begin{gather}
\abs{y}_0\leq \gamma \abs{n}_0\label{eq_lemma3.55_1}\\
\shortintertext{then}
\text{$y+n\in C$\quad if and only if\quad  $n\in C$.} \label{eq_lemma3.55_2}
\end{gather}
\end{subequations}
\noindent(1)\,  The $C^1$-map $\tau\colon V\to Y_1$ satisfies $\tau (0)=0$ and $D\tau (0)=0$, so that  
$$\lim_{\abs{v}\to 0}\dfrac{\abs{\tau (v)}_0}{\abs{v}_0}=0.$$
Consequently, there exists a relatively open neighborhood $V'\subset V$ of $0$ in $N$ such that $\abs{\tau (v)}_0\leq \gamma \abs{v}_0$ for all $v\in V'$.  Since $V'\subset V\subset C$, we conclude from \eqref{eq_lemma3.55_2} that 
$v+\tau (v)\in C$ for all $v\in V'$.  In addition, since  $V'$ consists of smooth points and $\tau (v)\in Y_1$,  we have $v+\tau (v)\in C_1$ for all $v\in V'$.\\[0.5ex]
\noindent  (2)\,  By assumption,  $E=N\oplus Y$ and $D\tau (v)n\in Y$ for all $n\in N$. If   $x=n+y\in N\oplus Y\in E$, then $y=(n+D\tau (v)n)+(y-D\tau (v)n)$, showing that $E=N_v +Y$.  If $(n+D\tau (v)n)+y=(n'+D\tau (v)n')+y'$ for some $n, n'\in N$ and $y, y'\in Y$, then $n-n'=(y'-y)+D\tau (v)(n'-n)\in N\cap Y$. Hence $n=n'$ and $y=y'$, showing that $N_v\cap Y=\{0\}$ and hence $E=N_v\oplus Y$. That the subspace $Y$ is a topological complement of $N_v$ follows from the fact that $N_v$ is a finite dimensional subspace of $E$. \\[0.5ex]


\noindent (3)\,  
Without loss of generality we may assume $E=\R^n\oplus W$ and $C=[0,\infty)^n\oplus W$. 
We recall that for  $x\in C$, the set $C_x$ is defined by   
$$
C_x=\{(a, w)\in \R^n\oplus W\ \vert \ \text{$a_i\geq 0$ for all $1\leq i\leq n$ for which $x_i=0$}\},$$ 
and if $x_i>0$ for all $\leq i\leq n$, we set $C_x=E$. Clearly, if $x=(0, w)\in \R^n\oplus W$, then $C_x=C$.\\[0.5ex] 

If $v=0$, then $x=v+\tau (v)=0$, so that $C_x=C$ and the statement of the corollary follows from the fact that $N_v=N$ and $N$ is in good position to $C$ having the good complement $Y$.

In order to prove the statement for $v\neq 0$ we distinguish thew  two cases $\dim N=1$ and $\dim N>1$. In 
the first case, the proof of Lemma \ref{big-pretzel_1a} shows that $v_j+\tau_j (v)>0$ for all $1\leq j\leq n$.
This implies that $C_x=E$ and that, in view of (2),  $Y$ is a good complement of $N_v$ with respect to $C_x$.\\[0.5ex]

 In the second case $\dim N>1$, and we denote by $\wt{N}$ the  algebraic complement of $N\cap W$ in $N$,  where $W$ is identified with $W=\{0\}^n\oplus W$,  so that $N=\wt{N}\oplus (N\cap W)$.   We may assume,  after a linear change of coordinates, that the following holds. 
\begin{itemize}
\item[(a)] The linear subspace $\wt{N}$ is spanned by the vectors $e^j$ for $1\leq j\leq m=\dim \wt{N}$ of the form 
$e^j=(a^j, b^j, w^j)$ where $a^j$ are the vectors of the standard basis in $\R^m$, $b^j=(b^j_1, \ldots ,b^j_{n-m})$ satisfy $b^j_i>0$ for all $1\leq i\leq n-m$, and $w^j\in W$. 
\item[(b)] The good complement $Y$ of $N$  in $E$  is contained in $\{0\}^m\oplus \R^{n-m}\oplus W.$
\end{itemize}
Denoting by   $e^j=(0, w^j)\in \{0\}^n\cap W$, $m+1\leq j\leq k=\dim N$, a basis of  $N\cap W$, the vectors $e^1,\ldots ,e^k$  form a basis of $N$.   
We choose $0<\varepsilon<\gamma/2$ where $\gamma$ is the constant from \eqref{eq_lemma3.55_1} and define the open neighborhood $V''$ of $0$ in $N\cap C$ consisting  of points $v\in V'$ satisfying $\abs{D\tau (v)n}_0<\varepsilon\abs{n}_0$ for all $n\in N$.

We consider points  $v\in V'$ belonging to $N\cap W$ where we identify $W$ with $\{0\}^n\oplus W$.  Hence  $v$ is  
of the form $v=(0, w)\in \{0\}^n\oplus W$.  Then $\abs{\tau(v)}_0\leq \gamma \abs{v}_0$ and,  since $v\in C$,  we conclude that $v+\tau (v)$ and $v-\tau (v)$ belong to the partial quadrant $C$. This implies that $\tau_i(v)=0$ for all $1\leq i\leq n$ and  $C_{x}=C$ where $x=v+\tau (v)=(0, w')$. We choose a positive constant $\gamma_v$ such that $\gamma_v(1+\varepsilon)+\varepsilon <\gamma$.  Abbreviating $h=n+D\tau (v)n\in N_v$ for some $n\in N$ and taking $y\in Y$, we assume that 
\begin{equation}\label{eq_lemma3.55_3}
\abs{y}_0\leq \gamma_0\abs{h}_0=\gamma_v \abs{n+D\tau (v)n}_0.
\end{equation}
 
 We claim that $y+h\in C_x$ if and only if $h\in C_x$. Suppose that $y+h=y+n+D\tau (v)n\in C_x=C$. Then we estimate using \eqref{eq_lemma3.55_3}, 
\begin{equation}\label{eeqq1}
\begin{split}
\abs{y+D\tau (v)n}_0&\leq \abs{y}_0+\abs{D\tau (v)n}_0\leq \gamma_v\abs{h+D\tau (v)n}+\abs{D\tau (v)n}_0\\
&\leq (\gamma_v+\varepsilon\gamma_v+\varepsilon)\abs{n}_0\leq \gamma \abs{n}_0.
\end{split}
\end{equation}

Since $y+D\tau (v)n\in Y$ and by assumption $y+h=y+D\tau (v)n+n\in C_x=C$, we deduce from \eqref{eeqq1}  and 
\eqref{eq_lemma3.55_2} that $n\in C$. 
Then the estimate $\abs{D\tau (v)n}_0\leq \varepsilon \abs{n}_0<\gamma\cdot \abs{n}_0$ and again \eqref{eq_lemma3.55_2}, imply that $h=n+D\tau (v)n\in C$, as claimed.\\
Conversely, the  assumption  that $h=n+D\tau (v)n\in C$ implies, in view of  $\abs{D\tau (v)h}_0\leq \abs{h}_0$ and \eqref{eq_lemma3.55_2}, that $n\in C$. Then using  \eqref{eeqq1} we find hat $y+D\tau (v)n +h=y+h\in C$. We have proved that if $h\in N_v$ and $y\in Y$ satisfy $\abs{y}_0\leq \gamma_v\abs{h}_0$, then $y+h\in C_x$ if and only if $h\in C_x$.\\[0.5ex]

Finally we consider points $v=(v_1,\ldots ,v_n, w)\in N\cap C\setminus N\cap W$. Then $v_i>0$ for some $1\leq i\leq m$ and $v_j>0$ for all $m+1\leq j\leq k$. Moreover, in view of the proof of the statement (1), 
$\abs{\tau (u)}_0\leq \gamma\abs{v}_0$ which implies that $v+\tau (v)\in C$ if $v$ is close to $0$.
It follows from the proof of Lemma \ref{big-pretzel_1a} that  $v_j+\tau_j (v)>0$ for all $m+1\leq j\leq n$. Abbreviating  $x=v+\tau (v)$ and denoting by $\Lambda_x$ the set of indices $1\leq j\leq n$ for which $x_j>0$,  we conclude 
$$C_x=\{(a, w)\in \R^n\oplus W\, \vert \, \text{$a_j\geq 0$ for all $1\leq j\leq m$ satisfying $j\not \in \Lambda_x$}\}.$$
Recall that $N_v=\{n+D\tau (v)n\, \vert \, n\in N\}$ and $E=N_v\oplus Y$ in view of the statement (2).  Now the inclusion $Y\subset  \{0\}^m\oplus \R^{n-m}\oplus W$ and the definition of $C_x$ above show that, 
if $h=n+D\tau (v)n\in N_v$ and  $y\in Y$, then $y+h\in C_x$ if and only if $h\in C_x$.  This completes the proof of Lemma \ref{good_pos}.

\end{proof}

\subsubsection{Diffeomorphisms Between Partial 
Quadrants in $\R^n$.}\label{finite_partial_inverse}

In the sections \ref{finite_partial_inverse} and \ref{implicit_finite_partial_quadrants} 
we shall prove an inverse function theorem and an implicit function theorem for smooth maps defined on partial quadrants in $\R^n$. We recall that the closed subset $C\subset \R^n$ is a partial quadrant, if there exists an isomorphism $L$ of $\R^n$ mapping $C$ onto $L(C)=[0,\infty)^k\oplus \R^{n-k}$ for some $k$.

We begin  with a definition.
\begin{definition}[{\bf Class $C^1$}] Let $C$ be a partial quadrant in $\R^n$ and $f\colon U\rightarrow \R^m$ a map defined on a relatively open subset $U\subset C$. The map  $f$ is said to be of {\bf class  $C^1$}  if,  for every $x\in U$,  there exists a bounded linear map $Df(x)\colon \R^n\rightarrow \R^m$
such that
$$
\lim_{x+h\in U,h\to  0} \dfrac{ \abs{f(x+h)-f(x)-Df(x)h}}{\abs{h}}=0,
$$
and,  in addition,  the map 
$$Df\colon U\to {\mathscr L}(\R^n, \R^m), \quad x\mapsto Df(x),$$
is continuous.
\end{definition}
We note that the tangent 
$TU=U\oplus {\mathbb R}^n$ is a relatively open subset of the partial quadrant $C\oplus \R^n$ and the map 
$Tf\colon TU\rightarrow T\R^m$,  defined by 
$$Tf(x,h)=\bigl(f(x), Df(x)h ),$$
is continuous. If the maps $f$ and $Tf$ are of class $C^1$, then  we say that $f$ is of class $C^2$. Inductively, the map $f$ is $C^k$ if the maps $f$ and $Tf$ are of class $C^{k-1}$.  Finally, $f$  is called smooth 
(or $C^\infty$) if  $f$ is of class $C^k$ for all $k\geq 1$.

We consider the following situation.  Assume  $C$ is  a partial quadrant in $\R^n$ and $U\subset C$  a relatively open neighborhood of $0$ in $C$. Let $f\colon U\to C$ be  a smooth map such that $f(0)=0$ and $Df(0)\colon \R^n\to \R^n$ an isomorphism. 
Even if $Df(0)C=C$ one cannot guarantee that the image $f(O)$  of an open neighborhood $O$ of $0$ in $C$  is an open neighborhood of $0$ in $C$ as Figure \ref{fig:pict6}  illustrates.

\begin{figure}[htb]
\begin{centering}
\def\svgwidth{69ex}
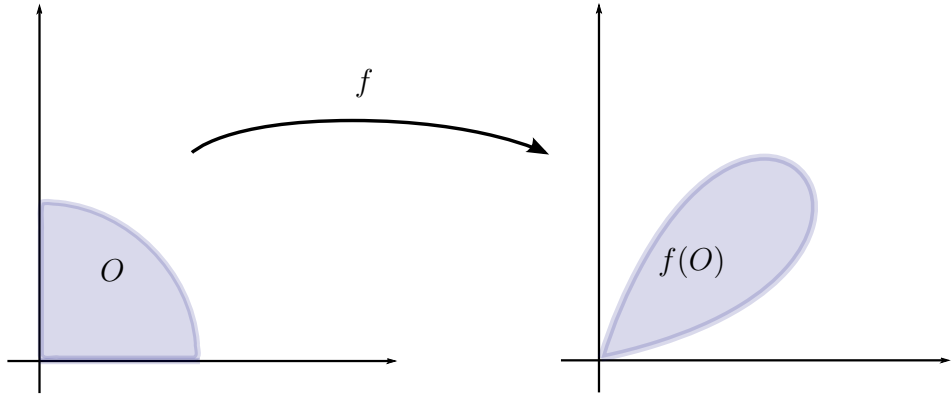
\caption{The boundary tangents at $(0,0)$ are the $x-$ and $y$-axis, but the domain in the right picture is not an open neighborhood of $(0,0)$.}
\label{fig:pict6}
\end{centering}
\end{figure}


Our aim in this section is to derive an inverse and implicit function theorem in this context.

\begin{theorem}[{\bf Quadrant Inverse Function Theorem}]\label{QIFT}
We assume that $C$ is a partial quadrant in ${\mathbb R}^n$ and $U\subset C$ a relatively open neighborhood of $0$, and consider  a smooth map  $f\colon U\rightarrow \R^n$ satisfying $f(U)\subset C$ and having the following properties:
\begin{itemize}
\item[{\em (1)}]  $f(0)=0$ and $Df(0)C=C$.
\item[{\em (2)}]  $d_C(x)=d_C(f(x))$ for every $x\in U$.
\end{itemize}
Then there exist two relatively open neighborhoods $U'$ and $V'$ of $0$ in $C$  such that $U'\subset U$ and the map 
$$f\colon U'\rightarrow V'$$ is a diffeomorphism.
\end{theorem}

From $\R^n=C-C$ and $Df(0)C=C$,  it follows that $Df(0)\colon \R^n\to \R^n$ is an isomorphism.  Therefore, it suffices to study the map $g\colon U\to C$, defined by 
$$g(x)=Df(0)^{-1}f(x).$$
The map $g$ has  the same properties (1) and (2), but has the simplifying feature that 
$Dg(0)={\mathbbm1}$. 
Theorem \ref{QIFT} is a consequence of the following proposition.

\begin{proposition}\label{hongkong}
We assume that $C$ is a partial quadrant in ${\mathbb R}^n$ and $U\subset C$ a relatively open convex neighborhood of $0$, and let $f\colon U\rightarrow C$ be a smooth map having the following properties:
\begin{itemize}
\item[{\em (1)}] $f(0)=0$.
\item[{\em (2)}] $\abs{Df(x)-{\mathbbm 1}}\leq 1/2$ for all $x\in U$.
\item[{\em (3)}]  $d_C(x)=d_C(f(x))$ for every $x\in U$.
\end{itemize}
Then $V:=f(U)$ is a relatively open neighborhood of $0$ in $C$,  and the map $f\colon U\rightarrow V$ is a diffeomorphism, that is, $f$ and $f^{-1}$ are smooth
in the sense defined above.
\end{proposition}

The proof of Proposition \ref{hongkong} follows from three lemmata proved below.

\begin{lemma}\label{new_lemma_3.73}
With  the assumptions  of Proposition \ref{hongkong} , 
the image  $V=f(U)$ is a relatively open subset of  $C$ and $f\colon U\rightarrow V$ is an open map. Moreover, 
\begin{equation}\label{new_equ_56}
\dfrac{1}{2}\abs{x-x'}\leq \abs{f(x)-f(x')}\leq \dfrac{3}{2} \abs{x-x'}, \quad \text{$x, x'\in U$}.
\end{equation}
\end{lemma}

\begin{proof}

Abbreviating $g(x)=x-f(x)$ and $G(x)= {\mathbbm 1}-Df(x)$, and using  the  convexity of the set $U$, we obtain the identity   
$$g(x')=g(x)+\biggl( \int_0^1G\bigl(\tau x'+(1-\tau )x\bigr)\ d\tau\biggr)\cdot (x'-x)$$
for $x, x'\in U$, from which the desired estimate 
\eqref{new_equ_56} follows, in view of  property (2).
We conclude, in particular,  that the map  $f\colon U\to V$ is injective.

In order to show that the image $V:=f(U)$ is a relatively open subset of $C$, we take a point $y'\in V$. Then $f(x')=y'$ for a unique $x'\in U$ and since 
$U$ is a relatively open subset of $C$, there exists $r>0$ such that $B(x', 3r)\cap C\subset U$. Here  we denoted by  $B(x', 3r)$  the open ball in $\R^n$ centered at $x'$ and having radius $3r$. We claim that $B(y', r)\cap C\subset V$.  In order to prove our claim, we first show that $\{y\in B(x', r)\cap C\, \vert \, d_C(y)=0\}\subset V$. To see this, 
we  take  a point $x_0\in B(x', 3r)\cap C$ such that $d_C(x_0)=0$ and let $y_0=f(x_0)$. By property (3), $d_C(y_0)=0$. Next  we take an arbitrary point $y_1\in B(y', r)\cap C$,  also satisfying $d_C(y_1)=0$,  and consider the points $y_\tau=(1-\tau )y_0+\tau y_1$ for $0\leq \tau \leq 1$.  Abbreviating  $\Sigma=\{\tau \in [0,1]\, \vert \, y_\tau\in V\}$ and $\tau^*=\sup \Sigma$, we assume that $\tau^*<1$.   
We note that $\Sigma$ is non-empty since $0\in \Sigma$ and $d_C(y_{\tau^*})=0$. Then we  choose a sequence $(\tau_n)\subset \Sigma$ such that $\tau_n\to \tau'$ and the  corresponding sequence of points $(y_{\tau_n})$  belonging  to $B(y', r)\cap C$ satisfying $y_{\tau_n}\to y_{\tau^*}$. Since $y_{\tau_n}\in V$, we find points $x_n\in U$ such that $f(x_n)=y_{\tau_n}$ and $d_C(x_n)=d_C(y_{\tau_n})=0$.
By \eqref{new_equ_56},  
$$\abs{x_n-x'}\leq 2\abs{f(x_n)-f(x')}=2\abs{y_{\tau_n}-y'}<2r$$
and 
$$\abs{x_n-x_m}\leq 2\abs{y_{\tau_n}-y_{\tau_m}},$$
which show that $(x_n)$ is a Cauchy sequence belonging to $B(x', 2r).$
Hence $(x_n)$ converges to some point $x^*$  belonging to $\ov{B}(x', 2r).$ By continuity of $f$,  $f(x^*)=y_{\tau^*}$, and the point $x^*$ belongs to the interior of $C$ since, by property (3),  $d_C(x^*)=d_C(y_{\tau^*})=0$.  Now  the classical  inverse function theorem implies that there are  open neighborhoods $W$ and $W'$ of $x^*$ and $y_{\tau^*}=f(x^*)$ both contained in the interior of $C$ such that the map $f\colon W\to W'$ has a continuous inverse $f^{-1}\colon W'\to W$. In particular, if  $\tau>0$ is small, then $y_{\tau^*+\tau}\in V$,  contradicting $\tau^*<1$. Summing up, we have proved our claim that 
$$\{y\in B(y', r)\cap C\, \vert \, d_C(y)=0\}\subset V.$$
Next we take any point $y\in B(y', r)\cap C$ satisfying $d_C(y)\geq 1$. We find sequences $(y_n)\subset B(y', r)\cap C$ and 
$(x_n)=(f^{-1}(y_n))$ satisfying  $d_C(y_n)=d_C(x_n)=0$ and $\abs{y_n-y}\to 0$. 
From  \eqref{new_equ_56}, we  have 
$\abs{x_n-x'}\leq 2\abs{y_n-y'}<2r$ and $\abs{x_n-x_m}\leq 2\abs{y_n-y_m}$
from which we deduce the convergence of the sequence $(x_n)$  to some point $x\in B(x', 3 r)\cap C.$  Hence  $f(x)=y$, implying that $y\in V$. This shows that  $ B(y', r)\cap C\subset V$ and that $V$ is a relatively open subset of $C$.

Finally, to see that $f\colon U\to V$  maps relatively open subsets onto relatively open subsets,  we take an open subset $U'$ of $U$. Since $U'$ can be written as a union of relatively open convex subsets of $C$,   employing the previous arguments we conclude that $V'=f(U')$ is an open subset of $V$. The proof of
 Lemma \ref{new_lemma_3.73} is complete.

\end{proof}

From the above lemma we obtain immediately the following corollary.
 \begin{corollary}\label{f_homeomrphism}
With the assumptions  of Proposition \ref{hongkong},  the map $f\colon U\to V$ is a smooth homeomorphism between the relatively open subsets $U$ and $V$ of $C$.
 \end{corollary}

We would like to point out that we cannot employ the usual implicit function theorem (it only applies at interior points). Therefore,  we must provide an argument  for the smoothness of the inverse map  $f^{-1}$. One could try to avoid work,  by first showing that $f$ can be extended to a smooth map defined on an open neighborhood
 of $0$ in ${\mathbb R}^n$. However,  our notion of smoothness does not stipulate that $f$ is near every point the restriction
 of a smooth map defined on an open subset of ${\mathbb R}^n$. This would complicate 
 the construction of a smooth extension.

 \begin{lemma}
With  the assumptions of Proposition \ref{hongkong},  the iterated tangent map $T^kf:T^kU\rightarrow T^kV$ is 
 a smooth homeomorphism.
 \end{lemma}
 
 \begin{proof}

The set $T^kC$ is a partial quadrant in $T^k\R^n$ and the sets $T^kU$ and $T^kV$ are relatively open subsets of $T^kC$. 
Then the iterated tangent map $T^kf\colon T^kU\rightarrow T^kV$ is smooth since the map $f\colon U\to V$ is smooth.  Hence we only have to show that $T^kf$ has an inverse 
$(T^kf)^{-1}\colon T^kV\to T^kU$ which is continuous. 
 In order to prove this we proceed by induction starting with $k=1$.  
 We introduce the map $\Phi_1\colon TV\to TU$, defined by 
 $$\Phi_1(y,l)=\bigl( f^{-1}(y),[ Df(f^{-1}(y)) ]^{-1}l).$$
 The map $\Phi_1$ is continuous since,  by Corollary \ref{f_homeomrphism},  the map $f\colon U\to V$ is a homeomorphism.
Moreover, 
 $$
 \Phi_1\circ Tf (x,h) = \bigl(f^{-1}(f(x)),[Df(f^{-1}(f(x))) ]^{-1}\circ Df(x)h\bigr)=(x,h), 
 $$
and similarly $Tf\circ \Phi_1=\mathbbm{1}$.  Hence $\Phi_1$ is an inverse of the tangent map $Tf$.  This  together with the continuity of $Tf$ and $\Phi_1$ show that $Tf$ is a homeomorphism, as claimed.

Now we assume that result holds for $k\geq 1$ and we show that $T^{k+1}f\colon T^{k+1}U\to T^{k+1}V$ is a homeomorphism.  Recalling  that $T^{k+1}U=T^kU\oplus T^k\R^n$, we define the map 
\begin{gather*}
\Phi_{k+1}\colon T^kV\oplus T^k\R^n \to T^kU\oplus T^k\R^n, \\
\Phi_{k+1}(y, l)=\bigl((T^kf)^{-1}(y), [D(T^kf)((T^kf)^{-1}(y))]^{-1}l\bigr).
\end{gather*}
By the inductive assumption, the map $T^kf\colon T^kV\to T^kU$ is a homeomorphism and hence $\Phi_{k+1}$ is well-defined and continuous. 
Moreover, 
\begin{equation*}
\begin{split}
\Phi_{k+1}&\circ T^{k+1}f (x, h)=\Phi_{k+1}\bigl(T^kf(x), (D(T^kf))(x)h\bigr)\\
&=\bigl((T^kf)^{-1}\bigl(T^kf(x)\bigl),  [D(T^kf)((T^kf)^{-1}(T^kf(x)   ))]^{-1}(D(T^kf))(x)h\bigr)\\
&=\bigl(x,  [D(T^kf)(x)]^{-1}(D(T^kf))(x)h\bigr)=\bigl(x, h),
\end{split}
\end{equation*} 
and similarly $ T^{k+1}f\circ \Phi_{k+1}=\mathbbm{1}$. Hence the map $\Phi_{k+1}$ is a continuous inverse of the iterated tangent map $T^{k+1}f$. This completes the inductive step and the proof of the lemma.

\end{proof}

\begin{lemma}\label{addd}
Let  $U$ be  a relatively open subset of a partial quadrant $C\subset {\mathbb R}^n$ and $f\colon U\rightarrow C$ a smooth map satisfying the assumptions of Theorem \ref{hongkong}.  We assume that the image $V=f(U)$ is a relatively open subset of $C$ and $f\colon U\rightarrow V$ a homeomorphism such that $Df(x)$ is invertible at every point $x\in U$. 
Then the inverse map $g=f^{-1}$ is of class $C^1$
 and its derivative at the point $y\in V$ is given by
 $$
 Dg(y) =[Df(g(y))]^{-1}.
 $$
 \end{lemma}
 
\begin{proof}

By assumption,  the map $f\colon U\to V$ is a homeomorphism. This implies that $g\colon V\to U$ is also a homeomorphism and the map 
$V\to {\mathscr L}(\R^n)$, defined by 
$$y\mapsto [Df(g(y))]^{-1},$$
is continuous.  Next we take  $y_0\in V$ and $h\in \R^n$ satisfying  $y_0+h\in V$. 
Then there are unique points  $x_0$ and $x_0+\delta\in U$ such that $f(x_0)=y_0$ and $f(x_0+\delta )=y_0+h.$
Recalling the abbreviation $g=f^{-1}$ we note that 
$g(y_0)=x_0$, $g(y_0+h)-g(y_0)=\delta$ and $f(x_0+\delta)-f(x_0)=h$ and compute,
\begin{equation*}
\begin{split}
&\dfrac{1}{\abs{h}}
\abs{g(y_0+h)-g(y_0)- Df(x_0)^{-1}h}\\
&\quad= \dfrac{1}{\abs{h}}\abs{\delta - Df(x_0)^{-1}(f(x_0+\delta )-f(x_0))}\\
&\quad =\dfrac{\abs{\delta}}{\abs{h}}\cdot 
\dfrac{1}{\abs{\delta}}\cdot  \abs{Df(x_0)^{-1}\bigl[ f(x_0+\delta)-f(x_0)-Df(x_0)\delta\bigr] }\\
\end{split}
\end{equation*}

From \eqref{new_equ_56} in Lemma \ref{new_lemma_3.73} we obtain 
$\frac{1}{2}\abs{\delta}\leq \abs{h}\leq \frac{3}{2}\abs{\delta}$ so that $\abs{h}\to 0$ if and only if $\abs{\delta}\to 0$. 
The limits vanish since $f$ is differentiable at $x_0$. The proof of  the lemma is finished.

\end{proof}
 
Now we are in the position to prove Proposition \ref{hongkong}.

\begin{proof}[Proof of Proposition \ref{hongkong}]
Under the hypotheses of the proposition the previous discussion shows for every $k$,  
that $T^kf:T^kU\rightarrow T^kV$ satisfies the hypotheses of Lemma \ref{addd} for a suitable choice 
of data, i.e.  taking $T^kC$ as partial quadrant in $T^k{\mathbb R}^n$. Hence we conclude that $(T^kf)^{-1}$ is $C^1$,
which precisely means that $T^kg$ is $C^1$.  In  other words,  $g$ is of class $C^{1+k}$. Since $k$ is arbitrary we conclude that $g$ is $C^\infty$.
\end{proof}

\subsubsection{An Implicit Function Theorem in Partial Quadrants}\label{implicit_finite_partial_quadrants}

We shall prove a version for the classical implicit function theorem for  maps defined on  partial quadrants. 

\begin{theorem}\label{help-you}

We assume that $U$ is relatively open neighborhood of $0$ in the  partial quadrant $C=[0,\infty)^k\oplus {\mathbb R}^{n-k}$ in $\R^n$, and consider a map 
$f\colon U\rightarrow {\mathbb R}^N$ of class $C^j$, $j\geq 1$,  satisfying $f(0)=0$. Moreover, we  assume that  
$Df(0)\colon \R^n\to \R^N $ is surjective  and the kernel $K:=\ker Df(0)$ is in  good position to $C$, and let  
$Y$ be a good  complement of $K$ in $\R^n$, so that $\R^n=K\oplus Y$. 

Then  there exist an open neighborhood $V$ of $0$ in the partial quadrant $K\cap C$, and a map  
$\tau\colon V\rightarrow Y$ of class $C^j$,  and  positive constants $\varepsilon, \sigma$ having the following properties.
\begin{itemize}
\item[{\em (1)}] $\tau(0)=0$, $D\tau(0)=0$, and $k+\tau(k)\in U$ if $k\in V$.
\item[{\em (2)}] If $k\in V$  satisfies  $\abs{k}\leq \varepsilon$, then $\abs{\tau(k)}\leq \sigma$.
\item[{\em (3)}] If $f(k+y)=0$ for $k\in V$ satisfying 
$\abs{k}\leq \varepsilon$  and $y\in Y$ 
satisfying $\abs{y}\leq\sigma$, then $y=\tau (k)$.
\end{itemize}

\end{theorem}

The proof will follow from several lemmata, where we shall us the notations 
$$x=(k, y)=k+y\in K\oplus Y$$
for $k\in K$ and $y\in Y$  interchangeably.

The restriction $Df(0)\vert_Y\colon Y\rightarrow {\mathbb R}^N$ is an isomorphism and we abbreviate its inverse by   $L=\bigl[ Df(0)\vert_Y\bigr]^{-1}\colon Y\to \R^N$. Instead of studying the  solutions $f(y)=0$ we can as well study  the solutions of $\wt{f}(x)=0$, where $\wt{f}$ is the composition $\wt{f}=L\circ f$.It satisfies $D\wt{f}(0)k=0$ if $k\in K$ and 
$D\wt{f}(0)y=y$ if $y\in Y$ .  In abuse of notation we shall in the following  denote  the composition  $\wt{f}=L\circ f$ by the old letter $f$ again. 
We  may therefore assume that  
$$
f\colon U\cap (K\oplus Y)\rightarrow Y
$$
has the form 
\begin{equation}\label{new_equ_57}
f(k, y)=y-B(k, y),
\end{equation}
where 
\begin{equation}\label{oipu}
B(0,0)=0\quad  \text{and}\quad DB(0,0)=0.
\end{equation}

By $B_K(a)$ , we denote the open ball in $K$ of radius $a$ centered at $0$. Similarly, $B_Y(b)$ is  the open ball in $Y$ of radius $b$ centered  at $0$.

\begin{lemma}\label{LEMMA1}
There exist constants $a>0$ and $b>0$ such  that 
\begin{itemize}
\item[{\em (1)}] $\bigl( B_K(a)\oplus B_Y(b)\bigr)\cap C\subset U$.
\item[{\em (2)}] $\abs{B(k,y)-B(k,y')}\leq \dfrac{1}{2}\abs{y-y'}$ for 
all $k\in B_K(a)\cap C$ and $y,y'\in B_Y(b)\cap C$. 
\end{itemize}
Moreover, if $f(k,y)=f(k,y')$ for  $(k, y)$ and $(k,y')\in \bigl( B_K(a)\oplus B_Y(b)\bigr)\cap C$, then $y=y'$. 
\end{lemma}

\begin{proof}

It is clear that there exists constants $a, b>0$ such that 
(1) holds. 
To verify (2) we estimate for 
$(k, y)$ and $(k, y')\in \bigl( B_K(a)\oplus B_Y(b)\bigr)\cap C$, 
\begin{equation}\label{est_contraction_B}
\abs{B(k, y)-B(k, y')}\leq\biggl[ \int_0^1\abs{D_2B(k, sy+(1-s)y')}ds\biggr](y-y').
\end{equation}

In view of $DB(0, 0)=0$ we find smaller $a, b$ such that $ \abs{D_2B(k, y'')}\leq 1/2$ for all $(k, y'')\in \bigl( B_K(a)\oplus B_Y(b)\bigr)\cap C$.  So, 
\begin{equation*}
\abs{B(k, y)-B(k, y')}\leq \dfrac{1}{2}\abs{y-y'}
\end{equation*}
as claimed in (2).

If $f(k,y)=f(k,y')$, for two point $(k,y), (k,y')\in \bigl( B_K(a)\oplus B_Y(b)\bigr)\cap C$, 
then, $y-B(k, y)=y'-B(k, y')$ and using (2),  
$$\abs{y-y'}=\abs{B(k,y)-B(k,y')}\leq \dfrac{1}{2}\abs{y-y'}$$
which implies  $y=y'$. This completes the proof of Lemma \ref{LEMMA1}.

\end{proof}

To continue with the proof of the theorem, we recall that $K$ is in  good position to $C$ and $Y$ is a good complement of $K$ in $\R^n$. Therefore,  there exists $\varepsilon>0$ such that for $k\in K$ and $y\in Y$ satisfying 
$\abs{y}\leq \varepsilon\abs{k}$  the statements
$k+y\in C$ and $k\in C$ are equivalent.

\begin{lemma}\label{LEMMA2}
Replacing $a$ by a smaller number, while keeping $b$,  we have
\begin{equation*}
\abs{B(k,y)}\leq \varepsilon \abs{k}\quad \text{ for all $k\in B_K(a)\cap C $ and $\abs{y}\leq \varepsilon \abs{k}$.}
\end{equation*}
\end{lemma}
\begin{proof}
First we replace $a$ by a  perhaps smaller number such that $\varepsilon a<b$. If $k\in B_K(a)\cap C $ and $\abs{y}\leq \varepsilon \abs{k}$, then $k+y\in C$  and $\abs{k+y}\leq (1+\varepsilon )\abs{k}<(1+\varepsilon )a$. Introducing 
$$c(a):= (1+\varepsilon)\cdot \max_{x\in B(1+\varepsilon )a)} \abs{DB(x)},$$
 we  observe that $c(a)\to 0$ as $a\to 0$,  in view of $DB(0, 0)=0$. Then we estimate for $k\in B_K(a)\cap C $  and $y\in Y$ satisfying $\abs{y}\leq \varepsilon \abs{k}$, 
 using that $B(0, 0)=0$, 
\begin{equation}\label{new_equ_70}
\begin{split}
\abs{B(k,y)}& \leq \biggl[ \int_0^1 \abs{DB(sk,sy)}\ ds\biggr] (k+y)\\
&\leq \max_{x\in B(1+\varepsilon )a)} \abs{DB(x)}(1+\varepsilon )\abs{k}=c(a)\abs{k}.
\end{split}
\end{equation}
The assertion follows if we choose $a$ so small that 
$c(a)\leq \varepsilon$.
\end{proof}

Now we take $a>0$ so small that $c(a)\leq \varepsilon/2$ and keep the original $b>0$.
For every $k\in B_K(a)\cap C$, we set 
$$X_k:=\ov{B}_Y(\varepsilon \abs{k})\quad\text{and}\quad X=\bigcup_{k\in B_K(a)\cap C}X_k. $$
Clearly, $X\subset U$. By  Lemma \ref{LEMMA2},  
$B(k,\cdot )\colon X_k\to X_k$, and by 
 Lemma \ref{LEMMA1}, the map $B(k, \cdot )$ is a contraction. Therefore, it  has a unique fixed point $\tau (k)\in X_k$ satisfying 
$$\tau (k)=B(k, \tau (k)).$$
We claim that if   
$y\in B_Y(b)\cap C$ and $B(k, y)=y$ for some $k\in B_K(a)\cap C$, then 
$y=\tau (k)$. 

Indeed, since $B(k, \tau (k))=\tau (k)$, it follows that 
$f(k, y)=f(k,\tau (k))$,  and Lemma \ref{LEMMA1} shows that 
$y=\tau (k)$, as claimed.

\begin{lemma}
The map $\tau\colon B_K(a)\cap C \to B_Y(b)$  is continuous
and satisfies $\abs{\tau(k)}\leq \varepsilon  \abs{k}/2$.
\end{lemma}
\begin{proof}
In view of \eqref{new_equ_70}, the estimate 
$\abs{\tau (k)}\leq \varepsilon \abs{k}/2$ follows from 
$\abs{\tau (k)}=\abs{B(k, \tau (k)}\leq c(a)\abs{k}$ and from the choice $c(a)\leq \varepsilon/2$.
 In order to prove continuity of $\tau$, we fix a point $k_0\in  B_K(a)\cap C$ and  use  Lemma \ref{LEMMA1}
 to estimate, 
\begin{equation*}
\begin{split}
\abs{\tau (k)-\tau (k_0)}&=\abs{B(k, \tau (k))-B(k_0, \tau (k_0))}\\
&\leq \abs{B(k, \tau (k))-B(k, \tau (k_0))}+\abs{B(k, \tau (k_0))-B(k_0, \tau (k_0))}\\
&\leq \dfrac{1}{2}\abs{\tau (k)-\tau (k_0)} +\abs{B(k, \tau (k_0))-B(k_0, \tau (k_0))}.
\end{split}
\end{equation*}
This implies, 
\begin{equation}\label{continuity_tau}
\begin{split}
\abs{\tau (k)-\tau (k_0)}&\leq 2\abs{B(k, \tau (k_0))-B(k_0, \tau (k_0))}\\
&\leq 2\biggl[ \int_0^1\abs{D_1B(sk+(1-s)k_0, \tau (k_0))}\ ds\biggr] \abs{k-k_0}\\
&\leq \varepsilon \abs{k-k_0},
\end{split}
\end{equation}
where we have used our choice of $a$ that $c(a)= (1+\varepsilon)\cdot \max_{x\in B(1+\varepsilon )a)} \abs{DB(x)}\leq \varepsilon/2$. This finishes the proof of the continuity of the map $\tau$.
\end{proof}


\begin{lemma}
If $f\colon U\to \R^N$ is of class $C^j$, $j\geq 1$, then the map $\tau\colon B_K(a)\cap C \to B_Y(b)$ is also of class $C^j$.
\end{lemma}

\begin{proof}

We fix the  point $k\in B_K(a)\cap C$ and assume that  
$k+\delta k\in B_K(a)\cap C.$
Then we abbreviate 
\begin{align*}
A(\delta k)&:=\int_0^1\bigl(D_2B(k+\delta k, s\tau (k+\delta k)+(1-s)\tau (k))-D_2B(k, \tau (k))\bigr)\ ds\\
R(\delta k)&:=B(k+\delta k, \tau (k))-
B(k, \tau (k))-D_1B(k, \tau (k))\delta k.
\end{align*}
Using the fixed point property $\tau (k)=B(k,\tau (k))$ and 
$\tau (k+\delta k)=B(k+\delta k,\tau (k+\delta k))$ we compute,
\begin{equation*}
\begin{split}
&\tau (k+\delta k)-\tau (k)-D_1B(k,\tau (k))\delta k\\
&\quad =B(k+\delta k, \tau (k+\delta k))-B(k, \tau (k))-D_1B(k,\tau (k))\delta k\\
&\quad=\bigl[ B(k+\delta k, \tau (k))-B(k, \tau (k))-D_1B(k_0,\tau (k_0))\bigr] \\
&\quad\phantom{=}+
\bigl[ B(k+\delta k, \tau (k+\delta k)) -B(k+\delta k, \tau (k))\bigr] \\
&\quad=R(\delta k)+A(\delta k)\bigl[ \tau (k+\delta k)-\tau (k)\bigr]+D_2B(k, \tau (k))\bigl[ \tau (k+\delta k)-\tau (k)\bigr].
\end{split}
\end{equation*}
Therefore,
\begin{equation*}
\begin{split}
 \bigl[{\mathbbm 1}-D_2B(k, \tau (k))\bigr]&(\tau (k+\delta k)-\tau (k))-D_1B(k, \tau (k))\delta k\\
&=R(\delta k)+A(\delta k)(\tau (k+\delta k)-\tau (k)).
\end{split}
\end{equation*}
Now, $R(\delta k)/\abs{\delta k}\to 0$ and $A(\delta k)\to 0$ as $\abs{\delta k}\to 0$. Moreover, in view of  
\eqref{continuity_tau}, 
$\abs{\tau(k+\delta k)-\tau (k)}/\abs{\delta k}\leq \varepsilon \abs{\delta k}/\abs{\delta k}\leq \varepsilon$. 
Consequently,
$$
\lim_{\abs{\delta k}\to 0}\dfrac{1}{\abs{\delta k}}\abs{\tau (k+\delta k)-\tau (k)- \bigl[{\mathbbm 1}-D_2B(k, \tau (k))\bigr]^{-1}D_1B(k, \tau (k))\delta k}=0.
$$
This shows that the map $\tau$ is differentiable at the point $k$ and its  derivative $D\tau (k)\in {\mathscr L}(K, Y)$ is given by the familiar formula
$$D\tau (k)\delta k= \bigl[{\mathbbm 1}-D_2B(k, \tau (k_0))\bigr]^{-1} D_1B(k,\tau (k))\delta k.$$
The formula shows that $\tau \mapsto D\tau (k)$ is continuous 
since $\tau $ and $DB$ are continuous maps. Consequently,  the map  $\tau$ is of class $C^1$.

We also note that, differentiating $\tau (k)=B(k, \tau (k))$, we obtain 
\begin{equation*}
\begin{split}
D\tau(k)\delta k& = D_2B(k,\tau(k))D\tau(k)\delta k + D_1B(k,\tau(k))\delta k\\
&=DB(k,\tau(k))(\delta k, D\tau(k)\delta k). 
\end{split}
\end{equation*}
If now $f$ is of class $C^2$, we define the $C^1$-map 
$$
f^{(1)}\colon V\oplus K\oplus Y\oplus Y\to Y\oplus Y
$$
by 
\begin{equation*}
\begin{split}
f^{(1)}(k, h, y,\eta)&=\bigl( y-B(k, y), \eta -DB(k, y)(h, \eta)\bigr)\\
&=(y,\eta)-\bigl( B(k, y), DB(k,y)(h, \eta)\bigr)
\end{split}
\end{equation*}
and find by the same reasoning as before, using that  $D_2B(k, y)$ for small $(k,y)$ is contractive, that the associated family of fixed points
$$(k,h)\mapsto  (\tau(k),D\tau(k)h)$$
is of class $C^1$ near $(0,0)$. In particular, the map 
$\tau$ is of class $C^2$ near $k=0$. Proceeding by induction we conclude that $\tau$ is of class $C^j$ 
on a possibly smaller neighborhood $V$ of $0$ in $K$.

\end{proof}
The proof of Theorem \ref{help-you} is finished.

\pagebreak
\section{Manifolds with Boundary with Corners}
\label{section_tame_manifolds}

The previous chapter 
showed a solution set of a sc-Fredholm section which is a sub-M-polyfold whose induced polyfold structure is equivalent to the structure of a finite dimensional smooth manifold with boundary with corners. In this chapter we shall study these objects in more details. 
The results of Chapter \ref{section_tame_manifolds} will not be used later on.

\subsection{Characterization}\label{subsect_characterization}

A smooth manifold with boundary with corners is a Hausdorff space $M$ which admits an atlas of smooth compatatible quadrant charts $(V, \varphi, (U, C,\R^n))$, where $\varphi\colon V\to U$ is a homeomorphism 
from an open subset $V\subset M$ onto a relatively open set $U$ of the partial quadrant $C$ of $\R^n$.

A M-polyfold $X$ is a Hausdorff space which, in addition, is equipped with a sc-structure. This prompts the following definition.

\begin{definition}
\index{D- Equivalent manifold structure}
A M-polyfold $X$ has a {\bf compatible smooth manifold structure with boundary with corners}, if  
it admits an atlas ${\mathcal A}$ consisting of sc-smoothly  compatible  charts 
$(V,\varphi, (U, C, \R^n))$, where $\varphi\colon V\to U$ is a sc-diffeomorphism from the open set $V\subset X$ onto the relatively open set $U\subset C$ of the partial quadrant 
$C$ in $\R^n$ (or in a finite-dimensional vector space $E$). 
\end{definition}

The transition maps between the  charts are sc-diffeomorphisms between relatively open subsets of partial quadrants in finite-dimensional vector spaces,  
and therefore are classically smooth maps, i.e., of class $C^\infty$. Consequently, the atlas ${\mathcal A}$ defines the structure of a smooth manifold with boundary with corners.

The aim of Section \ref{subsect_characterization}
is the proof of the following characterization.

\begin{theorem}[{\bf Characterization}]
\label{XXX--}\index{T- Characterization of smooth manifolds}
For a tame M-polyfold $X$ the following statements are equivalent.
\begin{itemize}
\item[{\em (i)}]  $X$ 
has a compatible smooth manifold structure with boundary with corners.
\item[{\em (ii)}] $X_0=X_\infty$ and the identity  map ${\mathbbm 1}^1_0\colon X\rightarrow X^1$, $x\mapsto x$ is sc-smooth. Moreover, every point $x\in X$ is contained in a tame sc-smooth polyfold chart $(V, \psi ,(O, C, E))$ satisfying $\psi (x)=0\in O$, and the tangent space $T_0O$ is finite dimensional and in good position to the partial quadrant $C$ in the sc-Banach space $E$.
\end{itemize}
\end{theorem}

\begin{remark}\label{r:r}

If $X_0=X_\infty$, the identity map 
${\mathbbm 1}^1_0\colon X\rightarrow X^1$ has as its inverse the identity map ${\mathbbm 1}^0_1\colon X^1\rightarrow X$. The map ${\mathbbm 1}^0_1$ is 
always sc-smooth. 
Hence, if  the map ${\mathbbm 1}^1_0\colon X\to X^1$ is sc-smooth, it is a sc-diffeomorphism. 
There is a  subtle point in (ii). If $X_0=X_\infty$,  then $X_m=X_0$ as sets for all $m\geq 0$.  However,  it is possible that $X^1\neq X$ as sc-smooth spaces
despite the fact that ${(X^1)}_m = X_{m+1}=X_i$ as sets. In  other words, it  is possible that  as M-polyfolds $X^1\neq X$  
even if the underlying sets are the same.
Indeed, recalling the definition
of sc-differentiability, one realizes that  the sc-smoothness of ${\mathbbm 1}\colon X\rightarrow X$ 
is  completely different from the 
sc-smoothness of ${\mathbbm 1}\colon X\rightarrow X^1$.
This is the reason we use the notation ${\mathbbm 1}_0^1$ 
the identity map $X\to X^1$.

In contrast, if $X$ is  a finite-dimensional vector space equipped with the constant sc-structure, then $X=X^1$ as sc-spaces.
\end{remark}

The proof of Theorem \ref{XXX--} requires some preparations and we start with a definition where we denote as usual with $U\subset C\subset E$ a relatively open set $U$ in the partial quadrant $C$ of the sc-Banach space $E$.

\begin{definition}[{\bf $\ssc^+$-retraction}] 
\label{sc_plus_retraction}\index{D- Sc$^+$-retraction}

 A sc-smooth retraction $r\colon U\to U$  is called a 
 {\bf $\ssc^+$-retraction}, if 
$r(U_m)\subset U_{m+1}$ for all $m\geq 0$ and if  $r\colon U\to U^1$  is sc-smooth.

Similarly, if $V\subset X$ is an open subset of a  M-polyfold $X$, we call the sc-smooth retraction $r\colon V\rightarrow V$ a {\bf $\ssc^+$-retraction},  if $r\colon V\rightarrow V^1$ is sc-smooth.

\end{definition}

\begin{lemma}\label{L-L}\index{L- Sc$^+$-retractions}
Let $(O,C,E)$ be  a sc-smooth retract and suppose that there exists a relatively open subset  $U\subset C$  and a 
$\ssc^+$-retraction
$t\colon U\rightarrow U$ onto  $t(U)=O$. Then every   sc-smooth retraction $s\colon V\rightarrow V$ of a relatively open subset $V\subset C$ satisfying $s(V)=O$
is a $\ssc^+$-retraction. 
\end{lemma}

\begin{proof}
By assumption, the sc-smooth map  $t:U\rightarrow U$  satisfies $t\circ t=t$ and $t(U)=O$. In addition, $t\colon U\to U^1$ is  sc-smooth. 
If now $s\colon V\rightarrow V$ is a sc-smooth retraction onto $s(V)=0$, then $t(s(v))=s(v)$ and hence 
$$
s = t\circ s.
$$
In view of the properties of $t$, the composition 
$t\circ s\colon V\to V^1$ is sc-smooth and hence 
$s\colon V\rightarrow V^1$ is sc-smooth, as claimed. 
\end{proof}

 
\begin{definition}[{\bf $\ssc^+$-retract}] \index{D- Sc$^+$-retract}
The sc-smooth retract $(O,C,E)$ is a 
{\bf $\ssc^+$-retract}, if there exists a relatively open subset 
$U\subset C$ and a $\ssc^+$-retraction $t\colon U\rightarrow U$ onto  $t(U)=O$.
\end{definition}

In view of the Lemma \ref{L-L} the choice of $t$ is irrelevant, being a $\ssc^+$-retract is an intrinsic property of the sc-smooth retract $(O, C, E)$.


\begin{lemma}\label{new_lemma_4.8}\index{L- Image of sc$^+$-retraction}
If $r\colon U\to U$ is a $\ssc^+$-retraction onto the sc-smooth retract $(O, C, E)$, then 
\begin{itemize}
\item[{\em (i)}] $r(U)=O=O_\infty$, i.e., consists of smooth points.
\item[{\em (ii)}] At every point $o\in O$, the tangent space $T_oO=Dr(o)E$ is finite-dimensional.
\end{itemize}
\end{lemma}
\begin{proof}\mbox{}
(i)\, 
 If $x\in U$, then, by definition, $r(x)\in U_1$ and using 
$r\circ r=r$, $r(x)=r(r(x))\in U_2$. Continuing this way we conclude that $r(x)\in \bigcap_{m\geq 0} U_m=U_\infty$. Hence $r(U)=O\cap U_\infty=O_\infty$.

\noindent (ii)\, At the point $o\in O$, the linearization 
$Dr(o)\colon E\to E$ is, in view of (i),  well-defined, and it is a 
$\ssc^+$-operator, since $r$ is a $\ssc^+$-map. Consequently, $Dr(o)\colon E\to E$ is a compact operator between every level. Therefore, the image of the projection $Dr(o)$, namely $Dr(o)E=T_oO$ must be finite-dimensional. This proves the lemma.

\end{proof}

The $\ssc^+$-retracts are characterized by the following proposition.
\begin{proposition}\label{frank}\index{P- Characterization of sc$^+$-retractions}
A sc-smooth retract $(O,C,E)$ is a $\ssc^+$-retract  if and only if  $O_0=O_\infty$ and ${\mathbbm 1}_0^1\colon O\rightarrow O^1$ is sc-smooth.
\end{proposition}

\begin{proof}

If $(O,C,E)$ is a $\ssc^+$-retract, there exists a $\ssc^+$-retraction $r\colon U\to U$ onto  $r(U)=O$. 
The retract has the induced M-polyfold structure  defined  by the scaling $O_m=r(U_m)$ and $O^1$ inherits this structure, so that 
$(O^1)_m=O_{m+1}$, for all $m\geq 0$.
In view of Lemma \ref{new_lemma_4.8}, $O=O_\infty$. Hence the identity map ${\mathbbm 1}_0^1\colon O\to O^1$ is well-defined. According to Definition \ref{tangent_retract}, the map  
${\mathbbm 1}_0^1$ is sc-smooth, provided  the composition 
${\mathbbm 1}_0^1\circ r \colon U\to E^1$ is sc-smooth, which is the case because  $r$ is a $\ssc^+$-map.

In order to prove the opposite direction we assume that $O=O_\infty$ and ${\mathbbm 1}_0^1\colon O\to O^1$ is sc-smooth. Let $r\colon U\to U$ be a sc-smooth retraction onto $O=r(U)$. Then the composition 
$$
U\xrightarrow{r} O\xrightarrow{{\mathbbm 1}_0^1} O^1\xrightarrow{\textrm{inclusion}} U^1
$$
is sc-smooth. The composition agrees with the map $r\colon U\to U^1$ and the lemma is proved.

\end{proof}

The proof of Theorem \ref{XXX--} will make use of the following technical result.

\begin{proposition}
\label{FFF}\index{T- Tangent representation}

Let $(O,C,E)$  be a $\ssc^+$-retract and $t\colon U\to U$ a $\ssc^+$-retraction of the the relatively open subset $U\subset C$ onto $O=t(U)$. We assume that $0\in O$ and the the tangent space $T_0O$ (which  by Lemma \ref{new_lemma_4.8} is finite-dimensional) is in good position to the partial quadrant $C$.   We denote by $Y$ a good complement of $T_0O$, so that $E=T_0O\oplus Y,$  and abbreviate by 
$$p=Dt(0)\colon E=T_0O\oplus Y\rightarrow T_0O$$ 
the sc-projection. 

Then there exist  open neighborhoods ${\mathcal U}$ of $0$ in  $O$ and ${\mathcal V}$ of $0$ in $T_0O\cap C$ such  that
$$
p\colon {\mathcal U}\rightarrow {\mathcal V}
$$
is a sc-diffeomorphism.

\end{proposition}
 
We 
note that $({\mathcal V},T_0O\cap C,T_0O)$ is a local M-polyfold model, because $T_0O\cap C$ is a  partial quadrant in $T_0O$, in view of Proposition \ref{pretzel}. 
The proof of Proposition \ref{FFF} is based on the implicit function theorem for sc-Fredholm sections proved in the previous chapter.



\begin{proof}[Proof of Proposition \ref{FFF}]

We recall that $U$ is a  relatively open neighborhood of $o$ in  the  partial quadrant  $C$ of the sc-Banach space $E$.
Moreover,  $t\colon U\rightarrow U$ is  a $\ssc^+$-retraction onto $O=t(U)$ which contains $0$. In view of 
Lemma \ref{new_lemma_4.8}, the tangent space $T_0O=Dt(0)E$ is a smooth subspace of finite dimensions. By assumption, $T_0O$ lies in good position to $C$. Accordingly, there exists a good sc-complement $Y$ of $T_0O$ in $E$ so that 
$$E=T_0O\oplus Y.$$
 It has the property that there 
exists  $\varepsilon_0>0$  such that for $a+y\in T_0O\oplus Y$ satisfying $\abs{y}_0\leq \varepsilon \abs{a}_0$ the  statements $a\in C$    
and $a+y\in C$ are equivalent.

We shall use in the following the notation $u=a+y\in T_0O\oplus Y$ and denote by $p$ the  sc-smooth projection 
$$p\colon E=T_0O\cap Y\rightarrow T_0O, $$
defined by  $p=Dt(0)$.

Now we consider the local strong bundle 
$$
\pi\colon U\triangleleft Y\rightarrow U,
$$
and  the sc-smooth section $f\colon U\to U\triangleleft Y$ , $f(u)  = (u,{\bf f}(u))$, whose principal part ${\bf f}\colon U\to Y$
is defined  by
$$
 {\bf f}(u) = ({\mathbbm 1}-p)\bigl( u-t(u)\bigr).
$$
We observe that if $u\in O$, then $u=t(u)$ and hence ${\bf f}(u)=0.$ We shall show later on that there are no other solutions $u\in U$ of ${\bf f}(u)=0$ locally near $u=0$.

\end{proof}

\begin{lemma}\label{new_lemma4.11}
The section $f$ is sc-Fredholm.
\end{lemma}


\begin{proof}

In order to verify that the sc-smooth section ${\bf f}$ is regularizing, we take $u=a+y\in U_m$ and assume that 
${\bf f}(u)=y-({\mathbbm 1}-p)\circ t (a+y)\in Y_{m+1}.$
Since $t$ is $\ssc^+$-smooth, we conclude  that $y\in Y_{m+1}$ and since $a$ is smooth, that  $a+y\in U_{m+1}$, showing that ${\bf f}$ is indeed regularizing.

The derivative of ${\bf f}$ at the smooth point $u\in U$ has the form
\begin{equation*}
\begin{split}
D{\bf f}(u)h &=({\mathbbm 1}-p)(h - Dt(u)h)\\
& = h -(p(h)-({\mathbbm 1}-p)Dt(u)h)
\end{split}
\end{equation*}
for $h\in E$. Since  $D{\bf f}(u)$ is a  $\ssc^+$-operator and $p$ is sc-smooth, the operator $D{\bf f}(0)$ is a perturbation of the identity operator by a $\ssc^+$-operator, and therefore a sc-Fredholm operator, in view of Proposition 
\ref{prop1.21}.

Next we have to verify that at a given smooth point $u\in U$, a filled version of ${\bf f}$, after modification by a $\ssc^+$-section ${\bf s}$, is conjugated to a basic germ. However, in the case at hand, we already 
work on relatively open subsets of a partial quadrant, so that  a filling is not needed. We merely have to find a suitable $\ssc^+$-section to obtain a section  which is conjugated to a basic germ.

In general, the good complement $Y$ in the decomposition 
$E=T_0O\oplus Y$ is not a subset of $C$. 
However, we claim that there exists  a finite-dimensional sc-Banach space $B$ and a sc-Banach space $W$ contained in $C$,  such  that $Y=B\oplus W$,  which leads to the sc-decomposition
$$
E= (T_0O\oplus B)\oplus W.
$$

Indeed, if $E=\R^n\oplus F$ and the partial quadrant $C\subset E$ is of  the form $C =[0, \infty)^n\oplus F$ where $F$ is a sc-Banach space, then, identifying the sc-Banach space $F$with $\{0\}^n\oplus F$, we let $B$ to be an  algebraic complement of $Y\cap F$ in $Y$ and $W=Y\cap F$, so that $Y=B\oplus  W$. Clearly, $W\subset C$ and in order to verify that $B$ is finite-dimensional, we take linearly independent vectors $(a^1, w^1), \ldots , (a^l, w^l)$ in $B$ and claim that the  vectors $a^1, \ldots ,  a^l$ are linearly independent in $\R^{n}$. If 
$\sum_{i=1}^l\lambda_ia^l=0$, then $\sum_{i=1}^l\lambda_i(a^i, w^i)=(0, w)\in B\cap W$. Since 
$B\cap W=\{0\}$, we conclude $\lambda_i=0$ and $l\leq n$. Hence $\dim B\leq n$. In the general case, there exists a sc-isomorphism $L\colon E\to \R^n\oplus F$ mapping $C$ onto 
the partial quadrant $C'=[0,\infty)^n\oplus F$.  Then the subspace $L(T_0O)$ is in good position to the partial quadrant $C'$ and $L(Y)$ is a good complement of $L(T_0O)$ in $\R^n\oplus F$. 
By the above argument, we find a finite-dimensional subspace $B'$ and a sc-Banach space $W'$ contained in $C'$, so that $L(T_0O)=B'\oplus W'$.  With $B=L^{-1}B'$ and $W=L^{-1}W'$,  we conclude that $B$ is finite-dimensional, $W$ is contained in $C$ and $Y=B\oplus W$, as claimed.

Since $W\subset C$,
$$C=\bigl(C\cap (T_0O\oplus B)\bigr)\oplus W.$$
The finite-dimensional space $T_0O\oplus B$ is in good position to $C$ because it has a sc-complement contained in $C$. Therefore, the set $C\cap (T_0O\oplus B)$ is a partial quadrant in $T_0O\oplus B$, in view of Proposition \ref{pretzel}. We shall represent an element $u\in U$ as 
$$u=a+b+w\in T_0O\oplus B\oplus W,$$
where $b+w=y\in Y=B\oplus W.$
Furthermore, we denote by  $P$ the projection 
$$P\colon Y=B\oplus W\to W.$$
In accordance with the notation in the definition of a basic germ, we view $Y$ as $\R^N\oplus W$, where 
$N=\dim(B)$,  and $P$  as the sc-projection $P\colon \R^N\oplus W\to W$. 

If $u\in C$ is a smooth point near $0$ we denote by $C_u$ the associated partial quadrant introduced in Definition \ref{new_def_2.33}. The partial quadrant $C_u$ contains $C$, so that 
$$C_u=(C_u\cap (T_0O\oplus B))\oplus W.$$
Fixing the smooth point $u\in C$, the map $v\mapsto u+v$ for $v$ near $0$ in $C_u$, maps an open neighborhood of $0$ in $C_u$ to an open neighborhood of $u$ in $C$. We now study the principal part ${\bf f}\colon U\rightarrow Y$ near the fixed smooth point $u$, make  the change of coordinate $v\mapsto  u+v$, and define the section 
${\bf g}(v)$ by 
$${\bf g}(v):={\bf f}(u+v).$$
In addition, we define a $\ssc^+$-section ${\bf s}$ near 
near $0$ in $C_u$,  by 
$${\bf s}(v) ={\bf f}(u) +( {\mathbbm 1}-p)[  t(u)-t(u+v)]. $$
It has the property that  at $v=0$, 
$$
{\bf g}(0)-{\bf s}(0)=0.
$$
In the decomposition $v=(a+b)+w$ we interpret $a+b\in C_u\cap (T_0O\oplus B)$ as a finite-dimensional parameter and consider 
the germ 
$$(a+b,w)\mapsto  ({\bf g}-{\bf s})(a+b+w)\in Y.$$
Then 
\begin{equation*}
\begin{split}
 P({\bf g}-{\bf s})(a+b+w)=P({\mathbbm 1}-p)(a+b+w) =P(b+w)=w.
\end{split}
\end{equation*}
Identifying an open neighborhood of $0$ in $C_u\cap (T_0O\oplus B)$ with an open neighborhood of $0$ in some partial quadrant $[0,\infty)^k\oplus {\mathbb R}^{n-k}$, we conclude that ${\bf g}-{\bf s}$ is a basic germ, whose contraction part happens to be  identically zero. We have used the local  strong bundle  isomorphism $(u+v,b+w)\mapsto (v,b+w)$.

Consequently,  $(f,u)$ is a sc-Fredholm germ and the proof of lemma \ref{new_lemma4.11}
is  complete.

\end{proof}

The sc-Fredholm section ${\bf f}$ vanishes at the point $u=0$, so that ${\bf f}(0)=0$. Its linearization is  is given by 
$$
D{\bf f}(0)(h) = ({\mathbbm 1}-p)h,\quad h\in E.
$$
Therefore, $Df(0)$ has the kernel $\ker(D{\bf f}(0))=T_0O=Dt(0)E$,  which, by assumption,  is in good position to the partial quadrant $C$. Moreover, its image is $ ({\mathbbm 1}-p)E=Y$, so that $Df(0)$ is surjective.

We are in position to apply the implicit function theorem, Corollary  \ref{LGS2},  and conclude that the local solution set of ${\bf f}$ near $u=0$ is represented by 
$$\{u\in U_0\, \vert \, {\bf f}(u)=0\}=\{u=a+\delta (a)\in T_0O\oplus Y\, \vert \, a\in V\}\subset U.$$
Here  $U_0\subset U$ is an open neighborhood of $0$ in $U$ and $V$ is an open neighborhood of $0$ in the partial quadrant $T_0O\cap C$ of $T_0O$. The map $\delta\colon V\to Y$ is a sc-smooth map satisfying $\delta (0)=0$ and $D\delta (0)=0$.

That all sufficiently small solutions of ${\bf f}(u)=0$ are of the form $a+\delta (a)$, is, in our special case, easily verified directly.

\begin{lemma}\label{new_lemma4.12}
Let $f(a+y)=0$ and $\abs{a}_0<\varepsilon$, $\abs{y}_0<\varepsilon$. If $\varepsilon$ is sufficiently small, then 
$y=\delta (a)$.
\end{lemma}

\begin{proof}
From $f(a+y)=0$ and $f(a+\delta (a))=0$, we obtain 
$y=({\mathbbm 1}-p)t(a+y)$ and 
$
\delta (a)=
({\mathbbm 1}-p)t(a+\delta (a))
$. 
Consequently,  
$$
y-\delta (a)=({\mathbbm 1}-p)[t(a+y)-t(a+\delta (a))].
$$
As $f$ is regularizing, the solutions are smooth points. Since $t$ is a $\ssc^+$-map, the map 
$u\mapsto ({\mathbbm 1}-p)t(u)$ is of class $C^1$ between 
the $1$-levels, and the derivative vanishes at $u=0$, 
$({\mathbbm 1}-p)Dt(0)=0$. Therefore, there exists 
$\varepsilon_1>0$ such that 
\begin{equation}\label{new_eq1_em4.12}
\norm{({\mathbbm 1}-p)Dt(u)}_{L(E_1,E_1)}\leq 1/2
\end{equation}
if $\abs{u}_1<\varepsilon_1$. 
Consequently, on level $1$,
\begin{equation}\label{new_eq2_em4.12}
y-\delta (a)=\biggl(\int_0^1({\mathbbm 1}-p)Dt(u(\tau))\ d\tau \biggr)\cdot (y-\delta (a)),
\end{equation}
where $u(\tau )=a+\tau y+(1-\tau )\delta (a)$.
Since $t$ is a $\ssc^+$-map satisfying $t(0)=0$, we conclude from 
$y=({\mathbbm 1}-p)t(a+y)$ that $\abs{y}_1$ is small if $a$ and $y$ are small on level $0$. Moreover, the map $\delta$ is sc-smooth and satisfies $\delta (0)=0$ and hence 
$\abs{\delta (a)}_1$ is small, if $a$ is small on level $0$. Summing up, 
$\abs{u(\tau)}_1<\varepsilon_1$ if $\varepsilon$ is sufficiently small,  and we conclude from \eqref{new_eq1_em4.12} and 
\eqref{new_eq2_em4.12} the estimate 
$$
\abs{y-\delta (a)}_1\leq \dfrac{1}{2}\abs{y-\delta (a)}_1,
$$
so that indeed $y=\delta (a)$ is $\varepsilon$ is sufficiently small, as claimed in Lemma \ref{new_lemma4.12}.
\end{proof}


We finally verify that the solutions $f(a+\delta (a))=0$ belong to $O$ if $a$ is small on level $0$.

\begin{lemma}\label{new_lemma_4.13}
Let $f(a+\delta (a))=0$ and $\abs{a}_0<\varepsilon.$
If $\varepsilon>0$ is sufficiently small, then 
$a+\delta (a)\in O$.
\end{lemma}

\begin{proof}

We have to confirm that $a+\delta (a)=t(a+\delta (a))$. From 
$f(a+\delta (a))=0$ we conclude that $\delta (a)=
t(a+\delta (a))-p\circ t(a+\delta (a))$, so that our aim is to prove that 
$$a=p\circ t(a+\delta (a)).$$
Applying the retraction $t$ to the identity
\begin{equation}\label{new_eq3_lem4.13}
p\circ t(a+\delta (a))+\delta (a)=t(a+\delta (a)),
\end{equation}
and using $t\circ t=t$,  we obtain
$$t\bigl(p\circ t(a+\delta (a))+\delta (a)\bigr)=t(a+\delta (a)).$$

Hence, abbreviating 
$$a_1:=p\circ t(a+\delta (a)),$$
we arrive at the identity
$$t(a_1+\delta (a))-t(a+\delta (a))=0.$$
Going to level $1$, abbreviating $a(\tau)=\tau a_1+(1-\tau)a+\delta (a)$ and observing that 
$Dt(0)(a_1-a)=a_1-a$,  we estimate
\begin{equation*}
\begin{split}
0&=\abs{t(a_1+\delta (a))-t(a+\delta (a))}_1\\
&=\left|\int_0^1Dt(a(\tau ))\ d\tau (a_1-a)\right|_1\\
& \geq \abs{a_1-a}-\biggl(\int_0^1\norm{Dt(a(\tau ))-Dt(0)}_{L(E_1, E_1)}\ dt \biggr)\cdot \abs{a_1-a}_1
\end{split}
\end{equation*}
Since $\norm{Dt(u)}_{L(E_1, E_1)}$ is continuous in $u$ on level $1$, there exists $\varepsilon_1>0$ such that 
$$\norm{Dt(v))-Dt(0)}_{L(E_1, E_1)}\leq 1/2$$
if $\abs{v}_1\leq \varepsilon_1$. Arguing as in the previous lemma, $\abs{a(\tau)}_1\leq \varepsilon_1$ if $\abs{a}_0<\varepsilon$ and $\varepsilon$ is sufficiently small. Consequently,
$$
0\geq \abs{a_1-a}_1-\dfrac{1}{2}\abs{a_1-a}_1=\dfrac{1}{2}\abs{a_1-a}_1.
$$
Therefore, $a_1=a$ and hence $p\circ t(a+\delta (a))=a$ if 
$\varepsilon>0$ is sufficiently small, and Lemma \ref{new_lemma_4.13} is proved.

\end{proof}
In order to complete the proof of Proposition \ref{FFF}, we set 
${\mathcal V}=\{a\in T_oO\cap C\, \vert \, \abs{a}_0<\varepsilon\}.$ 
The map $p$, satisfying 
$$p (a+\delta (a))=a,$$
is a sc-diffeomorphism from the open set 
${\mathcal U}=\{ a+\delta (a)\, \vert \, a\in {\mathcal V} \}$ in $O$, to the open set ${\mathcal V}$ having the sc-smooth map $a\mapsto a+\delta (a)$ as its inverse. This completes the proof of Proposition \ref{FFF}.
\hfill $\blacksquare$}
\begin{proof}[{\bf Proof of Theorem \ref{XXX--}}]

We assume that (i) holds true: the M-polyfold $X$  has a compatible smooth manifold structure with boundary with corners. 
Correspondingly there is an atlas of sc-smoothly compatible partial quadrants charts $(V, \varphi, (U, C, \R^n))$ where $\varphi\colon V\to U$ is a sc-diffeomorphism from the open subset $V\subset X$ onto the relatively open subset $U\subset C$ of the partial quadrant $C$ in $\R^n$. We take in $\R^n$ the unique constant sc-structure. 

Then the identity map ${\mathbbm 1}\colon U\to U$ is a tame sc-smooth retraction of $(U, C, \R^n)$.

 Clearly, $U=U_\infty$ and ${\mathbbm 1}_0^1\colon U\to U^1$ is sc-smooth. Consequently, these charts define a tame M-polyfold structure on $X$ for which  $X_0=X_\infty$ and
 ${\mathbbm 1}_0^1\colon X\to X^1$ is sc-smooth and the statement (ii) follows.
 
Conversely, if (ii) holds for the tame M-polyfold $X$, then $X_0=X_\infty$ and ${\mathbbm 1}_0^1\colon X\to X^1$ is sc-smooth. If $x\in X$, then we find a tame sc-smooth polyfold chart $(V,\psi, (O, C, E))$ such that the sc-diffeomorphism $\psi\colon V\to O$ satisfies $\psi (x)=0$. 
It follows that $O_0=O_\infty$ and 
${\mathbbm 1}_0^1\colon O\to O^1$ is sc-smooth. By Proposition \ref{frank} the sc-retract $(O, C, E)$ is a $\ssc^+$-retract and hence there is a $\ssc^+$-retraction $t\colon U\to U$ of a relatively open sbset $U$ in the partial quadrant $C$ satisfying $O=t(U)$. The tangent space $T_0O$ is a finite-dimensional smooth subspace of the sc-Banach space $E$ and, by assumption,  in good position to $C$.

Using Proposition \ref{FFF} we find an open neighborhood ${\mathcal O}'$ of $0$ in $O$ and an open neighborhood ${\mathcal V}$ of $0$ in $T_0O\cap C$ and a sc-diffeomorphsim $p\colon {\mathcal O}'\to {\mathcal V}$ onto the finite-dimensional polyfold model $({\mathcal V}, T_0O\cap C, T_0O)$. Taking the open set $V_0=\psi^{-1}({\mathcal O}')\subset V$, the composition 
$$
V_0\xrightarrow{\psi}{\mathcal O'}\xrightarrow{p}{\mathcal V}
$$
is a sc-diffeomorphism onto the finite-dimensional model $({\mathcal V}, T_0O\cap C, T_0O)$. Carrying out this construction around every point $x\in X$, we obtain a smoothly compatible atlas of tame charts and the statement (i) follows. This completes  the proof of Theorem \ref{XXX--}.

 \end{proof}

\subsection{Smooth Finite Dimensional  Submanifolds}

We start with the definition.

\begin{definition}\label{Def2.40}\index{D- Smooth submanifold}
A {\bf smooth finite-dimensional submanifold of the M-polyfold $X$}  is a subset $A\subset X$ having the property that every point $a\in A$ possesses 
 an open  neighborhood $V\subset X$ and a  $\ssc^+$-retraction $s\colon V\to V$ onto 
 $$s(V)=A\cap V.$$
\end{definition}

\mbox{}\\

In contrast to the sc-smooth sub-M-polyfold of an M-polyfold in Definition 
\ref{def_sc_smooth_sub_M_polyfold}, the retracts in the above definition are  $\ssc^+$-retracts.
We also point out, that a smooth finite dimensional submanifold is not defined as a sub-M-polyfold whose induced structure admits an equivalent structure of a smooth manifold with boundary with corners.

\mbox{}\\
\begin{remark} 
The subset $A\subset X$ in 
Definition \ref{Def2.40}  is called a smooth finite-dimensional manifold for the following reasons. The retraction $s\colon V\to V$, satisfying $s\circ s=s$, is a $\ssc^+$-map. Therefore, the image $s(V)\subset X_\infty$ consists of smooth points, so that every point $a\in s(V)$ possesses a tangent space $T_aA$ defined by $T_aA=Ts(a)(T_aX)$. From the $\ssc^+$-smoothness of the map $s$, it follows that the projection $Ts (a)\colon T_aX\to T_aX$ is a $\ssc^+$-operator  and 
therefore level-wise compact. 
Hence its image, $T_aA$ is finite-dimensional. Moreover, as we shall show below, the subsets $A\subset X$ are, near points $a\in A$ satisfying $d_X(a)=0$, 
necessarily smooth finite-dimensional manifolds in the classical sense. The same holds true in the case $d_X(a)>0$ under the additional assumption that the tangent space $T_aA$ is in good position to the partial quadrant $C_aX$ in $T_aX$.

\end{remark}

Next we are going to prove that a smooth finite-dimensional submanifold $A$ inherits 
from $X$ a natural M-polyfold structure. As 
such the degeneracy index $d_A(a)$ is well-defined,  and the boundary $\partial A$ is defined as the subset $\{a\in A\, \vert \, d_A(a)\geq 1\}$.

Let us first recall that a subset $A$ of  a topological space $X$ is called {\bf locally closed}\index{Locally closed} if every point $a\in A$ possesses an open neighborhood $V(a)\subset X$ having the property that a point $b\in V(a)$ belongs to $A$, if $U\cap A\neq \emptyset$ for all open neighborhoods $U$ of $b$.

\begin{proposition}\label{new_prop4.15}\index{P- Basic properties of submanifolds}
A smooth finite-dimensional submanifold $A\subset X$  of the M-polyfold $X$ has the
following properties.
\begin{itemize}
\item[{\em (1)}] $A\subset X_\infty$,  and  $A$  inherits the M-polyfold structure induced from $X$.  In particular,  $A$  possesses a tangent space at every point in $A$, and 
the degeneracy index $d_A$ is defined on $A$.
\item[{\em (2)}] $A$ is locally closed in $X$.
\item[{\em (3)}] $\partial A=\{a\in A\, \vert \, d_A(x)\geq 1\}\subset \partial X$. 
\item[{\em (4)}] ${\mathbbm 1}_0^1\colon A\to A^1$ is sc-smooth.
\end{itemize}
\end{proposition}

\begin{proof}\mbox{}

(1)\, A $\ssc^+$-retraction is, in particular, a sc-retraction and hence $A$ inherits the M-polyfold structure from $X$, in  view of Proposition \ref{sc_structure_sub_M_polyfold}. Since a $\ssc^+$-retraction $s$ has its image in $X_\infty$,  we conclude that $A\subset X_\infty$,  implying that every point in $a\in A$ has 
the  tangent space $T_aA=Ts(s)(T_aX)$. Since $A$ is a M-polyfold,  the degeneracy index $d_A$ is defined on $A$.\\[0.5ex]
(2)\, In order to prove (2) we choose a point   $a\in A$ and an open neighborhood $V\subset X$ of $a$ such that a  suitable 
$\ssc^+$-retraction $s\colon V\to V$  retracts onto $s(V)=A\cap V.$
If  $b\in V$ lies in the closure of $A$, then there exists a sequence $(a_k)\subset A$
satisfying $a_k\rightarrow b$. Since $V$ is open $b\in V$, we conclude that  
$a_k\in A\cap V$ for large $k$.  Hence $a_k=s(a_k)$ and,  using that $b\in V$,  we conclude that $s(b)=\lim_k s(a_k)=b,$ implying that $b\in A\cap V$.\\[0.5ex]
(3)\,  By definition,  $\partial A=\{a\in A\,  \vert \, d_A(x)\geq 1\}$.
We assume that $a\in A$ satisfies $d_A(a)\geq 1$ and show that $d_X(a)\geq 1$. If $d_X(a)=0$,  we find an open neighborhood
$V\subset X$  of $a$ which is sc-diffeomorphic to a retract $(O,E,E)$. This implies that there exists an open neighborhood 
$U\subset A$ of $a$ which is sc-diffeomorphic to a retract $(O',E,E)$, so that $d_A(a)=0$.  \\[0.5ex]
(4)\, The postulated $\ssc^+$-retraction is the identity on the retract $A$, and $A=A_\infty$, so that ${\mathbbm 1}_0^1\colon A\to A^1$ is sc-smooth.

The proof of 
Proposition \ref {new_prop4.15}
is complete.
\end{proof}

\begin{OQ} It is an open question,  whether  there 
is a difference between a smooth finite-dimensional submanifold $A$
and a sub-M-polyfold  satisfying $A=A_\infty$ and $\dim T_aA$ being finite and locally constant.  In fact, the question is if under the latter conditions ${\mathbbm 1}_0^1\colon A\rightarrow A^1$ is sc-smooth.  We conjecture that this is not always the case. Rather we would expect,
for example,  if the dimension of the tangent space is equal to $1$, the set $A$ might be something like a branched one-dimensional manifold
as defined in \cite{Mc}. It would be interesting to see if such type of examples  can be constructed. For example,  
consider the subset $T:=\{(x,0)\in {\mathbb R}^2\, \vert\, -1<x<1\}\cup\{(0,y)\in {\mathbb R}^2\, \vert \, y\in [0,1)\}$ with the induced topology.
Is it possible to find a sc-smooth retract $(O,E,E)$ where $E$ is a sc-Banach space so that $T$ is homeomorphic to $O$
and $O=O_\infty$?  From our previous discussion this is impossible if we require $O$ to be a $\ssc^+$-retract. 
\end{OQ}

So far not much can be said about the structure of the smooth finite-dimensional submanifold $A$ at the boundary without additional assumptions. In order to formulate such an assumption, we consider a M-polyfold $X$ which is required to be tame. Thus, at every smooth point $a\in X$, the cone $C_aX$ is a partial quadrant in the tangent space $T_aX$, in view of Proposition \ref{tame_equality}, and we can introduce the following definition.





\begin{definition}[{\bf good position at  $a\in A$}] \index{D- Good position at $a\in A$}
A smooth finite-dimensional submanifold $A$ of the tame M-polyfold $X$ is {\bf in good position at the point $a\in A$}, if the finite-dimensional linear subspace $T_aA\subset T_aX$ is in good position to the partial quadrant $C_aX$ in $T_aX$.
\end{definition}

We also need the next definition.

\begin{definition}[{\bf Tame submanifold}] \index{D- Tame submanifold}
A smooth finite-dimensional submanifold $A\subset X$ of the M-polyfold $X$ is called 
{\bf tame}, if, equipped with its induced M-polyfold structure, the M-polyfold $A$ is tame.
\end{definition}

The main result of the section is as follows.

\begin{theorem}\label{thm-basic}\index{T- Characterization of tame submanifolds}
Let $X$ be a  tame M-polyfold and  $A\subset X$ a smooth finite-dimensional submanifold of $X$. If $A$ is at every point $a\in A$ in good position, then the induced M-polyfold structure  on $A$ is equivalent to the structure of a smooth manifold with boundary with corners. In particular, the M-polyfold $A$ is tame.  
\end{theorem}

\begin{proof}

The result will be based on Proposition \ref{XXX--}. We focus on a point $a\in A$. Then $a$ is a smooth point, and we find a sc-smooth tame chart $\varphi\colon (V, a)\mapsto (O, 0)$, where $\varphi$ is a sc-diffeomorphism from the open neighborhood 
$V\subset X$ of $a$ onto the retract $O$ satisfying $\varphi (a)=0$. The retract $(O, C, E)$ is  a tame local model. By assumption, $T_aA$ is in good position to the partial quadrant $C_aX$ in $T_aX$ and there is a good complement $Y'\subset T_aX$, so that $T_aX=T_aA\oplus Y'$. Therefore, the tangent space $T_0O=T\varphi (a)(T_aX)$ has the sc-splitting 
$$T_0O=N\oplus Y,$$
in which $N=T\varphi(a)(T_aA)$ and $Y=T\varphi (a)Y'$. Since $(O, C, E)$ is tame, the tangent space $T_0O$ has a sc-complement $Z$ contained in $C$, in view of Proposition \ref{IAS-x}. Hence 
$$
E=T_0O\oplus Z= N\oplus Y\oplus Z.
$$

\begin{lemma}\label{new_lemma4.20}
The finite-dimensional subspace $N=T\varphi(a)(T_aA)$ is in good position to $C$ and $Y\oplus Z$ is a good complement in $E$.
\end{lemma}

\begin{proof}[{\bf Proof of Lemma \ref{new_lemma4.20}}] 
By assumption, $N$ is in good position to $C_0=T_0O\cap C=T\varphi (a)(C_aX)\subset T_0O$, with the good complement $Y$. Hence there exists $\gamma>0$ such that  for $(n, y)\in N\oplus Y$
satisfying $\abs{y}_0\leq \gamma \abs{n}_0$ we have $n\in C_0$ if and only if $n+y\in C_0$. Take the norm $\abs{(y, z)}_0=\abs{y}_0+\abs{z}_0$ on $Y\oplus Z$ and consider $(n, y, z)$ satisfying $\abs{(y, z)}_0\leq \gamma \abs{n}_0$, and note that $n\in C$ if and only if $n\in C_0$. If $n\in C_0$, we conclude from 
$\abs{(y, z)}_0\leq \gamma \abs{n}_0$  that $n+y\in C_0$ which implies $n+y\in C$. Since $z\in C$, we conclude that $n+y+z\in C$. Conversely, we assume that $n+y+z\in C$ satisfies $\abs{(y, z)}_0\leq \gamma \abs{n}_0$. Then it follows from $z\in Z\subset C$ that $n+y\in C$ and hence $n+y\in C_0=T_0O\cap C$. From $\abs{y}_0\leq \gamma \abs{n}_0$ we deduce that 
$n\in C_0$ and hence $n\in C$. Having verified that $N$ is in good position to $C$ and $Y\oplus Z$ is a good complement, the proof of lemma is complete.
\end{proof}

Continuing with the proof of Theorem \ref{thm-basic}, we recall that $(O, C, E)$ is  a tame retract. Hence there exists a relatively open  subset $U\subset C$ and a  sc-smooth tame retraction $r\colon U\to U$ onto $O=r(U)$. Since $A\subset X$ is a finite-dimensional smooth submanifold of $X$, we find an open neighborhood $V\subset X$ of $a$ and a $\ssc^+$-retraction $s\colon V\to V$ onto $s(V)=A\cap V$. Taking $V$ and $U$ small, we may assume that $V=\varphi^{-1}(O)$. 
Then, defining  $t\colon U\rightarrow U$ by
$$
t(u)=\varphi\circ s\circ \varphi^{-1}\circ r (u),
$$
we compute, using $r\circ \varphi =\varphi$ and $s\circ s=s$, 
\begin{equation*}
\begin{split}
t\circ t& = \varphi\circ s\circ\varphi^{-1}\circ r\circ  \varphi\circ s\circ \varphi^{-1}\circ r\\
&=\varphi\circ s\circ\varphi^{-1}\circ  \varphi\circ s\circ \varphi^{-1}\circ r\\
&= \varphi\circ s\circ s \circ \varphi^{-1}\circ r\\
&= \varphi\circ s\circ \varphi^{-1}\circ r\\
&=t.
\end{split}
\end{equation*}
We see that $t$ is 
a retraction. Moreover, it is a $\ssc^+$-retraction since $s$ is a 
$\ssc^+$-retraction. It retracts onto $t(U)=\varphi (A\cap V)$ and 
$Dt(0)E=T\varphi (a)Ts(a)(T_aX)=T\varphi (a)T_aA=N$. 
Therefore,  $Q:=t(U)\subset O$ is a $\ssc^+$-retract and $T_0Q=N$ 
is in good position to $C$, by Lemma \ref{new_lemma4.20}.

Now we can apply Proposition  \ref{FFF}  and conclude that 
that there are open neighborhoods $V_1$ of $0$ in $Q$ and $V_0$ of $0$ in $C\cap N$ such that $p=Dt(0)$ is a sc-diffeomorphism 
$$p\colon V_1\to V_0.$$
In view of Proposition \ref{pretzel},  $C\cap N$ is a partial quadrant in $N$, so that $(V_0, C\cap N, N)$ is a local model which is tame since $V_0\subset C\cap N$ is open.

Defining the open neighborhood $V(a)=\varphi^{-1}(V_0)\cap A$ of $a$ in $A$,  the map 
$$p\circ \varphi \colon V(a)\to V_0$$
is a sc-diffeomorphism defining a sc-smooth tame chart on the M-polyfold $A$. The collection of all these charts defines a sc-smoothly  compatible structure of a smooth manifold with boundary with corners. Moreover, the M-polyfold $A$ is tame. 
The proof of Theorem \ref{thm-basic} is complete.
%

\end{proof}

\begin{remark}
If $A$ is in good position at $a$, then there exists and open neighborhood $V(a)\subset X$, so that $A$ is in good position for all $b\in V(a)\cap A$. 
In order to see this we observe that Theorem \ref{FFF} defines a chart from the sole knowledge 
that we are in good position at $0$. Since the tangent spaces  of $A$  move only slowly, 
Lemma \ref{good_pos} implies that we are in good position at the nearby points as well.
\end{remark}

\begin{definition}[{\bf General position}] \index{D- Submanifold in general position}
Let $A\subset X$ be a smooth finite-dimensional submanifold of the 
M-polyfold $X$. Then $A$ is in {\bf general position} at the point $a\in A$, if 
 $T_aA$ has in $T_aX$ a sc-complement 
 contained in the reduced tangent space $T^R_aX$ as defined in Definition
 \ref{def_partial_cone_reduced_tangent}.
\end{definition}

We end this subsection with a local result.

\begin{theorem}[{\bf Local structure of smooth submanifolds}]\index{T- Local structure of submanifolds}\label{local-str}
Let $A$ be a smooth submanifold of the tame M-polyfold $X$. We assume that $A$ is in general position at the point $x\in A$, assuming that $T_xA$ has in $T_xX$ a sc-complement contained in $T_x^RX$. Then there exists an open neighborhood $U\subset X$  on which the following holds.
\begin{itemize}
\item[{\em (1)}] There are precisely $d=d_X(x)$  many local faces ${\mathcal F}_1,\ldots,{\mathcal F}_d$ contained in $U$. 
\item[{\em (2)}] $d_X(a)=d_A(a)$ for $a\in A\cap U$.
\item[{\em (3)}] If $a\in A\cap U$, then the tangent space $T_aA$ has in $T_aX$ a sc-complement contained in $T^R_aX$.
\item[{\em (4)}] If $\sigma\subset \{1, \ldots ,d\}$, then the intersection ${\mathcal F}_\sigma:=\bigcap_{i\in \sigma}{\mathcal F}_i$ of local faces is a tame M-polyfold and $A\cap {\mathcal F}_\sigma$ is a tame smooth submanifold of ${\mathcal F}_\sigma$. At every point $a\in (A\cap U)\cap {\mathcal F}_\sigma$, the tangent space $T_a(A\cap {\mathcal F}_\sigma)$ has in $T_a{\mathcal F}_\sigma$ a sc-complement contained in  $T_a^R{\mathcal F}_\sigma$.
 \end{itemize}
\end{theorem}

 
\begin{proof}

The result is local and,  going into a tame M-polyfold chart of $X$ around the point $x$,  we may assume that $X=O$ and $x=0\in O$, where $O$ belongs to the tame retract 
$(O, C, E)$ in which $C=[0,\infty)^n\oplus W$, $n=d_O(0)$, is a partial  quadrant in the sc-Banach space $E={\mathbb R}^n\oplus W$.
Moreover, $A\subset O$ is a smooth submanifold of $O$. We recall that the faces of $C$ are the subsets of $E$, defined by 
$$
F_i=\{(a_1,\ldots ,a_n,w)\in E\, \vert \,\text{$a_i=0$ and $a_j\geq 0$ for $j\neq  i$ and $w\in W$}\}.$$
For a subset $\sigma\subset \{1,\ldots ,n\}$ we introduce $F_\sigma=\bigcap_{i\in \sigma}F_i$. Then the M-polyfold $O$ has the faces ${\mathcal F}_i=O\cap F_i$ and we abbreviate the intersection by ${\mathcal F}_\sigma:=O\cap F_\sigma$.

By Definition \ref{reduced_cone_tangent}, $T_0^RO=T_0O\cap (\{0\}\oplus W)$, in our local coordinates.
By assumption of the theorem, the tangent space $T_0A\subset T_0O$ has in $T_0O$ a sc-complement $\wt{A}\subset T_0^RO$, so that $T_0O=T_0A\oplus \wt{A}$. In view of the definition of tame, the tangent space $T_0O$ has in $E$ a sc-complement $Z\subset \{0\}\oplus W$. Hence $E=T_0O\oplus Z=T_oA\oplus \wt{A}\oplus Z$ and it follows that $T_0A$ has in $E$ a sc-complement $Y$ contained in $W=\{0\}\oplus W$,
$$E=T_0A\oplus Y,\quad Y\subset W.$$

We denote by $N\subset T_aA$ a sc-complement of $T_0A\cap W$ in $T_0A$,
$$T_0A=N\oplus (T_0A\cap W).$$

From $E\cap W=W=(T_0A\cap W)\oplus (Y\cap W)=(T_0A\cap W)\oplus Y$, we obtain 
$$E=T_0A\oplus Y=N\oplus (T_0A\cap W)\oplus Y=N\oplus W.$$
Using $Y\subset W\subset C$, we deduce 
$$
C=(C\cap N)\oplus W$$
and
$$
C=(T_0A\cap C)\oplus Y.
$$
Therefore, the projection $\R^n\oplus W\to \R^n$ induces an isomorphism from $N$ to $\R^n$ and from $C\cap N$ onto $[0,\infty )^n$. Moreover, the tangent space $T_0A$ is in good position to $C$, and $Y$ is a good complement.
Hence $T_0A\cap C$ is a partial quadrant of $T_0A$ by Proposition \ref{pretzel}.

As proved in Proposition \ref{FFF} the smooth submanifold $A\subset O$ is represented local as the graph 
$$A=\{v+\delta (v)\, \vert \, v\in V\}$$
of a sc-smooth map $\delta\colon V\to Y$ defined on an open neighborhood $V$ of $0$ in the partial quadrant $T_0A\cap C$ of $T_0A$ satisfying $\delta (0)=0$ and $D\delta (0)=0$. 
The projection 
$$p\colon A\to (T_0A\cap C)\cap V,$$
defined by $p(v+\delta (v))=v$, is a sc-diffeomorphism from the M-polyfold $A$ and 
we conclude, by Proposition \ref{newprop2.24}, that 
$$d_A(a)=d_{T_0A\cap C}(p(a))$$
for all $a\in A\cap U$.

The points in $T_0A\cap C$  are represented by $(n, m)\in (N\cap C)\oplus (T_0A\cap W)$ so that the points $a\in A$ are represented by 
$$a=((n, m), \delta (n, m)).$$

Introducing 
$$
(T_0A)_\sigma:=T_0A\cap T_0{\mathcal F}_\sigma=\{(n, m)\in T_0A\, \vert \, 
\text{$n_i=0$ for $i \in \sigma$}\},
$$
the diffeomorphism $p\colon A\to  V\subset T_0A\cap C$, maps the intersection of faces $A\cap {\mathcal F}_\sigma$ to $(T_0A)_\sigma \cap C$. In particular, if $p(a)\in (T_0A)_\sigma\cap C$, then $a\in A\cap {\mathcal F}_\sigma$. This implies that 
$$d_{T_0A\cap C}(p(a))=d_C(a).$$
In view of proposition \ref{tame_equality}, $d_C(a)=d_O(a)$, and we have verified for $a\in A\cap U$ that $d_A(a)=d_O(a)$ as claimed in the statement (2) of the theorem.

In order to prove the statement (3) we choose $a\in A\cap U$. Then $a=v+\delta (v)$ with $v\in T_0A\cap C$. The tangent space $T_aA$ is the image of the linear map $\alpha\colon h\mapsto h+D\delta (v)h$, $h\in T_0A$. If $a=0$, then $T_0A+W=E$ and we conclude, by means of the particular form of the map $\alpha$, that also $T_aA+W=E$ for all $a\in A\cap U$. Intersecting  with $T_aO$ and using $T_aA\subset T_aO$ leads to 
\begin{equation*}
\begin{split}
T_aO&=T_aA+(W\cap T_aO)\\
&=T_aA+T^R_aO,
\end{split}
\end{equation*}
implying that $T_aA$  has in $T_aO$ a sc-complement which is contained in $T_a^RO$. This proves the statement (3) of the theorem.

As for the last statement we observe that, in view of the arguments in the proof of Proposition 2.43, the sets ${\mathcal F}_\sigma$ are tame M-polyfolds. Moreover, $A\cap {\mathcal F}_\sigma$ are smooth submanifolds of ${\mathcal F}_\sigma$. The previous arguments, but now applied to $A\cap {\mathcal F}_\sigma$ conclude the proof of Theorem \ref{local-str}

\end{proof}



\subsection{Families and an Application of Sard's Theorem}
If $X$ is a tame M-polyfold we introduce the family
$$Z:=\R^m\oplus X=\{(\lambda , x)\, \vert \, \lambda \in \R^m, x\in X\}.$$
The family $\R^m\oplus X$ has a natural tame M-polyfold structure defined by the product of the corresponding charts and we denote by 
$$P\colon \R^m\oplus X\to \R^m,\quad P(\lambda ,x)=\lambda$$
the sc-smooth projection.

Using Theorem \ref{local-str}
and Sard's theorem we are going to establish the following result.

\begin{theorem}\label{SARD}\index{T- Fibered families and Sard}
Let $A\subset Z$ be a smooth submanifold of the tame M-polyfold $Z$. We assume that the induced 
M-polyfold structure is tame and the closure of $A$ is compact. Denoting by $p:=P\vert A\colon A\to \R^m$ the restriction of $P$ to $A$ we assume, in addition, that there exists $\varepsilon>0$ such that the following holds.
\begin{itemize}
\item[(1)] For every $\lambda\in \R^m$ satisfying $\abs{\lambda}\leq \varepsilon$, the set $p^{-1}(\lambda )\subset A$ is compact and non-empty.
\item[(2)] For $z\in A$ there exists a sc-complement $T_zA$ in $T_zZ$ which is contained in $T_z^RZ$.
\end{itemize}
Then there exists a set of full measure $\Sigma\subset \{\lambda\in {\mathbb R}^m\, \vert \, \abs{\lambda}<\varepsilon\}$ of regular values of $p$, having full measure in $B^m(\varepsilon)$ such that for $\lambda\in\Sigma$,  the set 
$$A_\lambda:=\{x\in X \, \vert \,  (\lambda,x)\in A\}$$ is a smooth compact  manifold with boundary with corners, having the additional property that the tangent space $T_xA_{\lambda}$ at $x\in A_\lambda$ has in $T_xX$ a sc-complement contained in $T^R_xX$. 
\end{theorem}

\begin{remark}
By Proposition 2.36, 
the codimension of $T^R_xX$ is equal to $d_X(x)$. Therefore, if $\lambda\in \Sigma$, then the point $x\in A_\lambda$ can only belong to a corner if $d_X(x)\leq \dim (A_\lambda)$. For example, zero-dimensional manifolds $A_\lambda$ have to lie in $X\setminus \partial X$, $one$-dimensional $A_\lambda$ can only hit the $d_X=1$ part of the boundary,  etc.
\end{remark}

\begin{proof}

Fixing a point $z\in A$ there exists, in view of Theorem \ref{local-str} and open neighborhood $U(z)\subset Z$ of $z$ in $Z$ such that there are precisely $d=d_X(z)$ local faces ${\mathcal F}_1,\ldots ,{\mathcal F}_d$ in $Z\cap U$. Moreover, for all subsets $\sigma\subset \{1, 2,\ldots ,d\}$, the intersection of $A$ with ${\mathcal F}_\sigma=\bigcap_{i\in \sigma}{\mathcal F}_i$ is a tame smooth submanifold in ${\mathcal F}_\sigma$ of dimension $\dim A-\#\sigma$. Moreover, the tangent space $T_z(A\cap {\mathcal F}_\sigma)$ has a sc-complement in $T_z{\mathcal F}_\sigma$ which is contained in $T_z^R{\mathcal F}_\sigma$.

We cover the set $\{z\in A\, \vert \, \abs{p(z)}\leq \varepsilon\}$ with the finitely many such neighborhoods $U(z_1),\ldots , U(z_k)$ and first
 study the  geometry of the problem in one of these neighborhoods $U(z_i)$ which, for simplicity of notation, we denote by $U(z)$.

If $\sigma\subset \{1,\ldots ,d\}$
and ${\mathcal F}_\delta$ the  associated 
intersection of local faces, we denote by 
${\mathcal F}^{\circ}_\sigma$ the interior of ${\mathcal F}_\sigma$, i.e., the set ${\mathcal F}_\sigma$ with its boundary removed. The union of all 
${\mathcal F}^{\circ}_\sigma$ is equal to $U(z)$. Here we use the convention that for the empty subset of 
$\{1,\ldots ,d\}$, the empty intersection is equal to $U(z)$ from which the boundary is removed. We note that ${\mathcal F}_{\{1,\ldots ,d\}}^\circ$ is a M-polyfold without boundary. By Theorem 4.23, 
the intersection $A\cap {\mathcal F}^\circ_{\sigma}$ is a smooth manifold without boundary having dimension $\dim (A)-\#\sigma$.

With a subset $\sigma$ we associate the smooth projection 
$$p_\sigma\colon A\cap {\mathcal F}_\sigma^\circ\to \R^m,\quad p_\sigma (\lambda , x)=\lambda.$$
Using Sard's theorem we find a subset $\Sigma_\sigma\subset B^m(\varepsilon)$ of regular values of $p_{\sigma}$ having full measure. The intersection 
$$\Sigma=\bigcap_{\sigma \subset \{1,\ldots ,d\}}\Sigma_\sigma$$
has full measure in $B^m(\varepsilon)$ and consists of regular values for all the maps in (1).

Now we fix $\lambda \in \Sigma$. Then the preimage $p^{-1}(\lambda )\subset A$ has the form $\{\lambda \}\times A_\lambda$ and, by construction,
$$p_\sigma^{-1} (\lambda)=(\{\lambda\}\times A_\lambda)\cap {\mathcal F}_\sigma^\circ ,$$ 
which is a smooth submanifold of $A\cap {\mathcal F}_\sigma^\circ$. Moreover, the tangent space $T_{z'}(A\cap {\mathcal F}_\sigma^\circ )$ at the point $z'=(\lambda, x)\in A\cap {\mathcal F}_\sigma^\circ$ is equal to 
$$T_{z'}(A\cap {\mathcal F}_\sigma^\circ )=T_{z'}\bigl((\{\lambda \}\times A_\lambda)\cap {\mathcal F}_\sigma^\circ\bigr)\oplus \xi_{z'}$$
where the linearized projection 
$$Tp_\sigma (z')\colon T_{z'}(A\cap {\mathcal F}_\sigma^\circ )\to T_{p(z')}\R^m$$
maps $\xi_{z'}$  isomorphically onto $\R^m$.

From our discussion we conclude, in view of the assumption (2) of the theorem,  for $z'=(\lambda, x)\in U(z)$ that the tangent space $T_{z'}(\{\lambda \}\times A_\lambda)$ has a sc-complement in $T^R_{z'}X$. Consequently, $A_\lambda\cap U(z)$ is a smooth submanifold of $X$ in general position.

The argument above applies  to every $z_1,\ldots ,z_k$. Hence for every $z_i$ there exists a subset $\Sigma_{z_i}\subset B^m (\varepsilon)$ of full measure consisting of regular values. The subset $\wh{\Sigma}$, 
$$\wh{\Sigma}=\bigcap_{i=1}^k\Sigma_{z_i}\subset B^m(\varepsilon ),$$
has full measure and it follows, for every $\lambda \in \wh{\Sigma}$, that $A_\lambda\subset X$ is a smooth compact submanifold of the M-polyfold $X$ which, moreover, is in general position.

\end{proof}

\subsection{Sc-Differential Forms}\label{subs_sc_differential}
Following  \cite{HWZ7} we start with the definition.
\begin{definition}
Let $X$ be a M-polyfold and $TX\rightarrow X$ its tangent bundle. A sc-differential $k$-form  $\omega$ is a sc-smooth map
$$\omega:\bigoplus_k TX\rightarrow \R$$
 which is linear in each argument and skew-symmetric. 
The vector space of sc-differential k-forms on $X$ is denoted by $\Omega^k(X)$. \index{$\Omega^k(X)$}
\end{definition}

By means of the inclusion maps $X^i\rightarrow X$ the 
sc-differential $k$-form $\omega$ on $X$ pulls back to a sc-differential $k$-form on $X^i$, defining this way the directed system
$$
\Omega^k(X)\to \ldots \to  \Omega^k(X^i)\rightarrow\Omega^{k}(X^{i+1}\to \ldots 
$$
We denote the direct limit of the system by $\Omega_\infty^\ast (X)$\index{$\Omega_\infty^\ast (X)$}
and introduce the set $\Omega_\infty^\ast (X)$ of sc-differential form on $X_\infty$ by 
$$
\Omega_\infty^\ast (X)=\bigoplus_k \Omega_\infty^k (X).
$$

Next we define the exterior differential. For this we use the Lie bracket of vector fields which has to be generalized to our context.
As shown in \cite{HWZ7} Proposition 4.4,  the following holds.
\begin{proposition}\index{P- Lie bracket}
Let $X$ be a M-polyfold and given two sc-smooth vector fields $A$ and $B$ on $X$ one can define the Lie-bracket by the usual formula which defines
a sc-smooth vector field $[A,B]$ on $X^1$, that is,  
$[A, B]$ is a section of the tangent bundle $T(X^1)\to X^2.$\index{$[A,B]$}
\end{proposition}

In order to define the exterior derivative 
$$d:\Omega^k(X^{i+1})\to \Omega^{k+1}(X^{i}),$$
we take s sc-differential $k$-form and $(k+1)$ sc-smooth vector fields $A_0,A_1,\ldots, A_k$ on $X$ and define $(k+1)$-form $d\omega$ on $X$ by the following familiar formula
\begin{equation*}
\begin{split}
d\omega(A_0,A_1,\ldots, A_k)&=\sum_{i=0}^k(-1)^iD(\omega(A_0,\ldots,\what{A}_i,\ldots, A_k)\cdot A_i \\
&\phantom{=}+ \sum_{i<j} (-1)^{i+j}\omega([A_i,A_j],A_0,\ldots,\what{A}_i,\ldots,\what{A}_j,\ldots,A_k).
\end{split}
\end{equation*}

The exterior derivative $d$ commutes with the inclusion map $X^{i+1}\to X^i$  occurring in the directed system, and consequently induces a map
$$d:\Omega_\infty^\ast (X)\rightarrow\Omega_\infty^\ast (X)$$ 
having the property  $d^2=0$. The pair $(\Omega_\infty^\ast (X), d)$\index{$(\Omega_\infty^\ast (X), d)$} is a graded differential algebra which we call the de Rham complex of the M-polyfold $X$.

\begin{definition}\label{def_de_Rham_cohomology}
The sc-de Rham cohomology of the M-polyfold $X$ is defined as $H^\ast_{sc}(X):=\text{ker}(d)/\text{im}(d)$.\index{$H^\ast_{sc}(X)$}
\end{definition}

There is also a relative version.
If $X$ is a tame M-polyfold the inclusion map $\partial X\rightarrow X$ restricted to local faces is sc-smooth. Local faces are naturally tame M-polyfolds and the same is true for 
the intersection of local faces. Therefore it makes sense to talk about differential forms on $\partial X$ and $\partial X^i$. We define  the differential algebra $\Omega^\ast_\infty(X,\partial X)$ by
$$
\Omega_\infty^\ast(X,\partial X):=\Omega_\infty^\ast(X)\oplus \Omega_\infty^{\ast-1}(\partial X)\index{$\Omega_\infty^\ast(X,\partial X)$}
$$
with differential
$$
d(\omega,\tau)=(d\omega,j^\ast\omega-d\tau)
$$
where $j:\partial X\rightarrow X$ is the inclusion. One easily verifies that $d\circ d=0$ and we denote the associated cohomology by $H^{\ast}_{dR}(X,\partial X)$.

Clearly, a sc-differential $k$-form $\omega\in \Omega^\ast_{\infty}(X)$ induces a classical smooth  differential form 
on a smooth finite-dimensional submanifold $N$ of the M-polyfold $X$.  The following version of Stokes'  theorem holds true.

\begin{theorem}\index{T- Stokes}
Let $X$ be a M-polyfold and let $N$ be an oriented smooth n-dimensional compact tame submanifold of $X$ whose boundary 
$\partial N$, a  union of smooth faces $F$, is equipped with the induced orientation. If $\omega$ is  a sc-differential $(n-1)$-form  on $X$, then 
$$
\int_M d\omega  = \sum_{F} \int_F \omega.
$$
\end{theorem}
Here, the submanifold $N$ is not assumed to  be face structured.

The sc-smooth map $f:X\to Y$ between two M-polyfolds induces 
the map  $f^\ast:\Omega^\ast_\infty(Y)\rightarrow\Omega^\ast_\infty(X)$ in the usual way. There is also a version of the Poincar\'e Lemma, formulated and proved in \cite{HWZ7}.
The general the theory of sc-differential forms on M-polyfolds can be worked out as for the classical smooth manifolds. We leave it to the reader to carry  out the details.

\pagebreak
\section{The Fredholm Package for M-Polyfolds}\label{sec_package}
Chapter \ref{sec_package} 
is devoted to compactness properties 
of sc-Fredholm sections, to their perturbation theory, and to the transversality theory.

\subsection{Auxiliary Norms}

Recalling Section \ref{section2.5_sb}, we consider the strong bundle 
$$P\colon Y\rightarrow X$$
over the tame M-polyfold $X$.  The subset $Y_{0,1}$\index{$Y_{0,1}$ } of biregularity $(0, 1)$\index{Biregularity $(0, 1)$} is a topological space and
the map  $P\colon Y_{0,1}\rightarrow X$ is continuous. 
The fibers $Y_x:=P^{-1}(x)$ have the structure of Banach spaces.

We  first introduce the notion
of an auxiliary norm. This concept allows us to quantify the size of admissible perturbations in the transversality and perturbation theory.
We point out  that our definition is more general than the earlier one  given in \cite{HWZ3}, but it works just as well.

\begin{definition}
An {\bf auxiliary norm} \index{D- Auxiliary norm} is a continuous map $N\colon Y_{0,1}\rightarrow{\mathbb R}$, which has the following properties.
\begin{itemize}
\item[(1)] The  restriction of $N$ to a fiber is a complete norm.  
\item[(2)] If $(w_k)$ is a sequence in $Y_{0,1}$ such that $P(w_k)\rightarrow x$ in $X$
and $N(w_k)\rightarrow 0$, then $w_k\rightarrow 0_x$ in $Y_{0,1}$.
\end{itemize}
\end{definition}

Any two auxiliary norms are locally compatible according to the later Proposition \ref{peter},  which is an immediate corollary of  the following local comparison result.
\begin{lemma}\label{erde}\index{L- Comparison of  auxiliary norms}
Let $P\colon Y\rightarrow X$ be a strong bundle over the M-polyfold $X$,  and  let   $N\colon Y_{0,1}\rightarrow{\mathbb R}$ be  a continuous map which fiber-wise is a complete norm.
Then the following statements are equivalent.
\begin{itemize}
\item[{\em (i)}]  $N$ is an auxiliary norm.
\item[{\em (ii)}] For every $x\in X$  there exists,  for  a suitable open neighborhood $V$ of $x$,  a strong bundle isomorphism 
$\Phi\colon Y\vert V\rightarrow K$ whose underlying sc-diffeomorphism $\varphi\colon V\rightarrow O$ maps $x$ to a point $o\in O$,  and constants $0<c <C$ such that for all $w\in P^{-1}(V)$, 
$$
c\cdot N(w)\leq \abs{h}_1\leq C\cdot N(w),
$$
where $\Phi(w)=(p,h)$.
\end{itemize}
\end{lemma}
\begin{proof}

Assume that (i) holds.   We fix  $x\in X$ and choose for a suitable open neighborhood $V=V(x)\subset X$ a  strong bundle isomorphism $\Phi\colon Y\vert V\rightarrow K$, 
where $K\rightarrow O$ is a local strong bundle and  $K\subset V\triangleleft F\subset C\triangleleft F\subset E\triangleleft F$ is the retract  $K=R(V\triangleleft F)$.  Define $q\in O$ by  $(q,0)=\Phi(0_x)$.
An open neighborhood of $0_x\in  Y_{0,1}$ consists of all $w\in Y_{0,1}$ for which the set  of all  $(p,h)=\Phi(w)$ belongs to some open neighborhood of $(q,0)$ in $K_{0,1}$.
In view of the continuity of the function $N\colon Y_{0,1}\to \R$ there exists $\varepsilon>0$ such that $N(y)<1$ for all $y\in \Phi^{-1}(p, h)$ satisfying $\abs{p-q}_0+\abs{h}_1<\varepsilon.$
In particular,  $\abs{p-q}_0 <\varepsilon/2$ and $\abs{h}_1=\varepsilon/2$ imply  $N(w)\leq 1$.
Using that $N(\lambda w) =\abs{\lambda} N(w)$ and similarly for the norm $\abs{\cdot }_1$,  we infer  for $y$ close enough to $x$,  
$P(w)= y$, and $\Phi (w)=(p, h)$, that 
$$
N(w)\leq \frac{2}{\varepsilon}\cdot \abs{h}_1.
$$

On the other hand,  assume there is no constant $c>0$ such that 
$$
c\cdot \abs{h}_1\leq N(w)
$$
for $(p,h)=\Phi(w)$ and $p$ close to $q$.  Then we find  sequences $y_k\rightarrow x$ and $(w_k)$ satisfying  $P(w_k)=y_k$ and $\abs{h_k}_1=1$
such that $N(w_k)\rightarrow 0$. Since $N$ is an auxiliary norm,  we conclude that $w_k\rightarrow 0_x$ in $Y_{0,1}$ 
which implies the convergence $\Phi(w_k)=(q_k,h_k)\rightarrow (q,0)$, contradicting  $\abs{h_k}_1=1$.  At this point we have proved that,  given an auxiliary norm $N$,  there exist for every $x\in X$ constants $0<c<C<\infty$ depending on $x$
such  that 
$$
c\cdot N(w)\leq \abs{h}_1\leq C\cdot N(w)
$$
for all $w\in Y_{0,1}$  for which  $P(w)$ is close to $x$ and $(p,h)=\Phi(w)$,  for a suitable strong bundle isomorphism $\Phi$ to a local strong bundle model. Hence (i) implies (ii).

The other direction of the proof is obvious: since $N$ is continuous and fiber-wise a complete norm
we see that  the property (1) in  the definition of an auxiliary norm holds.  The estimate  in the statement of the lemma  implies that also property (2) holds.

\end{proof}

As a corollary of the lemma we immediately obtain the following proposition.

\begin{proposition}\index{P- Local equivalence of auxiliary norms}\label{peter}
Let $P\colon Y\rightarrow X$ be a strong bundle. Then there exists an auxiliary norm $N$.
Given two auxiliary norms $N_0$ and $N_1$, then there exists a continuous function $f\colon X\rightarrow (0,\infty)$ such that for all $w\in P^{-1}(x)\subset Y_{0,1}$,  
$$
f(x)\cdot N_0(h)\leq N_1(h)\leq \frac{1}{f(x)}\cdot N_0(h).
$$
\end{proposition}
\begin{proof}
The existence follows from a (continuous) partition of unity argument using the paracompactness of $X$, pulling back by strong bundle maps the standard norm $\abs{\cdot }_1$
to the fibers of the strong bundles. The local compatibility implies the existence of $f$.
\end{proof}

\subsection{Compactness Results}

There are several different kinds of compactness requirements on a sc-smooth section which are useful in practice. It will turn out that they are all equivalent for sc-Fredholm sections. We note that compactness is a notion on the $0$-level of $X$.

\begin{definition}
Let $f$ be a sc-smooth section of a strong bundle $P\colon Y\rightarrow X$. 
\begin{itemize}
\item[(1)] We say that 
$f$ has a {\bf compact solution set} \index{D- Compact solution set} if  $f^{-1}(0)=\{x\in X\ \vert \, f(x)=0\}$ is compact in $X$ (on level $0$).
\item[(2)] The section  $f$ is called {\bf weakly proper} \index{D- Weakly proper}
if  it has a compact solution set and if for every auxiliary norm $N$ there exists an open neighborhood
$U$ of $f^{-1}(0)$ such  that for every $\ssc^+$-section $s$ having support in $U$  and satisfying $N(s(y))\leq 1$ for all $y$,  the solution set 
$$
\{x\in X\,  \vert \, f(x)=s(x)\}
$$
is compact in $X$. 
\item[(3)] The section  $f$ is called {\bf  proper} \index{D- Proper} if $f$ has a compact solution set and if for  every auxiliary norm $N$ there exists an open neighborhood $U$ of
$f^{-1}(0)$ such that the closure in $X_0$ of the set $\{x\in U\,  \vert \,  N(f(x))\leq 1\}$ is compact.
\end{itemize}
\end{definition}

In point (3) we adopt the convention  that if $f(x)$ does not have bi-regularity $(0,1)$, then $N(f(x))=\infty$.
Obviously  proper implies weakly proper, which in turn implies compactness.  
$$
\text{\bf proper}\, \Longrightarrow\, \text{\bf weakly proper}\ \  \Longrightarrow\, \text{\bf compact solution set}.
$$

In general, for a sc-smooth section $f$,  these notions are not equivalent.  
The basic result that all previous compactness notions coincide for a sc-Fredholm section is given by the following theorem.

\begin{theorem}\label{x-cc}\index{T- Equivalence of compactness notions}
Assume that $P\colon Y\rightarrow X$ is a strong M-polyfold bundle over the tame $X$  and $f$ a sc-Fredholm section.
If $f$ has a compact solution set, then $f$ is  proper. In particular,  for a sc-Fredholm section the properties of being  proper,  or being weakly proper, or having a compact solution set are equivalent.
\end{theorem}

The proof is postponed to  Appendix \ref{a-x-cc}.

There are several other useful considerations. The first is the following  consequence of  the local Theorem \ref{save}.

\begin{theorem}[{\bf Local Compactness}]\label{save-1}\index{T- Local compactness}
Assume that $P\colon Y\rightarrow X$ is strong bundle over the M-polyfold $X$ and $f$ a sc-Fredholm section.
Then for a given solution  $x\in X$ of  $f(x)=0$ there exists a nested sequence of open neighborhoods $U(i)$ of $x$ on level zero,
say
$$
x\in U(i+1)\subset U(i)\subset X_0,\quad i\geq 0,
$$
such  that, for all $i\geq 0$,  $\cl_{X_0}(\{y\in U(i)\ |\ f(y)=0\})$ is a compact subset of $X_i$. 
\end{theorem}

The next result shows that compactness, a notion on the $0$-level, also  implies compactness on higher levels.
\begin{theorem}\index{T- Compactness properties}
Assume that $P\colon Y\rightarrow X$ is a strong bundle over the tame M-polyfold $X$ and $f$ is a sc-Fredholm section with compact solution set
$S=\{x\in X\, \vert \, f(x)=0\}$ in $X_0$.
Then $S$  is a compact subset of $X_\infty$.
\end{theorem}
\begin{proof}
By assumption $S$ is compact in $X_0$. Since $f$ is regularizing,  $S\subset X_\infty$.
As was previously shown $X_\infty$ is a metric space. Hence we can argue with sequences.
Take a sequence $(x_k)\subset S$. We have to show that it has a convergent subsequence in $X_\infty$.
After perhaps taking a subsequence we may assume that $x_k\rightarrow x\in S$ in $X_0$. 
From  Theorem \ref{save-1} we conclude that $x_k\rightarrow x$ on every level $i$. 
This implies the convergence  $x_k\rightarrow x$ in $X_\infty$.
\end{proof}

We recall that a sc-smooth section $f$ of the strong bundle $Y\rightarrow X$ defines,  by raising the index,  a sc-smooth section
$f^i$ of $Y^i\rightarrow X^i$. In view of the  previous theorem we conclude that a sc-Fredholm section with compact solution set
produces a sc-Fredholm section $f^i$ with compact solution set. Note that it is a priori clear that $f^i$ is sc-Fredholm. Hence we obtain the following corollary.
\begin{corollary}
Let $f$ be a sc-Fredholm section of the strong bundle $P\colon Y\rightarrow X$ over the tame M-polyfold $X$ and suppose that the solution set $f^{-1}(0)$ is compact.
Then $f^i$ is a sc-Fredholm section of $P^i\colon Y^i\to X^i$ with compact solution set.
\end{corollary}

\subsection{Perturbation Theory and Transversality}

Let us start with the case that $P\colon Y\rightarrow X$ is a tame strong M-polyfold bundle.
 We shall study  for a given sc-Fredholm section 
 $f$ the perturbed section $f+s$, where $s\in\Gamma^+(P)$.

We need a supply of $\ssc^+$-sections which we can only guarantee if we have enough sc-smooth bump functions
(this is a weaker requirement than having sc-smooth partitions of unity. See Appendix \ref{POU} for a more detailed discussion.)
This is,  for example,  the case if $X$ is built on sc-Hilbert spaces.

\begin{definition}
The  M-polyfold $X$ {\bf admits sc-smooth bump functions} if for every $x\in X$ and every  open neighborhood $U(x)$
there exists a sc-smooth map $f\colon X\rightarrow {\mathbb R}$ satisfying $f\neq 0$ and $\supp(f)\subset U(x)$.
\end{definition}

In Section  \ref{POU} (Proposition \ref{prop-x5.36}) the following useful statement is proved.

\begin{proposition}
If the M-polyfold $X$ admits sc-smooth bump functions then for every $x\in X$ and every open neighborhood $U(x)$ there exists
a sc-smooth function $f\colon X\rightarrow [0,1]$ satisfying  $f(x)=1$ and $\supp(f)\subset U(x)$. One can even achieve that $f(y)=1$ for all $y$ close to $x$.
\end{proposition}

We start with an existence result.

\begin{lemma}[{\bf Existence of $\ssc^+$-sections}]
\index{L- Sc$^+$-sections I}
We assume that $P\colon Y\rightarrow X$ is a tame strong bundle over the tame M-polyfold $X$ which admits sc-smooth bump functions.
Let $N$ be an auxiliary norm for $P$.
Then for every smooth point $x\in X$,  every smooth point $e$ in the fiber $Y_x=P^{-1}(x)$,  and every $\varepsilon>0$ there exists,  for  a given open neighborhood
$U$ of $x$,  a $\ssc^+$-section $s\in \Gamma^+(P)$ satisfying  $s(x)=e$, $\supp (s)\subset U$,  and $N(s(x))<N(e)+\varepsilon$.
\end{lemma}

The proof is an adaption of the proof in \cite{HWZ3}.
\begin{proof}
Since $X$ is metrizable,  it is a normal space. We choose  an open neighborhood $Q$ of $x$ so that $W\vert Q$ is strong bundle isomorphic
to the tame local bundle $p\colon K\rightarrow O$ mapping $x$ into $0\in O$. We find an open neighborhood $Q'$  of $x$  whose closure in $X$
is contained in $Q$.  If we construct $s\in\Gamma^+(W\vert Q)$ with support in $Q'$ we can extend it by $0$ to $X$. 
It suffices to construct a section $s$ suitably in local coordinates.
Hence we work in $K\rightarrow O$ 
and choose a smooth point $e\in p^{-1}(0)$. The open set $Q'$ corresponds to an open neighborhood $O'$ of $0$ contained in $O$.
We write $e=(0,\wt{e})$.
Let $N\colon K\rightarrow {\mathbb R}$
be the auxiliary norm. Using the local strong bundle retraction $R$, we define the $\ssc^+$-section $t\colon O\to K$ by $t(y)=R(y,\wt{e})$, so that $t(0)=(0, \wt{e})=e$.  
If $\varepsilon>0$ there exists $\delta>0$ such that $N(t(y))<N(t(0))+\varepsilon=N(e)+\varepsilon$ for $y\in O$ and $\abs{y}_0<\delta$. Using Proposition \ref{prop-x5.36}, we find a sc-smooth function $\beta\colon O\to [0,1]$ satisfying $\beta (0)=1$ and having support in $O'\cap \{y\in O\, \vert \, \abs{y}_0<\varepsilon\}$. The $\ssc^+$-sections $s(y)=\beta (y)t(y)$ of $K\to O$ has the required properties.
\end{proof}

Next we discuss a  perturbation and transversality result in the case that the M-polyfold does not have a boundary.
Our usual notation will be $P\colon Y\rightarrow X$ for the strong bundle. In case we have an auxiliary norm $N$
and an open subset $U$ of $X$ we denote by $\Gamma^{+,1}_U(P)$\index{Allowable sc$^+$-sections, $\Gamma^{+,1}_U(P)$} the space of all $s\in\Gamma^+(P)$ satisfying  $\supp(s)\subset U$ and $N(s(x))\leq 1$ for all $x\in X$. In our applications  $U$ is the open neighborhood
of the compact solution set of a sc-Fredholm section $f$.  We shall refer to $\Gamma^{+,1}_U(P)$ as the space of {\bf allowable sc$^+$-sections}.
The space  $\Gamma^{+,1}_U(P)$ becomes a metric space with respect to the uniform distance
defined by 
$$
\rho(s,s')=\sup_{x\in X}\{N(s(x)-s'(x))\,  \vert \, x\in X\}.
$$
The metric space  $(\Gamma^{+,1}_U(P),\rho)$ is not complete.

\begin{remark}
We note however that if a section  $s$ belongs to the completion of ${\Gamma}^{+,1}_U(P)$,  then the solution set of $f(x)+s(x)=0$ is still compact provided $f$ has a compact solution set and $U$ is an open neighborhood adapted
to the auxiliary norm $N$,  in the sense that $N(f(x))\leq 1$ for $x\in U$,  has the properties stipulated in the  properness result.
\end{remark}

\begin{theorem}[{\bf Perturbation and Transversality: interior case}]\label{p:=}\index{T- Perturbation and transversality: interior case}
Let $P\colon Y\rightarrow X$ be  a strong bundle over the   M-polyfold $X$ satisfying $\partial X=\emptyset$, and assume that $X$ admits sc-smooth bump functions.
Let $f$ be a sc-Fredholm section with a compact solution set and $N$ an auxiliary norm. Then there exists an open neighborhood
$U$ of the solution set $S=\{x\in X\, \vert \, f(x)=0\}$ such  that for sections $s\in\Gamma^+(P)$ having the property that  
$\supp(s)\subset U$ and  $N(f(x))\leq 1$ for all $x\in X$,  the solution set
$S_{f+s}=\{x\in X\,  \vert \, f(x)+s(x)=0\}$ is compact. Moreover,   there exists a dense subset ${\mathcal O}
$ of the metric space $(\Gamma^{+,1}_U(P),\rho)$ such  that for every $s\in {\mathcal O}$ the solution set $S_{f+s}$ has the property that
$(f+s)'(x)\colon T_xX\rightarrow Y_x$ is surjective for all $x\in S_{f+s}$, i.e. for all $x\in S_{f+s}$ the germ $(f,x)$ is in good position.
In particular,  $S_{f+s}$ is a M-subpolyfold whose induced structure is equivalent to the structure of a compact smooth manifold
without  boundary.
\end{theorem}

One can follow the proof of Theorem 5.21 in \cite{HWZ3}. In \cite{HWZ3}  we still worked with splicing cores.
For the convenience of the reader we therefore sketch the proof, the details can be filled in using the arguments  from  \cite{HWZ3}.
\begin{remark}
In \cite{HWZ3} we assume that the fibers of the strong bundle are separable sc-Hilbert spaces. This is in fact not needed
due to an improved  treatment of the compactness in the present text. Also,  originally an auxiliary norm
had to satisfy more properties involving weak convergence, which,  again due to the improved compactness results,  is not needed.
The strategies of the proofs in this more general context are the same.
\end{remark}

\begin{definition}\index{D- Transversal to the zero-section}
The  sc-Fredholm section $g$ of the strong M-polyfold bundle $P\colon Y\to X$ is called {\bf transversal to the zero-section} if,  at every point $x$ satisfying  $g(x)=0$, 
the linearization $g'(x)\colon T_xX\rightarrow Y_x$ is surjective.
\end{definition}

\begin{proof}[Proof of Theorem \ref{p:=}]
We  choose  $s_0'$ in $\Gamma^{+,1}_{U}(P)$ and for given $\varepsilon>0$ we find $\delta<1$ such  that
$$
\abs{s_0'(x)-\delta s_0'(x)}_1\leq \varepsilon/2 \quad  \text{for all $x\in X$}.
$$
Define $s_0=\delta s_0'$. Consider the solution set $S_{f+s_0}$ which we know is compact. 
For every $x\in S_{f+s_0}\subset X_\infty$ we find finitely  many allowable $\ssc^+$-sections $s_1^x,\ldots, s^{k_x}_x$ $\Gamma^{+,1}_U(P)$ such  that the range of 
$(f+s_0)'(x)$ together with the sections $s_j^x$ span $Y_x$.
Then, abbreviating $\lambda=(\lambda_1,\ldots ,\lambda_{k_x})$,  the map
$$
{\mathbb R}^{k_x}\oplus X\rightarrow W, \quad  (\lambda,y)\rightarrow f(y)+s_0(y)+\sum_{j=1}^{k_x} \lambda_j \cdot s^j_x(y)
$$
is a sc-Fredholm section of the obvious pull-back bundle,  in view of the theorem about parameterized perturbations,
Theorem \ref{corro}. We also note that the linearization at the point $(0,x)$ is surjective. In view of the the interior case 
of the implicit function theorem, Theorem \ref{implicit-x}, there exists an open neighborhood $U(x)\subset X$ of $x$
such  that for every $(0,y)$ with $y\in U(x)\cap S_{f+s_0}$ the linearization of the above section is surjective.
We can carry out the above construction for every $x\in S_{f+s_0}$, and  obtain an open covering $(U(x))_{x\in S_{f+s_0}}$
of $S_{f+s_0}$. Hence we find a finite open cover $(U(x_i))_{i=1,\ldots ,p}$.  For every $i$ we have sections 
$s^{j}_{x_i}$, $1\leq j\leq k_{x_i}$.  For simplicity of notation, we list all of  them as $t_1,\ldots,t_m$. Then,  by construction, the section
$$
{\mathbb R}^m\oplus X\rightarrow W, \quad (\lambda,y)\mapsto f(y)+s_0(y)+\sum_{j=1}^m \lambda_j\cdot t_j(y)
$$
is sc-Fredholm,  and its  linearization at every point $(0,x)$ with $x\in S_{f+s_0}$  is surjective. There is a number $\delta_0>0$ such  that
$\abs{\lambda}<\delta_0$ implies $N(\sum_{j=1}^m \lambda_j\cdot t_j(y))<\varepsilon/2$.  By the implicit function theorem
we find an open neighborhood ${\mathcal U}\subset \{\lambda \in \R^m\, \vert \, \abs{\lambda}<\delta_0 \}\oplus X$ of $\{0\}\times S_{f+s_0}$ such that the section 
$$
F\colon {\mathcal U}\rightarrow Y,\quad  F(\lambda,y)=f(y)+s(\lambda,y):=f(y)+s_0(y)+\sum_{j=1}^m \lambda_j\cdot t_j(y)
$$
has a surjective linearization at every solution 
$(\lambda, y)\in {\mathcal U}$ of the equation $F(\lambda,y)=0$. Moreover, $\wt{S}:=F^{-1}(0)$ is a smooth manifold containing $\{0\}\oplus S_{f+s_0}$. 
Taking a regular value $\lambda_0$ for the  projection
$$
\wt{S}\rightarrow {\mathbb R}^m, \quad (\lambda,y)\mapsto  \lambda,
$$ 
it is easily verified  that
$f+s_0+s(\lambda_0, \cdot )$ is transversal to the zero section, see Theorem 5.21 in \cite{HWZ3}. By construction,  $N(s(\lambda_0,y))\leq \varepsilon/2$
and $N(s_0(y)+s(\lambda_0,y))\leq 1-\varepsilon/2$ for all $y$. 
This completes the proof of Theorem \ref{p:=}. 
\end{proof}

\begin{remark}\label{rem_homotopy}
In practice we need to homotope from one sc-Fredholm operator
to the other. For example assume that $f_0$ and $f_1$ are sc-Fredholm sections for $P\colon Y\rightarrow X$, 
both transversal to the zero section and having compact solution sets.
 Suppose further that $f_t$, $t\in [0,1]$,  is an interpolating arc satisfying the following.
 First of all, the section 
 $$
 [0,1]\times X\rightarrow Y,\quad (t,x)\rightarrow f_t(x)
 $$
 is sc-Fredholm and has a compact solution set. Now we can use the above construction for a given auxiliary norm
 to find an open neighborhood $U$  of the solution set and a small perturbation $s$ supported in $U$ satisfying 
 $s(t,\cdot )=0$  for $t$ close to $t=0,1$ (we already have transversality at the boundaries) and such  that $(t,x)\mapsto  f_t(x)+s(t,x)$ is transversal to the 
 zero-section. Then the solution set is a compact smooth cobordism between the originally given solution sets  $S_{f_i}$ for $i=0,1$.
Here  we have to deal with the boundary situation which, in this 
special case,  is trivial. The reader will be able to carry out this construction in more detail
 once we have finished  our general discussion of the boundary case.
 \end{remark}
 
 The next result shows that under a generic perturbation we are able to bring the solution set into a general position to the boundary
 and achieve the transversality to the zero-section. The solution space is then a smooth manifold with boundary with corners.
 
 \begin{definition}[{\bf General Position}]\index{D- General position of $(f,x)$}
Let $P\colon  Y\rightarrow X$ is a strong bundle over the  M-polyfold $X$ and let $f$ be a  sc-Fredholm section. We say that $(f,x)$, where $f(x)=0$, is in {\bf general position}
if  $f'(x)\colon T_xX\rightarrow Y_x$ is surjective and the kernel $\ker(f'(x))$ has a sc-complement contained in the reduced tangent space $T^R_xX$.
\end{definition}

The associated result is the following.
\begin{theorem} [{\bf Perturbation and Transversality: general position}]\label{thm_pert_and_trans}\index{T- Perturbation and transversality: general position}
We assume that $P\colon  Y\rightarrow X$ is a strong bundle over the  M-polyfold $X$ which admits sc-smooth bump functions.
Let $f$ be a sc-Fredholm section with compact solution set and $N$ an auxiliary norm. Then there exists an open neighborhood $U$ of the solution set 
$S=\{x\in X\,  \vert \, f(x)=0\}$ such  that for a section $s\in\Gamma^+(P)$ satisfying  $\supp (s)\subset U$ and $N(s(x))\leq 1$ for all $x\in X$,  the solution set
$S_{f+s}=\{x\in X\,  \vert \,  f(x)+s(x)=0\}$ is compact. Moreover,   there exists a dense subset ${\mathcal O}
$ of the metric space $(\Gamma^{+,1}_U(P),\rho)$ such  that,  for every $s\in {\mathcal O}$,  the solution set $S_{f+s}$ has the property 
that for every $x\in S_{f+s}$,  the pair $(f+s,x)$ is in general position.
In particular,  $S_{f+s}$ is a sub-M-polyfold whose induced structure is  equivalent to the structure of a compact smooth manifold with boundary with corners and  $d_{S_{f+s}}(x)=d_X(x)$ for all $x\in S_{f+s}$. 
\end{theorem}

We follow the ideas of the proof of Theorem 5.22 in \cite{HWZ3}.
\begin{proof}

By assumption,  $f$ has a compact solution set $S_f$. Given an auxiliary norm  we can find an open neighborhood $U$ of $S_f$
such  that the solution set $S_{f+s}$ is compact for every section $s\in \Gamma^{+,1}_U(P)$. 

In order to prove the result we will choose  $s_0\in \Gamma^{+,1}_U(P)$ and perturb nearby by introducing
suitable $\ssc^+$-sections. 
We note that  a good approximation of $s_0\in\Gamma^{+,1}_U(P)$ is $\delta s_0$ where  $\delta<1$ is close to $1$ so that we have to find a small perturbation of the latter.
If we take $s_0$ satisfying  $N(s_0(y))\leq 1-\varepsilon$ for all $y\in X$, 
there is no loss of generality assuming that $s_0=0$ by replacing $f+s_0$ by $f$ and $N$ by $cN$ for some large $c$. 
 The general strategy already appears in the proof of Theorem \ref{p:=}.
Here the only complication is that we would like to achieve additional  properties of the perturbed problem. This requires a more sophisticated set-up. 

In the next step we choose enough sections
in $\Gamma^{+,1}_U(P)$, say $s_1,\ldots, s_m$ such  that,  near every $(0,x)$ with $x\in S_f$,  the section 
$$
F(\lambda,y)=f(y)+\sum_{i=1}^m\lambda_i\cdot s_i(y)
$$
has suitable properties. Namely,  we require the following properties:
\begin{itemize}
\item[(i)] $F'(0,x)\colon  {\mathbb R}^m\oplus T_xX\rightarrow Y_x$ is surjective.
\item[(ii)] $\ker(F'(0,x))$ is transversal to ${\mathbb R}^m\oplus  T^R_xX\subset {\mathbb R}^m\oplus T_xX$.
\end{itemize}

The strategy of the proof is the same as the strategy in  the proof of Theorem \ref{p:=}.  We fix a point $(0,x)$ with $x\in S_f$ and observe that if we have $\ssc^+$-sections
so that the properties (i)-(ii) hold at this specific $(0,x)$, then adding more sections,  the properties (i)-(ii) will still hold. 
Furthermore,   if for a section the properties (i)-(ii) hold at the specific $(0,x)$ , then they  will also hold 
at $(0,y)$ for  $y\in S_f$ close to $x$,  say for $y\in U(x)\cap S_f$. As a consequence we only have to find
the desired sections for a specific $x$ and then,  noting that the collection of neighborhoods $(U(x))$ is an open cover of $S_f$, 
we can choose  finitely many  points $x_1,\ldots ,x_p$ such  that the neighborhoods $U(x_1),\ldots , U(x_p)$ cover $S_f$. The collection of sections associated to these finitely many points 
then possesses  the  desired properties. Therefore,  it is enough to give the argument at a general point $(0,x)$ for  $x\in S_f$.
The way to  achieve  property (i) at $(0,x)$ is  as in the proof of Theorem \ref{p:=}. We take enough $\ssc^+$-sections to obtain the surjectivity.
We take a linear subspace $L$ complementing $T_x^RX$ in $T_xX$ and add sections,  which at $x$ span the image $f'(x)L$.  At this point the combined system of sections already satisfies (i) and (ii) and,  
taking the finite union of all these sections,  the desired properties at $(0,x)$ hold. 
By the previous discussion this completes the construction.

Since $S_f$ is compact  and since,  by construction, the section 
$(F,(x,0))$ is in general position at every  point $(0,x)$, we can apply the implicit function theorem to the section
$$
F(\lambda,y)=f(y)+\sum_{i=1}^m \lambda_i\cdot s_i(y).
$$
We deduce  the existence of $\varepsilon>0$ such  that the set $\wt{S}=\{(\lambda,y)\in {\mathbb R}^m\oplus X\,  \vert \, \text{$\abs{\lambda}<2 \varepsilon$ and $y\in X$}\}$
is a smooth manifold with boundary with corners containing $\{0\}\times S_f$. In addition, the  properties (i)-(ii) hold
for all $(\lambda,y)\in \wt{S}$ and not only for the points $(0,x)\in \{0\}\times S_f$.

To be precise,  $\wt{S}\subset {\mathbb R}^m\oplus X$ is a smooth submanifold with boundary with corners 
so that for every $z=(\lambda,x)\in \wt{S}$ the tangent space $T_z\wt{S}$ has a sc-complement in $ T^R_z({\mathbb R}^m\oplus X)$.  If  $p\colon  \wt{S}\rightarrow {\mathbb R}^m$
is  the projection,  the set $p^{-1}(\{\lambda\, \vert \ \abs{\lambda}\leq \varepsilon\})$ is compact. Hence we can apply 
Theorem \ref{SARD} and find,  for a subset $\Sigma$ of $B^m_\varepsilon$ of full measure,   that
the subset $S_\lambda=\{x\in X\,  \vert \,  (\lambda,x)\in \wt{S}\}$ of $X$ is a smooth submanifold with boundary with corners 
having the property  that every point is in general position, so that   for every  point $x\in S_\lambda$ and every parameter $\lambda\in\Sigma$, the tangent space $T_zS_\lambda$ has a sc-complement contained in $T^R_xX$.

\end{proof}

The third result is concerned with a relative perturbation, which vanishes at the boundary in case we already know that at the boundary we are in a good position. If we have a sc-Fredholm germ $(f,x)$ then a {\bf good position}\index{Good position of $(f,x)$} requires $f'(x)$ to be surjective and the
kernel to be in good position to the cone $C_xX$.

Such a germ would be in {\bf general  position}\index{General position of $(f,x)$} if we require in addition to the  surjectivity of $f'(x)$, that the kernel has a sc-complement in $T^R_xX$.  Clearly general position implies good position.
In SFT,  or more generally,  in a Fredholm theory with operations we have a lot of algebraic structure combining a possibly 
infinite family of Fredholm problems. In this case perturbations should respect the algebraic structure and genericity in these cases
might mean genericity within the algebraic contraints. In some of these cases general position is not achievable, but one can  still achieve a good
position. The perturbations occurring in the context of a Fredholm theory with operations are very often constructed inductively,
so that at each step the problem is already in good position at the boundary, but has to be extended to a generic problem.
The following theorem is a sample result along these lines.

\begin{theorem}[{\bf Perturbation and Transversality: good position}]\index{T- Perturbation and transversality: good position}
Assume that $P\colon  Y\rightarrow X$ is a strong bundle 
over the the tame  M-polyfold $X$ which admits sc-smooth bump functions.
Let $f$ be a sc-Fredholm section with compact solution set and $N$ an auxiliary norm. Further,  assume 
that,  for every $x\in \partial X$ solving  $f(x)=0$, the pair $(f,x)$ is in good position. Then there exists an open neighborhood
$U$ of the solution set $S=\{x\in X\, \vert \,  f(x)=0\}$ so that,  for every  section $s\in\Gamma^{+,1}_U(P)$,  the solution set
$S_{f+s}=\{x\in X\ |\ f(x)+s(x)=0\}$ is compact. Moreover,  there exists an arbitrarily small section $s\in\Gamma^{+,1}_U(P)$ satisfying $s(x)=0$ near $\partial X$
such  that $f+s$ is transversal to the zero-section and for every $x\in S_{f+s}$ the pair $(f,x)$ is in good position.
In particular,  $S_{f+s}$ is a M-subpolyfold whose  induced structure is equivalent to the structure of a compact smooth manifold with  boundary with corners.
\end{theorem}

\begin{proof}

By our previous compactness considerations there exists,  for a given auxiliary norm $N$,  an open neighborhood $U$ of $S_f$
so that for $s\in \Gamma^{+,1}_U(P)$ the solution set $S_{f+s}$ is compact.
By the usual recipe already used in the previous proofs we can find finitely many sections $s_1,\ldots ,s_m$ in $\Gamma^{+,1}_U(P)$ 
which are vanishing near $\partial X$ so that for every $x\in S_f$ the image $R(f'(x))$ and the $s_i(x)$ span $Y_x$.
Of course, in the present construction we are allowed to have sections which vanish near $\partial X$, since by assumption
for $x\in S_f\cap \partial X$ we are already in good position (which in fact implies that $S_f$ is already a manifold with boundary with corners
near $\partial X$). Then we consider as before the section 
$$
(\lambda,x)\mapsto   f(x)+\sum_{i=1}^m \lambda_i\cdot s_i(x),
$$
and,  for a generic value of $\lambda$, which we can take as small as we wish,  we conclude that the associated section $s_\lambda=
\sum_{i=1}^m \lambda_i\cdot s_i$ has the desired properties.

\end{proof}

The next result deals with a homotopy $t\mapsto f_t$ of sc-Fredholm sections during which also the bundle  changes. 
\begin{definition}[{\bf Generalized compact homotopy}]\label{gen_comp_hom}
Consider two sc-Fredholm sections $f_i$ of tame strong bundles $P_i\colon  Y_i \rightarrow X_i$ having  compact solution sets.
We shall refer to $(f_i,P_i)$ as two {\bf compact sc-Fredholm problem}\index{D- Compact sc-Fredholm problem}. 
Then a {\bf generalized compact homotopy between the two compact  sc-Fredholm problems}\index{D- Compact homotopy}  consists of a tame strong bundle $P\colon  Y\rightarrow X$
and a sc-Fredholm section $f$, where $X$ comes with a sc-smooth surjective map $t\colon  X\rightarrow [0,1]$, so that the preimages
$X_t$ are tame M-polyfolds and $f_t=f|X_t$ is a sc-Fredholm section of the bundle $Y\vert X_t$. Moreover,  $f$ has a compact solution set
and $(f\vert X_i,P\vert Y_i)=(f_i,P_i)$ for $i=0,1$.
\end{definition}

\begin{remark}
Instead of requiring $(f\vert X_i,P\vert Y_i)=(f_i,P_i)$ for $i=0,1$ one should better require that the problems 
are isomorphic and make this part of the data. But in applications the isomorphisms are mostly clear, so that we allow ourselves
to be somewhat sloppy.
\end{remark}

\begin{theorem}[{\bf Morse-type structure}]\label{MORSE-type}\index{T- Fredholm homotopy, Morse-type}
We assume that all occurring M-polyfolds admit sc-smooth bump functions. 
Let $f$ be a sc-Fredholm section of the tame strong bundle $P\colon Y\to X$ which is a generalized compact homotopy between the compact Fredholm  problems $(f_i, P_i)$ for $i=0, 1$, as in Definition \ref{gen_comp_hom}.
We assume that $P_i\colon  Y_i\rightarrow X_i$ are  strong bundles over M-polyfolds $X_i$ having no boundaries,  and that  the Fredholm sections $f_i$
are already generic in the sense that for all $x\in S_{f_i}$ the germ $(f_i,x)$ is in general position.
We assume further that $\partial X=X_0\coprod X_1$.  Let $N$ be an auxiliary norm on $P$. Then there exists an open neighborhood $U$ 
of the solution set $S_f$ in $X$ such  that,  for all sections $s\in \Gamma^{+,1}_U(P)$,  the solution set $S_{f+s}$ is compact. Moreover,  there exists an arbitrarily small section 
$s_0\in\Gamma^{+,1}_U(P)$ which vanishes near $\partial X$,  possessing the following properties.
\begin{itemize}
\item[{\em (1)}] For every $x\in S_{f+s_0}$,  the germ $(f+s_0,x)$ is in general position.
\item[{\em (2)}] The smooth function $t\colon  S_{f+s_0}\rightarrow [0,1]$ has only Morse-type critical points.
\end{itemize}
\end{theorem}

\begin{proof}

Applying the  previous discussions we can achieve property  (1) for a suitable section $s_0$. The idea then is to perturb $f+s_0$ further
to achieve also property (2). Note that (1) is still true after a small perturbation.  Hence it suffices to assume that $f$ already has the property that
$(f,x)$ is in good position for all $x$ solving  $f(x)=0$. The solution set $S=\{x\in X\, \vert \, f(x)=0\}$ is a compact manifold
with smooth boundary components. Moreover,  $d(t\vert S)$ has no critical points near $\partial S$ in view of  the assumption
that $(f_i,X_i)$ are  already in general position. If we take a finite number of $\ssc^+$-sections $s_1, \ldots, s_m$ of $P\colon Y\to X$, which vanish near $\partial X$ and are supported near $S$ (depending on the auxiliary  norm $N$),
then the solution set
$\wt{S}=\{(\lambda,x)\,\vert \,  f(x)+\sum_{i=1}^m\lambda_i\cdot s_i(x)=0,\ \abs{\lambda}<\varepsilon\}$ is a smooth manifold for $\varepsilon$ is small enough. Using the classical standard implicit function theorem, we find a family of smooth embeddings $\Phi_\lambda\colon  S\rightarrow \wt{S}$ having the property that $\Phi_0(x)=(0,x)$ and $\Phi_\lambda(S)=\{\lambda\}\times S_\lambda$, where
$S_\lambda=\{y\in X\, \vert \,  f(x)+\sum_{i=1}^m\lambda_i\cdot s_i(y)=0\}$. 

Our aim is to construct the above sections 
$s_i$ in such a way
that, in addition, the map
$$
(\lambda,y)\mapsto d(t\circ\Phi_\lambda(y))\in T_y^\ast S
$$
is transversal at $\{0\}\times S$ to the zero section in cotangent bundle $T^\ast S$.
Then, after having achieved this, the 
parameterized version of Sard's theorem 
will guarantee values of the parameter 
$\lambda$ arbitrarily close to $0$, for which the smooth section
$$
S\ni y\mapsto  d(t\circ\Phi_\lambda(y))\in T^\ast S
$$
is transversal to the zero-section and hence the function $y\mapsto  t\circ\Phi_\lambda(y)$ will be  a Morse function on $S$.
Since $\Phi_\lambda\colon  S\rightarrow S_\lambda$ is a diffeomorphism,  we conclude  that $t\vert S_\lambda$ is a Morse-function.
Having constructed this way the desired section $s_0$, the proof of the theorem will then be complete.

It  remains to construct the desired family of $\ssc^+$-sections $s_1, \ldots ,s_m$.

We fix a critical point $x\in S$ of the function $t\vert S\to [0,1]$, hence $d(t\vert S)(x)=0$. Then $T_xS\subset \ker (dt(x))$. Since $x$ is a smooth point,  
we find a one-dimensional smooth linear sc-subspace $Z\subset T_xX$ such that
$$
T_xX = Z\oplus \ker(dt(x)).
$$
The proof of the following trivial observation is left to the reader.
\begin{lemma}\label{lemma5.29}
For every element $\tau\in T_x^\ast S$ there exists a  uniquely determined sc-operator 
$$
b_\tau \colon  T_xS\to  Z
$$
satisfying 
$$
dt(x)\circ b_\tau = \tau.
$$
If $a\colon  T_xS\rightarrow \ker(dt(x))$ is a sc-operator, then
$$
dt(x)\circ (b_\tau + a) =\tau.
$$
\end{lemma}

Slightly more difficult is the next  lemma.
\begin{lemma}\label{lemma5.30}
Assume that $\tau \in T^\ast_xS$ is given.  Then there exists a $\ssc^+$-section $s$ with sufficiently small support around $x$ having the following properties.
\begin{itemize}
\item[{\em (1)}] $s(x)=0$.
\item[{\em (2)}] $f'(x)\circ b_\tau +  (s'(x)\vert T_xS)=0.$
\end{itemize} 
We note that if $a\colon  T_xS\rightarrow T_xS$,  then also property (2) holds with $b_\tau $ replaced by $b_\tau+a$,  in view of $T_xS = \ker(f'(x))$.
\end{lemma}

\begin{proof}
Since $f$ is a sc-Fredholm section and 
$f(x)=0$, 
the linearization $f'(x)\colon  T_xX\rightarrow Y_x$ is surjective and the kernel is equal to  $T_xS$.
Then $f'(x)\circ b_\tau \colon   T_xS\rightarrow Y_x$ is a sc-operator, i.e.,  the image of any vector in $T_xS$ belongs to $Y_\infty$. 
If this operator is the zero operator we can take $s=0$. Otherwise the operator has a one-dimensional image spanned
by some  smooth point $e\in Y_x$. We are done if we can construct a $\ssc^+$-section $s$ satisfying  $s(x)=0$,  and having support close to $x$, and  
 $s'(x)\vert T_xS\colon  T_xS\rightarrow Y_x$  has a one-dimensional image spanned by $e$,  and
$\ker(s'(0)\vert T_xS)= \ker(f'(x)\circ b_\tau)$. Then a suitable multiple of $s$ does the job. 

Denote by $K\subset T_xS$ the kernel of $f'(x)\circ b_\tau$  and by $L$ a complement of $K$ in $T_xS$. Then $L$ is one-dimensional.
We work now in local coordinates in order to construct $s$. We may assume that $x=0$ and represent $S$ near $0$ as a graph over the tangent space $T_xS$, say $q\mapsto  q+\delta(q)$ with $\delta(0)=0$ and $D\delta(0)=0$. Here $\delta\colon  {\mathcal O}(T_xS,0)\rightarrow V$, where $V$ is a sc-complement of $T_xS$ in the sc-Banach space $E$. The points in $E$ in a neighborhood of $0$ are of  the form
$$
q+\delta(q) + v,
$$
where  $v\in V$. We note that $q+\delta(q)\in O$, where $O$ is the local model for $X$ near $0$.
We split $T_xS=K\oplus L$ and correspondingly write $q=k+l$. Then we can represent  the points in a neighborhood of $0\in E$ in the form
$$
k+l +\delta(k+l) +v.
$$
Choosing  a linear isomorphism $j\colon  L\rightarrow {\mathbb R}$, we define the section $\wt{s}$ for $(k,l,v)$ small by
$$
k+l+\delta(k+l)+v\mapsto  R(k+l +\delta(k+l) +v,\beta(k+l+\delta(k+l)+v)j(l)e),
$$
where $\beta$ has support around $0$ (small) and $\beta$  takes the value $1$ near $0$. The section $s$ is then the restriction of $\wt{s}$ to
$O$. If we restrict $s$ near $0$ to $S$ we obtain 
$$
s(k+l+\delta(k+l)) = R(k+l+\delta(k+l),j(l) e).
$$
Hence  $s(0)=0$,  and the linearization of $s$ at $0$ restricted to $T_xS$ is given by
$$
s'(0)(\delta k+\delta l) = j(\delta l) e.
$$
This implies that $s'(0)\vert T_xS$ and $f'(0)\circ b_\tau$ have the same kernel and their image is spanned by $e$. Therefore,  $s$,  multiplied
by a suitable scalar,  has the desired properties and the proof of Lemma \ref{lemma5.30} is complete.

\end{proof}

Continuing with the proof of Theorem \ref{MORSE-type}  we focus as before on the critical point $x\in S$ of $t\vert S$, which satisfies   $d(t\vert S)(x)=0$. Associated with  a basis  $\tau_1,\ldots  ,\tau_m$ of $T_x^\ast S$ the previous lemma produces the sections $s_1, \ldots ,s_m$.
Consider the solution set $\wt{S}$ of solutions  $(\lambda,y)$ of $f(y)+\sum_{i=1}^m\lambda_i\cdot s_i(y)=0$.   Since $f'(y)$ is onto for all $y\in S$,  the solution set 
$S_\lambda=\{y\, \vert \,  f(y) +\sum_{i=1}^m \lambda_i\cdot s_i(y)=0\}$ is a compact manifold (with boundary) diffeomorphic to 
$S$ if $\lambda$ small. Moreover,  $\wt{S}$ fibers over a neighborhood of zero via the map $(\lambda,y)\mapsto  \lambda$. 

The smooth map
\begin{equation}\label{eqp}
(\lambda,y)\mapsto  d(t\vert S_\lambda)(y)\in T_yS_\lambda.
\end{equation}
is a smooth section of the bundle over $\wt{S}$ whose  fiber at $(\lambda,y)$ is equal to $T_y^\ast M_\lambda$.

We  show that the linearization of \eqref{eqp} at $(0,x),$ which is a map
$$
{\mathbb R}^d\oplus T_xS\rightarrow T_x^\ast S, 
$$
is surjective. Near $(0,x)\in \wt{S}$ we can parameterize $\wt{S}$,  using the implicit function theorem,  in the form
$$
(\lambda,y)\mapsto  (\lambda,\Phi_\lambda(y))
$$
where  $\Phi_0(y)=y$, and $\frac{\partial\Phi}{\partial\lambda_i}(0,x)=0$.
Using \eqref{eqp}  and the map $\Phi$ we obtain,  after a coordinate change on the base
for  $\lambda$ small and $z\in S$ near $x$, the map 
$$
(\lambda,z)\mapsto  d(t\circ \Phi_\lambda)(z) = dt(\Phi_\lambda(z))T\Phi_\lambda(z),
$$
where $d$ acts only on the $S$-part. By construction,  the section vanishes at $(0,x)$.
Recall that,  by construction,  $\Phi_\lambda(x)=x$. Hence for fixed $\delta \lambda$ the map 
$$
z\mapsto  \sum_{i=1}^m \delta\lambda_i \cdot \frac{\partial\Phi}{\partial\lambda_i}(0,z)
$$
is a vector field defined near  $x\in S$ which vanishes at $x$. Therefore, it has a well-defined linearization at $x$.
The linearization ${\mathbb R}^d\oplus T_xS\rightarrow T_x^\ast S$ at $(0,x)$ is  computed to be the mapping  
\begin{equation}\label{new_eq5_68}
\begin{split}
(\delta \lambda,\delta z)\mapsto &(dt\vert S)'(x)\delta z + dt(x)\biggl(\sum_{i=1}^m \delta\lambda_i\biggl( \frac{\partial\Phi}{\partial\lambda_i}(0,\cdot )\biggr)'(x)\biggr)\\
&=(dt\vert S)'(x)\delta z + \sum_{i=1}^m \delta\lambda_i \cdot dt(x)\circ \biggl(\frac{\partial\Phi}{\partial\lambda_i}(0,\cdot )\biggr)'(x).
\end{split}
\end{equation}
The derivative  $(dt\vert S)'(x)$ determines  the Hessian of the map $t\vert S$ at the point $x\in S$. The argument is complete if we can show that
\begin{equation}\label{question1}
\biggl(\frac{\partial\Phi}{\partial\lambda_i}(0,\cdot )\biggr)'(x) =b_{\tau_i} +a_i,
\end{equation}
where the image of $a_i$ belongs to  $T_xS$.
By the  previous discussion,   $dt(x)\circ b_{\tau_i}=\tau_i$ so that the map \eqref{new_eq5_68} can be rewritten, using $dt(x)\circ a_i=0$, as
$$
(\delta \lambda,\delta z)\mapsto   (dt\vert S)'(x)\delta z + \sum_{i=1}^m \delta\lambda_i \cdot \tau_i,
$$
which then proves our assertion. So, let us show that the identity \eqref{question1} holds. We first linearize the equation
$$
f(\Phi_\lambda(z))+\sum_{i=1}^d \lambda_i s_i(\Phi_\lambda(z))=0
$$
with respect to $\lambda$ at $\lambda=0$, which  gives 
$$
Tf(z)\biggl(\sum_{i=1}^m \delta\lambda_i\cdot \frac{\partial\Phi}{\partial\lambda_i}(0,z)\biggr) +\sum_{i=1}^d \delta\lambda_i\cdot s_i(z)=0.
$$
Next we linearize with respect to $z$ at $z=x$, leading  to 
$$
\sum_{i=1}^m \delta\lambda_i\cdot \biggl(f'(x)\circ \biggl(\frac{\partial\Phi}{\partial\lambda_i}\biggr)'(0,x) + s'_i(x)\vert T_xS\biggr)=0
$$
for all $i=1,\ldots  ,m$. Hence $f'(x)\circ \bigl(\frac{\partial\Phi}{\partial\lambda_i}\bigr)'(0,x) + s'_i(x)\vert T_xS=0$. This implies that 
$\bigl(\frac{\partial\Phi}{\partial\lambda_i}\bigr)'(0,x) =b_{\tau_i} +a_i$, where the image of $a_i$ lies  in the kernel of $f'(x)$, i.e., in  $T_xS$. 
At this point we have proved that the linearization of \eqref{eqp} at $(0,x)$ is surjective. Since the section is smooth,  there exists an open neighborhood $U(x)$ of $x$  in $S$ so that,  if at $(0,y)$ we have $d(t\vert S)(y)=0$, then the linearization of
$(\lambda,z)\rightarrow d(t\vert S_\lambda)(z)$ at $(0,y)$ is surjective.

We can now apply the previous discussion to all points $x$ solving $d(t\vert S)(x)=0$ and,  using the compactness,  we find finitely many such points $x_1,\ldots  ,x_k$
so that the union of all the $U(x_i)$ covers the critical points of $t\vert S$. For every $i$ we have $\ssc^+$-sections $s_1^i,\ldots  ,s_{m_i}^i$ possessing  the desired properties.  In order to simplify the notation we denote the union of these sections by $s_1,\ldots  ,s_d$. Then we consider the solutions of 
$$
f(y)+\sum_{i=1}^d \lambda_i\cdot s_i(y)=0.
$$
Again we denote the solution set by $\wt{S}$. It fibers over an open neighborhood of $0$ in ${\mathbb R}^d$. By construction, 
the smooth map
$$
(\lambda,z)\mapsto  d(t\vert S_\lambda)(z)\in T_zS_\lambda
$$
has,  at every point $(0,y)$ satisfying $d(t\vert S)(y)=0$,  a linearization
$$
{\mathbb R}^d\oplus T_yS\rightarrow T_y^\ast S
$$
which is surjective. Now we take a regular value $\lambda$ (small) for the projection $\wt{M}\rightarrow {\mathbb R}^d$
and find,   by the parameterized version of Sard's theorem, that
$$
d(t\vert S_\lambda)
$$
is indeed a Morse-function. The proof of  Theorem \ref{MORSE-type} is complete.

\end{proof}

We conclude this subsection by adding  two useful results. The first result shows that we can always
bring a proper Fredholm section into a good position  by a small perturbation.

\begin{theorem}[{\bf Perturbation into a good position}]
We assume that $P\colon  Y\rightarrow X$ is a strong bundle over the tame M-polyfold $X$ and $f$ a proper sc-Fredholm section. We assume that $X$ admits sc-smooth bump functions.
Fix an auxiliary norm $N$ and choose  an associated open neighborhood $U$ of $S=f^{-1}(0)$ such that the solution set $S_{f+s}=\{x\in X\, \vert \, f(x)+s(x)=0\}$ is compact for all $\ssc^+$-sections $s$ in $\Gamma^{+,1}_U(P)$. 
Then there exists a  $\ssc^+$-section $s\in \Gamma^{+,1}_U(P)$ such that in addition
the compact solution set $S_{f+s}=\{x\in X\, \vert \, f(x)+s(x)=0\}$ has the property that for every $x\in S_{f+s}$ the pair $(f+s,x)$ is in good position to the boundary $\partial X$.
In particular,  $S_{f+s}$ has in a natural way the structure of a manifold with boundary with corners.
\end{theorem}

\begin{proof}
The theorem is a consequence of Theorem \ref{thm_pert_and_trans} 
 and the following lemma showing that  general position implies good position.  
 \begin{lemma}\label{general_to_good}
Assume that $f\colon X\to Y$ is  a sc-Fredholm section of the strong M-polyfold bundle $P\colon Y\to X$ over the tame M-polyfold $X$. If $f$ is in general position, then $f$ is in good position.
\end{lemma}
\begin{proof}
We take a point $x\in X$ solving $f(x)=0$. Then the linearization $f'(x)\colon T_xX\to Y_x$ is surjective. This implies that $(f, x)$ is in good position if $d(x)=0$. Hence we assume that $d:=d(x)\geq 1$. Since $(f, x)$ is in general position, the kernel $\text{ker}f'(x)$ is transversal to the reduced tangent space $T^{R}_xX$ in $T_xX$. In particular, $\dim \text{ker}f'(x)\geq \codim T^{R}_xX$ which implies that $\dim \text{ker}f'(x)\geq d$.
Working in local coordinates, we may assume that $x=0$  belongs to the partial quadrant $C=[0,\infty)^d\oplus \R^{n-d}\oplus W$ and that 
$f\colon \R^n\oplus W\to \R^N\oplus W$. The linearization $f'(0)\colon \R^n\oplus W\to \R^N\oplus W$ is surjective and $\ker f'(0)$ is transversal to $\{0\}^d\oplus R^{n-d}\oplus W$. This implies that $K:=\ker f'(0)$ has a sc-complement $K^\perp$ contained in $C$. 
To verify that $K^\perp$ is a good complement of $\ker f'(0)$ we take $a\in K$ and $b\in K^\perp$. If $a\in C$, then, since $K^\perp\subset C$, we conclude that also $a+b\in C$. Conversely, assuming $a+b\in C$, it follows from $-b\in C$ and  $a=(a+b)+(-b)$ that $a\in C$.
We have proved that $K:=\ker f'(0)$ is in good position to $C$.
\end{proof}

 {
Now  fixing  an auxiliary norm $N$ of the strong M-polyfold bundle $P\colon Y\to X$, we find an open neighborhood $U$ of the compact solution set $S_f=f^{-1}(0)$ and having the property that for every $s\in \Gamma^{+,1}_U(P)$ the solution set $S_{f+s}$ is compact.  Taking $\varepsilon>0$, we choose $0<\delta <1$ and  a $\ssc^+$-section $s_0'$ having its support contained in $U$ and satisfying $(1-\delta )N(s_0'(x))<\varepsilon/2$ for all $x\in X$. Then we set $s_0=\delta s_0'$.  The solutions set $S_{f+s_0}=(f+s_0)^{-1}(0)$ is compact, contained in $U$, and consists of smooth points.  Arguing as in the proof of Theorem \ref{thm_pert_and_trans}, 
we choose finitely many $\ssc^+$-sections, $s_1$, \ldots ,$s_m$,  belonging to $\Gamma_U^{+,1}(P)$  such that the sections 
$F\colon \R^m\oplus X\to Y$, defined by 
$$F(\lambda, y)=(f+s_0)(y)+\sum_{j=1}^m\lambda_j s_j (y),$$
is in general position every point $(0, x)\in \R^m\oplus X$ with $x\in S_f$. There exists $\delta_0>0$ such that $\abs{\lambda}<\delta_0$ implies that $\sum_{j=1}^mN(s_j(y))<\varepsilon/2$. In view of Theorem \ref{IMPLICIT0} 
for every $(0, x)$ where $x\in S_f$ there exists an open neighborhood $U(x)$ such that at every point $(\lambda, y)\in U(x)$ solving $F(\lambda, y)=0$ the sc-Fredholm germ $(F, (\lambda ,y))$ is in general position to the boundary $\partial (\R^m\oplus X)=\R^m\oplus \partial X$. Using the compactness of the solution set $S_{f+s_0}$, there are finitely many sets $U(x_i)$, $1\leq i\leq k$, and possibly smaller $\delta_0$ such that 
$\wt{S}=\{(\lambda, x)\, \vert \, \text{$F(\lambda, x)=0$  and $\abs{\lambda}<\delta_0$}\}\subset \bigcup_{i=1}^kU(x_i)$. In particular, if $(\lambda , y)\in \wt{S}$, then $(F, (\lambda, y))$ is in general position, and hence, in view of the above lemma, in good position to $\partial (B_{\delta_0}\oplus X)$. Then Theorem 3.57 
implies that $\wt{S}$ is a sub M-polyfold
of $X$ and the induced M-polyfold structure on $S$ is equivalent to the structure of a smooth manifold with boundary with corners. Applying Sard's theorem as in the proof of Theorem 5.22 \cite{HWZ2}, we find $\lambda^*$ satisfying $\abs{\lambda^*}<\delta_0$ such that,  setting $s:=s_0+\sum_{j=1}^m\lambda_j^*s_j$, the sc-Fredholm section $f+s$ has the property that $(f+s, x)$ is in general position for every $x$ solving the equation $(f+s)=F(\lambda, x)=0$. In view of the above lemma, for any such $x$, the linearization $(f+s)'(x)$ is surjective and the $\ker(f+s)'(x)$ is in good position to $\partial X$. The proof of the theorem is complete.
}
\end{proof}

The next result deals with  the question how  different perturbations, which  bring  a sc-Fredholm section  into good position,  are related.
The proof is left to the reader.

\begin{theorem}[{\bf Cobordism between good position perturbations}]\index{T- Cobordism}
Let $P\colon  Y\rightarrow X$ be a strong bundle over the tame M-polyfold $X$, which is assumed to admit sc-smooth bump functions.
Assume that $f_t$, $t\in [0,1]$, is a proper homotopy
of sc-Fredholm sections. Assume $N$ is an auxiliary norm for $Y\rightarrow X\times [0,1]$ and $U$ an associated open neighborhood
of the compact solution set $S=\{(x,t)\, \vert\ f_t(x)=0\}$. Let $U_i=U\cap (\{i\}\times X)$ and $N_i=N(\cdot ,i)$ for $i=0,1$. Suppose further
that $s_i$ are sc$^+$-section of $Y\rightarrow X$ supported in $U_i$ and satisfying $N_i(s_i(x))<1$ for all $x\in X$, $i=0,1$, so that
$(f+s_i,x)$ is in good position to $\partial  X$ for all smooth $x$, with associated solution sets $S_0$ and $S_1$, which are compact manifolds with boundary with corners.
Then there exists a $\ssc^+$-section $s(x,t)$ with $N(s(x,t))<1$ for all $(x,t)$, supported in $U$,  and satisfying $s_i=s(x,i)$, so that
$F$ defined by $F(x,t)=f(x)+s(x,t)$ is at all smooth points $(x,t)$ in good position to the boundary. In particular,  the associated solution set 
$S$ is a compact manifold with boundary with corners intersecting $\{i\}\times X$ in the manifolds with boundary with corners $S_0$ and $S_1$.
\end{theorem}

\subsection{Appendix}
\subsubsection{Proof of Theorem \ref{x-cc}}\label{a-x-cc}

\begin{T5.5}
Let  $P\colon Y\rightarrow X$ be  a strong M-polyfold bundle over the M-polyfold $X$ and $f$ a sc-Fredholm section of $P$.
If $f$ has a compact solution set $S$, then $f$ is  proper. In particular,  for a sc-Fredholm section the properties of being  proper,  or being weakly proper, or having a compact solution set are equivalent.
\end{T5.5}

\begin{proof} 

We denote by  $N$ the auxiliary norm on a strong M-polyfold bundle  $P$.  Fixing a solution $x\in X$ of $f(x)=0$, we have  to find an open neighborhood $U\subset X$ of $x$, such  that the closure 
$\cl_{X}(\{y\in U\, \vert \,  N(f(y))\leq 1\})$ is compact. Since $f$ is regularizing,  the point $x$ is smooth. 
There is no loss of generality in assuming  that we work
with a filled version $g$ of $f$ for which we have a $\ssc^+$-section $s$ such  that $g-s$ is conjugated to a basic germ.
 Hence, without loss of generality,  we may  assume that  
 we work in local coordinates and $f=h+t$  where $h\colon {\mathcal O}(C)\to \R^N\oplus W$ is a basic germ and $t\colon {\mathcal O}(C)\to \R^N\oplus W$ a $\ssc^+$-germ satisfying $t(0)=0$. Here $C$ is the partial quadrant $C=[0,\infty)^k\oplus {\mathbb R}^{n-k}\oplus W$ in the sc-Banach space $E=\R^n\oplus W$. Then it suffices to find in local coordinates an open neighborhood $U$ of $0$ in the partial quadrant $C$ such that 
the closure (on level $0$) of the set 
\begin{equation}\label{eq_est_0_5}
\{(a, w)\in U\; \vert \; \abs{(h+t)(a, w)}_1\leq c\}
\end{equation}
is compact. 

The section  $t$ is a $\ssc^+$-section satisfying  $t(0)=0$. Therefore, we find  a constant $\tau'>0$ such that 

\begin{equation}\label{eq_est_1_5}
\text{$\abs{t(a, w)}_1\leq c$\quad for $(a, w)\in C$ satisfying   $\abs{a}_0<\tau'$, $\abs{w}_0<\tau'.$ }
\end{equation}

We denote by  $P\colon \R^N\oplus W\to W$ the sc-projection.   By assumption, $h$ is a basic germ and hence $P\circ h$ is of the form 
$$P\circ h(a, w)=w-B(a, w)\quad \text{for $(a, w)\in C$ near $0$.}$$
Here  $B$ is a sc-contraction germ. Moreover, $({\mathbbm 1}-P)h$ takes values in $\R^N$,  so that its range consists of smooth point. 
In view of the sc-contraction property of $B$ and since any two norms on $\R^N$ are equivalent, replacing $\tau'>0$ by a smaller number,  we may assume that the estimates 

\begin{equation}\label{eq_est_2_5}
\abs{B(a, w)-B(a, w')}_0\leq \dfrac{1}{4}\abs{w-w'}_0
\end{equation}
and
\begin{equation}\label{eq_est_2_5b}
\abs{({\mathbbm 1}-P)h(a, w)}_1\leq c
\end{equation}
are satisfied 
for all $(a, w), (a, w')\in C$ such that $\abs{a}_0\leq \tau'$, $\abs{w}_0\leq \tau'$,  and $\abs{w'}_0\leq \tau'$.

Since $B(0)=0$, we find $0<\tau\leq \tau'$ such that 
\begin{equation}\label{eq_est_3_5}
\abs{B(a, 0)}_0\leq \tau'/8\quad \text{for all $a\in \R^n$ such that  $\abs{a}_0\leq \tau$.}
\end{equation}
for all $\abs{a}_0\leq \tau$ and $\abs{w}_0\leq \tau'/4.$
We introduce the closed set 
$$\Sigma=\{(a, z)\in \R^n\oplus W\, \vert \, \abs{a}_0\leq \tau,\ \abs{z}_0\leq \tau'/2\}$$
and denote by $\ov{B}(\tau')$ the closed ball in $Y_0$ centered at $0$ and having radius $\tau'$.
Then we  define the map $F\colon \Sigma \times \ov{B}(\tau')\to Y_0$ by 
$$F(a, z, w)=B(a, w)+z.$$
If $(a, z, w)\in \Sigma \times \ov{B}(\tau')$, then,  in view of (2) and (3), 
\begin{equation*}
\begin{split}
\abs{F(a, z, w)}_0&\leq \abs{B(a, w)-B(a, 0)}_0+\abs{B(a, 0)}_0+\abs{z}_0\\
&\leq \dfrac{1}{4}\abs{w}_0+\abs{B(a, 0)}_0+\abs{z}_0\leq \tau'/4+\tau'/8+\tau'/2=3\tau'/4<\tau',
\end{split}
\end{equation*}
and, if  $(a, z)\in \Sigma$ and $w, w'\in \ov{B}(\tau')$, then 
$$\abs{F(a, z, w)-F(a, z, w')}_0=\abs{B(a, w)-B(a, w')}_0\leq \dfrac{1}{4}\abs{w-w'}_0.$$

We see that $F$ is a parametrized contraction of $\ov{B}(\tau')$, uniform in $(a, z)\in \Sigma$.
 In view of the parametrized Banach fixed point theorem, there exists a unique continuous map $\delta \colon \Sigma \to \ov{B}(\tau')$ satisfying 
 $F(a, z, \delta (a, z))=\delta (a, z)$ for every $(a, z)\in \Sigma$. 
Thus 
 $$\delta (a, z)=B(a, \delta (a, z))+z,\quad \text{for all $(a, z)\in \Sigma$}.$$
In particular, if $(a,z, w)\in \Sigma\times \ov{B}(\tau')$ and $z=w-B(a, w)$, then $w=\delta (a, z).$

Now we define the open neighborhood $U$ of $0$ in $C$ by 
 \begin{equation}\label{eq_est_4_5}
 U=\{(a, w)\, \vert \, \abs{a}_0<\tau,\, \abs{w}_0<\tau'/4\}.
 \end{equation}

Assume that $(a, w)\in U$ and  that $z'=h(a, w)$ belongs to $\R^n\oplus Y_1$ and satisfies $\abs{z'}_1\leq c$.  Then  
\begin{equation*}
\begin{split}
z'=h(a, w)+t(a, w)&=P\circ h(a, w)+\bigl(({\mathbbm 1}-P)\circ h+t\bigr) (a, w)\\
&=w-B(a, w)+\bigl(({\mathbbm 1}-P)\circ h+t\bigr)(a, w)
\end{split}
\end{equation*}
and 
\begin{equation}\label{eq_est_5}
w-B(a, w)=z\quad \text{where $z=z'-\bigl(({\mathbbm 1}-P)\circ h+t\bigr) (a, w)$}.
\end{equation}
In view of the estimates \eqref{eq_est_1_5} and \eqref{eq_est_2_5b}, 
\begin{equation}\label{eq_est_5_5}
\abs{z}_1\leq 3c. 
\end{equation}
The  norm of $z$ on level $0$,  can be estimated as 
\begin{equation}\label{eq_est_6_5}
\begin{split}
\abs{z}_0&=\abs{w-B(a, w)}_0\leq \abs{w}_0+\abs{B(a, w)}_0\\
&\leq \abs{w}_0+\abs{B(a, w)-B(a, 0)}_0+\abs{B(a, 0)}_0\\
&\leq \abs{w}_0+\dfrac{1}{4}\abs{w}_0+\abs{B(a, 0)}_0\\
&\leq \tau'/4+\tau'/16+\tau'/8=7\tau'/16<\tau'/2. 
\end{split}
\end{equation}
We conclude that  $(a, z,w )\in \Sigma\times \ov{B}(\tau')$ and $w=\delta (a, z).$ 

At this point we can verify the claim that the closure on level $0$ of the set defined in 
\eqref{eq_est_0_5} is compact. With a sequence $(a_n, w_n)\in \{(a, w)\in U\, \vert \, \abs{(h+t)(a, w)}_1<c\}$, we  consider the   corresponding sequence $z_n=w_n-B(a_n, w_n)$ defined by \eqref {eq_est_5}. Then the estimates  \eqref{eq_est_5_5} and \eqref{eq_est_6_5} give 
$$ \abs{z_n}_1\leq 3c \quad  \text{and}\quad  \abs{z_n}_0\leq \tau'/2, $$
which implies that $w_n=\delta (a_n, z_n)$. Since the embedding $E_1\to E_0$ is compact and $a_n\in \R^n$, we conclude, after taking a subsequence,  that $(a_n, z_n)\to (a, z)$ in $E_0=\R^n\oplus W_0$. From the continuity of the map  $\delta$, we deduce the convergence  
$w_n=\delta (a_n, z_n)\to w:=\delta (a, z)$ in $W_0$. Therefore, 
the sequence $(a_n, w_n)$ converges to $(a, w)$ in $E_0$. Since, by assumption, the solution set of $f$ is compact, the proof of the properness of the Fredholm section is complete.

\end{proof}

 \subsubsection{Notes on Partitions of Unity and Bump Functions}\label{POU}

An efficient tool for globalizing local construction in M-polyfolds are sc-smooth partitions of unity.
\begin{definition}
A M-polyfold $X$ {\bf admits sc-smooth partitions of unity} if  for every open covering of $X$ there exists a subordinate sc-smooth partition of unity.
\end{definition}

So far we did not  need sc-smooth partitions of unity for our constructions on M-polyfolds. We would need them, for example, for the construction of sc-connections  M-polyfolds. We point out that a (classically) smooth partition of unity on a Banach space, which is  equipped with a sc-structure,  induces a sc-smooth partition of unity, in view of  Corollary \ref{ABC-y}. However, many Banach spaces do not admit smooth partition of unity subordinate to a given open cover.
Our discussion in this section is based on the survey article \cite{Fry} by Fry  and McManusi on smooth bump functions on Banach spaces. The article contains many interesting open questions.

For many constructions one does not need sc-smooth partitions of unity, but only sc-smooth bump functions.
\begin{definition}
A M-polyfold $X$ admits {\bf admits sc-smooth bump functions} if  for every point $x\in X$
and every open neighborhood $U(x)$ there exists a sc-smooth function $f\colon  X\rightarrow {\mathbb R}$ which is not identically zero and has support in $U(x)$.
\end{definition}

For example, if $E$ is a Hilbert space with a scalar product $\langle \cdot, \cdot \rangle$ and associated norm $\norm{\cdot }$, the map $x\mapsto \norm{x}^2$, $x\in E$, is smooth. Choosing a smooth function $\beta\colon \R\to \R$ of  compact support and satisfying $\beta (0)=1$, the function $f(x)=\beta (\norm{x}^2)$, $x\in E$, is a smooth, non-vanishing function of bounded support. Therefore, if $E$ is equipped with a sc-structure, the map is a sc-smooth function on the Hilbert space in view of Corollary \ref{ABC-y}. Using the Hilbert structure to rescale and translate, we see that the Hilbert space $E$ admits sc-smooth bump functions.

Currently it is an open problem whether the existence of sc-smooth partitions of unity on M-polyfolds $X$ is equivalent to the existence of sc-smooth bump functions. The problem is related to an unsolved classical question in 
Banach spaces: is, in every Banach space, the existence of a single smooth bump function (we can use the Banach space structure to rescale and translate) equivalent 
to the existence of smooth partitions of unity subordinate to given  open  covers? 
For the discussion on this problem we refer to \cite{Fry}.

Sc-smooth bump functions can be used for the construction of functions having special properties, as the following example shows.

\begin{proposition}\label{prop-x5.36}
We assume that the M-polyfold $X$ admits sc-smooth bump functions. Then for every point $x\in X$ and every open neighborhood
$U(x)$ there exists a sc-smooth function $f\colon  X\rightarrow [0,1]$ with $f(x)=1$ and support in $U(x)$. In addition we can choose  $f$ in such a way 
that $f(y)=1$ for all $y$ near $x$.
\end{proposition}
\begin{proof}
By  assumption,  there exists  a sc-smooth bump function $g$ having support in $U(x)$   and satisfying $g(0)=1$. In order to achieve that the image is contained in $[0,1]$ we choose a smooth map $\sigma\colon  {\mathbb R}\rightarrow [0,1]$
satisfying $\sigma(s)=0$ for $s\leq 1$ and $\sigma(s)=1$ for $s\geq 1$ and define  the sc-smooth function $f$ by $f=\sigma\circ g$. If, in addition, we wish $f$ to be constant near $x$,  we take $f(y)=\sigma(\delta\cdot g(y))$ for  $\delta>1$. 
\end{proof}

The survey paper \cite{Fry} discusses, in particular, bump functions on Banach spaces, which are classically differentiable.

\begin{definition} [{\bf $C^k$-bump function}]  A Banach space $E$ 
{\bf admits a $C^k$- bump function}\index{D- $C^k$- bump function},  if there exists a $C^k$-function 
$f\colon  E\rightarrow {\mathbb R}$,  not identically zero and having bounded support.
Here $k\in \{0, 1, 2, \ldots\}\cup \{\infty\}$.
\end{definition}

The existence of $C^k$-bump functions on $L_p$ spaces is a consequence of the following result due to Bonic and Frampton, see Theorem 1 in \cite{Fry}.
\begin{theorem}[Bonic and Frampton]\label{Bonic_Frampton}
For the $L_p$-spaces the following holds for the usual $L^p$-norm $\norm{\cdot }$.
Let $p\geq 1$ and let $\norm{\cdot}$ be the usual norm on $L^p$.  

\begin{itemize}
\item[{\em (1)}] If is an even integer,  then $\norm{\cdot }^p$ is of class $C^\infty$.
\item[{\em (2)}] If is an odd integer, then $\norm{\cdot }^p$ is of class $C^{p-1}$.
\item[{\em (3)}] If $p\geq 1$ is not an integer, then $\norm{\cdot }^p$ is of class $C^{[p]}$, where $[p]$ is the integer part of $p$.
\end{itemize}
\end{theorem}
We deduce immediately for the Sobolev spaces $W^{k,p}(\Omega)$ that  the usual norms 
$$
\norm{u}_{W^{k,p}}^p=\sum_{|\alpha|\leq k} \norm{D^\alpha u}^p_{L_p}.
$$
have  the same differentiability as $\norm{\cdot}_{L_p}$.

Taking a non-vanishing smooth function $\beta\colon \R\to \R$ of compact support, 
the map $f(x)=\beta (\norm{x}^p)$, is a bump function of $L_p$, whose smoothness depends on $p$ as indicated in Theorem \ref{Bonic_Frampton}.

If 
$\Omega$ is a bounded domain in ${\mathbb R}^n$ and $E=W^{1,4}(\Omega)$ is equipped with the sc-structure $E_m=(W^{1+m,4}(\Omega)$, $m\geq 0$, we deduce from (1) in Theorem \ref{Bonic_Frampton} that $E$ admits sc-smooth bump functions.
In contrast, if  $E=W^{1,3/2}(\Omega)$ is equipped with the  sc-structure  $E_m=W^{1+m,3/2}(\Omega)$, then the straightforward bump function constructed by using (3) of Theorem \ref{Bonic_Frampton} would only be of class $\ssc^1$. Does there exist a sc-smooth bump function on $E$, i.e., sc-smooth and of bounded support in $E$?

The existence of sc-smooth bump functions is a local property.
\begin{definition}
A local M-polyfold model $(O,C,E)$ has the {\bf sc-smooth bump function property} if,   for every $x\in O$ and every open neighborhood
$U(x)\subset O$ satisfying  $\cl_E(U(x))\subset O$,  there exists a  sc-smooth function $f\colon  O\rightarrow {\mathbb R}$ satisfying  $f\neq 0$ and $\supp(f)\subset U(x)$.
\end{definition}

Clearly,  the following holds.
\begin{theorem}
A M-polyfold $X$ admits sc-smooth bump functions if and only if it admits a sc-smooth atlas whose  the local models have the sc-smooth bump function property.
\end{theorem}

The following class of spaces have the sc-bump function property.
\begin{proposition}
Assume that $(O,C,E)$ is a local M-polyfold model in which the $0$-level $E_0$ of the sc-Bananch space $E$ is a Hilbert space. 
Then $(O,C,E)$ has the sc-bump function property.
\end{proposition}

\begin{proof}
Let $\langle \cdot ,\cdot \rangle$ be the inner product and $\norm{\cdot}$ the associated norm of $E_0$. We choose a smooth function $\beta\colon \R\to \R$ of compact support and satisfying $\beta (0)=1$. Then the function $f(x)=\beta (\norm{x}^2)$, $x\in E$, defines, in view of Corollary \ref{ABC-y}, a sc-smooth function on $E$. Using scaling, translating, and composition with the sc-smooth retraction onto $O$ the proposition follows.
\end{proof}


Next we study the question of the existence of sc-smooth partitions of unity.
By definition, a M-polyfold $X$ is paracompact and, therefore, there exist  continuous partitions of unity. Hence  it is not surprising that the existence of sc-smooth partitions of unity is connected to local properties of $X$,  namely to   the local  approximability 
of continuous functions  by sc-smooth functions.

\begin{definition}\label{approximation_property}
A local M-polyfold model $(O,C,E)$ has the {\bf sc-smooth approximation property} provided the following holds.
Given $(f,V,\varepsilon)$,  where $V$ is an open subset of $O$ such that $\cl_C(V)\subset O$,  $f\colon  O\rightarrow [0,1]$ is a continuous function with support contained  $V$,  and $\varepsilon>0$,
there exists a sc-smooth map $g\colon  O\rightarrow [0,1]$ supported  in  $V$  and satisfying $\abs{f(x)-g(x)}\leq \varepsilon$ for all $x\in O$.
\end{definition}

\begin{theorem}\label{partition_approximation}
The following two statements are equivalent.
\begin{itemize}
\item[{\em (i)}] A M-polyfold $X$ admits sc-smooth partitions of unity subordinate to any given open cover.
\item[{\em (ii)}] The M-polyfold $X$ admits an atlas consisting of local models having the sc-smooth approximation property.
\end{itemize}
\end{theorem}

\begin{proof} 
Let us first show that (i) implies (ii). For the M-polyfold $X$,  we take an atlas of M-polyfold charts $\phi\colon U\to O$. We shall show that the local models $(O, C, E)$ posses  the sc-smooth approximation property. Let $(f,V,\varepsilon)$ be as in Definition \ref{approximation_property}. 
The sc-smooth function $f\circ \phi$ is defined on $\phi^{-1}(O)\subset X$ and we extend it by $0$ to all of $X$ and obtain a continuous function 
$g\colon X\to [0,1]$ whose  support is contained in the open set $W=\phi^{-1}(V)$, satisfying $\cl_X(W)\subset U$. 
Define the open subset $\wt{W}$ of $X$ by
$$
\wt{W}=\{x\in W\, \vert \, g(x)>\varepsilon/4\}.
$$
Then $\cl_X(\wt{W})\subset W$.
For every $x\in \cl_X(\wt{W})$,  there exists open neighborhood $U_x$ of $x$ such that
$\cl_X(U_x)\subset W$ and $\abs{g(x)-g(y)} <\varepsilon/2$ for all $y\in U_x$. Take the  open cover of $X$ consisting of $U_0=\{x\in X\, \vert \, g(x)<\varepsilon/2\}$ and $(U_x)$, $x\in \cl_X(\wt{W})$. By assumption, there exists a subordinate sc-smooth partition of unity consisting of $\beta_0$ with the support in $U_0$ and $(\beta_x)$ supported in $(U_x)$,  $x\in \cl_X(\wt{W})$.  Then we  define the function $\wh{g}\colon  X\rightarrow [0,1]$ by 
$$
\wh{g}(y)=\sum_{x\in \cl_X(\wt{W})} \beta_x(y)g(x).
$$
Since the collection of supports of $\beta_0$ and $(\beta_x)$ is locally finite, the sum is locally finite and, therefore, $\wh{g}$ is sc-smooth.

We claim that $\abs{g(x)-\wh{g}(x)}<\varepsilon$ for all $x\in X$. In order to show this we first show that $\supp (\wh{g})\subset W$. We assume that $y\in \supp (\wh{g})$ and let 
$(y_k)$ be a sequence  satisfying $\wh{g}(y_k)>0$ and $y_k\rightarrow y$. Since $\wh{g}(y_k)>0$, every open neighborhood of $y$ intersects $\supp (\beta_x)$ for some $x\in \cl_X(W)$. Since  the collection of supports $(\supp (\beta_x))$ is locally finite, there exists an open neighborhood $Q=Q(y)$ of $y$ in $X$ and finitely many points $x_1,\ldots ,x_m$ such that only the supports of the functions $\beta_{x_i}$ intersect  $Q$.
Hence
$$
0<\wh {g}(y_k) =\sum_{i=1}^m \beta_{x_i}(y_k)g(x_i), 
$$
showing  that  $y_k\subset \bigcup_{i=1}^m \cl_X(U_{x_i})\subset W$ for large $k\geq 1$.  Consequently,  $y\in W$ and hence $\supp (\wh{g})\subset W$.

Next, given $z\in U$,  there exists an  open neighborhood $Q=Q(z)$ of $z$ in $X$  such that 
$Q$ intersects only finitely many supports of functions belonging to the partition of unity.
If $Q$ intersects the support of $\beta_0$ and none of the supports of the functions $\beta_x$, then $Q\subset U_0$ and $\wh{g}(y)=0$ for all $y\in Q$. Consequently, 
$$\abs{g(y)-\wh{g}(y)}=\abs{g(y)}<\varepsilon/2<\varepsilon$$
for all $y\in Q$. If $Q$ intersects supports of the functions $\beta_x$, there are finitely many points $x_1,\ldots, x_l$ such that only the supports of the functions $\beta_{x_i}$ intersect  $Q$.
Hence 
$$
\wh{g}(y)=\sum_{i=1}^l \beta_{x_i}(y)g(x_i)\quad  \text{for all $y\in Q$}.
$$
Since $\abs{g(x)-g(y)}<\varepsilon/2$ if $y\in U_x$ and $g(y)<\varepsilon/2$ if $y\in U_0$,  we conclude that 
\begin{equation*}
\begin{split}
|g(y)-\wh{g}(y)|
&=\abs{\beta_0(x)g(y)+\sum_{i=1}^l\beta_{x_i}(y)g(y)-\sum_{i=1}^l \beta_{x_i}(y)g(x_i)}\\
&< \varepsilon/2 +\sum_{i=1}^l \beta_{x_i}(z)\abs{g(y)-g(x_i)}\\
&\leq \varepsilon/2 +(\varepsilon/2)\sum_{i=1}^l \beta_{x_i}(y)< \varepsilon.
\end{split}
\end{equation*}
for all $y\in Q$. 
Consequently, $\abs{g(y)-\wh{g}(y)}< \varepsilon$ for all $y\in U$ and 
$\abs{f(x)-\wh{g}\circ\phi^{-1}(x)}< \varepsilon$ for $x\in O$. This completes the proof that (i) implies (ii).

Next we show that (ii) implies (i).  For the M-polyfold $X$ we take an atlas of M-polyfold charts $\phi_\tau\colon U_\tau \to O_\tau$, $\tau \in T$, where the local models $(O_\tau, C_\tau, E_\tau)_{\tau \in T}$ posses  the sc-smooth approximation property. We assume that ${(V_\lambda)}_{\lambda\in\Lambda}$ is an open cover of $X$.
Then there exists a refinement $(W_\lambda)_{\lambda \in \Lambda}$ (some sets may be empty) with the same index set $\Lambda$, which is locally finite so that $W_\lambda\subset V_\lambda$ and  for every $\lambda\in \Lambda$ there exists an index $\tau(\lambda)\in T$ such that $\cl_X(W_\lambda)\subset  U_{\tau(\lambda)}$. 
Since $X$ is metrizable, we find open sets 
$Q_\lambda$ such  that
$$
 Q_\lambda\subset \cl_X(Q_\lambda)\subset W_\lambda
$$
and $(Q_\lambda)_{\lambda \in \Lambda}$ is a locally finite open cover of $X$. Using again  the metrizability  of $X$, we find continuous functions  $f_\lambda\colon  X\rightarrow [0,1]$ satisfying
$$
f\vert Q_\lambda\equiv 1\quad \text{and} \quad  \supp(f_\lambda)\subset W_\lambda.
 $$
 Let $\varepsilon=1/2$, $V_\lambda'=\phi_{\tau(\lambda)}(W_\lambda)$,   and $f_\lambda'=f_\lambda\circ\phi^{-1}_{\tau(\lambda)}$.   In view of the hypothesis (ii), for the triple $(f_\lambda', V_\lambda', 1/2)$ there exists  a sc-smooth function $g_\lambda\colon  O_{\tau(\lambda)}\rightarrow [0,1]$ having support in $V_\lambda'$ and satisfying 
 $$
\abs{ f_\lambda'(x)-g_\lambda(x)}< 1/2\quad  \text{for all $x\in O_{\tau(\lambda)}$}.
$$
Going back to $X$ and extending $g_\lambda\circ \phi_{\tau (\lambda)}^{-1}$ onto $X$ by $0$ outside of $W_\lambda$,  we obtain the sc-smooth functions $\wh{g}_\lambda\colon  X\rightarrow [0,1]$ satisfying
$\wh{g}_{\lambda}(x)>0$ for $x\in Q_\lambda$. Then we define 
$\gamma_\lambda\colon  X\rightarrow [0,1]$ by
$$
\gamma_\lambda(x)=\frac{\wh{g}_\lambda(x)}{\sum_{\lambda\in\Lambda} \wh{g}_{\lambda}(x)}.
$$
The family of functions  $(\gamma_\lambda)_{\lambda \in \Lambda}$ is the desired sc-smooth partition of unity subordinate to the given open cover $(V_\lambda)$ of $X$. 

\end{proof}

An immediate consequence of Theorem \ref{partition_approximation} is the following result.

\begin{proposition}
Assume that the sc-Banach space admits smooth partitions of unity. Then a local model $(O, C, E)$ has the sc-smooth approximation property.
\end{proposition}


The result below, due to Tor\'unczyk (see \cite{Fry}, Theorem 30),  gives a complete characterization of Banach spaces admitting $C^k$-partitions of unity.
This criterion reduces the question to a problem  in the geometry of Banach spaces. Though this criterion is not easy to apply it serves as one of the main tools in the investigation of the question, see \cite{Fry}. In order to formulate the theorem we need a definition.

\begin{definition}
If  $\Gamma$ is a set, we denote by $c_0(\Gamma)$ the Banach space of functions $f\colon  \Gamma\rightarrow {\mathbb R}$ having the property that for every $\varepsilon>0$
the number of $\gamma\in\Gamma$ with $\abs{f(\gamma)}>\varepsilon$ is finite.  The vector space operations are obvious and the norm is defined by
$$
\abs{f}_{c_0}=\text{max}_{\gamma\in\Gamma} \abs{f(\gamma)}.
$$
A  homeomorphic embedding $h\colon  E\rightarrow c_0(\Gamma)$ is {\bf coordinate-wise $C^k$}, \index{Coordinate-wise $C^k$-embedding}
if for every $\gamma\in\Gamma$ the map $E\rightarrow {\mathbb R}$,  $e\rightarrow (h(e))({\gamma})$ is $C^k$.
\end{definition}

If $\Gamma=\N$, then 
$c_0({\mathbb N})$ is the usual space $c_0$ of sequences converging to $0$.
\begin{theorem}[Tor\'unczyk's Theorem]
A Banach space $E$ admits a $C^k$-partition of unity if and only if there exists a set $\Gamma$ and a coordinate-wise $C^k$ homeomorphic embedding
of  $E$ into $c_0(\Gamma)$. 
\end{theorem}

An important class of Banach spaces are those which are are weakly compactly generated. They have good smoothness properties and  will provide
us with examples of sc-Banach spaces admitting sc-smooth partitions of unity.

\begin{definition}
A Banach space $E$ is called  {\bf weakly compactly generated} (WCG) if there exists a weakly compact set $K$ in $E$ such  that the closure of the span of $K$ is the whole space, 
$$
E=\cl_E(\text{span}(K)).
$$
\end{definition}

There are  two  useful examples of WCG Banach spaces.
\begin{proposition}
Let $E$ be a Banach space.
\begin{itemize}
\item[{\em (i)}] If $E$ is reflexive, then $E$ is WCG.
\item[{\em (ii)}] If $E$ is separable, then $E$ is WCG.
\end{itemize}
\end{proposition}
\begin{proof}
In case that $E$ is reflexive it is known that the closed unit ball ${B}$ is  compact. in the weak topology. Clearly,  $E=\text{span}(B)$.
If $E$ is separable, we  take a dense sequence ${(x_n)}_{n\geq  }$ in the unit ball and 
define $K=\{0\}\cup\{\frac{1}{n}x_n\, \vert \,  n\geq1\}$. Then $K$ is compact and,  in particular,  weakly compact. 
\end{proof}
The usefulness of WCG spaces lies in the following result from \cite{GTWZ}, see also \cite{Fry},  Theorem 31. 
\begin{theorem}[\cite{GTWZ}]\label{wcg_partition}If the  WCG-space $E$ admits a $C^k$-bump function, then it also admits  $C^k$-partitions of unity.
\end{theorem}


\begin{corollary}
Let $(O,C,E)$ be a local model, where $E_0$ is a Hilbert space. Then $(O,C,E)$ has the sc-smooth approximation property.
\end{corollary}
\begin{proof}
A Hilbert space is reflexive and hence a WCG-space. 
We have already seen that a Hilbert space equipped with a sc-structure admits sc-smooth bump functions and conclude from 
Theorem \ref{wcg_partition} that it admits sc-smooth partition of unity, and consequently has the smooth approximation property, in view of Theorem 
\ref{partition_approximation}.
\end{proof}

\pagebreak
\section{Linearizations, Orientations, and Invariants}

In this chapter we introduce the notion of a linearization of a sc-Fredholm section and  discuss orientations and invariants associated to proper sc-Fredholm sections. We refer the reader to \cite{DK} for some of the basic ideas around determinants of linear Fredholm operators,
and to \cite{FH} for applications of the more classical ideas to problems arising in symplectic geometry, i.e. linear Cauchy-Riemann type operators. In polyfold theory, the central issue is that the occuring linear Fredholm operators, which are linearizations of nonlinear sections,
 do in general not depend as operators continuously on the points where the linearization was taken. On the other hand there is some weak continuity
property which allows to introduce determinant bundles. However, it is necessary to develop some new ideas.

\subsection{Linearizations of Sc-Fredholm Sections}

Let $P\colon Y\rightarrow X$ be a strong bundle over the  M-polyfold $X$ and $f$ a sc-smooth section of $P$.

If $x$ is a smooth point in $X$ and $f(x)=0$, there exists a well-defined {\bf linearization}\index{Linearization} 
$$f'(x)\colon T_xX\rightarrow Y_x$$
which is a sc-operator.  In order to recall the definition, we identify, generalizing 
a classical fact of vector  bundles, the tangent space $T_{0_x}Y$ at the element $0_x$ with the sc-Banach space $T_xX\oplus Y_x$ where $Y_x=P^{-1}(x)$ is the fiber over $x$. Denoting by $P_x\colon T_xX\oplus Y_x\to Y_x$ the sc-projection, the linearization of $f$ at the point $x$ is the following operator,
$$f'(x):=P_x\circ Tf(x)\colon T_xX\to Y_x.$$

As in the case of vector bundles there is, in general, no intrinsic notion of a  linearization of the section $f$ at the smooth point $x$ if $f(x)\neq 0$. However, dealing with a strong bundle we can profit from the additional structure. We simply take a local $\ssc^+$-section $s$ defined near $x$ and satisfying $s(x)=f(x)$, so that the linearization 
$$(f-s)'(x)\colon T_xX\to Y_x$$
is well-defined. To find such a $\ssc^+$-section we take a strong bundle chart around $x$. Denoting the sections in the local charts by the same letters, we let $R$ be the local strong bundle retraction associated with the local strong bundle. It satisfies $R (x, f(x))=f(x)$ and we define the desired section $s$ by $s(y)=R(y, f(x))$.  Since $f(x)$ is a smooth point, $s$ is a $\ssc^+$-section satisfying $s(x)=f(x)$ at the distinguished point $x$, as desired.

If $t$ is another $\ssc^+$-section satisfying $t(x)=f(x)=s(x)$, then 
$$(f-s)'(x)=(f-t)'(x)+(t-s)'(x)$$
and the linearization $(t-s)'(x)$ is a $\ssc^+$-operator. It follows from Proposition \ref{prop1.21}, that $(f-s)'(x)$ is a sc-Fredholm operator if and only if $(f-s)'(x)$ is a sc-Fredholm operator, in which case their Fredholm indices agree because a $\ssc^+$-operator is level  wise a compact operator.

Let now $f$ be a sc-Fredholm section of the bundle $P$ and $x$ a smooth point in $X$. Then there exists, by definition, a $\ssc^+$-section $s$ satisfying $s(x)=f(x)$ and, moreover, $(f-s)'(x)$ is a sc-Fredholm operator.

\begin{definition}\index{D- Space of linearizations}\index{$\text{Lin}(f,x)$}
If $f$ is a sc-Fredholm section $f$ of the strong bundle $P\colon Y\rightarrow X$ and $x$ a smooth point in $X$, then the {\bf space of linearizations} of $f$ at $x$ is the set of sc-operators from $T_xX$ to $Y_x$ defined as 
$$
\text{Lin}(f,x)=\{(f-s)'(x)+a\, \vert \,  \text{$a\colon T_xX\to Y_x$ is a $\ssc^+$-operator}\}.
$$
\end{definition}

The operators in $\text{Lin}(f,x)$
all differ by  linear $\ssc^+$-operators and are all sc-Fredholm operators having the same Fredholm index. 
This allows us to define
the {\bf  index of the sc-Fredholm germ $(f,x)$} \index{Index of sc-Fredholm germ}
by
$$
\text{ind}(f,x):=\dim \ker \bigl((f-s)'(x)\bigr)-\dim\bigl(Y_x/(\text{Im}(f-s)'(x))\bigr).\index{$\text{ind}(f,x)$}
$$
We shall show that this index is locally constant. The proof has to cope with the difficulty caused by the fact that, in general, the linearizations do not depend continuously as operators on the smooth point $x$.

Recall that a M-polyfold
is locally path connected, and that,  moreover,  any two smooth points in the same path component can be connected by a $\ssc^+$-smooth path $\phi\colon [0,1]\rightarrow X$.
\begin{proposition}[Stability of $\text{ind}(f,x)$]\index{P- Stability of $\text{ind}(f,x)$}\label{sst}
Let $P\colon Y\rightarrow X$ be a strong bundle over the tame M-polyfold $X$ and $f$  a sc-Fredholm section. If 
$x_0$ and $x_1$ are smooth points in $X$ connected by a sc-smooth path $\phi\colon [0,1]\rightarrow X$, then
$$
\text{ind}(f,x_0)=\text{ind}(f,x_1).
$$
\end{proposition}
\begin{proof}

{
We shall show that the map $t\mapsto  \text{ind}(f,\phi(t))$ is locally constant. The difficulty is that the linearizations, even if picked sc-smoothly will,  in general,  not depend as operators continuously on $t$. On top of it we have possibly varying  dimensions of the spaces so that we need to change the filled version at every point. However, one can prove the result with a trick, which will also be used 
in dealing with orientation questions later on. We consider the tame M-polyfold $[0,1]\times X$ and consider the graph of the path $\phi$. We fix $t_0\in [0,1]$ and choose  a locally defined $\ssc^+$-section $s(t, x)$ satisfying $s(t,\phi(t))=f(\phi(t))$ for  $(t,x)\in [0,1]\times X$ near  $(t_0,\phi(t_0))$. (We do not need a sc-smooth partition of unity. If we had one  available then we could define such a section $s$ which satisfies  $s(t,\phi(t))=f(\phi(t))$ for all $t\in [0,1]$.) We choose  finitely many smooth points $e_1,\ldots ,e_m$
such  that the image of $(f-s(t_0,\cdot ))'(\phi(t_0))$ together with the $e_i$ span $Y_{\phi(t_0)}$.  Next we take a smooth finite-dimensional
linear subspace $L$ of $T_{\phi(t_0)}X$ which has a sc-complement in $T^R_{\phi(t_0)}X$. Then the image
of $L$ under $(f-s(t_0,\cdot ))'(\phi(t_0))$ is a smooth finite-dimensional subspace of $Y_{\phi(t_0)}$ of dimension $r$, say.
We choose  smooth vectors $p_1,\ldots ,p_r$ spanning this space. Next we take $m+r$ many $sc^+$-sections depending on $(t,x)$ (locally defined)
so that at $(t_0,\phi(t_0))$ they take the different values $e_1,\ldots ,e_m$ and $p_1,\ldots ,p_r$. Now the section
$$
F(\lambda,t,x)= f(x)-s(t,x)+\sum_{i=1}^{m+p} \lambda_i\cdot s_i(t,x)
$$
is defined near $(0,t_0,\phi(t_0))$ and takes values in $Y$. The linearization at $(0,t_0,\phi(t_0))$ with respect to the first and second variable
is surjective  and the kernel of the linearization has a complement contained in 
$$
T_{(0,t_0,\phi(t_0))}^R({\mathbb R}^{m+p}\oplus [0,1]\oplus X)={\mathbb R}^{m+p}\oplus T^R_{t_0}[0,1]\oplus T^R_{\phi(t_0)}X.
$$
Hence $\ker(F'(0,t_0,\phi(t_0)))$ is in good position to the boundary. Employing  the implicit function theorem
for the boundary case, we  obtain a solution manifold $S$ of $F=0$ containing $(0,t,\phi(t))$ for $t\in [0,1]$ close to $t_0$.
Moreover,  if  $(\lambda,t,x(t))\in S$, then  $\ker(F'(\lambda,t,\phi(t)))=T_{(\lambda,t,\phi(t))}S$ and
$F'(\lambda,t,\phi(t))$ is surjective. Hence 
$$
t\mapsto  \dim(T_{(0,t,\phi(t))}S)
$$
is locally constant for $t\in [0,1]$ near $t_0$. By construction, 
\begin{equation*}
\begin{split}
\text{ind}(f,\phi(t_0))+m+p&=\text{ind}(F,(0,t_0,\phi(t_0)))\\
&=\text{ind}(F,(0,t,\phi(t)))=\text{ind}(f,\phi(t))+m+p.
\end{split}
\end{equation*}
Therefore, $\ind (f,\phi (t_0))=\ind (f,\phi (t))$ for all $t$ near $t_0$.
}
\end{proof}

\subsection{Linear Algebra and Conventions}\label{sect_conventions}

In he following  we  are concerned with the orientation  which is  crucial in our applications.
We follow  to a large extent the  appendix in \cite{HWZ5}. The ideas of the previous proof are also useful
in dealing with orientation questions. There we did not use sc-smooth partitions of unity.  {\bf In the following, however,}
{\bf we shall assume the existence of sc-smooth partitions of unity} to simplify the presentation, but the proofs
could be modified arguing as in the index stability theorem.

We begin with standard facts about determinants and wedge products.
Basically all the constructions are natural, but  usually depend on conventions, which have  to be stated apriori.
Since different authors  use different conventions,  their natural isomorphisms can be different.  To avoid these difficulties we state our conventions carefully. 
We also would like to point out that A. Zinger has written a paper dealing with these type of issues, \cite{Zinger}.
He also describes  some of the mistakes occurring in the literature as well as deviating conventions by different authors.
Since the algebraic treatment of SFT (one of the important applications of the current theory)
relies on the orientations
of the moduli spaces and the underlying conventions we give a comprehensive treatment of orientation questions.

Using  the notation introduced by Zinger in  \cite{Zinger}, we define 
$$
\lambda(E):=\Lambda^{max} E\quad   \text{and}\quad   \lambda^\ast(E):=(\lambda(E))^\ast,\index{$\lambda(E)$}\index{$\lambda^\ast(E)$}
$$
where $(\lambda (E))^\ast$ is the dual of the vector space $\lambda(E)$.

A linear map $\Phi\colon E \rightarrow F$ between finite-dimensional vector spaces of the same dimension  induces the linear map
$$
\lambda(\Phi)\colon \lambda(E)\rightarrow\lambda (F), \index{$\lambda(\Phi)$}
$$
defined by $\lambda(\Phi)(a_1\wedge\ldots \wedge a_n):= \Phi(a_1)\wedge\ldots \wedge \Phi(a_n).
$
The  map $\lambda(\Phi)$ is nontrivial if and only if $\Phi$ is an isomorphism. The dual map $\Phi^\ast\colon F^\ast\to E^\ast$ of $\Phi\colon E \rightarrow F$ induces the map 
$$\lambda (\Phi^\ast)\colon \lambda (F^\ast )\to \lambda (E^\ast).$$
Moreover, we denote by 
$$\lambda^\ast  (\Phi)\colon \lambda^\ast  (F )\to \lambda^\ast  (E)$$
the dual of the map $\lambda(\Phi)\colon \lambda(E)\rightarrow\lambda(F)$.

The composition of  the two maps 
$$
E\xrightarrow{\Phi}F\xrightarrow{\Psi} G
$$
between vector spaces of the  same dimension satisfies 
$$
\lambda(\Psi\circ\Phi) =\lambda(\Psi)\circ \lambda(\Phi).
$$



{
There are different canonical isomorphisms
$$
\lambda ( E^\ast) \rightarrow \lambda^\ast( E)
$$
depending on different conventions. Our convention is the following.
}

{
\begin{definition}
If $E$ is a finite-dimensional real vector space and $E^\ast$ its dual,  the {\bf natural isomorphism}
$$
\iota\colon \lambda(E^\ast)\rightarrow\lambda^\ast(E)\index{$\iota\colon \lambda(E^\ast)\rightarrow\lambda^\ast(E)$}
$$
is defined by 
$$\iota( e_1^\ast\wedge\ldots \wedge e_n^\ast)(a_1\wedge \ldots \wedge a_n)=\det (e_i^\ast(a_j)),
$$
where $n=\dim(E)$. If $n=0$, we set 
$\lambda( E^\ast) =\lambda(\{0\}^\ast)=\R$.  
\end{definition}
}

{
From this definition we deduce  for a basis $e_1,\ldots ,e_n$ of $E$ and its dual basis $e_1^\ast,\ldots ,e_n^\ast$  the formula
$$
\iota(e^\ast_1\wedge\ldots \wedge e^\ast_n)=(e_1\wedge\ldots \wedge e_n)^\ast,
$$
where the dual vector $v^\ast$ of  a vector $v\neq 0$ in a one-dimensional vector space is determined by $v^\ast(v)=1$. 
}

{
Indeed, 
$$
(e_1\wedge\ldots \wedge e_n)^\ast (e_1\wedge\ldots \wedge e_n)=1= \det((e_i^\ast(e_j))=\iota(e_1^\ast\wedge\ldots \wedge e_n^\ast)(e_1\wedge\ldots \wedge e_n).
$$
}

{
The definition of $\iota$ is compatible with the previous definition of induced maps.
}

\begin{proposition}\index{P- Naturality of $\iota$}
If  $\Phi\colon E\rightarrow F$ is an  isomorphism between two finite-dimensional vector spaces and $\Phi^\ast\colon F^\ast\rightarrow E^\ast$ is its dual, then the following diagram is commutative, 
$$\begin{CD}
\lambda(F^\ast)@>\lambda(\Phi^\ast)>> \lambda(E^\ast)\\
@VV\iota V   @VV\iota V\\
\lambda^\ast(F)@>\lambda^\ast(\Phi)>> \lambda^\ast(E)
\end{CD}$$
\end{proposition}

\begin{proof}

{
Let $f_1,\ldots ,f_n$ be a basis of $F$ and $f_1^\ast,\ldots ,f^\ast_n$ its dual basis of $F^\ast$. Then we define the basis $e_1,\ldots ,e_n$  of $E$ by  $\Phi(e_i)=f_i$.
Its dual basis in $E^\ast$ is given by 
$e_1^\ast=f_1^\ast \circ \Phi,\ldots , 
e_n^\ast=f_n^\ast \circ \Phi$ and we compute,
\begin{equation*}
\begin{split}
&(\lambda^\ast  (\Phi)\circ\iota(f_1^\ast\wedge\ldots \wedge f^\ast_n))(e_1\wedge\ldots \wedge e_n)\\
&\quad= (\lambda (\Phi^\ast)((f_1\wedge\ldots \wedge f_n)^\ast))(e_1\wedge\ldots \wedge e_n)\\
&\quad = (f_1\wedge\ldots \wedge f_n)^\ast\circ \lambda(\Phi)(e_1\wedge\ldots \wedge e_n)\\
&\quad=(f_1\wedge\ldots \wedge f_n)^\ast(f_1\wedge\ldots \wedge f_n)\\
&\quad= 1.
\end{split}
\end{equation*}
Similarly,
\begin{equation*}
\begin{split}
&( \iota\circ\lambda(\Phi^\ast)(f_1^\ast\wedge\ldots \wedge f_n^\ast))(e_1\wedge\ldots \wedge e_n)\\
&\quad = (\iota( f_1^\ast\circ\Phi\wedge\ldots \wedge f_n^\ast\circ\Phi))(e_1\wedge\ldots \wedge e_n)\\
&\quad =(\iota(e_1^\ast\wedge\ldots \wedge e_n^\ast))(e_1\wedge\ldots \wedge e_n)\\
&\quad = {(e_1\wedge\ldots \wedge e_n)}^\ast(e_1\wedge\ldots \wedge e_n)\\
&\quad =1.
\end{split}
\end{equation*}
Hence
$
\lambda^\ast(\Phi)\circ \iota=\iota\circ \lambda(\Phi^\ast)
$
and the commutativity of the diagram is proved. 
}
\end{proof}

Next we consider the  exact sequence ${\bf E}$ of finite-dimensional linear vector spaces, \index{${\bf E}$, exact sequence}
$$
{\bf E}:\quad  0\rightarrow A\xrightarrow{\alpha} B\xrightarrow{\beta} C\xrightarrow{\gamma} D\rightarrow 0.
$$
We recall that, by definition, the sequence is exact at $B$, for example, if  $\im (\alpha)=\ker (\beta)$. We deal with the exact sequence as follows.

We take  a complement $Z\subset B$ of $\alpha (A)$ so that $B=\alpha (A)\oplus Z$, and 
a complement $V\subset C$ of $\beta (B)$ so that  $C=\beta (B)\oplus V$. Then the exact sequence ${\bf E}$ becomes 
$$
{\bf E}:\quad  0\rightarrow A\xrightarrow{\alpha} \alpha (A)\oplus Z \xrightarrow{\beta} \beta (B)\oplus V\xrightarrow{\gamma} D\rightarrow 0.$$
Here the first  nontrivial map is $a\mapsto (\alpha (a), 0)$, the second is $(b, z)\mapsto 
(\beta (z), 0)$, and the third is $(v, c)\mapsto \gamma (v)$. The maps $\alpha\colon A\to \alpha (A)$, $\beta \colon Z\to \beta (Z)$, and $\gamma \colon V\to D$ are isomorphisms. 

From the exact sequence ${\bf E}$ we are going  to construct several natural isomorphisms,  fixing again some  conventions.
The first natural isomorphism is the isomorphism
$$
\Phi_{\bf E}\colon \lambda(A)\otimes\lambda^\ast(D)\rightarrow \lambda(B)\otimes\lambda^\ast(C)\index{$\Phi_{\bf E}$}
$$
constructed as follows.

{
We abbreviate $n=\dim(A)$, $m=\dim(B)$, $k=\dim(C)$, and $l=\dim(D)$.  
}


{
$\Phi_{\bf E}$ maps $0$ to $0$. 
Next we take nonzero vector
$$
h:=(a_1\wedge\ldots \wedge a_n)\otimes (d_1\wedge\ldots \wedge d_l)^\ast\in \lambda (A)\otimes \lambda^\ast (D),$$
where  $a_1,\ldots ,a_n$ is a basis of  $A$ and $d_1,\ldots ,d_l$ is a basis for $D$.  Then we define the basis $b_1, \ldots ,b_n$ of $B$ by $b_i=\alpha (a_i)$ and the basis $c_1,\ldots ,c_l$ of $V$ by $\gamma (c_i)=d_i$, $i=1,\ldots ,l$.  Now we  choose a basis $b_1',\ldots ,b_{m-n}'$ of $Z$ and define the basis $c_1',\ldots ,c'_{m-n}$ of $\beta (B)\subset C$ by $c_i'=\beta (b_i')$.
}

Finally, we define $\Phi_{\bf E}(h)\in\lambda(B)\otimes\lambda^\ast(C)$ by 
\begin{equation}\label{eq_Phi_E}
\begin{split}
&\Phi_{\bf E}((a_1\wedge \ldots \wedge a_n)\otimes (d_1\wedge \ldots \wedge d_l)^\ast)\\
& = (\alpha(a_1)\wedge\ldots \wedge\alpha(a_n)\wedge b_1'\wedge\ldots \wedge b_{m-n}')\otimes (c_1\wedge\ldots \wedge c_l\wedge c_1'\wedge\ldots \wedge c_{m-n}')^\ast.
\end{split}
\end{equation}

The {\bf two conventions}  here are  that 
$b_1',\ldots ,b_{m-n}'$ are listed after the
$\alpha (a_1),\ldots ,\alpha (a_n)$ and
$c_1',\ldots ,c_{m-n}'$ are listed after $c_1,\ldots ,c_l$.
Apart from these two conventions how to list the vectors, the resulting definition does not depend on the choices involved.
\begin{lemma}\label{o6.6}\index{L- Well-definedness of $\Phi_{\bf E}$}
With the above two conventions, the definition of $\Phi_{\bf E}$ does not depend on the choices.
\end{lemma}
The  proof is carried out  in Appendix \ref{oo6.6}.


{
Associated with the exact sequence ${\bf E}$ there exists also a second natural isomorphism\index{$\Psi_{\bf E}$}
\begin{equation}\label{secon_iso}
\Psi_{\bf E}\colon  \lambda(C)\otimes\lambda(A)\otimes\lambda^\ast(D)\rightarrow \lambda(B)
\end{equation}
constructed as follows.
}

{
We first map  
$(c_1\wedge\ldots \wedge c_k)\otimes (a_1\wedge\ldots \wedge a_n)\otimes(d_1\wedge\ldots \wedge d_l)^\ast$ into
$(c_1\wedge\ldots \wedge c_k)\otimes\Phi_{\bf E}((a_1\wedge\ldots \wedge a_n)\otimes(d_1\wedge\ldots \wedge d_l)^\ast)$ which belongs to  $\lambda(C)\otimes\lambda(B)\otimes\lambda^\ast(C)$
and then we compose this map  with the  isomorphism
$$
\bar{\iota}\colon \lambda(C)\otimes\lambda(B)\otimes\lambda^\ast(C)\to \lambda (B), $$
defined by $v\otimes b\otimes v^\ast\rightarrow b.$ Since $\Phi_{\bf E}$ is well-defined, so is $\Psi_{\bf E}$. For convenience we present a more explicit formula for  $ \Psi_{\bf E}$, using the notations of $\Phi_{\bf E}$.
\begin{proposition}\label{oger}\index{P- Well-definedness of $\Psi_{\bf E}$}
We fix the basis  $a_1,\ldots ,a_n$ of $A$ and 
the basis $d_1,\ldots ,d_l$ of  $D$ and abbreviate their wedge products by $a$ and by $d$. In the complement of $\beta(B)\subset C$ we have the basis  $c_1,\ldots ,c_l$ defined by $\gamma(c_i)=d_i$.
The vectors  $b_1,\ldots ,b_n\in B$ defined by
$b_i=\alpha(a_i)$ are a basis of  of $\alpha (A)\subset B$, and in the complement of $\alpha(A)$ in $B$ we choose the basis  a $b_1',\ldots ,b_{m-n}'$ and define the basis $c_1',\ldots ,c_{m-n}'$ of $\beta (B)\subset C$ by $c_i'=\beta (b_i')$. Abbreviating 
$c=c_1'\wedge\ldots \wedge c_l'\wedge c_1\wedge\ldots \wedge c_{m-n}$, we obtain the formula
$$
\Psi_{\bf E}(c\otimes a\otimes d^\ast)= b_1\wedge\ldots \wedge b_n\wedge b_1'\wedge\ldots \wedge b_{m-n}'\in \lambda (B).
$$
\end{proposition}
}

{
\begin{proof}
We already know from Lemma \ref{o6.6} that $\Psi_{\bf E}$ is well-defined since $\Phi_{\bf E}$ is well-defined. By construction, 
$$
\Phi_{\bf E}(a\otimes d^\ast) = (b_1\wedge\ldots \wedge b_n\wedge b_1'\wedge\ldots \wedge b_{m-n}')\otimes (c_1'\wedge\ldots \wedge c'_l\wedge c_1\wedge\ldots \wedge c_{m-n})^\ast.
$$
Here $c_i=\beta(b_i)$ and $\gamma(c_i') =d_i$. Then,  abbreviating the wedge product of the vectors $b_1,\ldots b_n$, 
$b_1', \ldots, b_{n-m}'$ by $b$, we obtain 
\begin{equation*}
\Psi_{\bf E}(c\otimes a\otimes d^\ast)
=\bar{\iota}(c\otimes b\otimes c^\ast)
=b.
\end{equation*}
\end{proof}
}

\subsection{The Determinant of a Fredholm Operator}

For the determinants of sc-Fredholm operator later on we need some of the classical theory of determinants which can be found, for example,  in \cite{Zinger}.
\begin{definition}\label{def_determinant_1}\index{D- Determinant}\index{$\text{det}(T)$}
The {\bf determinant} $\text{det}(T)$ of a bounded linear Fredholm operator $T\colon E\to F$ between real Banach spaces is the $1$-dimensional real vector space defined by
$$
\text{det}(T)= \lambda( \ker(T))\otimes \lambda^\ast( \text{coker(T)}).
$$
\end{definition}
An alternative definition used by some authors is  $\text{det} (T)=\lambda(\ker(T))\otimes\lambda(\text{coker}(T)^\ast)$.
The two definitions are naturally isomorphic given a convention how to identify $\lambda^\ast(A)$ and $\lambda(A^\ast)$. 
\begin{definition}
An {\bf orientation}\index{D- Orientation of $T$} of an Fredholm operator $T\colon E\rightarrow F$ is an orientation of the line $\det(T)$.
\end{definition}

We begin this subsection by deriving exact sequences associated to Fredholm operators.  
\begin{definition}
Let $T\colon E\rightarrow F$ be a Fredholm operator between two Banach spaces. A {\bf good left-projection}\index{D- Good left-projection} for $T$ is a bounded projection $P\colon F\rightarrow F$ having the following two properties.
\begin{itemize}
\item[(1)] $\dim(F/R(P))<\infty$.
\item[(2)] $R(P\circ T)=R(P)$.
\end{itemize}
By $\Pi_T$\index{$\Pi_T$} we denote the collection of all good left-projections for $T$.
\end{definition}

In view of (1) the projection $P$ satisfies $\dim \text{coker} (P)=\dim \ker (P)<\infty$. Hence $P$ is a Fredholm operator of index $0$. Since $T$ is Fredholm, the composition $P\circ T$ is Fredholm and $\text{ind}(P\circ T)=\text{ind} (P)+\text{ind} (T)=\text{ind} (T).$

The set $\Pi_T$ of projections  possesses a partial ordering $\leq $ defined by
$$\text{$P\leq Q$ \quad if and only if \quad $P=PQ=QP$}.$$

For  $P$ and $P'$ belonging to $\Pi_T$,   the intersection $R(P)\cap R(P')$  has a finite codimension in $F$. 
\begin{lemma}\label{nd1}
Given $P,P'\in\Pi_T$,  there exists a projection  $P''\in\Pi_T$ satisfying $P''\leq P$ and $P''\leq P'$.
\end{lemma}

{
\begin{proof}
Let $X$ be a topological complement of 
$R(P)\cap R(P')$ in $R(P)$ and let $X'$ be a topological complement of $R(P)\cap R(P')$ in $R(P')$
so 
that 
\begin{equation}\label{eq_intersection}
\text{$X\oplus \bigl(R(P)\cap R(P')\bigr)= R(P)$\quad and\quad $X'\oplus \bigl(R(P)\cap R(P')\bigr)=R(P')$.}
\end{equation}
Then  $X\cap X'=\{0\}$. Indeed,  if $z\in X\cap X'$, then $z\in R(P)\cap R(P')$  and it follows from \eqref{eq_intersection} that $z=0$.  For the subspace 
$X\oplus X'\oplus \bigl(R(P)\cap R(P')\bigr)$ we choose  a topological complement $Z$ in $F$
so that 
$$
F= Z\oplus X\oplus X'\oplus \bigl(R(P)\cap R(P')\bigr).
$$
Let $P''\colon F\to F$ be the projection onto $R(P)\cap R(P')$ along $Z\oplus X\oplus X'$.
Then 
$$
w=P''(z+x+x'+w)=P''(x+w)=P''P(z+x+x'+w)=PP''(z+x+x'+w).
$$
Consequently, $P''=PP''=P''P$. 
Similarly,  $P''=P'P''=P''P$. Hence,  $P''\leq P$ and $P''\leq P'$ and, in view of $R(PT)=R(P)$, we also have  $R(P''T)=R(P'')$.
\end{proof}
}

{
Associated with  the projection $P\in\Pi_T$ there is the exact sequence\index{${\bf E}_{(T,P)}$, exact sequence}
$$
{\bf E}_{(T,P)}:\quad 0\rightarrow\ker(T)\xrightarrow{j_T^P} \ker(PT) \xrightarrow{\Phi^P_T} F/R(P)\xrightarrow{\pi_T^P}\text{coker}(T)\rightarrow 0,
$$
where $j_T^P$\index{${\bf E}_{(T,P)}$} is the inclusion map, $\pi_T^P$ is  defined by $\pi_T^P(f+R(P))= (I-P)f+R(T)$ for $f\in F$,  and 
$$
\Phi^P_T(x)=T(x)+R(P)
$$
for $x\in \ker (PT)$.
\begin{lemma}
The sequence ${\bf E}_{(T,P)}$ is exact.
\end{lemma}
}
\begin{proof}

{
We follow the proof in  \cite{HWZ5}.
The inclusion $j_T^P$ is injective and $\Phi^P_T\circ j_T^P=0$. If $\Phi^P_T(x)=0$,  then $T(x)\in R(P)$ implying $(I-P)T(x)=0$.
Since $x\in \ker(PT)$ we conclude $T(x)=0$. This proves exactness at $\ker(PT)$. Assuming  $x\in \ker(PT)$,  hence 
$PT(x)=0$, we compute
\begin{equation*}
\begin{split}
\pi_T^P\circ \Phi^P_T(x)&=\pi_T^P(T(x)+R(P))\\
&=\pi_T^P(T(x)+R(PT))\\
&=T(x)+R(T)\\
&=R(T),
\end{split}
\end{equation*}
i.e. the composition $\pi_T^P\circ \Phi^P_T$ vanishes. If $\pi_T^P(y+R(P))=0$, hence $(I-P)y\in R(T)$, there exists 
$x\in E$ solving  $T(x)=(I-P)y$. Consequently, 
$$
\Phi_T^P(x) = T(x) + R(P) = (I-P)y+R(P)= y+R(P),
$$
proving the  exactness at $F/R(P)$. Finally we show the surjectivity of the map $\pi_T^P$.
Given $f+R(T)\in \text{coker}(T)$, we  choose  $x\in E$ satisfying  $PT(x)=Pf$ and  compute,
$$
\pi_T^P(f-T(x)+R(P))= (I-P)(f-T(x)) +R(T)= f +R(T),
$$
which finishes the proof of the exactness.
}
\end{proof}

{
Starting with the  Fredholm operator $T\colon E\rightarrow F$ and the projection $P\in\Pi_T$ we obtain from  the exact sequence
$$
{\bf E}_{(T,P)}:\quad 0\rightarrow\ker(T)\xrightarrow{j_T^P} \ker(PT) \xrightarrow{\Phi^P_T} F/R(P)\xrightarrow{\pi_T^P}\text{coker}(T)\rightarrow 0,
$$
recalling $R(P)=R(PT)$, \index{${\bf E}_{(T,P)}$, exact sequence}
the isomorphism
$$
\Phi_{{\bf E}_{(T,P)}}\colon\lambda(\ker(T))\otimes\lambda^\ast(\text{coker}(T))\rightarrow \lambda(\ker(PT))\otimes \lambda^\ast(\text{coker}(PT)).
$$
By  definition of the determinant,  this means that
$$
\Phi_{{\bf E}_{(T,P)}}\colon  \det(T)\rightarrow \det(PT)
$$
is an isomorphism.
We rename this isomorphism for the further discussion and,  setting $\Phi_{{\bf E}_{(T,P)}}=\gamma_T^P$, we have the isomorphism
$$
\gamma_T^P\colon \det(T)\rightarrow\det(PT).\index{$\gamma^P_T$}
$$
}

{
We consider the Fredholm operator $T\colon E\rightarrow F$ and the projections $P,Q\in \Pi_T$ satisfying $P\leq Q$, so that  $P=PQ=QP$. Then, as we have already seen, the composition  $QT$ is Fredholm. Moreover, 
$P\in \Pi_{QT}$. Therefore, the associated exact sequences ${\bf E}_{(T,P)}$, ${\bf E}_{(T,Q)}$, and
${\bf E}_{(QT,P)}$ produce the following isomorphisms,
\begin{align*}
\gamma^P_T\colon&\det(T)\rightarrow\det(PT),\\ \gamma^Q_T\colon&\det(T)\rightarrow\det(QT),\\ \gamma^P_{QT}\colon&\det(QT)\rightarrow\det(PT).
\end{align*}
\index{$\gamma^P_T:\det(T)\rightarrow\det(PT)$}\index{$\gamma^Q_T:\det(T)\rightarrow\det(QT)$}
\index{$\gamma^P_{QT}:\det(QT)\rightarrow\det(PT)$}
The crucial observation is the following result, whose proof is postponed to Appendix \ref{oooo6.6}.
\begin{proposition}\label{ooo6.6}
If  $T\colon E\rightarrow F$ is a Fredholm operator and $Q,P\in\Pi_T$ satisfy $P\leq Q$, 
then $QT\colon E\rightarrow F$ is a Fredholm operator and $P\in\Pi_{QT}$. Moreover, 
$$
\gamma^P_{QT}\circ \gamma^Q_T=\gamma^P_T.
$$
\end{proposition}
}

{
From Proposition \ref{ooo6.6} we conclude for the three 
projections 
$P, R,S\in \Pi_T$ satisfying $P\leq R$ and $P\leq S$, the relation 
$\gamma^P_{TR}\circ \gamma_T^R=\gamma_T^P=\gamma_{ST}^P\circ \gamma_T^S$, so that the 
diagram
\begin{equation*}
\begin{CD}
\det(T)@> \gamma_T^R >> \det(RT)\\
@VV\gamma_T^S V   @VV\gamma_{RT}^P V\\
\det(ST) @>\gamma_{ST}^P>>  \det(PT)
\end{CD}
\end{equation*}
is commutative, implying the following corollary.
\begin{corollary}\label{KPD}
The map 
$$
\det(ST)\rightarrow \det(RT), \quad  h\mapsto {(\gamma_{RT}^P)}^{-1}\circ\gamma_{ST}^P(h)
$$
is independent of the choice of $P$ as long as $P\leq S$ and $P\leq R$. Moreover,  this map is equal to the isomorphism 
$$
\det(ST)\rightarrow \det(RT),\quad h\mapsto \gamma^R_T\circ {(\gamma_T^S)}^{-1}(h).
$$
\end{corollary}
}

{
The latter may be viewed as some kind of transition map. This will become clearer in the next subsection.
The first map in the corollary behaves smoothly with respect to certain bundle constructions.
}

{
It will turn out that the good left projection structures are  very useful in  the definition of  determinant bundles for classical Fredholm operators.
}

{
Postponing  the construction of determinant bundles there is an important stabilization construction,  which is the linearized version
of a construction occurring in perturbation theory,  and which has to be understood from the point of view of orientations of determinants. 
}

{
We start with a Fredholm operator $T\colon E\rightarrow F$ and assume that $\phi\colon {\mathbb R}^n\rightarrow F$ is a linear map
such  that the map $T_\phi\colon  E\oplus {\mathbb R}^n\rightarrow F$, defined by 
$$
T_\phi(e,r)=T(e)+\phi(r),\index{$T_{\phi}$}
$$ 
is surjective. Writing ${\mathbb R}^n$ after 
$E$ is convenient and goes hand in hand with the conventions in  the definition of the isomorphism $\Phi_{\bf E}$ associated with the  exact sequence ${\bf E}$.  We would like to introduce 
a convention relating the orientations of  $\det(T)$ and $\det(T_\phi)$,  knowing that  ${\mathbb R}^n$ possesses  the   preferred orientation as 
direct sum of $n$-many copies of ${\mathbb R}$,  each of which is  oriented by $[1]$.
We introduce  the exact sequence
$$
{\bf E}:\quad  0\rightarrow \ker(T)\xrightarrow{j} \ker(T_\phi)\xrightarrow{\pi} {\mathbb R}^n\xrightarrow{[\phi]} F/R(T)\rightarrow 0,
$$
in which  $j(e)=(e,0)$,  $\pi(e,r)= r$, and  $ [\phi](r)=\phi(r)+R(T)$. 
\begin{lemma}
The sequence ${\bf E}$ is exact.
\end{lemma}
\begin{proof}
Clearly,  $\pi\circ j=0$ and $j$ is injective. If $\pi(e,r)=0$, then $T(e)=0$ and $j(e)=(0,e)=(e,r)$. If $(e,r)\in\ker(T_\phi)$, then $T(e)+\phi(r)=0$
implying that $\phi(r)\in R(T)$. Hence $[\phi(r)]=R(T)$. If $[\phi](r)=0$,  we have $\phi(r)=T(-e)$ for some $e\in E$ and hence $(e,r)\in\ker(T_\phi)$.
Moreover,  $\pi(e,r)=r$. The last map is surjective since,   by  assumption,  $T_\phi$ is surjective. 
\end{proof}
}

From  the  exact sequence ${\bf E}$ we deduce the previously constructed natural isomorphism
$$
\Phi_{\bf E}\colon \det(T)\rightarrow \lambda(\ker(T_\phi))\otimes \lambda^\ast({\mathbb R}^n).
$$
In the applications the operator $T$ is oriented and the auxiliary constructions to obtain transversality  yield,  on the linearized level,  operators of the type $T_\phi$.

{
The above isomorphism can be used to relate the orientations of  $\det(T)$ and $\det(T_\phi)$. However, this requires  an additional convention. For this it suffices 
to fix an isomorphism $\lambda^\ast({\mathbb R}^n)\rightarrow {\mathbb R}^\ast =\lambda^\ast(\{0\})$. This gives rise to an isomorphism
$$
\lambda(\ker(T_\phi))\otimes \lambda^\ast({\mathbb R}^n)\rightarrow \lambda(\ker(T_\phi))\otimes {\mathbb R}^\ast=\det(T_\phi),
$$
which then yields an isomorphism $\det(T)\rightarrow \det(T_\phi)$. The isomorphism we choose is defined by 
$$
\psi\colon \lambda^\ast({\mathbb R}^n)\rightarrow \lambda^\ast(\{0\})={\mathbb R}^\ast,\quad  (e_1\wedge\ldots \wedge e_n)^\ast\mapsto 1^\ast,
$$
where $e_1,\ldots ,e_n$ is the standard basis of ${\mathbb R}^n$. It then follows for every basis $c_1,\ldots ,c_n$ of  ${\mathbb R}^n$ that 
\begin{equation}\label{iso--}
\psi \bigl((c_1\wedge\ldots \wedge c_n)^\ast\bigr)=\frac{1}{\det([c_1,\ldots ,c_n])}1^\ast,
\end{equation}
where $\det$ denotes the determinant of a $n\times n$ matrix of the column vectors.
}

{
Summing up the previous discussion and using the definition of the isomorphism $\Phi_{\bf E}$ associated to the exact sequence ${\bf E}$ we can summarize the findings 
in  the following proposition.
\begin{proposition}
With the chosen isomorphism $\phi\colon \lambda^\ast({\mathbb R}^n)\rightarrow\lambda^\ast(\{0\})={\mathbb R}^\ast$, $(e_1\wedge\ldots \wedge e_n)^\ast\rightarrow 1^\ast$, we define the  isomorphism \index{$\iota_\phi\colon \det(T)\rightarrow \det(T_\phi)$}
$$
\iota_\phi\colon \det(T)\rightarrow \det(T_\phi)
$$
by the following formula. We choose  a basis $a_1,\ldots ,a_k$ of $\ker(T)$ and a basis $\bar{d}_1=\phi(\bar{c}_1)+R(T),\ldots ,\bar{d}_m=\phi(\bar{c}_m)+R(T)$ of  $F/R(T)$.
Let $h=(a_1\wedge\ldots \wedge a_k)\otimes (\bar{d}_1\wedge\ldots \wedge \bar{d}_m)^\ast$.  
Extend $(a_1,0),\ldots ,(a_k,0)$ to a basis of  $\ker(T_\phi)$ by adding
$\bar{b}_1=(b_1,r_1),\ldots ,\bar{b}_l=(b_l,r_l)$,  and notice that $r_1=\pi(b_1,r_1),\ldots ,r_l=\pi(b_l,r_l)$ together with $\bar{c}_1,\ldots ,\bar{c}_m$ form a basis of ${\mathbb R}^n$.
Then the isomorphism $\iota_\phi$ is defined by the formula
$$
\iota_\phi(h) = \frac{1}{ \det(\bar{c}_1,\ldots , \bar{c}_m, r_1\,\ldots ,r_l)}((a_1,0)\wedge\ldots \wedge (a_k,0)\wedge \bar{b}_1\wedge\ldots \wedge \bar{b}_l)\otimes1^\ast.
$$
\end{proposition}
This follows immediately from the definition of the isomorphism $\Phi_{\bf E}$ associated with  the exact sequence ${\bf E}$. Our choice of isomorphism $\lambda^\ast({\mathbb R}^n)\rightarrow {\mathbb R}^\ast$ is, of course, dual
to a uniquely determined choice of isomorphism ${\mathbb R}\rightarrow\lambda({\mathbb R}^n)$,  defined by the mapping,  which maps $1$ to the wedge of the standard basis.
Its  inverse is the map 
\begin{equation}\label{iso--1}
c_1\wedge\ldots \wedge c_n \rightarrow \det([c_1,\ldots ,c_n]).
\end{equation}
Here again $\det$ is the determinant of a $n\times n$ matrix.  Having chosen  the isomorphisms (\ref{iso--}) and (\ref{iso--1}) we deduce the following natural identifications
for every finite dimensional vector space $V$,
$$
V\otimes \lambda({\mathbb R}^n)\rightarrow V\quad  v\otimes (e_1\wedge\ldots \wedge e_n)\mapsto  v\otimes 1\rightarrow v, 
$$
and
$$
V\otimes\lambda^\ast({\mathbb R}^n)\rightarrow V:v\otimes (e_1\wedge\ldots \wedge e_n)^\ast\rightarrow v\otimes 1^\ast\rightarrow v.
$$
}

{
We need another convention.
\begin{definition}\index{D- Natural isomorphism for $\det(T\oplus S)$}\index{$\det(T)\otimes \det(S)\rightarrow \det(T\oplus S)$, natural isomorphism}
 Assume that $T\colon E\rightarrow F$ and $S\colon E'\rightarrow F'$ are Fredholm operators. Then the direct sum $T\oplus S\colon E\oplus E'\rightarrow F\oplus F'$ is a Fredholm operator, and we define the {\bf natural isomorphism} 
$$
\det(T)\otimes \det(S)\rightarrow \det(T\oplus S)
$$
as follows. We choose  $h=(a_1\wedge\ldots \wedge a_n)\otimes (d_1\wedge\ldots \wedge d_l)^\ast\in \det (T)$ and  $h'=(a_1'\wedge\ldots \wedge a_{n'}')\otimes (d_1'\wedge\ldots \wedge d_{l'}')^\ast\in \det (S)$.
Then we map $h\otimes h'$ to the vector $g\in \det (T\oplus S)$,  defined by
$$
g=(a_1\wedge\ldots \wedge a_n\wedge a_1'\wedge\ldots \wedge a_{n'}')\otimes(d_1\wedge\ldots \wedge d_l\wedge d_1'\wedge\ldots \wedge d_{l'}')^\ast.
$$
\end{definition}
}

\subsection{Classical Local Determinant Bundles}

{
We continue  with some local constructions in the neighborhood of  a Fredholm operator $T\colon E\rightarrow F$ between two Banach spaces.
\begin{lemma}
Given a Fredholm operator $T\colon E\rightarrow F$ and a projection  $P\in \Pi_T$, we take a topological complement $Y$ of $\ker(PT)$ in $E$.
Then there exists $\varepsilon>0$ such  that for every $S\in {\mathcal L}(E,F)$ satisfying $\norm{S-T}<\varepsilon$,  the following holds.
\begin{itemize}
\item[{\em (1)}] $P\in\Pi_S$.
\item[{\em (2)}] $PS\vert Y\colon Y\rightarrow R(P)$ is a topological isomorphism.
\end{itemize}
\end{lemma} 
}
\begin{proof}

{
We estimate $\norm{PT\vert Y-PS\vert Y}\leq \norm{P}\cdot\norm{S-T}$. Since $PT\vert Y\colon Y\rightarrow R(P)$ is a topological linear isomorphism, 
it follows from the openness of invertible linear operators $Y\rightarrow R(P)$ that (2) holds for a suitable $\varepsilon$.
Since $PS\colon Y\rightarrow R(P)$ is a topological linear isomorphism,  we conclude that $R(PS)=R(P)$,  which implies
$P\in\Pi_S$.
}
\end{proof}

{
We assume the hypotheses of  the previous lemma, split $E=\ker(PT)\oplus Y$,  and  correspondingly write $e=k+y$.
For $S$ satisfying  $\norm{S-T}<\varepsilon$ we consider the equation
$PS(k+y)=0$ which can be rewritten as
$$
PS(y)=-PS(k).
$$
Hence $y=y(k,S))=-(PS\vert Y)^{-1}(PS(k))$. The map $(k,S)\mapsto  y(k,S)$ is continuous. 
We equipe  the topological space
$$
\bigcup_{\norm{S-T}<\varepsilon} \{S\}\times \ker(PS)\subset {\mathcal L}(E,F)\oplus E
$$
with the induced topology. 
The projection
$$
\pi\colon  \bigcup_{\norm{S-T}<\varepsilon} \{S\}\times \ker(PS)\rightarrow  \bigcup_{\norm{S-T}<\varepsilon} \{S\}
 $$
 is the restriction of the continuous projection ${\mathcal L}(E,F)\oplus E\rightarrow {\mathcal L}(E,F)$ and therefore continuous. The fibers of $\pi$ are finite-dimensional vector spaces of the same dimension.  
\begin{lemma}\label{reddat}
The topological  space $\bigcup_{\norm{S-T}<\varepsilon} \{S\}\times \ker(PS)$ together with the  continuous projection $\pi$ has the structure of a trivial  vector bundle.
\end{lemma}
\begin{proof}
The continuous and bijective map
$$
\{S\, \vert \,  \norm{S-T}<\varepsilon\}\times\ker(PT)\rightarrow \bigcup_{\norm{S-T}<\varepsilon} \{S\}\times \ker(PS) 
$$
is defined by
$$
(S,k)\mapsto (S,k+y(k,S)).
$$
Its inverse is the restriction of the continuous map
$$
\{S\,  \vert \, \norm{S-T}<\varepsilon\}\times E\rightarrow\{S\,  \vert \,  \norm{S-T}<\varepsilon\}\times \ker(PT),\quad (S,k+y)\mapsto (S,k), 
$$
so that our map is indeed a topological bundle trivialization. 
\end{proof}
}

{
Now we have the trivial bundle 
$$
\pi\colon \bigcup_{ \norm{S-T}<\varepsilon} \{S\}\times\ker(PS)\rightarrow \{S\,  \vert \, \norm{S-T}<\varepsilon\},
$$
and the  product bundle 
$$
\pi_0\colon \{S\,  \vert \, \norm{S-T}<\varepsilon\}\times(F/R(P))\rightarrow \{S\,  \vert \,  \norm{S-T}<\varepsilon\}.
$$
The obvious viewpoint about these bundles is the following. Given the family of Fredholm operators $S\mapsto  PS$ defined for $S$ satisfying  $\norm{S-T}<\varepsilon$, 
the kernel dimension and cokernel dimension is constant. Therefore,  we obtain the topological kernel bundle as well as the cokernel bundle.
We shall need the following standard lemma which can be derived from results in f.e. in \cite{LA}.
}

{
\begin{lemma}\label{lemma6.21}
Given finite-dimensional topological vector bundles $E$ and $F$ over the topological space $X$,  the bundles $\lambda(E)\rightarrow X$, $\lambda^\ast(F)\rightarrow X$
as well as $E\oplus F\rightarrow X$, and $E\otimes F\rightarrow X$ have the structure of topological vector bundles in a natural way.
\end{lemma}
}

{
From Lemma \ref{lemma6.21} we deduce  immediately the following result.
\begin{proposition}\label{erty-o}
The topological space 
$$
\bigcup_{\norm{S-T}<\varepsilon} \{S\}\times \det(PS), 
$$
together with the projection onto the set $\{S\, \vert \, \norm{S-T}<\varepsilon\}$ 
has in a natural way the structure of topological line bundle.
\end{proposition}
}

{We introduce,  for the Fredholm operator $T$ and the projection $P\in\Pi_T$,  the abbreviation
$$
\text{DET}(T,P,\varepsilon)=\bigcup_{\norm{S-T}<\varepsilon} \{S\}\times \det(PS),
$$
where $\varepsilon>0$ is guaranteed by Lemma \ref{reddat}. We shall call $\text{DET}(T,P,\varepsilon)$ the {\bf local
determinant bundle} \index{Local determinant bundle} associated with  the Fredholm operator  $T$, the projection $P\in\Pi_T$,  and $\varepsilon$.
}

{
For the projection  $Q\in\Pi_T$ satisfying  $Q\leq P$ we abbreviate
$$
\text{DET}(T,Q,\varepsilon)=\bigcup_{\norm{S-T}<\varepsilon} \{S\}\times \det(QS)
$$
which,  again in a natural way,  is a topological line bundle. 
\begin{lemma}\label{kimportant}
The algebraic isomorphism
$$
\what{\gamma}_{T,P,Q,\varepsilon}:\text{DET}(T,P,\varepsilon)\rightarrow \text{DET}(T,Q,\varepsilon):(S,h)\rightarrow (S,\gamma^Q_{PS}(h))
$$ 
is a topological line bundle isomorphism.
\end{lemma}
\begin{proof}
This is trivial and follows from an inspection of the maps $\gamma_{PS}^Q$.  A sketch of the proof goes as follows.
We start with the exact sequence
$$
0\rightarrow \ker(PS)\rightarrow \ker(QS)\rightarrow F/R(Q)\rightarrow F/R(P)\rightarrow 0.
$$
Here $\norm{S-T}<\varepsilon$. Since $Q,P\in\Pi_T$ and $Q\leq P$,  the kernels of $PS$ and $QS$ define topological bundles if 
we vary the Fredholm operators $S$. This is  trivially true for the cokernel bundles associated to $PS$ and $QS$,  which are honest product bundles,  in view of
$R(PS)=R(P)$ and $R(QS)=R(Q)$. We can take linearly independent continuous sections which span  the kernels of the $PS$. Similarly,  for the cokernels we can take constant sections. Now going through the construction of the maps $\gamma_{PS}^Q$,  we see that we can extend the kernel sections for the $PS$ to
a family of continuous sections which are  point-wise linearly independent and  span the kernels of the $QS$. Proceeding  the same way  with the cokernels, we obtain at the end a continuous family of point-wise linearly independent sections spanning $\ker(QS)$ and $F/R(Q)$. Taking the appropriate wedges
and tensor products we see  that a continuous section of the first line bundle is mapped to a continuous section of the second.
The argument can  be reversed to verify  the continuity of the inverse map.
\end{proof}
}

{
Now we introduce  the bundle
$$
\text{DET}_{(E,F)}=\bigcup_{S\in{\mathcal F}(E,F)} \{S\}\times \det (S)\index{$\text{DET}_{(E,F)}$}
$$
over the space of all Fredholm operators.
At first sight this set seems not to have a lot of structure globally (though near certain $S$ it has some due to the previous discussion). The problem
is that  in the definition of $\det(S)$ the `ingredients' 
$\ker(S)$ and $F/R(S)$ associated to  $S$ might have varying local dimensions. However, in view of the discussion in the previous subsection, we shall see, that we can equip $\text{DET}_{(E,F)}$  with the structure
of a topological line bundle over the space of Fredholm operators ${\mathcal F}(E,F)$. This structure will turn out to be natural, i.e.,  the structure  does not depend on the choices  involved in  its construction. Here, it is important to follow the conventions already introduced.
}

{
We have the projection map
$$
\text{DET}_{(E,F)}\rightarrow {\mathcal F}(E,F),\quad (S,h)\mapsto  S.
$$
The base ${\mathcal F}(E,F)$ has a topology, but the total space $\text{DET}_{(E,F)}$ at this point has not.
Given $T\in{\mathcal F}$ we can invoke Proposition \ref{erty-o} and find $\varepsilon>0$ and $P\in\Pi_T$
such  that 
$$
\gamma_{(T,P,\varepsilon)}\colon \text{DET}_{(E,F)}\vert \{S\, \vert \, \norm{S-T}<\varepsilon\}\rightarrow \text{DET}(T,P,\varepsilon),\quad (S,h)\mapsto  (S,\gamma^P_S(h))
$$
is a bijection which is fibers-wise linear and covers the identity on the base.
}

{
\begin{definition}
Denote by ${\mathcal B}$ the collection of all subsets $B$ of $\text{DET}_{(E,F)}$ having the property  that there exists $\gamma_{(T,P,\varepsilon)}$ for which 
$\gamma_{(T,P,\varepsilon)}(B)$ is open in $\text{DET}(T,P,\varepsilon)$. 
\end{definition} 
}

{
The fundamental topological observation is  the following proposition.
\begin{proposition}
The collection of sets ${\mathcal B}$ is a basis for a topology on $\text{DET}_{(E,F)}$.
\end{proposition}
}

{
Postponing the proof for the moment, we denote the topology associated with ${\mathcal B}$ by ${\mathcal T}$,  and  obtain the following result.
\begin{proposition}\index{P- Line bundle structure of $\text{DET}_{(E,F)}$}
The set $\text{DET}_{(E,F)}$ with its linear structure on the fibers, equipped with the topology ${\mathcal T}$, 
is a topological line bundle. The maps
$$
\gamma_{(T,P,\varepsilon)}\colon \text{DET}_{(E,F)}|\{S\ |\ \norm{S-T}<\varepsilon\}\rightarrow \text{DET}(T,P,\varepsilon),\quad (S,h)\rightarrow (S,\gamma^P_S(h))
$$
are topological bundle isomorphisms. 
\end{proposition}
}

{
In order to prove the previous two propositions we recall  the maps $\gamma_{(T,P,\varepsilon)}$ and $\gamma_{(T',P',\varepsilon')}$. We are interested in the transition map
$\gamma_{(T',P',\varepsilon')}\circ \gamma_{(T,P,\varepsilon)}^{-1}$. The domain of this map  consists of all pairs $(S,h)$ satisfying  $\norm{T-S}<\varepsilon$ and $\norm{S-T'}<\varepsilon'$ and $h\in \det(PS)$. The transition map preserves the base point and maps $h$ to an element in $\det(P'S)$. This transition map is obviously an algebraic isomorphism between
two topological line bundles. 
}

{
\begin{lemma}
The transition map $\gamma_{(T',P',\varepsilon')}\circ \gamma_{(T,P,\varepsilon)}^{-1}$ is a topological isomorphism between topological line bundles.
\end{lemma}
\begin{proof}
It suffices to prove the continuity of the transition map near a base point $S_0$, since the same argument also applies to the inverse map. We choose a projection $Q$ satisfying $Q\leq P$ and $Q\leq P'$. In view of Corollary \ref{KPD},  the transition map 
$\gamma_{(T',P',\varepsilon')}\circ \gamma_{(T,P,\varepsilon)}^{-1}$
is equal to 
$$
(S,h)\mapsto (S,(\gamma_{P'S}^Q)^{-1}\circ\gamma_{PS}^Q(h)).
$$
This map is continuous,  in view of Lemma \ref{kimportant}.
\end{proof}
}

{
Let us finally remark that it is known that the line bundle $\text{DET}_{(E,F)}$ is non-orientable. However,  over certain subsets it is orientable and that will be used in the discussion of sc-Fredholm sections.
}

\subsection{Local Orientation Propagation and Orientability}

{
Viewed locally, a classical Fredholm section $f$ is a smooth map $f\colon U\to F$ from an open subset $U$ of some Banach space $E$
into  a Banach space $F$. The derivatives $Df(x)$,  $x\in U$,  form a  family of Fredholm operators depending continuously as operators on $x\in U$.
The determinant line bundle of this family of Fredholm operators is a smooth line bundle over $U$. If $U$ is contractible, this line bundle possesses precisely two possible orientations because a line has precisely two possible orientations.  Indeed, in this case, the orientation of one single line in the bundle determines 
a natural orientation of all the other lines. We might view 
this procedure as a local continuous propagation of an orientation.
}

{
This method, however, is not applicable in the sc-Fredholm setting because the linearizations do not, in general,  depend continuously on the base point. But there are additional difficulties. For example, we are confronted with nontrivial bundles where the dimensions 
can locally jump. Nevertheless there is enough structure 
which allows to define a propagation of the orientation 
and this is the core of the orientation theory in the following section.
}

{
We consider  the sc-Fredholm section $f$ of the strong M-polyfold bundle $P\colon W\rightarrow X$ over the tame $M$-polyfold $X$, choose  for a smooth point $x\in X$, 
a locally defined $\ssc^+$-section $s$ satisfying  $s(x)=f(x)$,  and take the linearization $(f-s)'(x)\colon T_xX\rightarrow Y_x$.  The linearization $(f-s)'(x)$ belongs to
the space of linearizations 
$\text{Lin}(f,x)$ which is a subset of linear Fredholm 
 operators $T_xX\to Y_x$ possessing the induced metric defined by the norm of the space ${\mathscr L}(T_xX, Y_x)$ of bounded operators.
Moreover, $\text{Lin}(f,x)$ is a convex subset and therefore contractible. Introducing the line bundle
$$
\text{DET}(f,x):=\bigcup_{L\in \text{Lin}(f,x)}\{L\}\times \text{det}(L)\index{$\text{DET}(f,x)$}
$$
over the convex set $\text{Lin}(f,x)$,  the previous 
discussion shows that $\text{DET}(f,x)$ is a topological line bundle.
\begin{proposition}\index{P- Line bundle structure for $\text{DET}(f,x)$}
$\text{DET}(f,x)$ has in a natural way the structure of a topological line bundle over $\text{Lin}(f,x)$ and consequently has two possible orientations since the base space
is contractible.
\end{proposition}
}

{
\begin{definition}[{\bf Orientation}]\index{D- Orientations of $\text{DET}(f,x)$}
Let $f$ be a sc-Fredholm section of the strong M-polyfold bundle $P\colon Y\to X$ over the tame M-polyfold $X$ and $x\in X$ a smooth point. Then an {\bf orientation for the pair $(f,x)$},  denoted by 
$\mathfrak{o}_{(f,x)}$,  is a choice
of one of the two possible orientations of  
$\text{DET}(f,x)$.
\end{definition}
}

{
Let us denote by ${\mathcal O}_f$ the orientation space associated to $f$,  which is the collection of all pairs $(x,\mathfrak{o})$, in which $x\in X_\infty$ and $\mathfrak{o}$ is an orientation of $\text{DET}(f,x)$.
We consider a category whose  objects are the sets $Or_x^f:=\{\mathfrak{o}_x,-\mathfrak{o}_x\}$ of the two possible orientations of $\text{DET}(f,x)$. The reader should not be confused by our notation.
There is no distinguished class $\mathfrak{o}_x$. We only know there are two possible orientations, namely $\mathfrak{o}_x$ and $-\mathfrak{o}_x$.
For any two smooth points $x$ and $y$ 
in the same path component of $X$ the associated morphism set consists of two isomorphisms. The first one maps $\pm\mathfrak{o}_x\rightarrow \pm\mathfrak{o}_y$
and the second one $\pm\mathfrak{o}_x\rightarrow \mp\mathfrak{o}_y$.
}


{
\begin{proposition}\label{umbra}\index{P- Basic construction for orientation {I}}
If $f$ is a  sc-Fredholm section of the strong M-polyfold bundle $P\colon Y\rightarrow X$,  then every smooth point $x$ possesses  an open neighborhood $U=U(x)$ having the following properties.
\begin{itemize}
\item[{\em (1)}] $U$ is sc-smoothly contractible.
\item[{\em (2)}] The solution set 
$\{x\in \cl_X (U)\, \vert \, f(x)=0\}$ is compact.
\item[{\em (3)}] Given a sc-smooth path $\phi\colon [0,1]\rightarrow U$,  there exists a $\ssc^+$-section $s\colon U\times [0,1]\rightarrow Y$
having  the property that $s(\phi(t),t)=f(\phi(t))$. 
\item[{\em (4)}] If  $t_0\in [0,1]$, $y_0\in U_\infty$,  and $e_0$ is  a smooth point in $Y_{y_0}$, then  there exists a $\ssc^+$-section $q\colon U\times [0,1]\rightarrow Y$ satisfying $q(y_0,t_0)=e_0$.
\end{itemize}
\end{proposition}
}

\begin{proof}

{
Taking local strong bundle coordinates coordinates we may assume that the bundle is $P\colon K\rightarrow O$ and $x=0$.  Here $K=R(V\triangleleft F)$ is the strong bundle retract covering the sc-retraction $r$ satisfying  $r(V)=O$.
As usual,  $V\subset C$ is an open neighborhood of $0$ in the partial quadrant $C$ of the sc-Banach space $E$. 
}

{
There is no loss of generality
assuming that $V$ is convex. Note that in this case $O$ is sc-smoothly contractible by the contraction  
$$
\what{r}\colon [0,1]\times O\rightarrow O,\quad \what{r}(t,y)=r(ty).
$$
Indeed, for $y\in O$ we have that $y\in V$,  and since $0\in V$, we have  for $t\in [0,1]$ that $ty\in V$, so that $r(ty)$ is defined. 
If $t=1$,  then $r(1y)=r(y)=y$ and if  $t=0$, then $r(0y)=r(0)=0$. 
The  sc-Fredholm section $f$ possesses the local compactness property. We can therefore choose  a convex neighborhood $V'$ of $0\in C$
such  that $U=r(V')$ is sc-smoothly contractible and so small that $\cl_C(U)\subset O$ and the restriction $f\vert \cl_C(U)$ has a compact solution set.
}

{
At this point we have constructed an open neighborhood $U$ of $x$ which has the first two properties.
Shrinking $U$ further, so that it is still the retract of a convex open subset of $V$,  will keep these properties.  We  have to show that an additional shrinking will guarantee (3) and (4).
}

{
The section $f$ of $K\rightarrow O$ has the form $f(y)=(y,{\bf f}(y))$ where $R(f(y))=f(y)$.
The map  ${\bf f}$ is the  principle part of the section $f$ and is sc-smooth as a map from 
$O$ to $F$.
If $\phi\colon [0,1]\rightarrow U$ is a sc-smooth path, $t\in [0,1]$ and $y\in U$, then 
$(y,{\bf f}(\phi(t)))\in O\triangleleft F\subset V\triangleleft F$  and we define the map
$s\colon U\times [0,1]\rightarrow K$ by 
$$
s(y,t)=R(y,{\bf f}(\phi(t))).
$$
}

{The map  $s$ is a $\ssc^+$-section of the pull-back of the bundle $K\rightarrow O$ by the map $U\times [0,1]\rightarrow O$, $(y, t)\mapsto y$.
By construction,  $s(\phi(t),t)=R(\phi(t),{\bf f}(\phi(t)))=R(f(\phi(t)))=f(\phi(t))$.  This proves (3).
}

{
The statement (4) was already proved in some variation  in the sections about Fredholm theory.  Formulated in the local
coordinates, the statement (4) assumes that  the  smooth point $e_0=(y_0,{\bf e_0})\in K$ and 
$t_0\in 0,1]$ are given.
The required  section $q\colon O\times [0,1]\to K$ can be then defined as the section 
$$
q(y,t)=R(y,{\bf e_0}).
$$
It satisfies  $q(y_0,t_0)=R(y_0,{\bf e_0})=(y_0,{\bf e_0})=e_0$, as desired.
}
\end{proof}

{
Now we are in the position to define a propagation mechanism. Since it will involve several choices we have to make sure that the end result is independent of the choices involved. 
The set up is as follows. We have a strong bundle $P\colon Y\rightarrow X$ over the M-polyfold $X$ and an sc-Fredholm section $f$ of $P$. Around a smooth point $x\in X$ we choose  an open neighborhood $U=U(x)$ so that the statements (1)-(4) of Proposition \ref{umbra} hold.
}

{
There is no loss of generality assuming that $X=U$. If 
 $\phi\colon [0,1]\rightarrow X$ is a sc-smooth path, we employ Proposition \ref{umbra}
and choose  a $\ssc^+$-section 
$s:X\times [0,1]\rightarrow W$ satisfying  $s(\phi(t),t)=f(\phi(t))$. 
}

{
Adding  finitely many $\ssc^+$-sections $s_1,\ldots ,s_k$ defined on $X\times [0,1]$, we obtain the 
sc-Fredholm section
$$
F\colon  [0,1]\times X\times {\mathbb R}^k\rightarrow W,\quad F(t,y,\lambda)=f(y)-s(y,t)+\sum_{i=1}^k\lambda_i\cdot s_i(y,t)
$$
having the property that  for $\lambda=0$ the points $(t,\phi(t),0)$ are solutions of  $F(t,\phi(t),0)=0$. For fixed $t\in [0,1]$ we introduce the  sc-Fredholm section
$$
F_t\colon X\times {\mathbb R}^k\rightarrow Y,\quad (y,\lambda)\rightarrow F(t,y,\lambda).
$$
The discussion, so far,  is true for all choices $s_1,\ldots ,s_k$. Using the results from the transversality theory we can choose  
$s_1,\ldots ,s_k$ such  that, in addition, $F_t$ has a good boundary behavior.
}

{
\begin{lemma}  \label{lem_property_star}
There exist finitely many $\ssc^+$-sections $s_1,\ldots ,s_k$ of the bundle  $Y\rightarrow [0,1]\times X$
such  that the sc-smooth Fredholm section $F$ of the bundle $Y\rightarrow [0,1]\times {\mathbb R}^k\times X$,  defined by
$$
F(t,y,\lambda) = f(y)-s(y,t)+\sum_{i=1}^k \lambda_i\cdot s_i(y,t),
$$
has the following property (P).
\begin{itemize}
\item[{\em (P)}] For every fixed $t\in [0,1]$,  the section $F_{t}$ is at the point $(\phi(t),0)$  in general position to the boundary of $X\times {\mathbb R}^k$.
\end{itemize}
\end{lemma}
Property (P) automatically implies that $F$ is in general position to the boundary of $[0,1]\times X\times {\mathbb R}^k$
at all points $(t,\phi(t),0)$ for $t\in [0,1]$. As a consequence of the implicit function theorem for the  boundary case we obtain the following result.
\begin{lemma}
The solution set $S=\{ (t, y, \lambda)\, \vert \,  F(t,y,\lambda)=0\}$ for $(t,y,0)$ near $(t,\phi(t),0)$ is a smooth manifold with boundary with corners.
Hence $\{(t,\phi(t),0)\, \vert \,  t\in [0,1]\}\subset S$. Moreover,  it follows from the property (P) that the projection 
$$
\pi\colon S\rightarrow [0,1], \quad (t,y,\lambda)\mapsto  t
$$
is a submersion. The set  $S_t$ defined by  $\{t\}\times S_t:=\pi^{-1}(t)$  is a manifold with boundary with corners contained
in $ X\times {\mathbb R}^k$.
\end{lemma}
We do not claim that $S_t$ is compact, but the manifold $S_t$  lies in such a way in $X\times {\mathbb R}^k$ that its intersection with $(\partial X)\times {\mathbb R}^k$
carves out a boundary with corners on the manifold $S_t$.
}

{
The tangent space $T_{(\phi (t), 0)}S_t$ at the point $(\phi (t), 0)\in S_t$ agrees with the kernel $\ker (F_t'(\phi (t), 0))$ and, abbreviating  
$$
L_t:=T_{(\phi(t),0)}S_t=\ker (F_t'(\phi (t), 0))\quad \text{for $t\in [0,1]$},$$
we introduce the bundle  $L$ of tangent spaces along the path $\phi (t)\in S$, for $t\in [0,1]$, by 
$$
L=\bigcup_{t\in [0,1]} \{t\}\times L_t.
$$
The bundle $L$ is a smooth vector bundle over $[0,1]$. By $\lambda (L)$ we denote the line bundle associated with $L$,
$$\lambda (L)=\bigcup_{t\in [0,1]} \{t\}\times \lambda (L_t).$$
}

{
An orientation of the lines $\lambda (L_t)$ determines by continuation an orientation of all the other lines. In particular, an orientation of $\lambda (L_0)$ at $t=0$ determines an orientation of $\lambda (L_1)$ at $t=1$, and we shall relate these orientations to the orientations of
$\text{DET}(f,\phi(0))$ and $\text{DET}(f,\phi(1))$.
To this aim we introduce, for every fixed $t\in [0,1]$, the exact sequence ${\bf E}_t$ defined by 
\begin{equation}\label{eexact}
\begin{gathered}
{\bf E}_t:\quad 0\rightarrow \ker( (f-s(\cdot,t ))'(\phi(t)) )\xrightarrow{j} \ker ({F}_t'(\phi(t),0))\xrightarrow{p}\\
\xrightarrow{p} {\mathbb R}^k\xrightarrow{c} Y_{\phi(t)}/R((f-s(\cdot,t ))'(\phi(t)))\rightarrow 0
\end{gathered}
\end{equation}
in which  $j$ is the inclusion map, $p$ the projection onto the ${\mathbb R}^k$-factor, and the  map $c$ is defined
by 
$$
c(\lambda) =\left(\sum_{i=1}^k \lambda_is_i(t,\phi(t))\right)+ R((f-s(\cdot,t ))'(\phi(t))).
$$
}
\begin{lemma}\label{lem_exact_seq_2}
The sequence \eqref{eexact} is exact.
\end{lemma}
\begin{proof}

{
The inclusion map $j$ is injective and $p\circ j=0$. From  $p(h,\lambda)=0$ it follows that 
$\lambda=0$
so that $h\in \ker((f-s(\cdot,t ))'(\phi(t)))$. If $(h,\lambda)\in \ker(F_t'(\phi(t),0))$,  then $\sum \lambda_is_i(\phi(t),t)$
belongs to the image of $(f-s(\cdot,t ))'(\phi(t))$ which implies $c\circ p=0$. It is also immediate that an
element $\lambda\in {\mathbb R}^k$ satisfying  $c(\lambda)=0$ implies that $\sum \lambda_is_i(\phi(t),t)$ belongs to the image
of $(f-s(\cdot,t ))'(\phi(t))$. This allows us to construct an element $(h,\lambda)\in \ker(F_t'(\phi(t),0))$ satisfying  $p(h,\lambda,)=\lambda$. 
Finally, it follows from the property $(P)$ in Lemma \ref{lem_property_star}  that the map $c$ is  surjective. 
The proof of Lemma 
\ref{lem_exact_seq_2} is complete.
}
\end{proof}

{
From the exact sequence ${\bf E}_t$ we deduce the natural isomorphism $\Phi_{{\bf E}_t}$ introduced in Section \ref{sect_conventions}, 
\begin{equation*}
\begin{split}
\Phi_{{\bf E}_t}&\colon \lambda \bigl(\ker (f-s(\cdot , t))'(\phi (t))\bigr)\otimes \lambda^\ast (\coker  (f-s(\cdot , t))'(\phi (t))\bigr)\\
&\to \lambda \bigl(\ker F_t'(\phi (t), 0)\otimes \lambda^\ast(\R^k),
\end{split}
\end{equation*}
and obtain, in view of the the definition of the determinant, the isomorphism
$$\Phi_{{\bf E}_t}\colon \det ((f-s(\cdot ,t))'(\varphi (t)))\to \lambda (L_t)\otimes \lambda^\ast(\R^k).$$
}

{
Now we assume that we have chosen the orientation 
$\mathfrak{o}_0$ of $\text{DET}(f,\phi(0))$  at $t=0$. It induces an orientation of $\det ((f-s(\cdot ,0))'(\phi (0))).$
The isomorphism $\Phi_{{\bf E}_0}$ determines an orientation of $\lambda (L_0)\otimes \lambda^\ast (\R^k)$. By continuation we extend this orientation to an orientation of  $\lambda (L_1)\otimes \lambda^\ast (\R^k).$ Then the isomorphism $(\Phi_{{\bf E}_1})^{-1}$ at $t=1$ gives us an orientation of $\det ((f-s(\cdot ,1))'(\varphi (1)))$ and consequently an orientation $\mathfrak{o}_1$ of 
$\text{DET}(f,\phi(1))$  at $t=1$, denoted by 
$$\mathfrak{o}_1=\phi_\ast \mathfrak{o}_0,$$
and called the {\bf push forward orientation}.
}

{
A priori our procedure might depend on the choices of the sections $s(y, t)$ and $s_1\ldots ,s_k$ and we shall prove next that it actually does not.}

{
\begin{proposition}\label{local-xyz}\index{P- Basic construction for orientation {II}}
We assume that $f$ is the sc-Fredholm section of the strong M-polyfold bundle $P\colon Y\rightarrow X$ and $U=U(x)$ an open neighborhood  around a smooth point $x$, for which the conclusions of Proposition \ref{umbra} hold. Let $\phi\colon [0,1]\rightarrow U$ be a  sc-smooth path. Then the construction of the map
$\mathfrak{o}\rightarrow\phi_\ast\mathfrak{o}$,
associating with  an orientation of $\text{DET}(f,\phi(0))$ an orientation $\phi_\ast\mathfrak{o}$  of $ \text{DET}(f,\phi(1)) $ does not depend on the choices involved as long as the hypotheses of Lemma \ref{lem_property_star} hold. 
\end{proposition}
}

{The proof of the proposition follows from two lemmata.
\begin{lemma}\label{opas}
Under the assumption of the proposition we assume that 
the sc-Fredholm section 
$$
F(t,y,\lambda)=f(y)-s(y,t)+\sum_{i=1}^k\lambda_i\cdot s_i(y,t),
$$ 
satisfies the property  (P) from Lemma \ref{lem_property_star}.
We view $F$ as section of the bundle $Y\rightarrow [0,1]\times X\times {\mathbb R}^k$. Adding more $\ssc^+$-sections we 
also introduce the second sc-Fredholm section 
$$
\ov{F}(t,y,\lambda,\mu)=F(t,y,\lambda)+\sum_{j=1}^l\mu_j\cdot \bar{s}_j(y,t),
$$
viewed as a section of the bundle 
$Y\rightarrow [0,1]\times X\times {\mathbb R}^{k+l}$. 
Then both sections define the same propagation of the orientation along the sc-smooth path $\phi$.
\end{lemma}
}

\begin{proof}

{
The key is to view $F$ as the section
$$
\wh{F}(\tau, t,y,\lambda, \mu)= f(y)-s(y,t)+\sum_{i=1}^k\lambda_i\cdot s_i(y,t)
$$
of the strong M-polyfold bundle $Y\rightarrow [0,1]\times [0,1]\times X\times {\mathbb R}^{k+l}$.
There is a sc-Fredholm section 
$$\wt{F}\colon
[0,1]\times[0,1]\times X\times{\mathbb R}^{k+l}\rightarrow Y, 
$$
defined 
by
$$
\wt{F}(\tau,t,y,\lambda,\mu)=\wh{F}(\tau, t,y,\lambda,\mu)+\sum_{j=1}^l \mu_j\cdot \bar{s}_j(y,t).
$$
The induced map $X\times{\mathbb R}^{k+l}\rightarrow Y$ for fixed $(\tau,t)\in [0,1]\times [0,1]$ satisfies the property (P).
We obtain a solution manifold $\tilde{S}$ with boundary with corners  defined near the points $(\tau,t,\phi(t),0)$,which fibers over $[0,1]\times [0,1]$.
Let $\tilde{\pi}\colon \tilde{S}\rightarrow [0,1]\times [0,1]$ and consider its tangent map.
Then take the vector bundle $\tilde{L}\rightarrow [0,1]\times [0,1]$ whose fiber over $(\tau,t)$ is the preimage of the zero section of $T([0,1]\times [0,1])$ under $T\tilde{\pi}$.
For $\tau=0$ we obtain the bundle $\tilde{L}^0\rightarrow \{0\}\times [0,1]=[0,1]$. This is of the form
$$
L\oplus ([0,1]\times {\mathbb R}^l)\rightarrow [0,1],
$$
where $L$ is the bundle associated to $F$ which we used originally to define the propagation along $\phi$. 
Applying the construction now using $\tilde{L}^0$ one verifies by a simple computation that it defines the same 
propagation.  
}

{
Hence it remains to verify that every bundle $\wt{L}^\tau$, $\tau\in [0,1]$,  defines  the same propagation. 
First of all we note that as $\tau$ varies,  the bundle varies continuously. Moreover, we relate the end points 
to $\det((f-s(\cdot,0))'(\phi(0)))$ and $\det((f-s(\cdot,1))'(\phi(1)))$. The data in the occurring exact sequence
relating  the latter with  the corresponding orientations of $\tilde{L}^\tau$,  vary continuously.  This implies that the propagation definition 
does not depend on $\tau$.
}
\end{proof}

{
Next we show that the choice of the $\ssc^+$-section $s$ satisfying  $s(\phi(t),t)=f(\phi(t))$ does not affect the definition of the propagation.
}

{
\begin{lemma}\label{opas1}
We assume that the $\ssc^+$-sections $s^i$ for $i=0,1$ satisfy 
$$
s^i(\phi(t),t)=f(\phi(t))
$$
 and consider the homotopy $s^\tau=\tau \cdot s^1+(1-\tau)\cdot s^0$.
We choose  additional sections $s^\tau_1,\ldots ,s^\tau_k$  so that the property (P) holds together with $s^\tau$  for every $\tau$.
Then the  propagation of the orientation along the sc-smooth path $\phi$ using each of the collections $s^\tau,s_1^\tau,\ldots ,s_k^\tau$ is the same.
 \end{lemma}
 }
 
 {
 \begin{proof}
 The proof is similar as the lemma above. We obtain a vector bundle $\wtilde{L}\rightarrow [0,1]\times [0,1]$
 with parameters $(\tau,t)$ in the base. Then we can define for fixed $\tau$ a vector bundle over $t\in[0,1]$, say $\wt{L}^\tau$.
 Clearly $\wt{L}^\tau$ varies continuously in $\tau$.  At  $t=0,1$ the linearizations
 $(f-s^\tau(\cdot,0))'(\phi(0))$ and $(f-s^\tau(\cdot,1))'(\phi(1))$ vary continuously as operators in $\text{DET}(f,\phi(0))$ and $\text{DET}(f,\phi(1))$, respectively.
 This implies that the bundle $\tilde{L}^\tau$ induces,  independently of $\tau\in [0,1]$,  the same propagation of the orientation  along the path $\phi$.
\end{proof}
}

{
Now we are in the position to finish the proof of Proposition \ref{local-xyz}.
We start with two  collections $s^i,s_1^i,\ldots ,s_{k^i}^i$ to define the propagation and add for $i=0$ and $i=1$
additional $\ssc^+$-sections such that the following holds.
\begin{itemize}
\item[(1)] For both situations $i=0$ and $i=1$ we have the same number of sections.
\item[(2)] For  a convex homotopy parametrized by $\tau\in [0,1]$, the property (P) holds for each fixed $\tau$.
\end{itemize}
By Lemma \ref{opas} adding section does not change the propagation. We can apply this for $i=0,1$.
Then we can use Lemma \ref{opas1} to show that these two propagations (after adding sections) are the same.
}

{
Next assume we have the same hypotheses and $U=U(x)$ has the initially stated properties.
If the two sc-smooth paths $\phi^0,\phi^1\colon [0,1]\rightarrow U$ have the same starting and end points,  we homotope sc-smoothly from one to the other with end points fixed,  using that $U$ is sc-smoothly contractible. We 
denote this sc-smooth homotopy by
$$
\Phi\colon [0,1]\times [0,1]\rightarrow U
$$
where $\phi^i=\Phi(i,0)$ for $i=0,1$. Then we  construct the $\ssc^+$-section $s$ satisfying $s(\Phi(\tau,t),\tau,t)=f(\Phi(\tau,t))$ and define the sc-Fredholm section  
$$
(\tau,t,y,\lambda)\rightarrow f(y)-s(y,\tau,t)+\sum_{i=1}^k\lambda_i\cdot s_i(y,\tau,t)
$$
possesing the obvious properties. 
For every $\tau$ we obtain a vector bundle $\wt{L}^\tau\rightarrow [0,1]$. Since all the data change continuously in these
bundles and the operators for $t=0,1$ in $\text{DET}(\ast)$  change continuously in $\tau$,  we see that every $\wt{L}^\tau$ defines the same 
propagation of the orientation along paths $\phi^\tau=\Phi(\tau,\cdot )$. Hence we have proved the following statement.
}

{
\begin{theorem}\label{thm_6.35}\index{T- Propagation of orientation}
We assume that $P\colon Y\rightarrow X$ is a strong bundle over the tame M-polyfold and  $x$ a smooth point in $X$.
Then there exists an open sc-smoothly contractible neighborhood $U=U(x)$ such  that for every sc-smooth path $\phi\colon [0,1]\rightarrow U$
there exists a well-defined propagation 
$$
\mathfrak{o}\rightarrow \phi_\ast\mathfrak{o}
$$
of an orientation $\mathfrak{o}$ of $\text{DET}(f,\phi(0))$. Moreover,  if $\phi^0$ and $\phi^1\colon [0,1]\rightarrow U$ are two sc-smooth paths from the  same starting points 
to the same end points, then the  propagation along $\phi^0$ and $\phi^1$ is the same, i.e., 
$$
\phi^0_\ast\mathfrak{o}=\phi^1_\ast\mathfrak{o}.
$$
\end{theorem}
}

{
Having this theorem there are precisely two possible ways to orient the family $\text{DET}(f,y)$, $y\in U_\infty$,  so that 
these orientations are related by propagation along paths. Namely,  we fix $x$ and take for $y\in U$ a sc-smooth path
$\phi\colon [0,1]\rightarrow U$, starting at $x$ and ending at $y$.  Fixing  an orientation $\mathfrak{o}_x$ of $\text{DET}(f,x)$, 
we define $\mathfrak{o}_y=\phi_\ast\mathfrak{o}_x$. The definition is independent of the choice of the path $\phi$.
At this point we have a map which associates to a smooth point $y\in U$, i.e. $y\in U_\infty$,  an orientation $\mathfrak{o}_y$ of $\text{DET}(f,y)$. It follows from Theorem \ref{thm_6.35} that if  $\psi$ is a sc-smooth path connecting $y_1$ with $y_2$, then $\psi_\ast\mathfrak{o}_{y_1}=\mathfrak{o}_{y_2}$.
}

{
\begin{definition}
Given a strong bundle $P\colon Y\rightarrow X$ over the tame M-polyfold $X$ and a sc-Fredholm section $f$, we call an open contractible
neighborhood $U(x)$ of a smooth point $x$ on which the local propagation construction can be carried out an {\bf orientable neighborhood}\index{D- Orientable neighborhood} of $(f,x)$. Any two of the possible orientations of $\text{DET}(f,y)$ for $y\in U_\infty$, which have the propagation property,  is called  a  {\bf continuous orientation}\index{D- Continuous orientations}.
\end{definition}
An orientable neighborhood has precisely two continuous orientations. 
}

{
Now we globalize the local propagation of orientation constructions to a more global procedure.
}

{
Let $X$ be an M-polyfold and $\phi\colon [0,1]\rightarrow X$ a sc-smooth map. We denote by $[\phi]$ its sc-smooth homotopy class with end points fixed. Associated with  $[\phi]$ we have the source $s([\phi])$ which is the starting point $\phi(0)$ and the target $t([\phi])$ which is the end point $\phi(1)$.
Given two sc-smooth homotopy classes  $[\phi]$ and $[\psi]$ satisfying $t([\phi])=s([\psi])$,  we  define the composition $[\psi]\ast[\phi]$ as the class of $\gamma$ defined as follows.
Take a smooth map $\beta\colon [0,1]\rightarrow [0,2]$ satisfying  that $\beta([0,1/2])=[0,1]$, $\beta(0)=0$ and $\beta(s)=1$ for $s$ near $1/2$. Moreover, 
$\beta([1/2,1])=[1,2]$ and $\beta(1)=2$. Then we define the sc-smooth path $\gamma$ by $\gamma(t)=\phi(\beta(s))$ for $s\in [0,1/2]$ and $\gamma(t)=\psi(\beta(s))$ for $s\in [1/2,1]$.
This way we obtain a category ${\mathcal P}_X$ whose objects are the smooth points in $X$ and whose morphisms $x\rightarrow y$ are the homotopy  classes $[\phi]$ with source $x$ and target $y$.
}

{
The main result of this section is  described  by the following theorem which is a consequence of the local constructions.
}

\begin{theorem}\index{T- Main orientation result}

{
Let $P\colon Y\rightarrow X$ be a strong bundle over the tame M-polyfold $X$ and $f$ a sc-Fredholm section.
Then there exists a uniquely determined functor 
$$
{\mathcal P}_X\xrightarrow{\Gamma_f}{\mathcal O}_f,\index{${\mathcal P}_X\xrightarrow{\Gamma_f}{\mathcal O}_f$}
$$
which associates
with  the smooth point $x$ the set $Or_x^f$,  and with the sc-smooth homotopy class  $[\phi]\colon x\rightarrow y$ a morphism $\Gamma_f([\phi])\colon Or_x^f\rightarrow Or_y^f$,  which has the following property.\\[1ex]
$({\bf \ast})$\quad 
Given any smooth point $x$ in $X$  and  a sc-smoothly contractible open neighborhood $U=U(x)$ having the property that $f$ on $U$ has the local propagation property, then 
the maps $\phi_\ast\colon Or_{\phi(0)}^f\rightarrow Or_{\phi(1)}^f$ and $\Gamma^f([\phi])\colon Or_{\phi(0)}^f\rightarrow Or_{\phi(1)}^f$ coincide for a sc-smooth path $\phi\colon [0,1]\rightarrow U$.
}
\end{theorem}
\begin{proof}

{
We consider the  sc-smooth path $\phi:[0,1]\rightarrow X$, fix $t_0\in [0,1]$, and choose an orientable neighborhood $U(\phi(t_0))$.
Then we find an open interval $I(t_0)$ so that an orientation
$\mathfrak{o}_t$ for $(f,\phi(t))$, where $t\in I(t_0)\cap [0,1]$, determines an orientation for all $s\in I(t_0)\cap [0,1]$.  
Now using the compactness of $[0,1]$ and a finite covering map, we can transport a given orientation $\mathfrak{o}_0$ of $\text{DET}(f,\phi(0))$ to an orientation 
$\mathfrak{o}_1$ of $\text{DET}(f,\phi(1))$. The map $\pm\mathfrak{o}_0\rightarrow \pm\mathfrak{o}_1$ is denoted by $\phi_\ast$. By a covering argument one verifies that the map $\phi_\ast$ only depends on the homotopy class $[\phi]$ for fixed end-points. It is obvious that $\Gamma_f$ is unique.
}
\end{proof}
 
 {
Finally we can give two equivalent definitions of orientability of a sc-Fredholm section.
}

{
\begin{definition}\index{D- Orientable sc-Fredholm section {I}}
A sc-Fredholm section  $f$  of the strong bundle $P\colon Y\rightarrow X$ over the tame M-polyfold $X$ is called  {\bf orientable}  provided for every pair of smooth points $x$ and $y$
in the same path component,  the morphism $\Gamma_f([\phi])$ does not depend on the choice of the homotopy class $[\phi]\colon x\rightarrow y$. 
\end{definition}
}
An equivalent version is the following.

{
\begin{definition} \index{D- Orientable sc-Fredholm section {II}}
Let $f$ be a sc-Fredholm section of the strong bundle $P\colon Y\rightarrow X$ over the tame M-polyfold $X$. Then $f$ is  called {\bf orientable}  if there exists 
a map which associates with a point  $y\in U_\infty$ an orientation $\mathfrak{o}_y$ of $\text{DET}(f,y)$ having the following property. For every smooth point $x$ there exists
an orientable neighborhood $U(x)$ for which  $U_\infty(x)\ni y\mapsto \mathfrak{o}_y$ is one of the two continuous orientations.
\end{definition}
}

{
Finally we  define an orientation of the sc-Fredholm section as follows.
}

{
\begin{definition}[{\bf Orientation of sc-Fredholm sections}]\index{D- Orientation of an sc-Fredholm section}
Let $P\colon Y\rightarrow X$ be a strong bundle over the tame M-polyfold $X$ and assume that $f$ is a  sc-Fredholm section.
An orientation $\mathfrak{o}^f$ for $f$ is a map which associates with every  smooth  point $y\in X$ an orientation $\mathfrak{o}_y^f$ of $\text{DET}(f,y)$
such  that,  for every smooth point $x$ and orientable neighborhood $U=U(x)$,  the restriction $\mathfrak{o}^f\vert U_\infty$ is one of the two possible continuous orientations.
\end{definition}
}

{
We close this subsection with a general result.
}

{
\begin{theorem}\index{T- Orientation of solution set}
We assume that $P\colon Y\rightarrow X$ is a strong bundle over the tame M-polyfold $X$. Let $f$ be a proper sc-Fredholm section
such  that for every $x\in X$ solving $f(x)=0$ the pair $(f,x)$ is in 
general position to the boundary $\partial X$.
Suppose $\mathfrak{o}^f$ is an orientation for $f$. Then the solution set $S=f^{-1}(0)$ is a smooth compact manifold with boundary with corners
possessing  a natural orientation.
\end{theorem}
}
\begin{proof}

{
We already know from the transversality discussion that $S$ is a compact manifold with boundary with corners.
If  $x\in S$, then $\ker(f'(x))=T_xS$ and because  $f'(x)$ is surjective,  every tangent space $T_xS$ is oriented by $\mathfrak{o}_x^f$.
Since $\mathfrak{o}^f$ has the local continuation property,  one verifies that the differential geometric local prolongation of the orientation on $S$ is
the same as the Fredholm one.
}
\end{proof}

{
If $f$ is not  generic to start with,  we have to take a small perturbation $s$ supported near the zero set of $f$, so that  the solution set $S_{f+s} $ has a natural orientation coming from $\mathfrak{o}^f$ since $\det((f+s)'(x))$ has an orientation coming from $\text{DET}(f,x)$ for $x\in S_{f+s}$.
}

{
\begin{remark} 
If $(f,\mathfrak{o})$ is an oriented Fredholm section of the strong bundle $P\colon Y\rightarrow X$ and $s$ is a $\ssc^+$-section of $Y\rightarrow X\times [0,1]$, then the sc-Fredholm section 
$\wh{f}$,  defined on  $X\times[0,1]$ by $\wh{f}(x,t)=f(x)+s(x,t)\in Y$, 
has a natural orientation associated with  
$\mathfrak{o}$ and the standard orientation of $[0,1]$. 
We choose  a local $\ssc^+$-section $t$ near the smooth point $x$  satisfying  $t(x)=f(x)$ and consider the orientation 
$\mathfrak{o}_x$ of  $\det((f-t)'(x))$, where $L=(f-t)'(x)\colon T_xX\rightarrow Y_x$. Taking the vector 
$h=(a_1\wedge\ldots \wedge a_n)\otimes {(b_1\wedge\ldots \wedge b_l)}^\ast$ determining the orientation of $L$, we orient 
$$
\wt{L}\colon T_{(x,t)}(X\times [0,1])\rightarrow Y_x,\quad (a,b)\rightarrow La
$$
by the vector 
$\wt{h}=(a_1\wedge\ldots \wedge a_n\wedge e)\otimes (b_1\wedge\ldots \wedge b_l)^\ast$,
where the vector $e$ defines the standard orientation of $[0,1].$
The linear map $\wt{L}$ belongs to the linearization space of $\wh{f}$ at $(x,t)$, so that the latter obtains an orientation.
\end{remark}
}

\subsection{Invariants Associated to Proper Sc-Fredholm Sections}

{
The next result generalizes a classical fact of smooth compact manifolds to polyfolds. For the convenience of the reader we first recall the relevant previous constructions.
}

{We consider the oriented and proper sc-Fredholm section $(f, \mathfrak{o})$ of the strong bundle $P\colon Y\to X$ over the tame M-polyfold $X$ admitting sc-smooth bump functions. The solution set $S=\{x\in X\, \vert \, f(x)=0\}$ is compact. If $N$ is an auxiliary norm on $P$ we know from the perturbation and transversality result, Theorem \ref{thm_pert_and_trans}, that there exists an open neighborhood $U\subset X$ of $S$ such that for every $\ssc^+$-section $s$ of $P$ which is supported in $U$ and satisfies $N(s(x))\leq 1$ for all $x\in X$, the solution sets $S_{f+s}=\{x\in X\, \vert \, f(x)+s(x)=0\}$ is compact. Moreover, there exist distinguished such 
$\ssc^+$-sections having the additional property that,  for all $x\in S_{f+s}$, the  pair $(f+s, x)$ is in general position to the boundary $\partial X$. Its solution set is a compact manifold with boundary with corners and, as we have seen above, possesses a natural orientation induced from the orientation $\mathfrak{o}$ of the sc-Fredholm section $f$.
}

{
In the following theorem we shall call these distinguished $\ssc^+$-sections $s$ {\bf admissible for the pair $(N, U)$.}
}

{
In Section \ref{subs_sc_differential} we have introduced the differential algebra 
$\Omega^\ast_\infty(X,\partial X)=\Omega_\infty^\ast(X)\oplus \Omega_\infty^{\ast-1}(\partial X).$ Its differential $d(w, \tau)=(d\omega, j^\ast w-d\tau)$, in which $j\colon \partial X\to X$ is the inclusion map, satisfies $d\circ d=0$ and we denote the associated cohomology by 
$H^\ast_{\textrm{dR}}(X, \partial X)$. 
}

{
\begin{theorem}\index{T- Invariants for oriented proper sc-Fredholm sections}
Let $P\colon Y\rightarrow X$ be a strong bundle over the tame M-polyfold $X$. We assume that $X$ admits sc-smooth bump functions.
Then there exists a well-defined map which associates with  a proper and oriented sc-Fredholm section $(f, \mathfrak{o})$ of the bundle  $P$ a linear map 
$$
\Psi_{(f,\mathfrak{o})}\colon H^\ast_{\textrm{dR}}(X,\partial X)\rightarrow {\mathbb R}
$$
having the following properties.
\begin{itemize}
\item[{\em (1)}] If $N$ is an auxiliary norm on $P$ and $s$ a $\ssc^+$-section of $P$ which is admissible for the pair $(N, U)$, then the solution set  $S_{f+s}=\{x\in X\, \vert \, f(x)+s(x)=0\}$ is an oriented and compact manifold with boundary with corners, and 
$$
\Psi_{(f,\mathfrak{o})}([\omega,\tau]) = \int_{S_{f+s}} \omega -\int_{\partial S_{f+s}}\tau,
$$ 
holds for every cohomology class $[\omega,\tau]$ in $H^\ast_{dR}(X,\partial X)$. The integrals on the right-hand side are defined to be zero, if the dimensions of the forms and manifolds do not agree.  
\item[{\em (2)}] For  a proper and oriented homotopy between two oriented sc-Fredholm sections $(f_0,\mathfrak{o}_0)$ and $(f_1,\mathfrak{o}_1)$ of the bundle $P$  we have the identity
$$
\Psi_{(f_0,\mathfrak{o}_0)}=\Psi_{(f_1,\mathfrak{o}_1)}.
$$
\end{itemize}
\end{theorem}
}
\begin{proof}

{
(1)\,  The properness of the sc-Fredholm section $f$ implies the compactness of its solution set $S_f=\{x\in X\,\vert \, f(x)=0\}$. If $N$  is an auxiliary norm $N$, then there exist, by Theorem 
\ref{thm_pert_and_trans}, an open neighborhood $U\subset X$ of $S_f$ and a $\ssc^+$-section of $s_0$ of $P$ which is admissible for the pair $(N, U)$ such that the solution set $S_0=S_{f+s_0}=\{x\in X\, \vert \, f(x)+s_0(x)=0\}$ is, in addition, an oriented compact manifold with boundary with corners of dimension $n=\dim S_{f+s_0}=\ind (f'(x))$ for $x\in S_{f+s_0}$. 
The orientation $(f, \mathfrak{o})$ induces an orientation on $S_{f+s}$ and its local faces.  We define the map $\Psi_{(f,\mathfrak{o})}$ by 
\begin{equation}\label{int_eq_1}
\Psi_{(f,\mathfrak{o})}([\omega,\tau]):=\int_{S_{f+s_0}}\omega -\int_{\partial S_{f+s_0}}\tau
\end{equation}
for a cohomology class $[\omega, \tau]$ in $H^\ast_{\textrm{dR}}(X,\partial X)$. The integrals on the right-hand side are defined to be zero 
if the dimensions of the forms and manifolds do not agree. 
}

{
In order to verify that the 
definition \eqref{int_eq_1} does not depend on the choice of admissible section $s_0$,  we take a second $\ssc^+$-section $s_1$ admissible for the pair $(N, U)$ so that the solution set $S_1=S_{f+s_1}=\{x\in X\, \vert \, f(x)+s_1(x)=0\}$ is a compact oriented manifold with boundary with corners. As explained in Remark \ref{rem_homotopy}, 
there exists a proper $\ssc^+$- homotopy $s_t$ connecting  $s_0$ with  $s_1$
such the sc-Fredholm section  $F$, defined on $X\times [0,1]$ by $F(x,t)=f(x)+s_t(x)$, is in 
general position to the boundary of $ X\times [0,1]$ for every  $(x,t)$ in the solution set   $S_F=\{(x, t)\in X\times [0,1]\, \vert \, f(x)+s_t(x)=0\}$, which is an 
oriented compact manifold with boundary with corners, whose orientation is induced from the orientation  $\mathfrak{o}$ of $f$ and the standard orientation of $[0,1]$.  We recall that  $[\omega,\tau]$ satisfies $d\omega=0$ and $j^\ast\omega=d\tau$ on $\partial X$ where $j\colon \partial X\to X$ is the inclusion map. Extending $\omega$ to the whole space $X\times [0,1]$, we introduce the form $\ov{\omega}$ on $X\times [0,1]$ by 
$$\ov{\omega}=p_1^\ast\omega\wedge p_2^\ast dt,$$
where $p_1\colon X\times [0,1]\to X$ and $p_2\colon X\times [0,1]\to [0,1]$ are the sc-projection maps. It follows that $d\ov{\omega}=0$ and $\ov{\omega}\vert \partial X\times [0,1]=d\ov{\tau}$, where $d\ov{\tau}=\ov{p}_1^\ast (d\tau)\wedge \ov{p}_2^\ast dt.$ The maps $\ov{p}_i$ for $i=1, 2$ are the restrictions of $p_i$ to $\partial X\times [0,1]$. 
By Stokes theorem,
$$\int_{S_F}d\ov{\omega}=\int_{\partial S_F}\ov{\omega}=0.$$
We decompose the boundary $\partial S_F$ of the manifold $S_F$ into 
$$\partial S_F=S_0\cup S_1\cup \wt{S},$$
where $S_0\subset \partial X\times \{0\}$, $S_1\subset\partial X\times \{1\}$, and $\wt{S}\subset \partial X\times (0, 1)$, as illustrated in Figure \ref{pict7}.

\begin{figure}[htb]
\begin{centering}
\def\svgwidth{70ex}
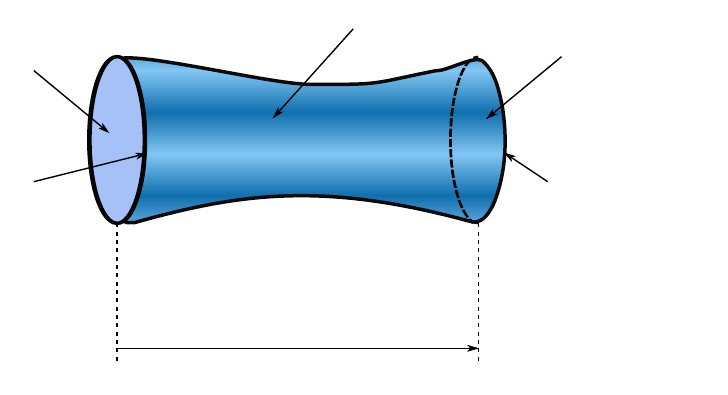
\caption{}\label{pict7}
\end{centering}
\end{figure}
}

{
Hence,
$$
0=\int_{\partial S_F}\ov{\omega}=\int_{\partial S_F}d\ov{\tau},
$$
and we compute the two integrals, taking the orientations into account,
$$0=\int_{\partial S_F}\ov{\omega}=\int_{S_0}\omega-\int_{S_1}\omega-\int_{\wt{S}}\ov{\omega}.$$
Integration of $d\ov{\tau}$ over $\partial S_F$ using Stokes theorem gives 
\begin{equation*}
\begin{split}
0&=\int_{\partial S_F}d\ov{\tau}=\int_{S_0}d\tau-
\int_{S_1}d\tau+\int_{\wt{S}}\ov{\omega}\\
&=\int_{\partial S_0}\tau-\int_{\partial S_1}\tau
+\int_{\wt{S}}\ov{\omega}.
\end{split}
\end{equation*}
Comparing the integrals, we 
obtain 
$$\int_{S_0}\omega-\int_{\partial S_0}\tau=\int_{S_1}\omega-\int_{\partial S_1}\tau,$$
which shows that the definition \eqref{int_eq_1} does indeed not depend on the choice of the distinguished section $s_0$.
}

{
The formula in (2) is verified by the same homotopy argument as in (1).
}
\end{proof}

{In the special situation $\partial X=\emptyset$ of no boundary, the map $\Psi_{(f,\mathfrak{o})}\colon H^n_{\textrm{dR}}(X)\rightarrow {\mathbb R}$ is defined by 
$$
\Psi_{(f,\mathfrak{o})}([\omega])=\int_{S_{f+s}}\omega
$$ 
for $[\omega]\in H^n_{\textrm{dR}}(X)$,  if $n$ is equal to the index of the sc-Fredholm section $f$, and zero otherwise.
}
\subsection{Appendix}
\subsubsection{Proof of Lemma \ref{o6.6}}\label{oo6.6}
\begin{proof}[Proof of Lemma \ref{o6.6}]

{Recalling the exact sequence 
$$
{\bf E}:\quad  0\rightarrow A\xrightarrow{\alpha} B\xrightarrow{\beta} C\xrightarrow{\gamma} D\rightarrow 0,
$$
we choose linear subspaces $V\subset B$ and $W\subset C$ such  that
$$
B=\alpha(A)\oplus V\quad   \text{and}\quad    C=\beta(B) \oplus W
$$
and denote this choice of subspaces by $(V, W)$. The maps 
$$\beta_V\colon V\to \beta (B)\quad \text{and}\quad \gamma_W\colon W\to D$$
are isomorphisms and  $\beta(B)=\beta(V)$ by exactness. Fixing the basis 
$a_1,\ldots ,a_n$ of $A$ and  $d_1,\ldots ,d_l$ of $D$, we choose any basis $b_1,\ldots ,b_{m-n}$ of $V\subset B$ and define the basis 
$c_1,\ldots ,c_{m-n}$ of $\beta (B)\subset C$ by  $c_i=\beta_V(b_i)$, and  define the basis  $c_1',\ldots ,c_l'$ of $W$ by $\gamma_W(c_i')=d_i$.
}

{In order to show that  the vector
$$
(\alpha(a_1)\wedge\ldots \wedge\alpha(a_n)\wedge b_1\wedge\ldots \wedge b_{m-n})\otimes ( c_1'\wedge\ldots \wedge c_l'\wedge c_1\wedge\ldots \wedge c_{m-n})^\ast
$$
does not depend on the choice of the basis $b_1,\ldots ,b_{m-n}$ of $V$,  we choose a second basis
 $\ov{b}_1,\ldots ,\ov{b}_{m-n}$ of $V$, so that there is a linear isomorphism
$\sigma\colon V\rightarrow V$ mapping one basis into the other by $\ov{b}=\sigma (b)$.
Hence 
$$\ov{b}_1\wedge \ldots \wedge \ov{b}_{m-n}=\det (\sigma )\cdot b_1\wedge \ldots \wedge b_{m-n},$$
where $\det(\sigma)$ is the usual determinant of the  linear map $\sigma$. 
The associated basis $\ov{c}_1,\ldots ,\ov{c}_{m-n}$ of $\beta (B)$  is then defined by 
$\ov{c}_i =\beta_V(\ov{b}_i)$. The two basis 
$c_i$ and $\ov{c}_i$ of $\beta (B)$ are therefore related by the isomorphism
$$
 \beta_V\circ \sigma\circ \beta_V^{-1}\colon \beta(B)\rightarrow\beta(B)
 $$
and consequently,
$$
\ov{c}_1\wedge\ldots \wedge \ov{c}_{m-n} =\det( \beta_V\circ \sigma\circ \beta_V^{-1})c_1\wedge\ldots \wedge c_{m-n}=\det(\sigma)c_1\wedge\ldots \wedge c_{m-n}.
$$
From this we obtain
\begin{equation*}
\begin{split}
&(\alpha(a_1)\wedge\ldots \wedge\alpha(a_n)\wedge\bar{b}_1\wedge\ldots \wedge \ov{b}_{m-n})\otimes ( c_1'\wedge\ldots \wedge c_l'\wedge \ov{c}_1\wedge\ldots \wedge \ov{c}_{m-n})^\ast\\
&\quad \quad =\det (\sigma) (\alpha(a_1)\wedge\ldots \wedge\alpha(a_n)\wedge b_1\wedge\ldots \wedge b_{m-n})\\
&\phantom{\quad \quad =}\otimes \dfrac{1}{\det (\sigma)}( c_1'\wedge\ldots \wedge c_l'\wedge c_1\wedge\ldots \wedge c_{m-n})^\ast,
\end{split}
\end{equation*}
so that the basis change has indeed no influence. 
}

{
Next we replace the basis  $a_1,\ldots ,a_n$ of $A$  by the basis $\ov{a}_1,\ldots ,\ov{a}_n$ by  means of the linear isomorphism  $\sigma\colon A\rightarrow A$ and replace  the basis $d_1,\ldots ,d_l$ of $D$ 
by the basis $\bar{d}_1,\ldots ,\bar{d}_l$  by means of the isomorphism $\varepsilon\colon D\rightarrow D$. Then 
$$(\ov{a}_1\wedge \ldots \wedge \ov{a}_n)\wedge (\ov{d}_1\wedge \ldots \wedge \ov{d}_l)^\ast=
\dfrac{\det (\tau)}{\det (\varepsilon)}\cdot 
(a_1\wedge \ldots \wedge a_n)\wedge (d_1\wedge \ldots \wedge d_l)^\ast,$$
and the vector 
$(a_1\wedge \ldots \wedge a_n)\wedge (d_1\wedge \ldots \wedge d_l)^\ast$ is independent of the choices of the basis if and only if $\det (\tau)=\det (\varepsilon)$. So far, the definition of $\Phi_{\bf E}$ could  only depend on the choice 
 of the complements $(V,W)$.
}

{
So, we  assume that $(V,W)$ and $(V',W')$ are two choices. We assume that $a_1,\ldots ,a_n$ and $d_1,\ldots ,d_l$ are fixed bases for $A$ and $D$, respectively.
Now fixing  any basis $b_1,\ldots ,b_{m-n}$ of  $V$, we find vectors $q_1,\ldots ,q_{m-n}$ of  $\alpha (A)$ such  that $\ov{b}_1,\ldots ,\ov{b}_{m-n}$ with $\ov{b}_i=b_i+q_i$ is a basis of  $V'$. We define $c_i\in W$ by  $\gamma_W(c_i)=d_i$. Then we choose  $p_i\in\beta(C)$ so that $\ov{c}_i=c_i +p_i\in W'$
and $\gamma_{W'}(\ov{c}_i)=d_i$. Now we use the fact that for  $(V,W)$ and $(V',W')$ other choices do not matter.
Using that $q_j$ is a linear combination of 
$\alpha (a_1),\ldots ,\alpha (a_n)$ and $p_j$ a linear combination of $c_1',\ldots ,c_l'$, it follows that 
\begin{equation*}
\begin{split}
&(\alpha(a_1)\wedge\ldots \wedge\alpha(a_n)\wedge b_1\wedge\ldots \wedge b_{m-n})\otimes ( c_1'\wedge\ldots \wedge c_l'\wedge c_1\wedge\ldots \wedge c_{m-n})^\ast\\
&\quad =(\alpha(a_1)\wedge\ldots \wedge\alpha(a_n)\wedge\ov{b}_1\wedge\ldots \wedge \ov{b}_{m-n})\otimes ( c_1'\wedge\ldots \wedge c_l'\wedge \ov{c}_1\wedge\ldots \wedge \ov{c}_{m-n})^\ast,
\end{split}
\end{equation*}
which completes the proof that $\Phi_{\bf E}$ is well-defined.
}
\end{proof}

\subsubsection{Proof of Proposition \ref{ooo6.6}}\label{oooo6.6}
\begin{proof}[Proof of Proposition \ref{ooo6.6}]

{
Let $P, Q\in\Pi_T$, $P\leq Q$,  and recall the exact sequence \begin{equation*}
{\bf E}_{(T,Q)}\colon\quad 0\rightarrow\ker(T)\xrightarrow{j^Q_T}\ker(QT)\xrightarrow{\Phi^Q_T} F/R(Q)\xrightarrow{\pi^Q_T}  \text{coker}(T)\rightarrow 0.
\end{equation*}
By definition of the determinant,  $h\in \det(T)$ is of  the form
$$
h=(a_1\wedge,\ldots ,\wedge a_n)\otimes ((d_1+R(T))\wedge\ldots \wedge (d_l+R(T)))^\ast,
$$
where $a_1,\ldots ,a_n$ is a basis of  $\ker(T)$ and $d_1+R(T),\ldots , d_l+R(T)$ is a basis of $F/R(T)$ satisfying $d_i\in R(I-Q)$.
The latter condition can be achieved since, by definition of $\Pi_T$,  $R(QT)=R(Q)$.
We choose  the  vectors $b_1,\ldots ,b_m$ in $\ker(QT)$
so that $a_1,\ldots ,a_n,b_1,\ldots ,b_m$ form a basis of  $\ker(QT)$ and take the linearly independent vectors  $T(b_1)+R(Q),\ldots ,T(b_m)+R(Q)$ in $F/R(Q)$.
By definition of the isomorphism $\gamma^Q_T\colon \det T\to \det (QT)$, 
\begin{equation*}
\begin{split}
&\gamma^Q_T(h)\\
&\quad =(a_1\wedge\ldots \wedge a_n\wedge b_1\wedge\ldots \wedge b_m)\otimes\\
&\quad ((d_1+R(Q))\wedge\ldots \wedge (d_l+R(Q))\wedge(T(b_1)+R(Q))\wedge\ldots \wedge (T(b_m)+R(Q)))^\ast .
\end{split}
\end{equation*}
We note that 
\begin{equation} 
T(a_1)=\ldots =T(a_n)=0\ \text{and}\ \ QT(b_1)=\ldots =QT(b_m)=0.
\end{equation}
Next we consider the exact sequence 
\begin{equation*}
0\rightarrow\ker(QT)\xrightarrow{j_{QT}^P}\ker(PT)\xrightarrow{\Phi^P_{QT}} F/R(P)\xrightarrow{\pi_{QT}^P}\text{coker}(QT)\rightarrow 0.
\end{equation*}
In order to  compute $\gamma^P_{QT}(\gamma^Q_T(h))$ we take as basis of  $\ker(QT)$ the vectors
$a_1,\ldots ,a_n,b_1,\ldots ,b_m$,  and as basis of $F/R(Q)$ the vectors $d_1+R(Q),\ldots ,d_l +R(Q),T(b_1)+R(Q),\ldots ,T(b_m)+R(Q)$.
Since  $\ker(QT)\subset\ker(PT)$,  we choose the vectors  $\wt{b}_1,\ldots ,\wt{b}_k$  in $\ker (PT)$ so that
$a_1,\ldots ,a_n,b_1,\ldots ,b_m,\wt{b}_1,\ldots ,\wt{b}_k$ form  a basis of  $\ker(PT)$.  We define 
$$
\alpha^P=a_1\wedge\ldots \wedge a_n\wedge b_1\wedge \ldots \wedge b_m\wedge \wt{b}_1\wedge\ldots \wedge \wt{b}_k,
$$
and note  that $\pi^P_{QT}(d_i+R(P))=(I-P)d_i +R(Q)=d_i+R(Q)$.  The vectors 
\begin{gather*}
d_1+R(P),\ldots ,d_l+R(P),\quad T(b_1)+R(P),\ldots ,T(b_m)+R(P),\\
QT(\wt{b}_1)+R(P),\ldots ,QT(\wt{b}_k)+R(P)
\end{gather*}
are a basis 
of $F/R(P)$ and we abbreviate  their  wedge product by $\beta^P$, so that 
$$
\gamma^P_{QT}(\gamma^Q_T(h)) = \alpha^P\otimes {(\beta^P)}^\ast.
$$
Recall that $d_i=(I-Q)d_i$ and $(I-Q)T(b_i)=T(b_i)$ for $i=1,\ldots ,m$. Abbreviating 
$[d]=d+R(P)$, the projection  $Q\colon F\rightarrow F$ induces the  projection
$\wt{Q}\colon F/R(P)\rightarrow F/R(P)$,  defined by $[d]\mapsto [Qd]$. Then $(I-\wt{Q})[d]=[(I-Q)d]$,  and therefore   
$$
F/R(P)= \ker(\wt{Q})\oplus \ker(I-\wt{Q}).
$$
The vectors 
 \begin{gather*}
 [d_1],\ldots ,[d_l],  [T(b_1)],\ldots ,[T(b_m)]\in \ker(\wt{Q}),
[QT(\wt{b}_1)],\ldots ,[QT(\wt{b}_k)]\in\ker(I-\wt{Q})
\end{gather*}
form  a basis for $F/R(P)$. By the  standard properties of the wedge product, 
\begin{equation*}
\begin{split}
\beta^P&=[d_1]\wedge\ldots \wedge[d_l]\wedge [T(b_1)]\wedge\ldots \wedge[T(b_m)]\wedge[QT(\wt{b}_1)]\wedge\ldots \wedge [QT(\wt{b}_k)]\\
&=[d_1]\wedge\ldots \wedge[d_l]\wedge [T(b_1)]\wedge\ldots \wedge[T(b_m)]\wedge[T(\wt{b}_1)]\wedge\ldots \wedge [T(\wt{b}_k)].
\end{split}
\end{equation*}
Hence we arrive for the composition 
$\gamma^P_{QT}\circ \gamma^Q_T$ at the formula
$$
\gamma^P_{QT}(\gamma^Q_T(h)) = \alpha^P\otimes ([d_1]\wedge\ldots \wedge[d_l]\wedge [T(b_1)]\wedge\ldots \wedge[T(b_m)]\wedge[T(\wt{b}_1)]\wedge\ldots \wedge [T(\wt{b}_k)])^\ast.
$$
Next we compute $\gamma_T^P(h)$. In order to do so we start with the basis $a_1,\ldots ,a_n$ of  $\ker(T)$
and extend it to a basis of  $\ker(PT)$  by choosing   
 $b_1,\ldots ,b_m,\wt{b}_1,\ldots ,\wt{b}_k$. The wedge of all these vectors  is $\alpha^q$.
For $F/R(T)$ we have the basis 
$d_1+R(T),\ldots ,d_l +R(T)$.   Recall that $(I-Q)d_i=d_i$ so that $d_i\in \ker(\wt{Q})$. Further, 
\begin{equation*}
\begin{split}
\pi^P_T(d_i +R(P))&=(I-P) d_i+R(T)\\
&=(I-P)(I-Q)d_i+R(T)\\
&=(I-Q)d_i+R(T)\\
&=d_i+R(T).
\end{split}
\end{equation*}
Then we take  the basis of $F/R(P)$ formed  (with the previous convention) by the vectors
$$
[d_1],\ldots ,[d_l], [T(b_1)],\ldots ,[T(b_m)],[T(\wt{b}_1)],\ldots ,[T(\wt{b}_k)],
$$
and note that their  wedge product is equal to $\beta^P$. By definition,
$$
\gamma^P_T(h)=\alpha^P\otimes {(\beta^P)}^\ast,
$$
which agrees with  
$\gamma^P_{QT}\circ \gamma^Q_T(h)$.  The proof  of Proposition \ref{ooo6.6} complete.
}
\end{proof}

\bibliographystyle{amsalpha}

\printindex
\end{document}